\documentclass[nonblindrev,opre]{informs3} 

\SingleSpacedXI

\usepackage{xfrac}  

\usepackage{natbib}
 \bibpunct[, ]{(}{)}{,}{a}{}{,}%
 \def\bibfont{\small}%
 \def\bibhang{24pt}%

 \usepackage[caption=false]{subfig}
 
 \usepackage{float}
\newfloat{algorithm}{t}{lop}
\floatname{algorithm}{Algorithm}

\TheoremsNumberedThrough     
\ECRepeatTheorems

\EquationsNumberedThrough    

\MANUSCRIPTNO{}

\usepackage{afterpage}

\usepackage[normalem]{ulem}
\newcommand\redsout{\bgroup\markoverwith{\textcolor{red}{\rule[0.5ex]{2pt}{0.4pt}}}\ULon}
\newcommand\gsout{\bgroup\markoverwith{\textcolor{green}{\rule[0.5ex]{2pt}{0.4pt}}}\ULon}

\usepackage{float}

\usepackage{array}
\newcommand{\PreserveBackslash}[1]{\let\temp=\\#1\let\\=\temp}
\newcolumntype{C}[1]{>{\PreserveBackslash\centering}p{#1}}
\newcolumntype{R}[1]{>{\PreserveBackslash\raggedleft}p{#1}}
\newcolumntype{L}[1]{>{\PreserveBackslash\raggedright}p{#1}}

\RequirePackage[shortlabels]{enumitem}
\usepackage{mathtools} 
\usepackage{enumitem} 
\usepackage{dsfont} 
\usepackage{mathrsfs}  
\usepackage{verbatim}
\usepackage{romannum}
\AtBeginDocument{\pagenumbering{arabic}}

\newenvironment{itemize*}%
  {\begin{itemize}%
    \setlength{\itemsep}{0.5em}%
    \setlength{\parskip}{0pt}}%
  {\end{itemize}}

\newcommand{\Prb}{\mathbb{P}}  
\newcommand{\Exp}{\mathbb{E}}

\let\N\Natural

\let\R\Real

\usepackage{graphicx}
\usepackage{bbm}
\usepackage{tikz,ifthen}
\usepackage{fixltx2e}    
   \usetikzlibrary{math}
   \usetikzlibrary{shapes.misc}
   \usetikzlibrary{shapes.multipart,positioning,patterns,backgrounds}

\usepackage{multirow}
\usepackage{array}

\newcommand{\ie}{{\it{i.e.}}, }
\newcommand{\eg}{{\it{e.g.}}, }

\makeatletter
\let\save@mathaccent\mathaccent
\newcommand*\if@single[3]{%
  \setbox0\hbox{${\mathaccent"0362{#1}}^H$}%
  \setbox2\hbox{${\mathaccent"0362{\kern0pt#1}}^H$}%
  \ifdim\ht0=\ht2 #3\else #2\fi
  }
\newcommand*\rel@kern[1]{\kern#1\dimexpr\macc@kerna}
\newcommand*\widebar[1]{\@ifnextchar^{{\wide@bar{#1}{0}}}{\wide@bar{#1}{1}}}
\newcommand*\wide@bar[2]{\if@single{#1}{\wide@bar@{#1}{#2}{1}}{\wide@bar@{#1}{#2}{2}}}
\newcommand*\wide@bar@[3]{%
  \begingroup
  \def\mathaccent##1##2{%
    \let\mathaccent\save@mathaccent
    \if#32 \let\macc@nucleus\first@char \fi
    \setbox\z@\hbox{$\macc@style{\macc@nucleus}_{}$}%
    \setbox\tw@\hbox{$\macc@style{\macc@nucleus}{}_{}$}%
    \dimen@\wd\tw@
    \advance\dimen@-\wd\z@
    \divide\dimen@ 3
    \@tempdima\wd\tw@
    \advance\@tempdima-\scriptspace
    \divide\@tempdima 10
    \advance\dimen@-\@tempdima
    \ifdim\dimen@>\z@ \dimen@0pt\fi
    \rel@kern{0.6}\kern-\dimen@
    \if#31
      \overline{\rel@kern{-0.6}\kern\dimen@\macc@nucleus\rel@kern{0.4}\kern\dimen@}%
      \advance\dimen@0.4\dimexpr\macc@kerna
      \let\final@kern#2%
      \ifdim\dimen@<\z@ \let\final@kern1\fi
      \if\final@kern1 \kern-\dimen@\fi
    \else
      \overline{\rel@kern{-0.6}\kern\dimen@#1}%
    \fi
  }%
  \macc@depth\@ne
  \let\math@bgroup\@empty \let\math@egroup\macc@set@skewchar
  \mathsurround\z@ \frozen@everymath{\mathgroup\macc@group\relax}%
  \macc@set@skewchar\relax
  \let\mathaccentV\macc@nested@a
  \if#31
    \macc@nested@a\relax111{#1}%
  \else
    \def\gobble@till@marker##1\endmarker{}%
    \futurelet\first@char\gobble@till@marker#1\endmarker
    \ifcat\noexpand\first@char A\else
      \def\first@char{}%
    \fi
    \macc@nested@a\relax111{\first@char}%
  \fi
  \endgroup
}
\makeatother

\begin{document}

\defcitealias{bertsimasdata}{BSS22}
 \RUNAUTHOR{Sturt}


\RUNTITLE{A nonparametric algorithm for optimal stopping based on robust optimization}

\TITLE{A nonparametric algorithm for optimal stopping based on robust optimization}

\ARTICLEAUTHORS{%
\AUTHOR{Bradley Sturt}
\AFF{Department of Information and Decision Sciences\\
University of Illinois at Chicago, \EMAIL{bsturt@uic.edu}}
} 

\ABSTRACT{%
{
{\color{black}Optimal stopping is a fundamental class of stochastic dynamic optimization problems with numerous applications in finance and operations management. 
We introduce a new approach for solving computationally-demanding stochastic optimal stopping problems with known probability distributions.  
The approach uses  {simulation} to construct a {robust optimization} problem that approximates the stochastic optimal stopping problem to any arbitrary accuracy; we then solve the robust optimization problem to obtain near-optimal Markovian stopping rules for the stochastic optimal stopping problem.  
\vspace{0.5em}

 In this paper, we focus on designing algorithms for solving the robust optimization problems that approximate the stochastic optimal stopping problems. 
 These robust optimization problems are challenging to solve because they require optimizing over the infinite-dimensional  space of all Markovian stopping rules.
We overcome this challenge by characterizing the structure of optimal Markovian stopping rules for the robust optimization problems.
In particular,  we show that optimal Markovian stopping rules for the robust optimization problems have a structure that is {surprisingly simple} and {finite-dimensional}.  
We leverage this structure  to develop an exact reformulation of the robust optimization problem as a  zero-one bilinear program over totally unimodular constraints. 
 We show that the bilinear program can be solved in polynomial time in special cases, establish computational complexity results for general cases, and develop polynomial-time heuristics by relating the bilinear program to the maximal closure problem from graph theory. 
 Numerical experiments demonstrate that our algorithms for solving the robust optimization problems are practical and can outperform 
 state-of-the-art simulation-based algorithms in the context of widely-studied stochastic optimal stopping problems from high-dimensional option pricing.
}
 }
}%


\HISTORY{{\color{black}First version: March 4, 2021. {\color{black}Revisions submitted on June 17, 2022 and  December 21, 2022}. Accepted for publication on March 16, 2023.  }}

\maketitle




\section{Introduction} \label{sec:introduction}

Consider the following class of stochastic dynamic optimization problems{\color{black}:} A sequence of random states 
are incrementally revealed to a decision maker. After observing the state in each period, the decision maker chooses whether to continue to the next period or stop and receive a reward that depends on the current state. 
The problem is to find a control policy, called a {stopping rule}, for selecting when to stop the process to maximize the expected reward. 

Such {optimal stopping} problems are widely studied and arise in a variety of domains like finance, promotion planning \citep{feng1995optimal}, and organ transplantation \citep{david1985time}. 
In particular, optimal stopping has considerable importance to industry for the pricing of financial derivatives. 
With a record trading volume that exceeded seven billion contracts in 2020, equity options are among the most widely-traded type of financial derivative 
 \citep{reuters2021}, 
and financial  firms depend on solving optimal stopping problems to determine accurate prices for American-style options, the most common type of equity option.

In this paper, we study a general class of optimal stopping problems in which the sequence of random states is driven by a {non-Markovian} probability distribution. {We recall that a sequence of random states is non-Markovian if the state in the next time period (\eg a stock's price tomorrow) has a probability distribution which depends both on the state in the current time period (\eg the stock's price today) as well as the states in the past periods (\eg the stock's price yesterday). This  class of optimal stopping problems has witnessed a surge of interest as financial firms increasingly use non-Markovian probability distributions to accurately model the volatility patterns of stocks \citep{gatheral2018volatility,leao2019discrete,becker2019deep,bezerra2020discrete,goudenege2020machine,bayer2020pricing}. 
Optimal stopping problems with non-Markovian probability distributions  also occur when using  popular dimensionality-reduction techniques for pricing high-dimensional basket options \citep[p. 372]{bayer2019implied} and pricing options when the probability distribution of underlying assets is accessed  via a black-box simulator constructed from historical data \citep[\S 5.5]{ciocan2020interpretable}. 

Despite their importance in practice, non-Markovian optimal stopping (NMOS) problems are ``not easy to solve" \cite[p. 982]{leao2019discrete}.   
The difficulty of these problems arises because the optimal decision in each period may depend on the entire history of the state process. 
In principle, an NMOS problem can be transformed into an equivalent Markovian optimal stopping problem by converting the original sequence of random states $x_1,\ldots,x_T \in \mathcal{X}$ into a Markovian stochastic process $
X_1,\ldots,X_T \in \mathcal{X}^T, 
$
where the new state $X_t \coloneqq (x_1,\ldots,x_t,0,\ldots,0) \in \mathcal{X}^T$ in each period $t$ includes the entire state history of the original process. Unfortunately, the enlarged state space $\mathcal{X}^T$ will be high-dimensional when the optimal stopping problem has many periods, and the difficulty of solving a Markovian optimal stopping problem  explodes in the dimensionality of the state space.
To contend with the curse-of-dimensionality that arises in NMOS problems, a natural approximation technique is to search only for stopping rules that are {Markovian}. 
Rather than depending on the entire history of the original sequence, a Markovian stopping rule makes a decision in each period $t$ based only on the current original state $x_t$ and knowledge that the sequence of random states was not stopped in any of the previous time periods. 
In general, the best Markovian stopping rule for an NMOS problem is not guaranteed to be an optimal stopping rule for the NMOS problem. 
However, recent numerical evidence demonstrates that Markovian stopping rules can lead to highly accurate approximations of optimal stopping rules in   NMOS problems from options pricing; see 
\citet[\S9.2]{goudenege2020machine}, \citet[\S 5.5]{ciocan2020interpretable}, \citet[\S 3.5]{bayer2020pricing}.





As far as we are aware, only two papers until now 
have suggested methods that are theoretically capable  of finding the best Markovian stopping rules to NMOS problems. 
 \cite{belomestny2011rates}  analyzes simulation-based methods which optimize directly over spaces of Markovian stopping rules and suggests a nonparametric space of Markovian stopping rules  based on  $\epsilon$-nets; however, he does not propose any concrete algorithms for optimizing over this  nonparametric space of Markovian stopping rules. \cite{ciocan2020interpretable} propose optimizing over Markovian stopping rules which are restricted to decision trees with fixed depth
. Due to the computational intractability of optimizing over all decision trees of fixed depth, the authors develop greedy heuristics which are shown to limit the range of attainable decision trees, and overcoming this limitation of the heuristics ``is not obvious, especially in light of the structure of the optimal stopping problem that is leveraged to efficiently optimize split points in our construction" \cite[p. 22]{ciocan2020interpretable}.


We take a different approach to the aforementioned literature, and in doing so make our contributions to optimal stopping, by drawing on the traditionally unrelated field of \emph{robust optimization}. 
Over the past two decades, robust optimization  has emerged as  a leading tool in operations research for dynamic decision-making when uncertainty is driven by {unknown} or {ambiguous} probability distributions \citep{ben2009robust,delage2015robust}. 
In this paper, we show  that robust optimization can be combined with {simulation} to develop algorithms for finding Markovian stopping rules to NMOS problems with known probability distributions. 
 Compared to \cite{belomestny2011rates} and \cite{ciocan2020interpretable}, our approach for optimal stopping does not restrict the space of Markovian stopping rules to {any} parametric class, and we develop concrete algorithms that are guaranteed to yield $\epsilon$-optimal Markovian stopping rules for general classes of NMOS problems.  
 
{\color{black}

{\color{black}





 In greater detail, this paper introduces a new approach for computing Markovian stopping rules for NMOS problems with known probability distributions that is based on a combination of simulation and robust optimization.  
At a high level, our approach is  comprised of the following steps:

\begin{description}
\item[{Step 1}:] We use Monte-Carlo simulation to generate sample paths of the sequence of random states. 
\item[Step 2:] From those sample paths, we 
 construct a robust optimization problem {\color{black}that approximates the NMOS problem}. 
\item[Step 3:] We solve the robust optimization problem to obtain stopping rules for the NMOS problem. 
\end{description}
The robust optimization problem constructed in Step 2 can be interpreted as a \emph{proxy} or \emph{surrogate} for the NMOS problem. 
Indeed, if the number of simulated sample paths in Step 1 is sufficiently large, then every Markovian stopping rule that is optimal for the robust optimization problem constructed in Step 2 is guaranteed with high probability 
to be an $\epsilon$-optimal Markovian stopping rule for the NMOS problem; see \S\ref{sec:opt}. Thus, the approach comprised of the above steps enables  the task of computing near-optimal Markovian stopping rules for an NMOS problem to be reduced to  the task of solving a robust optimization problem.

 With respect to the robust optimization literature, Step 2 of our approach  follows   a recent paper by \citet*[henceforth referred to as \citetalias{bertsimasdata}]{bertsimasdata}. 
In that paper, the authors showed that a general class of stochastic dynamic optimization problems with unknown probability distributions can in principle be approximated to arbitrary accuracy by a robust optimization problem constructed from historical data. In particular, our work draws on \citetalias{bertsimasdata} in two specific ways. 
First, our robust optimization problem constructed in Step 2 is a variant of a robust optimization formulation that is proposed in  \citetalias{bertsimasdata}. Second, by focusing on the application of optimal stopping, we strengthen  theoretical developments from \citetalias{bertsimasdata} to prove under relatively {mild} and {verifiable} assumptions that the optimal objective value and optimal Markovian stopping rules of the robust optimization problem from Step 2 converge almost surely to those of the NMOS problem as the number of simulated sample paths in Step 1 grows to infinity (Theorems~\ref{thm:conv:asympt}-\ref{thm:conv:ub} in \S\ref{sec:opt}). {\color{black}For the interested reader, a discussion of our improvements to the theoretical convergence guarantees from \citetalias{bertsimasdata} is provided at the beginning of Appendix~\ref{appx:proof_conv}. }

In contrast to \citetalias{bertsimasdata}, the key novelty of the present paper lies \emph{not} in showing that robust optimization can be used to construct approximations of stochastic dynamic optimization problems; rather, the main contributions of the present paper are \emph{algorithmic}. 
 Indeed, in order for a combination of simulation and robust optimization to yield near-optimal algorithms for a class of stochastic dynamic optimization problems with known probability distributions, it is not sufficient to construct a robust optimization problem that approximates the stochastic problem to arbitrary accuracy: one must also have exact or provably near-optimal algorithms for solving the robust optimization problem.  
  However, up to this point, there have been no exact algorithms in the literature for solving the type of robust optimization problems proposed by \citetalias{bertsimasdata} in any  class of stochastic dynamic optimization problems where uncertainty unfolds over two or more periods.  In particular, because solving the type of robust optimization problems from \citetalias{bertsimasdata} requires optimizing over {infinite-dimensional}  spaces of  control policies, it has been unknown in any application whether optimal  control policies for these robust optimization problems even \emph{exist}, let alone whether they can ever be tractably computed. These  hurdles have, practically speaking, prevented the robust optimization techniques from  \citetalias{bertsimasdata} from being combined with simulation to develop algorithms for computing near-optimal control policies for any class of stochastic dynamic optimization problems with known probability distributions until now.

In this paper, we resolve the  aforementioned gaps in the robust optimization literature by establishing the first characterization of optimal control policies for the type of robust optimization problems proposed by  \citetalias{bertsimasdata} in an application where uncertainty unfolds over two or more periods (Theorem~\ref{thm:characterization} in \S\ref{sec:main}). 
Specifically, we consider the task of solving the robust optimization problems constructed in Step 2 over the infinite-dimensional space of all Markovian stopping rules.
Using a novel pruning technique, we prove  in Theorem~\ref{thm:characterization} that optimal Markovian stopping rules for these robust optimization problems not only {exist}, but also have a structure that is  simple and {finite-dimensional}. In fact, 
our characterization reveals that optimal Markovian stopping rules for these robust optimization problems 
 can be compactly parameterized  by integer variables $\sigma^1,\ldots,\sigma^N \in \{1,\ldots,T\}$, 
where $N$ is the number of simulated sample paths chosen in Step 1 and $T$ is the number of periods in the NMOS problem. 


Leveraging the structure of optimal Markovian stopping rules for the robust optimization problems  constructed in Step 2, we develop exact and heuristic algorithms for solving  the robust optimization problem in Step 3. 
Specifically, we make the following algorithmic contributions: 
\begin{enumerate}[(a)]
\item We use our characterization of optimal Markovian stopping rules for the robust optimization problems  constructed in Step 2 to transform the robust optimization problem into a finite-dimensional optimization problem  over integers $\sigma^1,\ldots,\sigma^N \in \{1,\ldots,T\}$ (Theorem~\ref{thm:reform} in \S\ref{sec:algorithms:prelim}).
 \item For NMOS problems with two periods, we show that optimal Markovian stopping rules for the robust optimization problem can be computed in $\mathcal{O}(N^3)$ time (Theorem~\ref{thm:twoperiods} in \S\ref{sec:exact}).

 \item For NMOS problems with three or more periods, we prove that finding optimal Markovian stopping rules for the robust optimization problem is   NP-hard (Theorem~\ref{thm:hard} in \S\ref{sec:exact}).

\item We develop a general-purpose {exact algorithm} for computing optimal Markovian stopping rules for the robust optimization problem that consists of solving a zero-one bilinear program over totally unimodular constraints (Theorem~\ref{thm:bp_main} in \S\ref{sec:exact}).

\item  We design and analyze a polynomial-time heuristic algorithm for approximately solving the robust optimization problem by relating the bilinear program to the maximal closure problem from graph theory  (Propositions~\ref{prop:similarity_of_formulations}-\ref{prop:lb_on_h2:N} in \S\ref{sec:approx}). 

 \end{enumerate}
 In summary, our exact and heuristic algorithms allow us to solve the robust optimization problem constructed in Step 2. Since the robust optimization problem constructed via Steps 1 and 2 serves as an approximation of the NMOS problem to any arbitrary accuracy, our algorithms for solving the robust optimization problem in turn allow us to compute near-optimal Markovian stopping rules for the NMOS problem. 

}}

We conclude with numerical experiments that demonstrate the value of our robust optimization-based algorithms in several settings. First, we consider a simple one-dimensional non-Markovian optimal stopping problem with fifty periods, and we compare the robust optimization algorithm to existing methods based on approximate dynamic programming \citep{longstaff2001valuing} and parametric stopping rules \citep{ciocan2020interpretable}. The experiments show that our method can find stopping rules that significantly outperform those found by the other techniques, while maintaining a comparable computational cost. 
 In particular, 
  the experiments reveal that our method can strictly outperform alternative algorithms for finding Markovian stopping rules to NMOS problems that are based on backwards recursion. 
  Second, we consider a widely-studied and important problem of pricing high-dimensional Bermudan barrier options with over fifty periods. 
Across several variants of this problem, we demonstrate that our combination of robust optimization and simulation can find stopping rules that match, and in some cases significantly outperform, those from state-of-the-art algorithms {\color{black}by \cite{longstaff2001valuing,ciocan2020interpretable} as well as the duality-based pathwise optimization method of  \cite{desai2012pathwise}.  }

The rest of our paper has the following organization.  \S\ref{sec:lit} provides a review of methods for solving optimal stopping problems with Markovian and non-Markovian probability distributions. 
\S\ref{sec:setting}  formalizes the problem setting and introduces our robust optimization-based method.  \S\ref{sec:main} characterizes the structure of optimal policies for the robust optimization problem. \S\ref{sec:algorithms} develops tractable algorithms and computational complexity results for the robust optimization problem. \S\ref{sec:experiments} illustrates the performance of our algorithms in numerical experiments.  
 {\color{black}Unless stated otherwise, all  technical proofs can be found in the appendices.  }


\subsection{Other Related Literature} \label{sec:lit}

Many methods based on approximate dynamic programming (ADP) have been developed for optimal stopping problems with high-dimensional Markovian stochastic processes. The most popular ADP methods for these optimal stopping problems are based on Monte-Carlo simulation and regression, which originate with 
\cite{carriere1996valuation,longstaff2001valuing} and \cite{tsitsiklis2001regression}. Given sample paths of the entire stochastic process, these methods use backwards recursion and regression to obtain approximations of the value function, and exercise policies are then obtained by proceeding greedily with respect to the approximate value functions.  
  The efficacy of regression-based methods hinges on selecting a parametrization of basis functions for the value function that strikes a balance between approximation quality and sample complexity.    Nonparametric choices for the basis functions, \eg Lagurerre polynomials, are discussed in the aforementioned works and subsequently analyzed in works such as \cite{clement2002analysis,glasserman2004number,egloff2005monte,belomestny2011pricing} and \cite{zanger2020general}. In \S\ref{sec:experiments}, we provide numerical   comparisons of our proposed algorithms to ADP techniques  in the context of NMOS problems.
  
A variety of other nonparametric methods have been developed for solving optimal stopping problems with Markovian probability distributions, such as quantization-based approximations of value functions \citep{bally2003quantization} and scenario tree discretizations of the sequence of random states \citep{broadie1997pricing}. Recent works have also considered using deep learning to learn the continuation function, including  \cite{becker2019deep} and \cite{goudenege2020machine}. {\color{black}The efficacy of deep learning algorithms for optimal stopping hinges on carefully selecting the topology of the neural network and choosing the right  tuning algorithm, and performing these tasks effectively  in  the context of optimal stopping is an ongoing area of research \citep{fathan2021deep}. 
}  
Methods to compute upper bounds on optimal stopping problems grew in interest due to the  independent works of \cite{haugh2004pricing} and \cite{rogers2002monte}, and duality-based algorithms to obtain upper bounds which combine simulation and suboptimal stopping rules were first proposed by  \cite{andersen2004primal}. Other works that harness dual representations to solve optimal stopping {and other stochastic dynamic optimization} problems include {\color{black}\cite{brown2010information}}, \cite{desai2012pathwise,belomestny2013solving}, and \cite{goldberg2018beating}, among many others.  {\color{black} In \S\ref{sec:experiments}, we provide numerical comparisons of our proposed algorithms to the duality-based pathwise optimization method of  \cite{desai2012pathwise}.}  
  

%


In the context of non-Markovian optimal stopping, methods have been developed which address settings that are different from ours in non-trivial ways.  \cite{leao2019discrete} and \cite{bezerra2020discrete} develop discretization schemes for  NMOS problems over continuous time and restrict the class of probability distributions to those based on the Brownian motion.  In contrast to these works, our paper develops algorithms for finding Markovian stopping rules for discrete-time optimal stopping problems, and we do not require any parametric assumptions on the probability distributions of the underlying stochastic processes.  NMOS problems can also be addressed by the scenario tree method of  \cite{broadie1997pricing} and the recursive-dual algorithm of \cite{goldberg2018beating}, provided that one can perform Monte-Carlo simulation on the conditional probability distribution of the stochastic process in each time period. In contrast to these methods, the algorithms in this paper require only the ability to simulate sample paths of the  entire stochastic process and are shown in numerical experiments to be practically tractable in low-dimensional NMOS problems with dozens of time periods. 
Within the optimal stopping literature, our method is most closely related to a class of simulation-based methods which optimize directly over spaces of  deterministic  stopping rules, as explored by \cite{garcia2003convergence,andersen1999simple,belomestny2011rates,gemmrich2012,ciocan2020interpretable}, and  \citet[\S 8.2]{glasserman2013monte}. We discuss  connections between our {\color{black}approach} and this stream of literature in \S\ref{sec:relatedwork}, and a discussion of the challenges of using dynamic programming to find the best Markovian stopping rules for NMOS problems can be found in Appendix~\ref{appx:dp_limit}. 

{\color{black}Finally, we note that prior research in operations research and economics have studied robust optimal stopping problems in which the goal is to find stopping rules that perform well under worst-case probability distributions \citep{bayraktar2014robust,riedel2009optimal} or under worst-case state trajectories \citep{iancu2021monitoring}. That stream of research differs significantly from ours, as that stream of research does not consider robust optimization problems that are approximations of stochastic optimal stopping problems with known probability distributions.  
  }

\section{Robust Optimization for Stochastic Optimal Stopping} \label{sec:setting}

\subsection{Problem Setting} \label{sec:setting:notation}
We consider stochastic optimal stopping problems defined by the following components:

\vspace{-1em}
\paragraph{\textbf{States:}} Let $x \equiv (x_1,\ldots,x_T)$ denote a sequence of random states, where the state $x_t \in \mathcal{X}\; {\color{black}\equiv}\; \R^d$ in each period $t$ is a random vector of dimension $d$.  For example, the state in each period may represent the prices of multiple assets at that point in time.  The joint probability distribution of this 
 stochastic process is  assumed to be known and accessible through a simulator which generates independent sample paths of the entire stochastic process. 
 
 \vspace{-1em}
\paragraph{\textbf{Policies:}} Let $\mu \equiv (\mu_1,\ldots,\mu_T)$ represent a collection of exercise policies, where the exercise policy in each period $t$ is a {\color{black}measurable} function of the form $\mu_t: \mathcal{X} \to \{\textsc{Stop}, \textsc{Continue} \}$. Speaking intuitively, each exercise policy is a partitioning of the state space into regions for stopping and continuing. 
From the exercise policies, the corresponding Markovian stopping rule $\tau_\mu: \mathcal{X}^T \to \{1,\ldots,T\} \cup \{\infty \}$ is a function that maps a realization of the stochastic process to {\color{black}a }stopping period:
\begin{align*}
\tau_\mu(x) \triangleq \min \{t \in \{1,\ldots, T\}:\; \mu_t(x_t) = \textsc{Stop} \}.
\end{align*}
 Throughout this paper, a minimization problem with no feasible solutions is defined equal to $\infty$.\footnote{{\color{black}We remark that the Markovian stopping rule $\tau_\mu$ is a non-anticipative control policy, meaning that the event $\{\tau_\mu(x) = t\}$ does not depend on the future states $x_{t+1},\ldots,x_T$ for each period $t \in \{1,\ldots,T\}$.  To see why this is the case, consider any two realizations of the sequence of random states $x \equiv (x_1,\ldots,x_T)$ and $x' \equiv (x'_1,\ldots,x'_T)$, and suppose for a given period $t \in \{1,\ldots,T\}$ that the two realizations satisfy $x_s = x_s'$ for all $s \in \{1,\ldots,t\}$. Then we readily observe from our definition of Markovian stopping rules and from algebra that the Markovian stopping rule $\tau_\mu$ satisfies $\tau_\mu(x) = t$ if and only if $\tau_\mu(x') = t$ is satisfied. }} 

\vspace{-1em}
\paragraph{\textbf{Rewards:}}  Let $g: \{1,\ldots,T\} \cup \{\infty \}\times \mathcal{X}^T \to \R_+$ be a known and deterministic function that maps a stopping period and a realization of the entire stochastic process to a reward. The assumption that the reward function is nonnegative is common in many applications of optimal stopping, and we assume throughout the paper that a stochastic process that is never stopped yields a reward of zero: $g(\infty,x) \equiv 0$. It follows from the definition of the reward function that the reward from stopping on any period $t$ may in general depend on the states of the stochastic process in previous or future time periods. 

\vspace{-1em}
\paragraph{\textbf{Problem:}}    With the above notation and  inputs, the goal of this paper is to solve stochastic optimal stopping problems of the form 
\begin{align} \label{prob:main} \tag{OPT}
\begin{aligned}
&\sup_{\mu} &&\Exp \left[g(\tau_\mu(x),x) \right],
\end{aligned}
\end{align}
where the optimization is taken over the space of all Markovian stopping rules. 
 In the following sections, we introduce and analyze a new simulation-based method for solving this class of stochastic dynamic optimization problems. 

\subsection{The Robust Optimization {\color{black}Approach}} \label{sec:rob}
Our proposed {\color{black}approach} for solving stochastic optimal stopping problems  of the form \eqref{prob:main} consists of the following steps. 
We first simulate sample paths of the stochastic process  $x \equiv (x_1,\ldots,x_T)$. Let $N$ denote the number of sample paths{\color{black},} and {\color{black}let} the values of the sample paths be denoted by $$x^i \equiv (x^i_1,\ldots,x^i_T) \text{ for each }i=1,\ldots,N.$$ 
 We assume that the sample paths are independent and identically distributed realizations of the entire (possibly non-Markovian) stochastic process. 
We next choose the following robustness parameter: $$\epsilon \ge 0.$$ The purpose of the robustness parameter will become clear momentarily, and a discussion on how to choose the number of sample paths and the robustness parameter is deferred until  \S\ref{sec:params}.  
With these parameters, let the uncertainty set around sample path $i$ on period $t$ be defined as
\begin{align*}
\mathcal{U}^i_t \triangleq \left \{ y_t \in \mathcal{X}: \; \| y_t - x_t^i \|_\infty \le \epsilon  \right \}.
\end{align*}
For notational convenience, denote the uncertainty set around sample path $i$ across all periods by
\begin{align*}
\mathcal{U}^i &\triangleq \mathcal{U}^i_1 \times \cdots \times \mathcal{U}^i_T.
\end{align*}
Hence, we observe that the role of the robustness parameter is to control the size of these sets. 
Given the sample paths and choice of the robustness parameter, 
our approach obtains an approximation of \eqref{prob:main} by solving the following robust optimization problem:
\begin{align} \label{prob:sro} \tag{RO}
\begin{aligned}
&\sup_{\mu} &&\frac{1}{N} \sum_{i=1}^N \inf_{y \in \mathcal{U}^i}  g(\tau_\mu(y), x^i).
\end{aligned}
\end{align}
By solving the above robust optimization problem, we obtain {\color{black}exercise} policies $\hat{\mu} \equiv (\hat{\mu}_1,\ldots,\hat{\mu}_T)$. These {\color{black}exercise} policies constitute our approximate solution to the stochastic optimal stopping problem \eqref{prob:main}. 
{\color{black}

\begin{remark}
 In the above formulation of the robust optimization problem \eqref{prob:sro}, the sample paths in the objective function are allowed to be perturbed by an adversary in the evaluation of the Markovian stopping rule, $\tau_\mu(y)$,   but not in the evaluation of the reward, $g(\cdot,x^i)$. This is a deviation from the robust optimization {\color{black}formulation} from   \citetalias{bertsimasdata}, presented below as \eqref{prob:sro_orig}, in which the worst-case reward over each uncertainty set has the form $ \inf_{y \in \mathcal{U}^i}  g(\tau_\mu(y), y)$: 
 \begin{align} \label{prob:sro_orig} \tag{RO'}
\begin{aligned}
&\sup_{\mu} &&\frac{1}{N} \sum_{i=1}^N \inf_{y \in \mathcal{U}^i}  g(\tau_\mu(y), y).
\end{aligned}
\end{align}
  The difference between the objective functions of \eqref{prob:sro} and \eqref{prob:sro_orig} turns out to be inconsequential from the perspective of establishing convergence guarantees (see \S\ref{sec:opt}) or characterizing the structure of optimal Markovian stopping rules (see \S\ref{sec:main}). However, \eqref{prob:sro} is significantly simpler from the perspective of algorithm design. 
For the interested reader, an extended discussion on the similarities and differences between   \eqref{prob:sro} and the alternative robust optimization formulation \eqref{prob:sro_orig} can be found in Appendix~\ref{appx:motivation_for_sro}.
\end{remark}
}


\subsection{Background and Motivation} \label{sec:relatedwork}
In contrast to traditional robust optimization or distributionally robust optimization, our motivation behind adding adversarial noise to the sample paths in \eqref{prob:sro} is {not} to find stopping rules which have worst-case performance guarantees, are attractive in risk-averse settings, or perform well the presence of an ambiguous probability distribution. Rather, this paper proposes using robust optimization purely as an {{algorithmic tool}} for solving stochastic optimal stopping problems  of the form \eqref{prob:main} when the joint probability distributions are known. The present section 
elaborates on this motivation and positions our use of robust optimization within the optimal stopping literature. 


 For the sake of developing intuition, let us suppose for the moment that the robustness parameter of the uncertainty sets in \eqref{prob:sro} was set equal to zero. In this case, for any fixed exercise policies $\mu = (\mu_1,\ldots,\mu_T)$, the expected reward of those exercise policies,
  $$J^*(\mu) \triangleq \Exp \left[g(\tau_\mu(x),x) \right],$$ would be approximated in \eqref{prob:sro} by the sample average approximation: 
 \begin{align*}
 \widehat{J}_{N,0}(\mu) \triangleq \frac{1}{N} \sum_{i=1}^N \inf_{y \in \mathcal{X}^T: \| y - x^i \|_\infty \le 0}  g(\tau_\mu(y), x^i) = \frac{1}{N} \sum_{i=1}^N g(\tau_\mu(x^i),x^i).
 \end{align*}
 For these fixed exercise policies, we observe that the sample average approximation is a \emph{consistent} estimator of the expected reward. In other words, for the fixed exercise policies $\mu$, it follows from the strong law of large numbers under relatively mild assumptions\footnote{For example, Assumptions~\ref{ass:lighttail} and \ref{ass:bound} in \S\ref{sec:opt}.} that $\widehat{J}_{N,0}(\mu)$ will converge almost surely to $J^*(\mu)$ as the number of simulated sample paths is taken to infinity.  

However, it is well known in the optimal stopping literature that these desirable asymptotic properties of $\widehat{J}_{N,0}(\mu)$ are generally not retained when considering the  problem of optimizing over the space of all exercise policies.
For instance, \citet[EC.1.2]{ciocan2020interpretable} provide simple examples in which the following statements hold almost surely: 
\begin{align}
\lim_{N \to \infty} \sup_\mu \widehat{J}_{N,0}(\mu) &\gg \sup_\mu {J^*}(\mu);\quad \lim_{N \to \infty} J^*(\hat{\mu}_{N,0}) \ll \sup_\mu {J^*}(\mu).  
\label{line:subopt}
\end{align}
The asymptotic suboptimality of the optimal objective value  and optimal policies for the problem $ \sup_\mu \widehat{J}_{N,0}(\mu)$ can be intuitively understood as a type of {overfitting}. 
To see why line~\eqref{line:subopt} occurs, we recall  for any fixed choice of exercise policies $\mu$ that the sample average approximation $\widehat{J}_{N,0}(\mu)$ is an unbiased estimate of the expected reward $J^*(\mu)$. However, when simultaneously considering the space of all exercise policies, there exists  for each $N \in \N$ with high probability  a collection of exercise policies that satisfies $\widehat{J}_{N,0}(\mu) \gg J^*(\mu)$. The  problem $\sup_{\mu} \widehat{J}_{N,0}(\mu)$ will thus be biased towards choosing those exercise policies, which in general will be suboptimal for the problem $\sup_\mu J^*(\mu)$. Because the set of all $\mu$ is an infinite-dimensional space, the gap between the objective values $\widehat{J}_{N,0}(\mu)$ and $J^*(\mu)$ does not converge to zero uniformly over the set of all $\mu$ as the number of sample paths tends to infinity. 

To circumvent this overfitting in the context of optimal stopping in line~\eqref{line:subopt}, a vast literature has focused on restricting the functional form of exercise policies to a finite-dimensional space, such as 
\cite{garcia2003convergence,andersen1999simple,belomestny2011rates,gemmrich2012,ciocan2020interpretable}. 
In this approach, the choice of the parameterization for the space of exercise policies must be made very carefully. 
On one hand, the effective dimension of the restricted space of exercise policies must be small relative to the number of simulated sample paths to ensure that the sample average approximation problem finds the parametric exercise policies that are `best-in-class' with respect to the stochastic optimal stopping problem \citep[\S 3]{belomestny2011rates}. 
On the other hand, the parameterization must be chosen appropriately in order for the sample average approximation problem to obtain a good approximation of \eqref{prob:main}. 
Choosing such an appropriate parameterization ``may be counterfactual in some cases", as explained by \citet[p. 1859]{garcia2003convergence},  
``since we may not have a good understanding of what the early exercise rule should depend on." 

Our approach, in view of the above discussion, provides an alternative means to  circumvent overfitting. The proposed robust optimization problem allows the space of exercise policies to remain {general}, and thus relieves the decision maker from the need to select and impose a parametric structure on the  exercise policies. 
Moreover, we show in the following section that our use of robust optimization provably overcomes the asymptotic overfitting  described in line~\eqref{line:subopt}.  %

\subsection{Optimality Guarantees} \label{sec:opt}
{\color{black}In this subsection, we establish theoretical justification for our combination of robust optimization and simulation that is presented in \S\ref{sec:rob}. Specifically, we strengthen  convergence guarantees from \citetalias{bertsimasdata} to the specific problem of optimal stopping to prove that the optimal objective value and optimal exercise policies of \eqref{prob:sro} converge almost surely to those of the stochastic optimal stopping problem \eqref{prob:main} under mild and verifiable conditions. Establishing these convergence guarantees in the context of optimal stopping is necessary to ensure that \eqref{prob:sro} will provide a high-quality approximation of the stochastic optimal stopping problem \eqref{prob:main} when the robustness parameter is sufficiently small and the number of simulated sample paths is sufficiently large. 
 } 

{\color{black}
To establish our theoretical results, we make four relatively mild assumptions on the stochastic optimal stopping problem \eqref{prob:main}. Our first assumption, denoted below by Assumption~\ref{ass:weird}, concerns the structure of the reward functions in the optimal stopping problem, and can be roughly interpreted as a requirement that the reward function changes continuously as a function of the states: 
\begin{assumption} \label{ass:weird}
 $\lim \limits_{\epsilon \to 0} \Delta_\epsilon(x) = 0 $ almost surely, where 
 \begin{align*}
 \Delta_\epsilon(x) \triangleq  \min_{t \in \{1,\ldots,T\}} \left \{ \inf \limits_{y \in \mathcal{X}^T: \| y - x \|_\infty \le \epsilon} g(t,y) - g(t,x)\right \}.
 \end{align*}
 \end{assumption}
\noindent From a practical standpoint, it is easy to see that Assumption~\ref{ass:weird} holds whenever the functions $g(1, \cdot),\ldots,g(T,\cdot)$ are continuous, and it can also hold in  important  stochastic optimal stopping problems with discontinuous reward functions.\footnote{To illustrate, consider Robbin's problem \citep{bruss2005known}, in which the reward functions $g(t,x) = \text{rank}(x_t ; x_1,\ldots,x_T)$ are discontinuous and the probability distribution is $x_1,\ldots,x_T \overset{\text{iid}}{\sim} \text{Uniform}[0,1]$. To show that Assumption~\ref{ass:weird} is satisfied, we observe that the random variable $\bar{\epsilon} \triangleq \min_{ s < t} |x_s - x_t|$ is strictly positive with probability one, which implies that $\Delta_\epsilon(x) = 0$ for all $\epsilon < \bar{\epsilon}$. } 
 
 Our second and third assumptions concern the structure of the probability distribution in \eqref{prob:main}. Specifically,  Assumption~\ref{ass:lighttail} enforces that the stochastic process has a light tail, and Assumption~\ref{ass:continuous} says the stochastic process is drawn from a continuous probability distribution. 
 {\color{black}
   \begin{assumption}\label{ass:lighttail}
 The stochastic process satisfies $\Exp[\textnormal{exp}(\| x \|_\infty^a)] < \infty$ for some $a > 1$. 
 \end{assumption}}
   \begin{assumption}\label{ass:continuous}
The stochastic process $x \equiv (x_1,\ldots,x_T)$ has a probability density function. 
 \end{assumption}
  Let us reflect on the practical restrictiveness of these two assumptions. The second assumption, 
Assumption~\ref{ass:lighttail}, is a standard light-tail assumption on the stochastic process which is satisfied, for example, if the stochastic process is bounded or has a multivariate normal distribution. This assumption greatly simplifies our analysis, as it allows us to invoke a convergence result by \citetalias{bertsimasdata} in our proofs (see Appendix~\ref{appx:proof_conv}).  We impose the third assumption, Assumption~\ref{ass:continuous},   to ensure that there exist arbitrarily near-optimal Markovian stopping rules for \eqref{prob:main} that satisfy a certain technical continuity structure that we can exploit in our proof. As far as we can tell, these assumptions on the probability distribution are relatively mild and routinely satisfied in applications of optimal stopping in the context of the options pricing literature.  
Nonetheless, we do not preclude the possibility that these assumptions on the probability distribution can be weakened while still establishing convergence guarantees.

Our fourth and final assumption imposes boundedness on the reward function. This assumption, presented below as Assumption~\ref{ass:bound}, leads to a considerably simpler proof and statement of the results, but can generally be relaxed to reward functions bounded above by an integrable, Lipschitz-continuous function. 
\begin{assumption} \label{ass:bound}
 The reward function satisfies $0 \le g(t,y) \le U$ for all $t \in \{1,\ldots,T\}$ and $y \in \mathcal{X}^T$. 
 \end{assumption}

We emphasize that each of the aforementioned four assumptions  on the stochastic optimal stopping problem~\eqref{prob:main} can be verified \emph{a priori}. In particular, they do not require any knowledge of the structure of optimal exercise policies for the stochastic optimal stopping problem~\eqref{prob:main}. As a result, each of these assumptions can be verified using the information typically available in practice. We note that these assumptions are considerably weaker than those in \citetalias{bertsimasdata}, which require knowledge of the structure of optimal control policies to establish convergence results. 
 }

Under the above conditions, the following theorems 
provide justification for using the robust optimization problem as a proxy for the stochastic optimal stopping problem.  In a nutshell, the following Theorems~\ref{thm:conv:asympt}-\ref{thm:conv:ub} show that \eqref{prob:sro} will, for all sufficiently small choices of the robustness parameter and all sufficiently large choices of the number of simulated sample paths,  yield a near-optimal approximation of \eqref{prob:main}. Stated another way, the following theorems show that our use of robust optimization provably overcomes the asymptotic overfitting  described in line~\eqref{line:subopt} {\color{black}of \S\ref{sec:relatedwork}}. While the following theorems do not specify how to choose the robustness parameter and number of simulated sample paths for any particular optimal stopping problem, we provide guidance (\S \ref{sec:params}) and numerical evidence (\S \ref{sec:experiments}) which suggest that these parameters can be found effectively in practice. {\color{black}A discussion of the technical innovations as well the proofs of the following theorems in this subsection can be found in Appendix~\ref{appx:proof_conv}.}

Our first theorem shows that the optimal objective value of the robust optimization problem~\eqref{prob:sro} will converge almost surely to that of the stochastic problem \eqref{prob:main} as the robustness parameter tends to zero and the number of sample paths tends to infinity.  In the following result, we use the  notation $\widehat{J}_{N,\epsilon}(\mu)$  to denote the objective value of the robust optimization problem~\eqref{prob:sro} corresponding to exercise policies $\mu$. 

 \begin{theorem}[Consistency of optimal objective value] \label{thm:conv:asympt}
Under {\color{black}Assumptions~\ref{ass:weird}, \ref{ass:lighttail}, \ref{ass:continuous}, and  \ref{ass:bound}}, 
$$
 \lim \limits_{ \epsilon \to 0} \lim \limits_{N \to \infty} \sup \limits_{\mu} \widehat{J}_{N,\epsilon}(\mu) = \sup \limits_{\mu} J^*(\mu) \quad \textnormal{almost surely}.
$$ 
\end{theorem}

Our second theorem shows that the expected reward of the optimal exercise policies for the robust optimization problem~\eqref{prob:sro} will converge almost surely to the optimal objective value of the stochastic problem~\eqref{prob:main}.   We let $\hat{\mu}_{N,\epsilon}$ denote optimal exercise policies for \eqref{prob:sro}, and we remark that the existence of optimal exercise policies for the robust optimization problem will be established in \S\ref{sec:main}. 
\begin{theorem}[Consistency of optimal policies]  Under {\color{black}Assumptions~\ref{ass:weird}, \ref{ass:lighttail}, \ref{ass:continuous}, and  \ref{ass:bound}},  \label{thm:conv:policy}
\begin{align*}
\lim_{ \epsilon \to 0} \liminf_{N \to \infty} J^*(\hat{\mu}_{N,\epsilon})  = \lim_{ \epsilon \to 0} \limsup_{N \to \infty} J^*(\hat{\mu}_{N,\epsilon}) &=\sup_{\mu} J^*(\mu) \quad \textnormal{almost surely}.
\end{align*}
\end{theorem}

 Because we will develop algorithms that solve the robust optimization problem {approximately} as well as exactly, it is imperative for us to have theoretical guarantees {\color{black}that} hold for any Markovian stopping rule that can be found by the robust optimization problem. 
To this end, our third and final theorem of this section shows that the (in-sample) robust objective value will asymptotically provide a low-bias estimate of the expected reward, and this bound holds uniformly over all exercise policies.  The result yields theoretical assurance, provided that the robustness parameter is sufficiently small and the number of sample paths is sufficiently large, that searching for exercise policies with high robust objective values $\widehat{J}_{N,\epsilon}(\mu)$ will typically result in exercise policies with high expected rewards $J^*(\mu)$. 

\begin{theorem}[Asymptotic low-bias] \label{thm:conv:ub}
Under {\color{black}Assumptions~\ref{ass:weird}, \ref{ass:lighttail}, and \ref{ass:bound}},
$$\lim \limits_{ \epsilon \to 0}\liminf \limits_{N \to \infty} \inf \limits_{\mu} \left \{J^*(\mu)  -  \widehat{J}_{N,\epsilon}(\mu) \right \} \ge 0 \quad \textnormal{almost surely}.$$ 
\end{theorem}



\afterpage{%
\null
\vfill
\begin{algorithm}[H] 
\begin{center}
\fbox{\begin{minipage}{\linewidth}
\begin{center}
\textbf{\underline{The Robust Optimization {\color{black}Approach} for Stochastic Optimal Stopping}}\\
\end{center}
\vspace{1em}
\textbf{Inputs}: 
\begin{itemize}
\item \emph{Sizes of Training Sets}: A collection of integers $\mathcal{N} \triangleq \{N_1,\ldots,N_K \}$ sorted in ascending order.
\item \emph{Size of Validation Set}: An integer $\bar{N} \in \N$.
\item \emph{Size of Testing Set}: An integer $\tilde{N} \in \N$.
\item \emph{Robustness Parameters}: A collection of nonnegative real numbers $\mathcal{E} \triangleq \{\epsilon_1,\ldots,\epsilon_L \}$. 
\end{itemize}
\vspace{1em}
\textbf{Outputs}:
\begin{itemize}
\item Exercise policies $\hat{\mu} \equiv (\hat{\mu}_1,\ldots,\hat{\mu}_T)$ for the stochastic optimal stopping problem \eqref{prob:main}. 
\item An (unbiased) estimate of the expected reward $J^*(\hat{\mu})$ for the exercise policies. 
\end{itemize} 
\vspace{1em}
\textbf{Procedure}: 
\begin{enumerate}
\item Simulate the {validation set} $\bar{\mathcal{S}} \triangleq \{\bar{x}^1,\ldots,\bar{x}^{\bar{N}}\}$ and {testing set} $\tilde{\mathcal{S}} \triangleq \{\tilde{x}^1,\ldots,\tilde{x}^{\tilde{N}}\}$. 
\item For each $N \in \mathcal{N}$:
\begin{enumerate}
\item Simulate a {training set} $\mathcal{S} \triangleq \{ x^1,\ldots,x^N \}$.  \label{line:simulate}
\item For each $\epsilon \in \mathcal{E}$: \label{line:chooseeps}
\begin{enumerate}
\item Obtain exercise policies $\hat{\mu}_{N,\epsilon}$ by solving the robust optimization problem \eqref{prob:sro} constructed from the training set  $\mathcal{S}$ and robustness parameter $\epsilon$. 
\item Estimate the expected reward for these exercise policies using the validation set: 
\begin{align*}
 \bar{J}(\hat{\mu}_{N,\epsilon}) \triangleq \frac{1}{\bar{N}} \sum_{i=1}^{\bar{N}} g(\tau_{\hat{\mu}_{N,\epsilon}}(\bar{x}^i), \bar{x}^i ). 
\end{align*}
\end{enumerate}
\item If the total computation time of steps \eqref{line:simulate}-\eqref{line:chooseeps} reached the allowed budget (or if $N = N_K$), go to step~\eqref{line:finalstep} with the current value of $N$.  
\end{enumerate}
\item Output the exercise policies which maximized the estimated expected reward:  $$\hat{\mu} \leftarrow \argmax_{\epsilon \in \mathcal{E}} \bar{J}(\hat{\mu}_{N,\epsilon}),$$ 
and output an unbiased estimate of their expected reward using the testing set:
\begin{align*}
 J^*(\hat{\mu}) \approx \frac{1}{\tilde{N}} \sum_{i=1}^{\tilde{N}} g(\tau_{\hat{\mu}}(\tilde{x}^i), \tilde{x}^i ). 
\end{align*} \label{line:finalstep}
\end{enumerate}
\vspace{1em}
\end{minipage}}
\vspace{1em}
\end{center}
\caption{A procedure for selecting the robustness parameter and number of sample paths.}\label{fig:heuristic}
\end{algorithm}
\null
\vfill
\clearpage
}

\subsection{Implementation Details} \label{sec:params}
In anticipation of algorithmic techniques for solving the robust optimization problem \eqref{prob:sro} in the remainder of the paper, it remains to be specified   how the parameters of the robust optimization problem (the number of simulated sample paths $N \in \N$ and the robustness parameter $\epsilon \ge 0$) should be selected in practice. 
For the sake of concreteness, we conclude \S\ref{sec:setting} by briefly providing guidance for choosing these parameters  and applying the robust optimization {\color{black}approach} in practice. The procedures described below are formalized in Algorithm~\ref{fig:heuristic} and implemented in our numerical experiments in \S\ref{sec:experiments}. 

As described previously, 
this paper addresses  stochastic optimal stopping problems in which the probability distributions are {known}. Consequently, the decision-maker is granted flexibility in choosing the number of sample paths $N$ to simulate. On one hand,  we have established in the previous section that larger choices of the number of simulated sample paths will generally lead to tighter approximations of the stochastic optimal stopping problem. On the other hand, larger choices of $N$  require a greater computation cost in performing the Monte-Carlo simulation and creates a robust optimization problem of a larger size. To balance these tradeoffs in particular applications, we recommend using a straightforward procedure of starting out with a small choice of $N$ and iteratively increasing the number of simulated sample paths until the total computational cost meets the allocated computational budget.  

Given  a fixed number of sample paths, the choice of the robustness parameter $\epsilon \ge 0$ can have a significant impact on the policies produced by the robust optimization problem. To this end, we recommend solving \eqref{prob:sro} over a grid of possible choices for the robustness parameter. Because the probability distribution is known, we can generate a second set of `validation' sample paths to select the best choice of the robustness parameter. Specifically, for each choice of the robustness parameter, one solves the robust optimization problem to obtain exercise policies. The expected reward of the exercise policies is then estimated using the validation set of sample paths. Finally, we select the value of the robustness parameter (and the corresponding exercise policies) which maximizes the average reward with respect to the validation set. 

In summary, we have described straightforward and easy-to-implement heuristics for choosing the parameters of the robust optimization problem. Applying the heuristics and solving the robust optimization problem yields exercise policies for the stochastic optimal stopping problem, and an unbiased estimate of the expected reward of these exercise policies can similarly be obtained by simulating a set of `testing' sample paths  (see {\color{black}Algorithm}~\ref{fig:heuristic}). Because the exercise policies obtained from the robust optimization problem are feasible for the stochastic optimal stopping problem, the expected reward of these exercise policies is thus a {lower bound} on the optimal objective value of the stochastic optimal stopping problem. Finally, we remark that under a stronger assumption in which one has the ability to  perform conditional Monte Carlo simulation,  the exercise policies obtained from solving the robust optimization {\color{black}problem} can be combined with the method of  \cite{andersen2004primal} to obtain an {upper bound} on the optimal objective value  of the stochastic optimal stopping problem.


\section{Characterization of Optimal {\color{black}Markovian Stopping Rules}} \label{sec:main}
{\color{black}
In \S\ref{sec:setting}, we showed that our combination of robust optimization and simulation (\S\ref{sec:rob}) can yield an arbitrarily close approximation of the stochastic optimal stopping problem (\S\ref{sec:opt}-\S\ref{sec:params}). 
 In this section, we develop the key technical result of this paper, Theorem~\ref{thm:characterization}, which will enable us to design exact and heuristic algorithms for solving the robust optimization problem. Specifically, Theorem~\ref{thm:characterization} {establishes the existence} and {characterizes the structure} of optimal Markovian stopping rules for the robust optimization problem~\eqref{prob:sro}. {\color{black}In \S\ref{sec:main:statement} and \S\ref{sec:main:proof}, we present the statement of Theorem~\ref{thm:characterization} and provide a sketch of its proof.  In \S\ref{sec:algorithms:prelim}, we use the characterization of optimal Markovian stopping rules to transform \eqref{prob:sro} from an optimization problem over an {infinite-dimensional} space of exercise policies into a {finite-dimensional} optimization problem over integer decision variables $\sigma^1,\ldots,\sigma^N \in \{1,\ldots,T \}$ (Theorem~\ref{thm:reform}).} 
 
 {\color{black}
 \subsection{Statement of Theorem~\ref{thm:characterization}} \label{sec:main:statement}
 We begin by introducing the notation that will be used in our characterization of the structure of optimal Markovian stopping rules for the robust optimization problem~\eqref{prob:sro}. 
Consider any instance of the robust optimization problem~\eqref{prob:sro}. For any choice of integers $\sigma^1,\ldots,\sigma^N \in \{1,\ldots,T\}$, we define  $\mu^{\sigma^1 \cdots \sigma^N} \equiv ( \mu^{\sigma^1 \cdots \sigma^N}_1,\ldots,\mu^{\sigma^1 \cdots \sigma^N}_T )$ as the exercise policy that satisfies the following equality for each period $t \in \{1,\ldots,T\}$ and each state $y_t \in \mathcal{X}$: 
   \begin{align}
\mu_t^{ \sigma^1 \cdots \sigma^N} (y_t) \triangleq \begin{cases}
 \textsc{Stop},& \text{ if }  y_t \in \bigcup \limits_{i:\; \sigma^i = t} \mathcal{U}^i_t,\\
 \textsc{Continue},& \text{ if }  y_t \notin \bigcup \limits_{i:\; \sigma^i = t} \mathcal{U}^i_t.
 \end{cases} \label{line:equality_for_policies}
 \end{align} 
 
To develop intuition of the above exercise policy, we remark for each sample path $i \in \{1,\ldots,N\}$ that the integer $\sigma^i \in \{1,\ldots,T\}$ can be interpreted as a selection of one of the uncertainty sets $\mathcal{U}^i_1,\ldots,\mathcal{U}^i_T$. Specifically, if $\sigma^i = t$, then we observe from the above definition that the exercise policy  will satisfy  $\mu_t^{\sigma^1 \cdots \sigma^N}(y_t) = \textsc{Stop}$ for all $y_t \in \mathcal{U}^i_t$. 
More generally,  we observe that the exercise policy will satisfy    $\mu_t^{\sigma^1 \cdots \sigma^N}(y_t) = \textsc{Stop}$ for period $t \in \{1,\ldots,T\}$ and state $y_t \in \mathcal{X}$  
 if and only if there exists an uncertainty set $\mathcal{U}^i_t$ such that the state is contained in the uncertainty set, $y_t \in \mathcal{U}^i_t$, and the integer corresponding to the $i$th sample path is equal to the current period, $\sigma^i = t$. 
A visualization of the exercise policy $\mu^{\sigma^1 \cdots \sigma^N}$ is found in Figure~\ref{fig:simple_1}. 


\begin{figure}[t]
\centering
\subfloat[Sample paths and their corresponding uncertainty sets]{%
\includegraphics[width=0.5\linewidth]{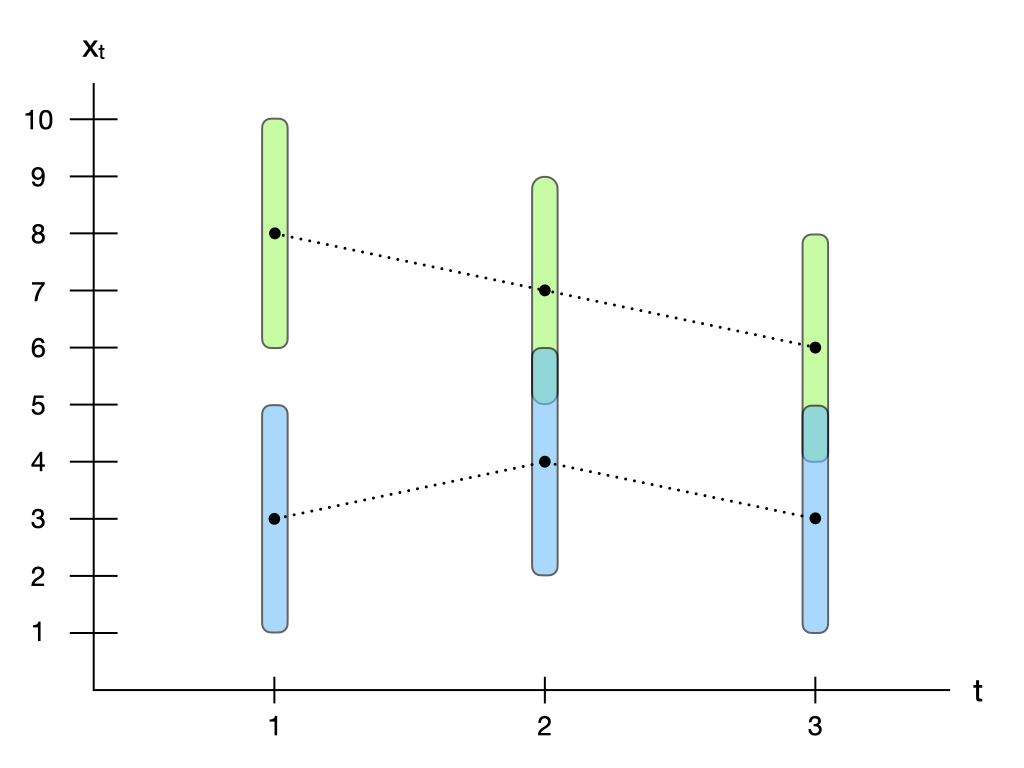}\label{fig:sub1}%
}
\subfloat[Stopping regions of $\mu^{2 3} \equiv (\mu^{2 3}_1,\mu^{2 3}_2,\mu^{2 3}_3 )$ ]{%
\includegraphics[width=0.5\linewidth]{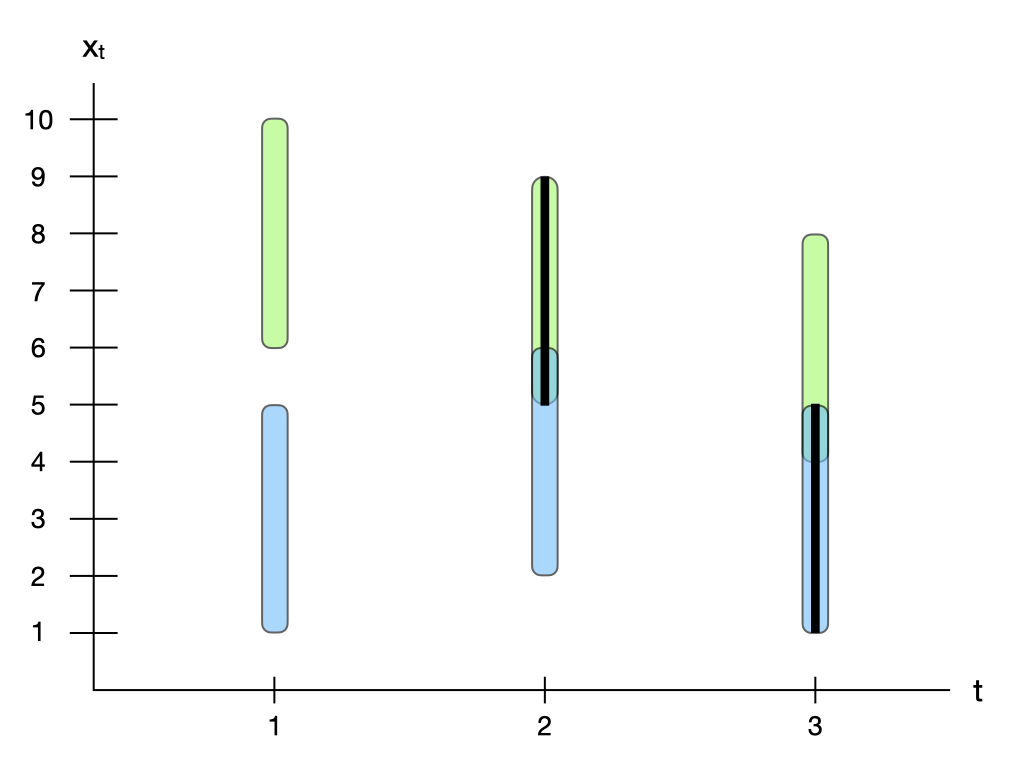} \label{fig:sub2}
}
\medskip
\caption{{\color{black}Visualization of a robust optimization problem with $N = 2$ sample paths  and $T = 3$ periods. The values of the sample paths in the robust optimization problem are given by $x^1 = (x^1_1,x^1_2,x^1_3) = (8,7,6)$ and $x^2 = (x^2_1,x^2_2,x^2_3) = (3,4,3)$, the state space is the real numbers ($\mathcal{X} = \R^1$) and the robustness parameter is  $\epsilon = 2.0$.  The green intervals correspond to the uncertainty sets $\mathcal{U}^1_1 = [6,10]$, $\mathcal{U}^1_2 = [5,9]$, and $\mathcal{U}^1_3 = [4,8]$ around the first sample path, and the blue intervals correspond to the uncertainty sets $\mathcal{U}^2_1 = [1,5]$, $\mathcal{U}^2_2 = [2,6]$, and $\mathcal{U}^2_3 = [1,5]$ around the second sample path.  In Figure~\ref{fig:sub1}, we  visualize the two sample paths using dashed lines.  In Figure~\ref{fig:sub2}, we   visualize  the stopping regions induced by the exercise policy  $\mu^{\sigma^1 \sigma^2} \equiv ( \mu^{\sigma^1 \sigma^2}_1,  \mu^{\sigma^1 \sigma^2}_2, \mu^{\sigma^1 \sigma^2}_3 )$  in the case where the integers $\sigma^1,\sigma^2 \in \{1,2,3\}$ satisfy $\sigma^1 = 2$ and $\sigma^2 = 3$.  Specifically, the  black vertical lines in Figure~\ref{fig:sub2} show the stopping regions  $\{y_1 \in \R: \mu^{23}_1(y_1) = \textsc{Stop} \}$, $\{y_2 \in \R: \mu^{23}_2(y_2) = \textsc{Stop} \}$, and $\{y_3 \in \R: \mu^{23}_3(y_3) = \textsc{Stop} \}$ induced by $\mu^{23} \equiv (\mu^{23}_1,\mu^{23}_2,\mu^{23}_3)$.  }   } \label{fig:simple_1}
\end{figure}



Let the set of all exercise polices generated by integers $\sigma^1,\ldots,\sigma^N \in \{1,\ldots,T\}$ be denoted by 
\begin{align*}
\mathcal{M} \triangleq \left \{ \mu \equiv (\mu_1,\ldots,\mu_T): \textnormal{ there exist } \sigma^1,\ldots,\sigma^N \in \{1,\ldots,T\} \textnormal{ such that } \mu = \mu^{\sigma^1 \cdots \sigma^N}  \right \}. 
\end{align*}
It is clear from the above definition that each of the exercise policies $\mu \in \mathcal{M}$ is parameterized by integers  $\sigma^1,\ldots,\sigma^N \in \{1,\ldots,T\}$; thus, we readily observe that the cardinality of $\mathcal{M}$ is always finite and  upper bounded by $T^N$. Moreover, we observe that the definition of the set of exercise policies is {sample-path dependent}, in the sense that $\mathcal{M}$ depends on the number and realizations of the simulated sample paths  $x^1 \equiv (x^1_1,\ldots,x^1_T),\ldots, x^N \equiv (x^N_1,\ldots,x^N_T)$ and on the choice of the robustness parameter $\epsilon \ge 0$. 
In Figure~\ref{fig:exhaustive}, we present a visualization of the set of exercise policies $\mathcal{M}$. 

   \afterpage{%
\null
\vfill
\begin{figure}[H]

\centering

\subfloat[$\sigma^1 = 1,\; \sigma^2 = 1$]{%
  \includegraphics[width=0.333\textwidth]{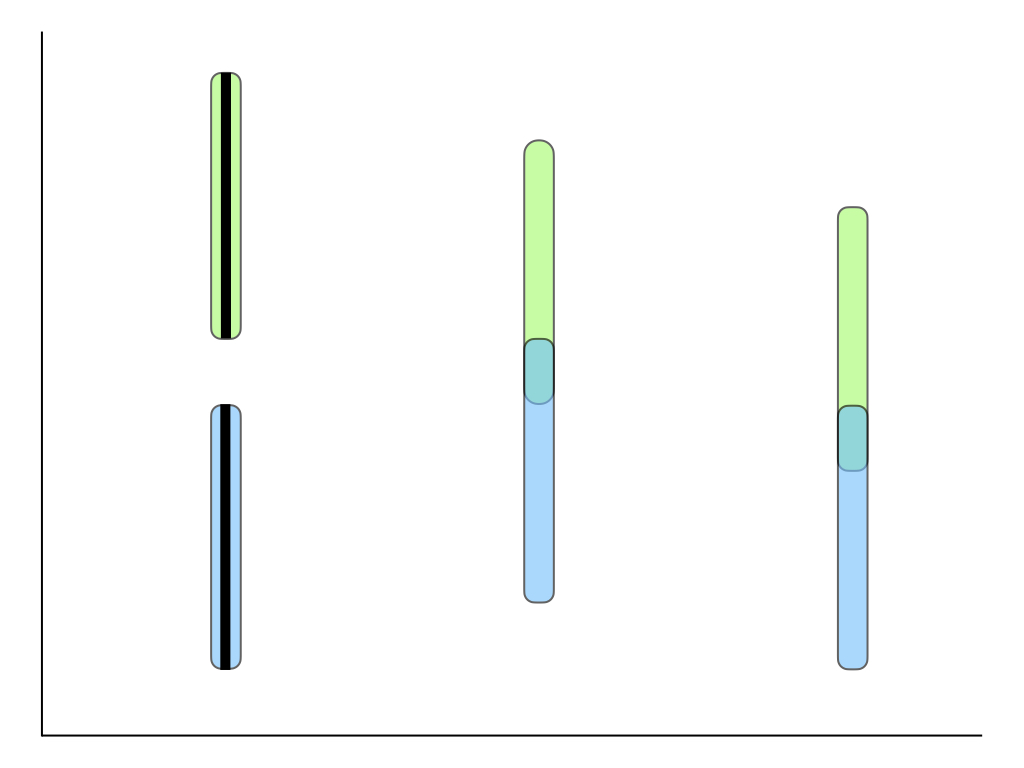}\label{fig:exhaustive:a}%
}
\subfloat[$\sigma^1 = 1,\; \sigma^2 = 2$]{%
  \includegraphics[width=0.333\textwidth]{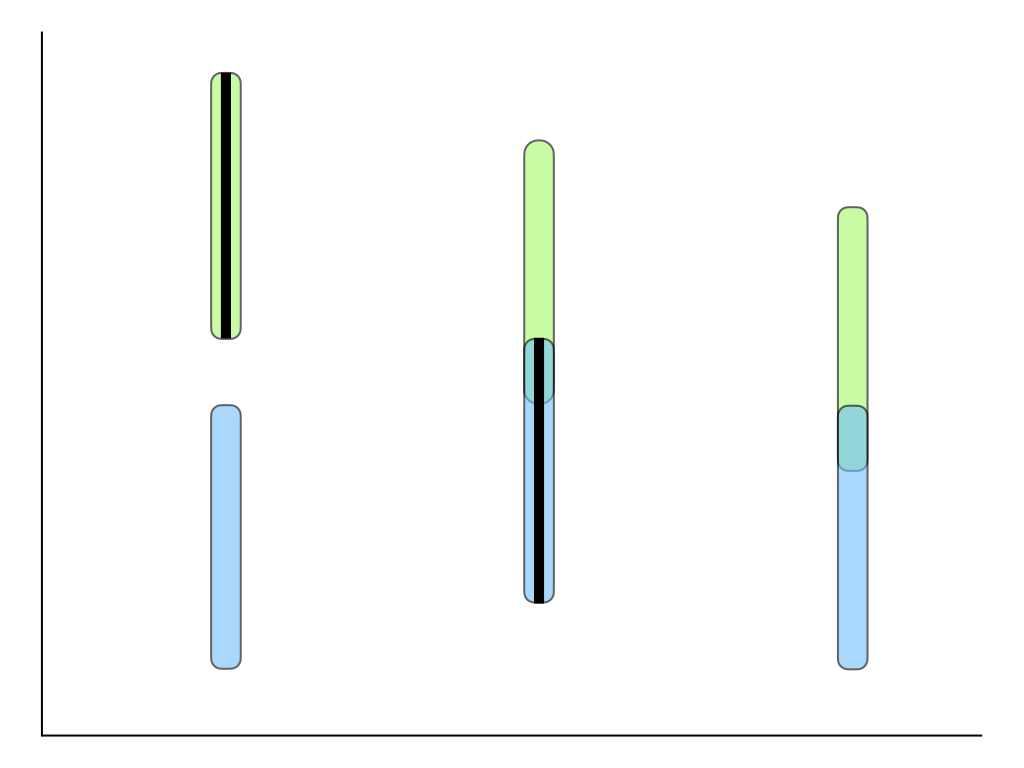}%
}
\subfloat[$\sigma^1 = 1,\; \sigma^2 = 3$]{%
  \includegraphics[width=0.333\textwidth]{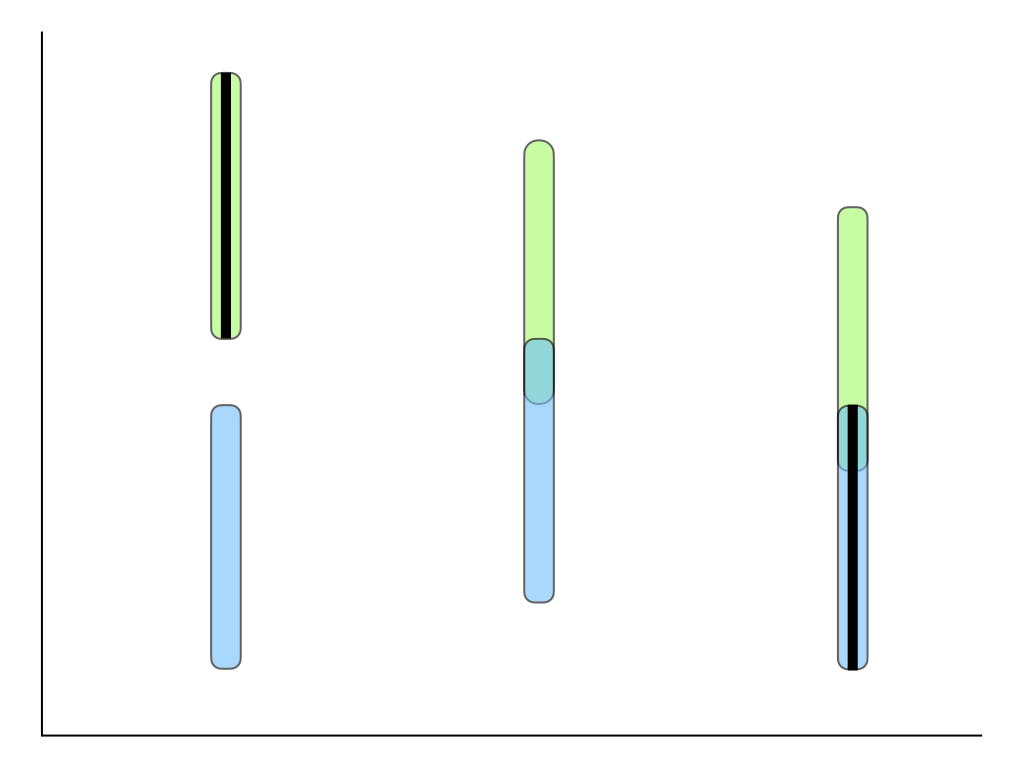}%
}\\
\subfloat[$\sigma^1 = 2,\; \sigma^2 = 1$]{%
  \includegraphics[width=0.333\textwidth]{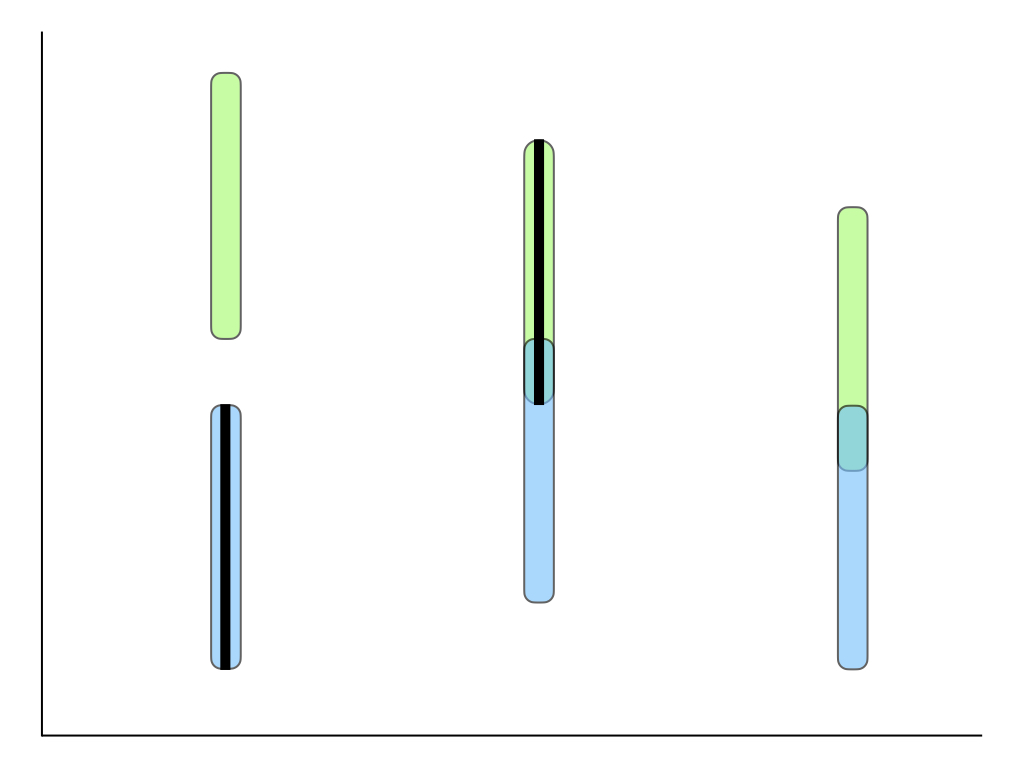}%
}
\subfloat[$\sigma^1 = 2,\; \sigma^2 = 2$]{%
  \includegraphics[width=0.333\textwidth]{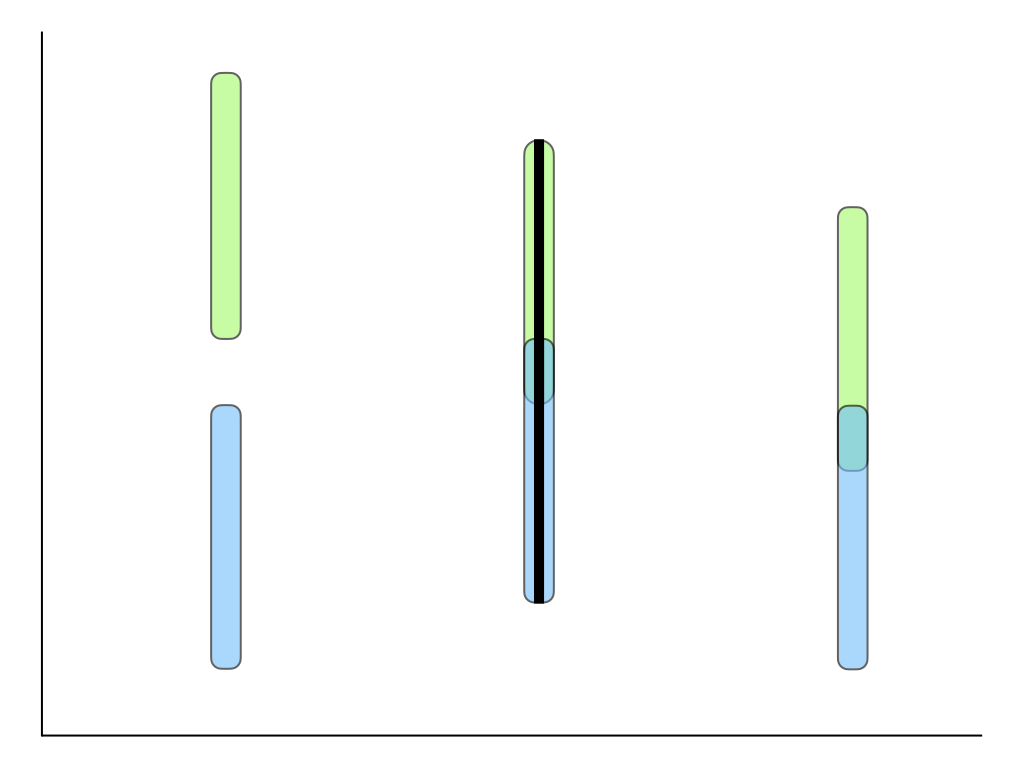}%
}
\subfloat[$\sigma^1 = 2,\; \sigma^2 = 3$]{%
  \includegraphics[width=0.333\textwidth]{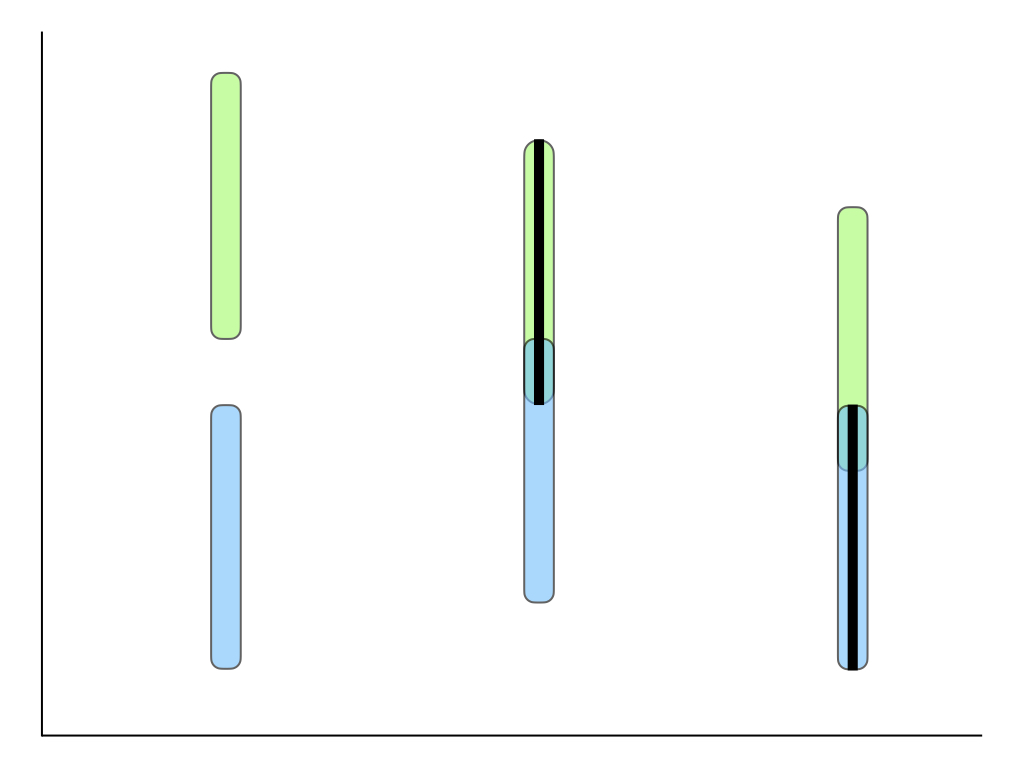}%
}\\
\subfloat[$\sigma^1 = 3,\; \sigma^2 = 1$]{%
  \includegraphics[width=0.333\textwidth]{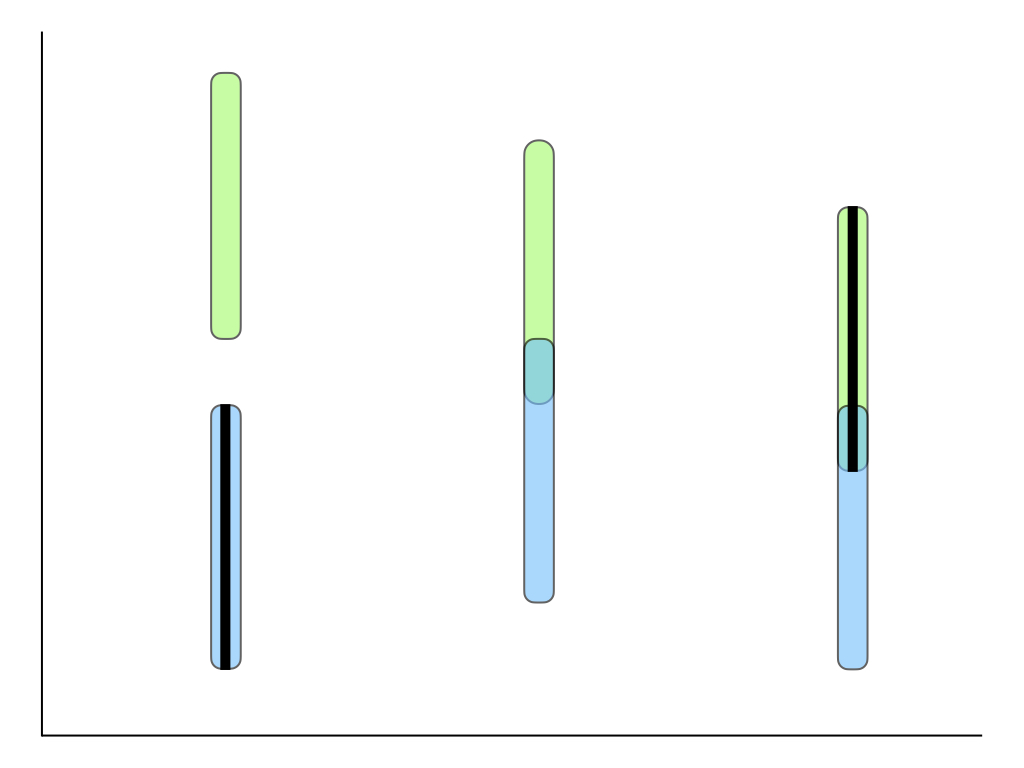}%
}
\subfloat[$\sigma^1 = 3,\; \sigma^2 = 2$]{%
  \includegraphics[width=0.333\textwidth]{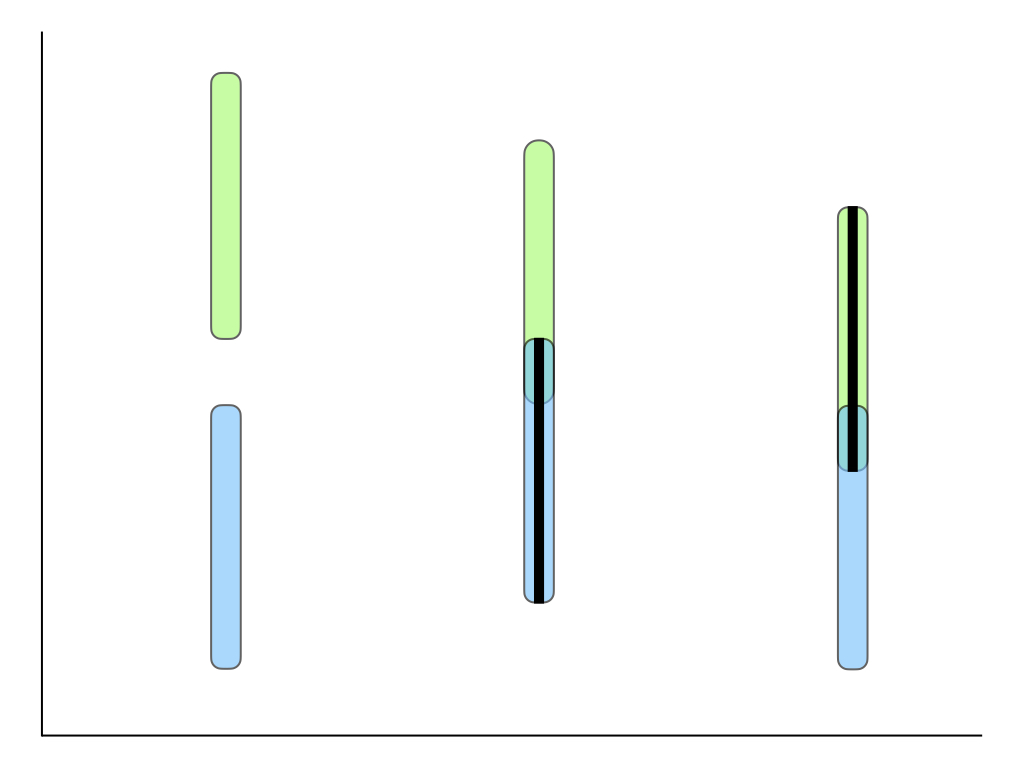}%
}
\subfloat[$\sigma^1 = 3,\; \sigma^2 = 3$]{%
  \includegraphics[width=0.333\textwidth]{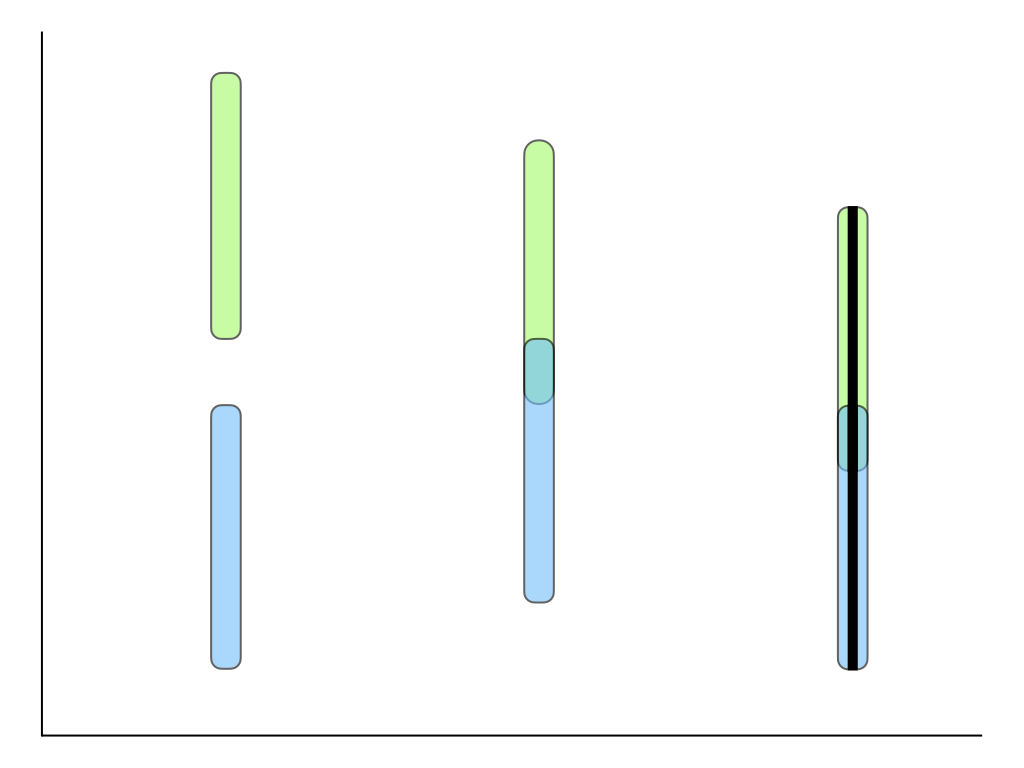}\label{fig:exhaustive:i}%
}
\medskip
\caption{The {\color{black}set} of exercise policies $\mathcal{M}$ for the robust optimization problem {\color{black}that is visualized in Figure~\ref{fig:sub1}. Specifically, each of the Figures~\ref{fig:exhaustive:a}-\ref{fig:exhaustive:i} visualizes the stopping regions induced by the exercise policy $\mu^{\sigma^1 \sigma^2} \equiv (\mu^{\sigma^1 \sigma^2}_1,\mu^{\sigma^1 \sigma^2}_2,\mu^{\sigma^1 \sigma^2}_3)$ overlayed on top of the uncertainty set around the first sample path $\mathcal{U}^1 \equiv \mathcal{U}^1_1 \times \mathcal{U}^1_2 \times \mathcal{U}^1_3 \equiv [6,10] \times [5,9] \times [4,8]$ and the uncertainty set around the second sample path $\mathcal{U}^2 \equiv \mathcal{U}^2_1 \times \mathcal{U}^2_2 \times \mathcal{U}^2_3 \equiv [1,5] \times [2,6] \times [1,5]$. We observe that none of the exercise policies in the Figures~\ref{fig:exhaustive:a}-\ref{fig:exhaustive:i} induce identical stopping regions; hence,  for this robust optimization problem, the cardinality of the set of exercise policies $\mathcal{M}$ is equal to $3^2 = 9$. 
}} \label{fig:exhaustive}
\end{figure}
\null
\vfill
\clearpage
}

In view of the above notation, we now present our main result: 
 \begin{theorem}
\label{thm:characterization}
There exists $\mu \in \mathcal{M}$ that is optimal for \eqref{prob:sro}.
\end{theorem}
The above theorem shows that there always exists an exercise policy in the set $\mathcal{M}$ that is optimal for the robust optimization problem. 
The result is significant because it will allow us to transform \eqref{prob:sro} from an optimization problem over an {infinite-dimensional} space of exercise policies into a {finite-dimensional} combinatorial optimization problem over the set $\mathcal{M}$. Moreover, Theorem~\ref{thm:characterization} is important because it establishes the existence of optimal Markovian stopping rules for the robust optimization problem~\eqref{prob:sro}. 


\subsection{Proof  Sketch of Theorem~\ref{thm:characterization}} \label{sec:main:proof}
Our proof of Theorem~\ref{thm:characterization} follows an exchange argument that is rooted in a technique that we refer to as \emph{pruning}.  The  technique consists of starting with an initial exercise policy $\mu \equiv (\mu_1,\ldots,\mu_T)$ and then modifying the exercise policy to reduce the size of the stopping regions $\{y_t \in \mathcal{X}: \mu_t(y_t) =  \textsc{Stop} \}$. By carefully pruning (\emph{i.e.}, reducing the size of) the stopping regions of the initial exercise policy $\mu \equiv (\mu_1,\ldots,\mu_T)$, we will show that any initial exercise policy can be transformed into a new exercise policy $\mu' \in \mathcal{M}$ with the same or better objective value in the robust optimization problem~\eqref{prob:sro}. 

To discuss the proof in greater detail, we begin by stating two preliminary lemmas:

\begin{lemma} \label{lem:restrict_to_completion}
The optimal objective value of \eqref{prob:sro} is equal to the optimal objective value of 
\begin{align} \label{prob:sro_2} \tag{RO$_T$}
\begin{aligned}
&\sup_{\mu} &&\frac{1}{N} \sum_{i=1}^N \inf_{y \in \mathcal{U}^i}  g(\tau_\mu(y), x^i)\\
&\textnormal{subject to}&& \textnormal{for each } i \in \{1,\ldots,N\}, \; \textnormal{ there exists } t \in \{1,\ldots,T\} \\
&&& \textnormal{such that } \mu_t(y_t) = \textsc{Stop} \textnormal{ for all } y_t \in \mathcal{U}^i_t .
\end{aligned}
\end{align}
\end{lemma}
 \begin{lemma} \label{lem:obj_sigma}
Consider any $\mu \equiv (\mu_1,\ldots,\mu_T)$ that satisfies the constraints of \eqref{prob:sro_2}, and define
\begin{align*}
\sigma^i \triangleq 
\min  \left \{ t \in \{1,\ldots,T\}: \; \mu_t(y_t) = \textsc{Stop} \textnormal{ for all } y_t \in \mathcal{U}^i_t \right \}  \quad \forall i \in \{1,\ldots,N\}. 
\end{align*}
Then the following equality holds for each $i \in \{1,\ldots,N\}$: 
\begin{align*}
& \inf_{y \in \mathcal{U}^i} g(\tau_{\mu}(y),x^i) = \min \limits_{t \in \{1,\ldots,\sigma^i\}} \left \{ g(t,x^i) :  \textnormal{there exists }  y_t \in \mathcal{U}^i_t \textnormal{ such that } \mu_t(y_t) = \textsc{Stop} \right \}.
\end{align*}
\end{lemma}
The first preliminary lemma  shows that a structural constraint can  be imposed onto the exercise policies in the robust optimization problem~\eqref{prob:sro} without any loss of optimality. In particular, Lemma~\ref{lem:restrict_to_completion} says that we can restrict to exercise policies in which the resulting Markovian stopping rule satisfies $\tau_\mu(y) < \infty$ for each trajectory $y$ in each of the uncertainty sets $\mathcal{U}^1,\ldots,\mathcal{U}^N$.\footnote{\color{black}Recall that the definition of a Markovian stopping rule is $\tau_\mu(y) \triangleq \min \{ t \in \{1,\ldots,T\}: \mu_t(y_t) = \textsc{Stop} \}$. Therefore, we readily observe  for each sample path $i$ that an exercise policy $\mu \equiv (\mu_1,\ldots,\mu_T)$ satisfies $\tau_\mu(y) < \infty$ for all $y \in \mathcal{U}^i \equiv \mathcal{U}^i_1 \times \cdots \times \mathcal{U}^i_T$ if and only if there exists a period $t \in \{1,\ldots,T\}$ that satisfies $\mu_t(y_t) = \textsc{Stop}$ for all $y_t \in \mathcal{U}^i_t$.  }
 The second preliminary lemma develops a convenient representation of the objective function of \eqref{prob:sro_2}. Specifically, Lemma~\ref{lem:obj_sigma} shows for each sample path $i$ that the quantity $\inf_{y \in \mathcal{U}^i} g(\tau_\mu(y),x^i)$ in the objective function of \eqref{prob:sro_2} is equal to the the minimum $g(t,x^i)$ among all 
periods $t \in \{1,\ldots, \max_{y \in \mathcal{U}^i} \tau_\mu(y)\}$ for which the uncertainty set $\mathcal{U}^i_t$ has a nonempty intersection with the stopping region $\{ y_t \in \mathcal{X}_t: \mu_t(y_t) = \textsc{Stop} \}$.



 
 Equipped with the above preliminary lemmas, we now formally define the {pruning technique} that underpins our proof of Theorem~\ref{thm:characterization}. 
 \begin{definition} \label{defn:pruning}
\emph{ Let $\mu$ be an exercise policy that is feasible for \eqref{prob:sro_2}. We say that $\mu'$ is a \emph{pruned version} of $\mu$ if the stopping regions of $\mu'$ are a subset of the stopping regions of $\mu$, \emph{i.e.},
 \begin{align*}
 &\left \{y_t \in \mathcal{X}: \mu_t'(y_t) = \textsc{Stop} \right \} \subseteq  \left \{y_t \in \mathcal{X}: \mu_t(y_t) = \textsc{Stop} \right \} \quad \forall t \in \{1,\ldots,T\},
 \end{align*}
 and if the following equality holds for each sample path $i \in \{1,\ldots,N\}$: 
 \begin{align*}
\min \left \{ t \in \{1,\ldots,T\}: \; \mu_t'(y_t) = \textsc{Stop}\; \forall y_t \in \mathcal{U}^i_t \right \}= \min \left \{ t \in \{1,\ldots,T\}: \; \mu_t(y_t) = \textsc{Stop} \;\forall y_t \in \mathcal{U}^i_t \right \}.
 \end{align*}
 }
  \end{definition}
    Speaking intuitively, an exercise policy $\mu'$ is a pruned version of an exercise policy $\mu$ if the stopping regions of $\mu'$ are a subset of the stopping regions of $\mu$ and if $\max_{y \in \mathcal{U}^i} \tau_\mu(y) = \max_{y \in \mathcal{U}^i} \tau_{\mu'}(y)$ for each sample path $i$. 
   The significance of Definition~\ref{defn:pruning} in combination with Lemma~\ref{lem:obj_sigma} is established by the following Lemma~\ref{lem:pruning}; specifically, the following Lemma~\ref{lem:pruning}  shows that if $\mu'$ is a pruned version of $\mu$, then the robust objective value associated with $\mu'$ is greater than or equal to the robust objective value associated with $\mu$. 
\begin{lemma} \label{lem:pruning}
If $\mu'$ is a pruned version of $\mu$, then $ \inf \limits_{y \in \mathcal{U}^i} g(\tau_{\mu'}(y),x^i)  \ge \inf \limits_{y \in \mathcal{U}^i} g(\tau_{\mu}(y),x^i)$ $\forall i \in \{1,\ldots,N\}$. 
\end{lemma}

In summary, we have shown in Lemma~\ref{lem:restrict_to_completion} that the robust optimization problem~\eqref{prob:sro} is equivalent to the robust optimization problem~\eqref{prob:sro_2}. Moreover, for every initial  exercise policy $\mu$ that is feasible for \eqref{prob:sro_2}, Lemma~\ref{lem:pruning} shows that the objective value associated with $\mu$ is less than or equal to the objective value associated with every exercise policy $\mu'$ that is a pruned version of $\mu$. The final step of our proof of Theorem~\ref{thm:characterization} is given by the following lemma:
\begin{lemma} \label{lem:M_is_pruned}
Consider any $\mu \equiv (\mu_1,\ldots,\mu_T)$ that satisfies the constraints of \eqref{prob:sro_2}, and define
\begin{align*}
\sigma^i \triangleq \min  \left \{ t \in \{1,\ldots,T\}: \; \mu_t(y_t) = \textsc{Stop} \textnormal{ for all } y_t \in \mathcal{U}^i_t \right \}  \quad \forall i \in \{1,\ldots,N\}. 
\end{align*}
Then $\mu^{\sigma^1 \cdots \sigma^N}$ is a pruned version of $\mu$. 
\end{lemma}
The above lemma implies that if $\mu$ is a feasible exercise policy for \eqref{prob:sro_2}, then there always exists an exercise policy $\mu' \in \mathcal{M}$ that is a pruned version of $\mu$. Hence, we conclude from Lemmas~\ref{lem:restrict_to_completion}, \ref{lem:pruning}, and \ref{lem:M_is_pruned} that there always exists an exercise policy $\mu \in \mathcal{M}$ that is optimal for \eqref{prob:sro}. The omitted details of the proofs of Theorem~\ref{thm:characterization} and Lemmas~\ref{lem:restrict_to_completion}, \ref{lem:obj_sigma}, \ref{lem:pruning}, and \ref{lem:M_is_pruned} are found in Appendix~\ref{appx:characterization}.

\subsection{Reformulation of \eqref{prob:sro} as a Finite-Dimensional Optimization Problem}  \label{sec:algorithms:prelim}
We conclude \S\ref{sec:main} by using our characterization of optimal Markovian stopping rules (Theorem~\ref{thm:characterization})  to develop a reformulation of \eqref{prob:sro} as a finite-dimensional   optimization problem over integer decision variables $\sigma^1,\ldots,\sigma^N \in \{1,\ldots,T \}$.  
The reformulation from this subsection is important because it establishes a natural combinatorial interpretation of the robust optimization problem  \eqref{prob:sro}, which will provide the foundation for our algorithmic developments in \S\ref{sec:exact} and \S\ref{sec:approx}. 

 Our finite-dimensional reformulation is presented as the following optimization problem~\eqref{prob:mip}:
 {\color{black}
 \begin{align} \label{prob:mip} \tag{IP}
\underset{ \sigma^1,\ldots,\sigma^N \in \{1,\ldots,T\}}{\textnormal{maximize}}\;  \frac{1}{N} \sum_{i=1}^N \min_{t \in \{1,\ldots,\sigma^i \}} \left \{g(t,x^i): \textnormal{ there exists } j \textnormal{ such that } \mathcal{U}^i_t \cap \mathcal{U}^j_t \neq \emptyset \textnormal{ and } \sigma^j = t\right \}. 
\end{align}
The decision variables in the above optimization problem are the integers $\sigma^1,\ldots,\sigma^N \in \{1,\ldots,T\}$. We further observe that the inner minimization problems in \eqref{prob:mip} involves constraints that depend on whether the intersection $\mathcal{U}^i_t \cap \mathcal{U}^j_t$ is nonempty for each pair of sample paths $i,j$ and each period $t$. An important insight is that these intersections can be {precomputed}; that is, given the construction of the uncertainty sets from \S\ref{sec:rob}, we can  efficiently precompute  the set of all pairs of sample paths and periods that satisfy  $\mathcal{U}^i_t \cap \mathcal{U}^j_t \neq \emptyset$.\footnote{{\color{black}We recall from \S\ref{sec:rob} that each uncertainty set $\mathcal{U}^i_t \subseteq \R^d$ is a hypercube. Thus, determining whether $\mathcal{U}^i_t \cap \mathcal{U}^j_t$ is nonempty for each pair of sample paths $i,j \in \{1,\ldots,N\}$ and each period $t \in \{1,\ldots,T\}$ can be precomputed in a total of $\mathcal{O}(N^2T d)$ computation time.}} 

The equivalence of \eqref{prob:mip} and \eqref{prob:sro} is formalized by the following theorem.

  \begin{theorem}\label{thm:reform}
The optimal objective value of \eqref{prob:mip} is equal to the optimal objective value of \eqref{prob:sro}. Furthermore, for any choice of integers $\sigma^1,\ldots,\sigma^N \in \{1,\ldots,T\}$, 
 \begin{align*}
 &\frac{1}{N} \sum_{i=1}^N \inf_{y \in \mathcal{U}^i}  g \left(\tau_{\mu^{\sigma^1 \cdots \sigma^N}}(y), x^i \right) \\
 &\ge \frac{1}{N} \sum_{i=1}^N \min_{t \in \{1,\ldots,\sigma^i \}} \left \{g(t,x^i): \textnormal{ there exists } j \textnormal{ such that } \mathcal{U}^i_t \cap \mathcal{U}^j_t \neq \emptyset \textnormal{ and } \sigma^j = t\right \}.
  \end{align*}
\end{theorem}
The proof of Theorem~\ref{thm:reform}, which is found in Appendix~\ref{appx:characterization}, follows readily from Lemmas~\ref{lem:restrict_to_completion}, \ref{lem:obj_sigma}, \ref{lem:pruning}, and \ref{lem:M_is_pruned}. 
Stated in words, the above theorem shows, for any given feasible solution $\sigma^1,\ldots,\sigma^N \in \{1,\ldots,T\}$ for  the optimization problem~\eqref{prob:mip}, that the objective value associated with the exercise policy $\mu^{\sigma^1 \cdots \sigma^N}$ in the optimization problem~\eqref{prob:sro} is greater than or equal to the objective value associated with $\sigma^1,\ldots,\sigma^N$ in the optimization problem~\eqref{prob:mip}. Because the above theorem also shows that optimal objective value of \eqref{prob:sro} is equal to the optimal objective value of \eqref{prob:mip}, it follows immediately from  Theorem~\ref{thm:reform} that any optimal solution for \eqref{prob:mip} can be transformed into an optimal solution for \eqref{prob:sro}.

}}

\section{Computation of Optimal {\color{black}Markovian Stopping Rules}} \label{sec:algorithms}

{\color{black} In this section, we use our characterization of optimal Markovian stopping rules from \S\ref{sec:main} to develop exact and heuristic algorithms for solving the robust optimization problem  \eqref{prob:sro}.  In \S\ref{sec:exact}, we use this finite-dimensional reformulation of  \eqref{prob:sro} to design exact algorithms and hardness results for solving the robust optimization problem. In \S\ref{sec:approx}, we propose and analyze an efficient heuristic algorithm for approximately solving the robust optimization problem. We emphasize that the algorithms and analysis in this section are general and do not require any of the assumptions that were made in \S\ref{sec:opt}. 
}

\subsection{Exact Algorithm} \label{sec:exact}



{\color{black} In {\color{black}\S\ref{sec:algorithms:prelim}}, we developed a finite-dimensional optimization problem \eqref{prob:mip} with  integer decision variables $\sigma^1,\ldots,\sigma^N \in \{1,\ldots,T\}$ that is equivalent to the infinite-dimensional  robust optimization problem~\eqref{prob:sro}. Specifically, we showed that the optimal objective values of the optimization problems \eqref{prob:mip} and \eqref{prob:sro} are  always equal. Moreover, we readily observe from Theorem~\ref{thm:reform} that any optimal solution for \eqref{prob:mip} can further be transformed into an optimal exercise policy for \eqref{prob:sro} through the transformation described by {\color{black}line}~\eqref{line:equality_for_policies} in \S\ref{sec:main}. 
Harnessing this optimization problem~\eqref{prob:mip}, we now design exact algorithms for solving the robust optimization problem~\eqref{prob:sro}. 

}

We begin our discussion by considering the case of \eqref{prob:mip} when there are two periods. Optimal stopping problems with two periods has been studied in the literature as a testbed for understanding the complexity of solving optimal stopping problems, \eg \cite{glasserman2004number}. We defer the proof of the following result to \S\ref{sec:approx:theory}. 
{\color{black}
\begin{theorem} \label{thm:twoperiods}
If $T=2$, then \eqref{prob:mip} can be solved in $\mathcal{O}(N^3)$ time. 
\end{theorem}}
Continuing with our discussion on the theoretical tractability of \eqref{prob:mip}, we next consider  problems with three or more periods. The following negative result shows that the computational tractability of \eqref{prob:mip} in the case of two periods {\color{black} established in Theorem~\ref{thm:twoperiods}} does not generally extend to optimal stopping problems with three or more periods. The proof of the following result{\color{black}, which is found in Appendix~\ref{appx:hard},} consists of a reduction from MIN-2-SAT, which is shown to be strongly NP-hard by \cite{ kohli1994minimum}.  

\begin{theorem} \label{thm:hard}
 \eqref{prob:mip} is strongly NP-hard for any fixed $T \ge 3$. 
\end{theorem}

{\color{black}
Motivated by the above hardness result, we proceed to develop an exact algorithm for solving \eqref{prob:mip} by reformulating it as a zero-one bilinear program {\color{black}over totally unimodular constraints. This reformulation, presented below as \eqref{prob:bp}, is valuable for three primary reasons. First, it is well known that bilinear programs can generally be transformed into mixed-integer linear optimization problems using linearization techniques; see Appendix~\ref{appx:reform} for details. Thus,  for robust optimization problems with small values of $T$ and $N$, the exact reformulation from this section can directly solved by off-the-shelf commercial optimization software. Second,  the exact reformulation~\eqref{prob:bp} will provide the foundation for our heuristic algorithm for the robust optimization problem in the following \S\ref{sec:approx}. Third, the exact reformulation~\eqref{prob:bp} is relatively compact, requiring only $\mathcal{O}(N T^2)$ decision variables and $\mathcal{O}(N T (T + N))$ constraints. In particular, the mild dependence of the size of the exact reformulation on the number of simulated sample paths is attractive in practical settings where the number of simulated sample paths is much larger than the number of time periods. 
}

{\color{black}Our  reformulation of \eqref{prob:mip} as a zero-one bilinear program over totally unimodular constraints~\eqref{prob:bp} 
requires the following additional notation. For each sample path $i \in \{1,\ldots,N\}$, we define the following set: 
\begin{align*}
{\mathcal{K}}^i &\triangleq \left \{ \kappa: \text{ there exists } t  \in \{1,\ldots,T\}\text{ such that } g(t,x^i) = \kappa \right \} \cup \{ 0\}.
\end{align*} 
The above set can be interpreted as the set of distinct values among $g(1,x^i),\ldots,g(T,x^i)$ and $0$. 
 For notational convenience, let the elements of $\mathcal{K}^i$ be indexed in ascending order, $\kappa^i_1 < \cdots < \kappa^i_{|\mathcal{K}^i|}$, and 
 let $L^i_t \in \{1,\ldots,|\mathcal{K}^i| \}$ be defined for each period $t \in \{1,\ldots,T\}$ as the unique index that satisfies the equality $g(t,x^i) = \kappa^i_{L^i_t}$. We readily observe that the quantities $L^i_t$ and $\kappa^i_1,\ldots,\kappa^i_{L^i_t}$ can be efficiently precomputed for each sample path $i$ and period $t$.   With the above notation,  our exact reformulation of \eqref{prob:mip} is stated as follows: 
 \begin{equation} \label{prob:bp} \tag{BP}
\begin{aligned}
 &\underset{ b,w}{\textnormal{maximize}}&& \frac{1}{N} \sum_{i=1}^N \sum_{t=1}^T \sum_{\ell=1}^{L^i_t-1}( \kappa^i_{\ell+1} - \kappa^i_\ell)b^i_t (1 - w^i_{t\ell}) \\
&\textnormal{subject to}&& \begin{aligned}[t]
&w^{i}_{t\ell} \le  w^i_{t+1,\ell} && \textnormal{for all } i \in \{1,\ldots,N\}, \;t \in \{1,\ldots,T-1\},\; \ell \in \{1,\ldots,|\mathcal{K}^i|\}\\
&w^{i}_{t\ell} \le w^i_{t,\ell+1} && \textnormal{for all } i \in \{1,\ldots,N\},\; t \in \{1,\ldots,T\},\;\ell \in \{1,\ldots,|\mathcal{K}^i|-1\} \\
&b^i_t \le w^i_{t+1,1}&&\textnormal{for all }i \in \{1,\ldots,N\},\; t \in \{1,\ldots,T-1\}\\
&b^j_t \le w^i_{t\ell}&& \textnormal{for all } i, j \in \{1,\ldots,N\} \textnormal{ and } t \in \{1,\dots,T\} \textnormal{ such that } g(t,x^i) = \kappa^i_\ell\\
&&& \textnormal{and } \mathcal{U}^i_t \cap \mathcal{U}^j_t \neq \emptyset\\
&b^i_t \in \{0,1\} && \textnormal{for all } i \in \{1,\ldots,N\}, \; t \in \{1,\dots,T\}\\
&w^i_{t\ell} \in \R  && \textnormal{for all } i \in \{1,\ldots,N\}, \; t \in \{1,\ldots,T\}, \; \ell \in \{1,\ldots,| \mathcal{K}^i|\}.
\end{aligned}
\end{aligned}
\end{equation}

Let us reflect on the structure of the above optimization problem. First, we observe that \eqref{prob:bp} is a {bilinear program} because its objective function is the sum of products of decision variables $b^i_t$ and $w^i_{t\ell}$. Second, we observe that the bilinear program~\eqref{prob:bp} is comprised of binary decision variables $b^i_t \in \{0,1\}$ as well as continuous decision variables $w^i_{t \ell} \in \R$. 
In particular, we readily observe from the structure of the constraints of \eqref{prob:bp} and from the fact that each  quantity $\kappa^i_{\ell+1} - \kappa^i_\ell$ is strictly positive that there always exists an optimal solution for \eqref{prob:bp} in which each decision variable $w^i_{t\ell}$ is equal to zero or one. For this reason, \eqref{prob:bp} will be henceforth referred to as a {zero-one bilinear program} without any ambiguity in terminology.  Finally, 
 we remark that the constraints of  \eqref{prob:bp} are totally unimodular,  which implies that every extreme point of  the polyhedron defined by the constraints of \eqref{prob:bp} is integral  \citep[\S 4.2]{conforti2014integer}. We will utilize the total unimodularity of \eqref{prob:bp}  in  \S\ref{sec:approx} when designing our heuristic algorithm for the robust optimization problem. 

We next show that \eqref{prob:bp} is indeed equivalent to \eqref{prob:mip}. 
In the following Theorem~\ref{thm:bp_main}, we establish this equivalence and show that any feasible solution for \eqref{prob:bp} can be transformed into a feasible solution for \eqref{prob:mip} with the same or greater objective value. 
The proof of Theorem~\ref{thm:bp_main}, as well as the proofs of the subsequent lemmas in \S\ref{sec:exact}, are found in Appendix~\ref{appx:exact}. 
\begin{theorem} \label{thm:bp_main}
 The optimal objective value of \eqref{prob:mip} is equal to the optimal objective value of \eqref{prob:bp}. Furthermore, if
  $b,w$ is a feasible solution for \eqref{prob:bp} and if $\sigma^i \triangleq \min \{\min \{t \in \{1,\ldots,T\}:  b^i_t = 1 \}, T \}$ for each $i \in \{1,\ldots,N\}$, then 
  \begin{align*}
&\frac{1}{N} \sum_{i=1}^N \min_{t \in \{1,\ldots,\sigma^i \}} \left \{g(t,x^i): \textnormal{ there exists } j \textnormal{ such that } \mathcal{U}^i_t \cap \mathcal{U}^j_t \neq \emptyset \textnormal{ and } \sigma^j = t\right \} \\ 
 &\ge  \frac{1}{N} \sum_{i=1}^N \sum_{t=1}^T \sum_{\ell=1}^{L^i_t-1}( \kappa^i_{\ell+1} - \kappa^i_\ell)b^i_t (1 - w^i_{t\ell}). &
  \end{align*}
  \end{theorem}
The above theorem shows for each sample path $i$ that the decision variables $b^i_1,\ldots,b^i_T \in \{0,1\}$ in \eqref{prob:bp} can be interpreted as an encoding of an integer $\sigma^i \in \{1,\ldots,T\}$ for \eqref{prob:mip}. In other words, given any feasible solution $b,w$ for \eqref{prob:bp}, Theorem~\ref{thm:bp_main} shows how to construct a feasible solution $\sigma^1,\ldots,\sigma^N \in \{1,\ldots,T\}$ for \eqref{prob:mip} such that the objective value associated with $\sigma^1,\ldots,\sigma^N$ in the optimization problem~\eqref{prob:mip} is greater than or equal to the objective value associated with $b,w$ in the optimization problem~\eqref{prob:bp}. Because the above theorem also shows that optimal objective value of \eqref{prob:mip} is equal to the optimal objective value of \eqref{prob:bp}, it follows immediately from  Theorem~\ref{thm:bp_main} that any optimal solution for \eqref{prob:bp} can be transformed into an optimal solution for \eqref{prob:mip}. 

 
We conclude \S\ref{sec:exact} by providing intuition for the role of the decision variables $w^i_{t \ell}$ in the optimization problem \eqref{prob:bp}. To do this,  we outline the two key steps of our derivation of \eqref{prob:bp}. 
Indeed, recall that the optimization problem~\eqref{prob:mip} is comprised of integer decision variables $\sigma^1,\ldots,\sigma^N \in \{1,\ldots,T \}$. 
The first step in our derivation of \eqref{prob:bp} is to introduce binary decision variables $b^i_1,\ldots,b^i_T\in  \{0,1\}$ for each sample path $i$ to represent each integer decision variable $\sigma^i$. Specifically, consider the following nonlinear binary optimization problem:
\begin{equation} \label{prob:bp_1} \tag{BP-1}
\begin{aligned}
&\underset{b}{\textnormal{maximize}}&& \frac{1}{N} \sum_{i=1}^N   \psi^i \left(b \right) \\
&\textnormal{subject to}&&  b^i_t \in \{0,1\} \quad \textnormal{for all } i \in \{1,\ldots,N\} \text{ and } t \in \{1,\ldots,T\},
\end{aligned}
\end{equation}
where the function $\psi^i \left(b \right)$ is defined for each sample path $ i \in \{1,\ldots,N\}$  as 
\begin{align*}
\psi^i \left(b \right) &\triangleq \sum_{t=1}^T  b^i_t \left( \prod_{s=1}^{t-1}(1 - b^i_s)  \right) \min_{s \in \{1,\ldots,t \}} \left \{g(s,x^i): \textnormal{there exists } j \textnormal{ such that } \mathcal{U}^i_s \cap \mathcal{U}^j_s \neq \emptyset \textnormal{ and } b^j_s = 1\right \}.
\end{align*}
To make sense of the function $\psi^i(b)$, we remark that each quantity $ b^i_t \prod_{s=1}^{t-1}(1 - b^i_s) $ will be equal to one if and only if $t$ is the earliest period for sample path $i$ that satisfies the equality  $b^i_t = 1$. 
 With this observation, the first step in our derivation  of \eqref{prob:bp} is comprised of the following lemma:
\begin{lemma} \label{lem:exact:bp_1}
 The optimal objective value of \eqref{prob:mip} is equal to the optimal objective value of \eqref{prob:bp_1}. Furthermore, if
  $b$ is a feasible solution for \eqref{prob:bp_1} and if $\sigma^i \triangleq \min \{\min \{t \in \{1,\ldots,T\}:  b^i_t = 1 \}, T \}$ for each $i \in \{1,\ldots,N\}$, then
  \begin{align*}
&\frac{1}{N} \sum_{i=1}^N \min_{t \in \{1,\ldots,\sigma^i \}} \left \{g(t,x^i): \textnormal{ there exists } j \textnormal{ such that } \mathcal{U}^i_t \cap \mathcal{U}^j_t \neq \emptyset \textnormal{ and } \sigma^j = t\right \} \ge  \frac{1}{N} \sum_{i=1}^N\psi^i(b). 
  \end{align*}
\end{lemma}
The above lemma establishes the equivalence of \eqref{prob:mip} of \eqref{prob:bp_1}, and it shows that any feasible solution for \eqref{prob:bp_1} can be transformed into a feasible solution for \eqref{prob:mip} with the same or greater objective value.  In the second and final step in our derivation of \eqref{prob:bp}, denoted below by Lemma~\ref{lem:submodular_reform}, we show that each function $\psi^i(b)$ can be represented as the optimal objective function of a linear optimization problem.  
 \begin{lemma} \label{lem:submodular_reform}
 For each sample path $i \in \{1,\ldots,N\}$, $\psi^i(b)$ is equal to the optimal objective value of the following linear optimization problem:
   \begin{equation} \label{line:partial_reform}
 \begin{aligned}
  &\underset{w^i}{\textnormal{maximize}}&& \sum_{t=1}^T  \sum_{\ell=1}^{L^i_t-1}( \kappa^i_{\ell+1} - \kappa^i_\ell) b^i_t (1 - w^i_{t\ell}) \\
  &\textnormal{subject to}&& \begin{aligned}[t]
  &w^{i}_{t \ell} \le  w^i_{t+1,\ell} && \textnormal{for all } t \in \{1,\ldots,T-1\},\; \ell \in \{1,\ldots,|\mathcal{K}^i|\}\\
&w^{i}_{t\ell} \le w^i_{t,\ell+1} && \textnormal{for all } t \in \{1,\ldots,T\},\; \ell \in \{1,\ldots,|\mathcal{K}^i|-1\} \\
&b^j_t \le w^i_{t\ell}&& \textnormal{for all } j \in \{1,\ldots,N\} \textnormal{ and } t \in \{1,\dots,T\} \\
&&& \textnormal{such that }g(t,x^i)  = \kappa^i_\ell \textnormal{ and }  \mathcal{U}^i_t \cap \mathcal{U}^j_t \neq \emptyset\\
& b^i_t \le w^i_{t+1,1}  && \textnormal{for all }  \; t \in \{1,\ldots,T-1\}\\
&w^i_{t \ell} \in \R&& \textnormal{for all }\; t \in \{1,\ldots,T\}, \; \ell \in \{1,\ldots,|\mathcal{K}^i| \}.
  \end{aligned}
  \end{aligned} 
 \end{equation}
 \end{lemma}
Lemma~\ref{lem:submodular_reform} thus shows that the purpose of the decision variables $w^i_s$ and the constraints of \eqref{prob:bp} is to encode the value of $\psi^i(b)$ for each sample path $i \in \{1,\ldots,N\}$. Finally, we readily observe that  the combination of Lemmas~\ref{lem:exact:bp_1} and \ref{lem:submodular_reform} immediately yields the proof of Theorem~\ref{thm:bp_main}.

 }}

\subsection{{\color{black}Heuristic Algorithm}} \label{sec:approx}
{\color{black}
In \S\ref{sec:exact}, we developed an exact reformulation of the robust optimization problem \eqref{prob:sro} as a zero-one bilinear program over totally unimodular constraints \eqref{prob:bp}. That bilinear program can be solved by off-the-shelf software for mixed-integer linear optimization; see Appendix~\ref{appx:reform}. Thus, the previous subsection can be viewed as a concrete and easily implementable exact algorithm for solving the robust optimization problem. 

Building upon \eqref{prob:bp}, we now turn to the task of developing efficient algorithms that can approximately solve the robust optimization problem. Specifically, the main contribution of \S\ref{sec:approx} is a practical heuristic for solving the robust optimization problem with computation time that is polynomial in both the number of sample paths $N$ and the number of periods $T$. The proposed heuristic is thus significantly more computational tractable than solving \eqref{prob:bp}, which was shown in Theorem~\ref{thm:hard} to be {NP}-hard. Moreover, we will provide theoretical and empirical evidence that the proposed heuristic can perform surprisingly well, both with respect to approximation quality and computational tractability. All omitted proofs of results from \S\ref{sec:approx} can be found in Appendix~\ref{appx:approx}. 

This subsection is organized as follows. In \S\ref{sec:approx:prelim}, we discuss the high-level motivation and intuition behind our proposed heuristic for the robust optimization problem. In \S\ref{sec:approx:main} and \S\ref{sec:approx:computation}, we formalize the heuristic and offer a strongly polynomial-time algorithm for implementing it. In \S\ref{sec:approx:theory}, we provide theoretical guarantees which show that the approximation gap of the heuristic cannot be arbitrarily bad and, in some cases, is guaranteed to be equal to zero. We study the empirical performance of the heuristic in the subsequent \S\ref{sec:experiments}. 

\subsubsection{Preliminaries.} \label{sec:approx:prelim}
 Our heuristic for approximately solving the robust optimization problem is motivated by the structure of the exact reformulation~\eqref{prob:bp}  from \S\ref{sec:exact} for the robust optimization problem. Recall that \eqref{prob:bp} is a zero-one optimization problem with a bilinear objective function and totally unimodular constraints. Because the constraints of \eqref{prob:bp} are totally unimodular, the computational difficulty of solving~\eqref{prob:bp} can thus be attributed to the nonlinearity of its objective function. The proposed heuristic, which is formalized in \S\ref{sec:approx:main},   
aims to contend with this difficulty by approximating the nonlinear objective function of \eqref{prob:bp} with a linear function. 

To describe the motivation behind our heuristic in greater detail, consider any function $f(b,w)$ that is linear in the vectors $b,w$ and is a lower bound on the objective function of the optimization problem~\eqref{prob:bp}  for all $b,w$ that satisfy the constraints of \eqref{prob:bp}. That is, let $f(b,w)$ be a linear function that satisfies the following condition:
\begin{align}
f(b,w) \le \frac{1}{N} \sum_{i=1}^N \sum_{t=1}^T \sum_{\ell=1}^{L^i_t-1}( \kappa^i_{\ell+1} - \kappa^i_\ell)b^i_t (1 - w^i_{t\ell})\; \; \; \; \textnormal{for all } b,w \textnormal{ that are feasible for \eqref{prob:bp}.} \label{line:f_lb}
\end{align}
Then it follows immediately from the condition on line~\eqref{line:f_lb} that a conservative, lower-bound approximation of the optimization problem~\eqref{prob:bp} is given by the following optimization problem:
%
 \begin{equation} \label{prob:h} \tag{H}
\begin{aligned}
 &\underset{ b,w}{\textnormal{maximize}}&&f(b,w) \\
&\textnormal{subject to}&& \begin{aligned}[t]
&w^{i}_{t\ell} \le  w^i_{t+1,\ell} && \textnormal{for all } i \in \{1,\ldots,N\}, \;t \in \{1,\ldots,T-1\},\; \ell \in \{1,\ldots,|\mathcal{K}^i|\}\\
&w^{i}_{t\ell} \le w^i_{t,\ell+1} && \textnormal{for all } i \in \{1,\ldots,N\},\; t \in \{1,\ldots,T\},\;\ell \in \{1,\ldots,|\mathcal{K}^i|-1\} \\
&b^i_t \le w^i_{t+1,1}&&\textnormal{for all }i \in \{1,\ldots,N\},\; t \in \{1,\ldots,T-1\}\\
&b^j_t \le w^i_{t\ell}&& \textnormal{for all } i, j \in \{1,\ldots,N\} \textnormal{ and } t \in \{1,\dots,T\} \textnormal{ such that } g(t,x^i) = \kappa^i_\ell\\
&&& \textnormal{and } \mathcal{U}^i_t \cap \mathcal{U}^j_t \neq \emptyset\\
&b^i_t \in \{0,1\} && \textnormal{for all } i \in \{1,\ldots,N\}, \; t \in \{1,\dots,T\}\\
&w^i_{t\ell} \in \R  && \textnormal{for all } i \in \{1,\ldots,N\}, \; t \in \{1,\ldots,T\}, \; \ell \in \{1,\ldots,| \mathcal{K}^i|\}.
\end{aligned}
\end{aligned}
\end{equation}
Indeed, we observe that the constraints of \eqref{prob:h} are identical to the constraints of \eqref{prob:bp}. The only difference between  these two optimization problems is that the nonlinear objective function of  \eqref{prob:bp} has been replaced with the linear function $f(b,w)$ that satisfies the condition from line~\eqref{line:f_lb}. 
Therefore, 
 we conclude that every optimal solution for the optimization problem~\eqref{prob:h} will be a feasible solution for the optimization problem~\eqref{prob:bp}, and the optimal objective value of \eqref{prob:h} will always be less than or equal to the optimal objective value of \eqref{prob:bp}.

Given any linear function $f(b,w)$ that satisfies the condition from line~\eqref{line:f_lb}, 
 the above optimization problem~\eqref{prob:h} is useful because it provides a means to computing approximate solutions for \eqref{prob:bp}. Indeed, the fact that \eqref{prob:h} and \eqref{prob:bp} have the same constraints implies that any optimal solution for \eqref{prob:h} is a feasible solution for \eqref{prob:bp}. Moreover, it is always theoretically possible to choose the function $f(b,w)$ such that every optimal solution for \eqref{prob:h} is an optimal solution for \eqref{prob:bp}, as shown by the following Proposition~\ref{prop:similarity_of_formulations}. We thus conclude that solving the optimization problem~\eqref{prob:h} has the potential to yield high-quality approximate solutions for \eqref{prob:bp}. 

\begin{proposition} \label{prop:similarity_of_formulations}
For any instance of \eqref{prob:bp}, there exists a linear function $f(b,w)$ satisfying the condition from line~\eqref{line:f_lb} such that  the optimal objective value of \eqref{prob:h} is equal to the optimal objective value of \eqref{prob:bp}, and every optimal solution for \eqref{prob:h} is also an optimal solution for \eqref{prob:bp}.
\end{proposition}


The optimization problem~\eqref{prob:h} is ultimately attractive compared to \eqref{prob:bp} from the perspective of computational tractability. Indeed,  we observe that \eqref{prob:h} is a zero-one linear optimization problem over totally unimodular constraints. This implies that the integrality constraints on the decision variables can be relaxed without loss of generality; that is, \eqref{prob:h} can be solved as a linear optimization problem in which each constraints $b^i_t \in \{0,1\}$ are replaced with $0 \le b^i_t \le 1$  \citep[\S 4.2]{conforti2014integer}. The optimization problem~\eqref{prob:h} is thus particularly convenient from an implementation standpoint: as a linear optimization problem, \eqref{prob:h}  can be easily formulated and solved directly by commercial linear optimization solvers such as CPLEX or Gurobi.
 
  In fact, the optimization problem \eqref{prob:h} can be solved \emph{very} efficiently by exploiting the relationship between \eqref{prob:h} 
and the {maximal closure problem}. 
 The maximal closure problem is a problem from combinatorial optimization that has been widely studied in the operations research literature dating back to  \cite{rhys1970selection} and \cite{picard1976maximal}, with applications ranging from  project scheduling to open-pit mining.\footnote{{\color{black}For further background  on the maximal closure problem, we refer the interested reader to \cite{hochbaum200450th}.}} The goal in the maximal closure problem is to find a maximum-weight closure in a directed graph, where a closure is defined as a subset of vertices without edges that leave the subset.   In particular, it follows immediately from  \citet[\S 3]{picard1976maximal} that the linear optimization problem \eqref{prob:h} with $\mathcal{O}(N T^2)$ decision variables and $\mathcal{O}(N T^2)$ constraints is equivalent to computing the maximal closure of a directed graph with $\mathcal{O}(N T^2)$ nodes and $\mathcal{O}(NT(T+N))$ edges.  
 An important algorithmic property is that maximal closure problems can be solved by computing the maximum flow in the directed graph \citep[\S 4]{picard1976maximal}.  Thus, as we will formalize in \S\ref{sec:approx:computation}, the optimization problem \eqref{prob:h} can be solved with running time that is {strongly} polynomial with respect to both the number of simulated sample paths $N$ and the number of time periods $T$.\footnote{{\color{black}We note that Proposition~\ref{prop:similarity_of_formulations} does not contradict our hardness result for the robust optimization problem~\eqref{prob:sro} from Theorem~\ref{thm:hard}, as the right choice of the linear function $f(b,w)$ in Proposition~\ref{prop:similarity_of_formulations} may be difficult to identify.}}

 \subsubsection{Description of Heuristic. } \label{sec:approx:main}  In \S\ref{sec:approx:prelim}, we discussed the high-level motivation and intuition behind our proposed heuristic. In view of that motivation, we now formally describe our proposed heuristic for approximately solving the optimization problem~\eqref{prob:bp}. Specifically, our proposed heuristic consists of solving the optimization problem~\eqref{prob:h} with a particular linear objective function, denoted below by $\bar{f}(b,w)$. 

 To formally describe our proposed heuristic, we require the following notation.  Recall that $T$ is the number of periods in the optimal stopping problem. For each sample path $i$, let $T^i \triangleq  \argmax_{t \in \{1,\ldots,T\}} \; {\color{black} g(t,x^i)}$ be defined as the period in which the sample path $i$ achieves its maximum reward, and if there are multiple optimal solutions, we choose the optimal solution that is smallest. 
 To simplify our notation, let us define $\mathcal{T}^i \triangleq \{ T^i \} \cup \{T\}$ as the set that contains the period in which sample path $i$ achieves its maximum reward as well as the last period of the optimal stopping period. 
 
 With this additional notation, the proposed heuristic obtains approximate solutions $b,w$ for the optimization problem~\eqref{prob:bp} by solving the optimization problem~\eqref{prob:h} with the objective function 
 \begin{align*}
\bar{f}(b,w) \triangleq \frac{1}{N} \sum_{i=1}^N \sum_{t \in \mathcal{T}^i} \sum_{\ell=1}^{L^i_t-1}( \kappa^i_{\ell+1} - \kappa^i_\ell) \left( b^i_t  - w^i_{t\ell} \right). 
 \end{align*}
 In other words, our heuristic for solving the optimization problem~\eqref{prob:bp} consists of solving the optimization problem~\eqref{prob:h} with objective function given by $f(b,w) = \bar{f}(b,w)$. The optimal solution for \eqref{prob:h} with objective function  $\bar{f}(b,w)$ constitutes a feasible solution for \eqref{prob:bp}, since the constraints of \eqref{prob:h} and \eqref{prob:bp} are identical. This solution can thus be transformed into a Markovian stopping rule using the transformations described in Theorem~\ref{thm:reform} and \ref{thm:bp_main}. 
 
 \begin{remark}\label{rem:f_bar}

 We readily observe that the function $\bar{f}(b,w)$ is linear in $b$ and $w$. Moreover, it follows from algebra that $\bar{f}(b,w)$ is less than or equal to the objective function of \eqref{prob:bp} for all feasible solutions of \eqref{prob:bp}; that is, the linear function $\bar{f}(b,w)$ satisfies the condition from line~\eqref{line:f_lb}.\footnote{{\color{black}To see why $\bar{f}(b,w)$ is a lower bound on the objective function of \eqref{prob:bp}, consider any arbitrary vectors $b,w$ that satisfy the constraints of \eqref{prob:bp}. Since feasibility for the optimization problem~\eqref{prob:bp} implies that $b$ is a binary vector, we observe that the equality  $b^i_t(1 - w^i_{t \ell}) = \max \left \{ b^i_t - w^i_{t\ell}, 0 \right \}$ holds for each $i \in \{1,\ldots,N\}$, $t \in \{1,\ldots,T\}$, and $\ell \in \{1,\ldots,| \mathcal{K}^i|\}$. 
Therefore, a lower-bound approximation of the objective function of the optimization problem~\eqref{prob:bp} can be obtained by replacing each term $b^i_t(1 - w^i_{t \ell})$ with $b^i_t - w^i_{t\ell}$ if $t \in \mathcal{T}^i$ and with $0$ if $t \notin \mathcal{T}^i$.} } Thus, we conclude that solving the optimization problem~\eqref{prob:h} with objective function $f(b,w) = \bar{f}(b,w)$ will provide a lower-bound approximation of the optimization problem~\eqref{prob:bp}, and any optimal solution for \eqref{prob:h} will be a feasible solution for \eqref{prob:bp}.
\end{remark}

Our motivations for using the heuristic outlined above are two-fold.  First, we find that the optimization problem~\eqref{prob:h} with the objective function $\bar{f}(b,w)$ is highly tractable from both a theoretical and empirical standpoint on realistic sizes of robust optimization problems. Second, we provide theoretical and empirical evidence that the above heuristic can find high-quality and in some cases optimal solutions for \eqref{prob:bp}. We elaborate on the theoretical aspects of these two motivations in \S\ref{sec:approx:computation} and \S\ref{sec:approx:theory}, and we explore the empirical performance of the heuristic in the subsequent \S\ref{sec:experiments}.

\subsubsection{Algorithms for Implementing Heuristic.} \label{sec:approx:computation}  In \S\ref{sec:approx:prelim},  we argued  generally that the optimization problem~\eqref{prob:h} with any linear objective function $f(b,w)$ can be solved in strongly polynomial time by reducing \eqref{prob:h} to a maximal closure problem. We now formalize those earlier arguments to derive an explicit, strongly polynomial-time algorithm for solving the optimization problem~\eqref{prob:h} in the particular case where $f(b,w) = \bar{f}(b,w)$. Specifically, the main contribution of \S\ref{sec:approx:computation} is the development of an algorithm that achieves the running time that is specified in the following proposition: 
\begin{proposition} \label{prop:approx:tract}
If $f(b,w) = \bar{f}(b,w)$, then \eqref{prob:h} can be solved in $\mathcal{O}(N^2 T(N + T))$ time. 
\end{proposition}
The above proposition establishes the computational tractability of our heuristic by demonstrating that an optimal solution for the optimization problem~\eqref{prob:h} with objective function $f(b,w) = \bar{f}(b,w)$ can be computed with running  time that is polynomial in both the number of simulated sample paths as well as the number of time periods. In particular, if the number of periods $T$ is held constant,  then the running time in Proposition~\ref{prop:approx:tract} scales cubically in the number of simulated sample paths $N$. Such a tractability guarantee is ultimately important from a practical perspective: indeed, in the following  \S\ref{sec:experiments}, we will present numerical experiments which show that our heuristic can run in seconds on realistic problem sizes with over fifty periods and thousands of sample paths.  An outline of the proof of Proposition~\ref{prop:approx:tract} is found throughout the rest of  \S\ref{sec:approx:computation}.

To develop an algorithm with the computation time that is specified in Proposition~\ref{prop:approx:tract}, we begin by deriving a compact reformulation of \eqref{prob:h} when the objective function satisfies $f(b,w) = \bar{f}(b,w)$. The compact reformulation is denoted below as the optimization problem~\eqref{prob:h_bar}. To derive this compact reformulation,  we first show that many of the decision variables in the optimization problem~\eqref{prob:h}  can, without loss of generality, be removed from~\eqref{prob:h}  when the objective function satisfies $f(b,w) = \bar{f}(b,w)$. In particular, our compact reformulation will utilize the following lemma: 
 \begin{lemma} \label{lem:optimal_h_bar}
 If $f(b,w) = \bar{f}(b,w)$, then there exists an optimal solution $b,w$ for \eqref{prob:h} that satisfies  $b^1_T = \cdots = b^N_T = 1$ and satisfies $b^i_t = 0$ for each sample path $i \in \{1,\ldots,N\}$ and period $t \in \{1,\ldots,T\} \setminus \mathcal{T}^i$. 
 \end{lemma}
 The above lemma demonstrates that the optimal values for many of the decision variables of \eqref{prob:h} can be known in advance when the  objective function satisfies $f(b,w) = \bar{f}(b,w)$. Hence, 
Lemma~\ref{lem:optimal_h_bar} implies that each decision variable $b^i_t$ can be fixed to its optimal value and removed from \eqref{prob:h} when $f(b,w) = \bar{f}(b,w)$, except for the decision variables $b^i_t$ in which the sample path $i \in \{1,\ldots,N\}$ and period $t \in \{1,\ldots,T\}$ satisfy $t = T^i$ and $T^i < T$. 
In view of Lemma~\ref{lem:optimal_h_bar}, we now consider the following optimization problem:
 \begin{equation} \label{prob:h_bar} \tag{$\bar{\textnormal{H}}$}
\begin{aligned}
  &\underset{ b,w}{\textnormal{maximize}}&& \frac{1}{N} \sum_{i \in \{1,\ldots,N\}: T^i < T} \sum_{\ell=1}^{L^i_{T^i}-1}( \kappa^i_{\ell+1} - \kappa^i_\ell) \left( b^i_{T^i}  - w^i_{T^i \ell} \right)  + \frac{1}{N} \sum_{i=1}^N \sum_{\ell=1}^{L^i_T - 1} ( \kappa^i_{\ell+1} - \kappa^i_\ell)  \left(1 - w^i_{T \ell} \right)  \\
&\textnormal{subject to}&& \begin{aligned}[t]
&w^{i}_{t\ell} \le w^i_{t,\ell+1} && \textnormal{for all } i \in \{1,\ldots,N\},\; t \in \mathcal{T}^i,\;\ell \in \{1,\ldots,|\mathcal{K}^i|-1\} \\
&b^i_{T^i} \le w^i_{T1}&&\textnormal{for all }i \in \{1,\ldots,N\} \textnormal{ such that }  T^i < T\\
&b^j_{T^j} \le w^i_{t \ell}&& \textnormal{for all } i, j \in \{1,\ldots,N\} \textnormal{ and } t \in \mathcal{T}^i \textnormal{ such that } g(T^j,x^i) = \kappa^i_\ell,\\
&&&T^j \le t, \;  T^j < T, \; \textnormal{and } \mathcal{U}^i_{T^j} \cap \mathcal{U}^j_{T^j} \neq \emptyset\\
&b^i_{T^i} \in \{0,1\} && \textnormal{for all } i \in \{1,\ldots,N\} \textnormal{ such that } T^i < T\\
&w^i_{t\ell} \in \R  && \textnormal{for all } i \in \{1,\ldots,N\}, \; t \in \mathcal{T}^i, \; \ell \in \{1,\ldots,| \mathcal{K}^i|\}.
\end{aligned}
\end{aligned}
\end{equation}
It can be shown using straightforward algebra that the above optimization problem~\eqref{prob:h_bar} is equivalent to the optimization problem~\eqref{prob:h} when the objective function of \eqref{prob:h} satisfies $f(b,w) = \bar{f}(b,w)$ and when the decision variables in  \eqref{prob:h} are restricted to satisfy $b^1_T = \cdots = b^N_T = 1$ and $b^i_t = 0$ for all $i \in \{1,\ldots,N\}$ and $t \in \{1,\ldots,T\} \setminus \mathcal{T}^i$; see proof of the following Lemma~\ref{lem:equivalence_of_h_bar} for details. Hence, it follows from Lemma~\ref{lem:optimal_h_bar} that any optimal solution for \eqref{prob:h_bar} can be transformed into an optimal solution for \eqref{prob:h} when the objective function satisfies $f(b,w) = \bar{f}(b,w)$. We formalize the equivalence of \eqref{prob:h} and \eqref{prob:h_bar} in the following Lemma~\ref{lem:equivalence_of_h_bar}: 
\begin{lemma}\label{lem:equivalence_of_h_bar}
Let $f(b,w) = \bar{f}(b,w)$. Then, the optimal objective value of \eqref{prob:h_bar} is equal to the optimal objective value of \eqref{prob:h}. Moreover, let  $\bar{b},\bar{w}$  denote an optimal solution for \eqref{prob:h_bar}. Then there exists an optimal solution $b,w$ for \eqref{prob:h} that satisfies the following equality for each sample path $i \in \{1,\ldots,N\}$ and period $t \in \{1,\ldots,T\}$:
\begin{align*}
{b}^i_t &= \begin{cases}
\bar{b}^i_t,&\textnormal{if } t = T^i \textnormal{ and } T^i < T,\\
1,&\textnormal{if } t = T,\\
0,&\textnormal{otherwise}. 
\end{cases}
\end{align*}
\end{lemma}


To summarize, we have shown in Lemmas~\ref{lem:optimal_h_bar} and \ref{lem:equivalence_of_h_bar} that solving the optimization problem~\eqref{prob:h} can be reduced to solving the optimization problem~\eqref{prob:h_bar} when the objective function of \eqref{prob:h} satisfies $f(b,w) = \bar{f}(b,w)$. More precisely, we have shown that any optimal solution for \eqref{prob:h_bar} can be transformed into an optimal choice for the decision variables ${b}$ in \eqref{prob:h} when the objective function of \eqref{prob:h} satisfies $f(b,w) = \bar{f}(b,w)$. By applying the transformations described in Theorems~\ref{thm:reform} and \ref{thm:bp_main}, we can construct exercise policies $\mu \equiv (\mu_1,\ldots,\mu_T)$ from the decision variables $b$ whose robust objective value in the robust optimization problem $\widehat{J}_{N,\epsilon}(\mu)$ is greater than or equal to the optimal objective value of \eqref{prob:h_bar}.  

From a computational tractability standpoint, the optimization problem~\eqref{prob:h_bar} is ultimately attractive  because it has significantly fewer decision variables and fewer constraints than the optimization problem~\eqref{prob:h}. Indeed, we observe from inspection that \eqref{prob:h_bar} has $\mathcal{O}(N T)$ decision variables and $\mathcal{O}(NT + N^2)$ constraints, whereas \eqref{prob:h} has $\mathcal{O}(N T^2)$ decision variables and $\mathcal{O}(NT^2 + N^2 T )$ constraints. Because of the same reasoning as given in \S\ref{sec:approx:prelim},  the integrality constraints on the decision variables can also be relaxed without loss of generality; that is, \eqref{prob:h} can be solved as a linear optimization problem in which each constraints $b^i_t \in \{0,1\}$ are replaced with $0 \le b^i_t \le 1$. Thus,   \eqref{prob:h_bar}  can be easily formulated and solved directly by commercial linear optimization solvers.  

We conclude \S\ref{sec:approx:computation} by using the above optimization problem~\eqref{prob:h_bar} to outline the proof of Proposition~\ref{prop:approx:tract}. Specifically, our proof of Proposition~\ref{prop:approx:tract} consists of reformulating \eqref{prob:h_bar} as a maximal closure problem in a directed graph. We then use the efficient maximum-flow algorithm of  \cite{orlin2013max} to solve the maximal closure problem. The remaining details for the proof of Proposition~\ref{prop:approx:tract} are found in Appendix~\ref{appx:approx}.

\subsubsection{Approximation Guarantees.} \label{sec:approx:theory} 
 We conclude \S\ref{sec:approx} by performing a theoretical analysis of the approximation quality of the proposed heuristic~\eqref{prob:h_bar}  for the optimization problem~\eqref{prob:bp}. 
 Our motivation here is to understand whether the proposed heuristic is ever guaranteed to find optimal solutions for \eqref{prob:bp} and, conversely, whether it is ever possible for the  gap between the optimal objective values of \eqref{prob:h_bar} and \eqref{prob:bp} to be arbitrarily large. Our answers to these theoretical questions  are presented below in Propositions~\ref{prop:equivalence_2}, \ref{prop:lb_on_h2:T}, and \ref{prop:lb_on_h2:N}.  

We begin by focusing on optimal stopping problems in which the number of periods is equal to two.  We recall from Theorem~\ref{thm:twoperiods} in \S\ref{sec:exact} that any instance of the robust optimization problem with two periods can be solved exactly in strongly polynomial-time. We will now prove  Theorem~\ref{thm:twoperiods} by establishing that the approximation gap between \eqref{prob:h_bar}  and \eqref{prob:bp} is always equal to zero when the number of periods is equal to two: 
\begin{proposition} \label{prop:equivalence_2}
If $T=2$, then the optimal objective values of \eqref{prob:h_bar} and \eqref{prob:bp} are equal.
\end{proposition}
Our takeaways from the above proposition are two-fold. First, Proposition~\ref{prop:equivalence_2} shows that there indeed exist settings in which the proposed heuristic is guaranteed to find optimal solutions for \eqref{prob:bp}. Second, Proposition~\ref{prop:equivalence_2} in combination with   Theorems~\ref{thm:reform} and \ref{thm:bp_main} implies that solving the robust optimization problem~\eqref{prob:sro} can reduced to solving the optimization problem~\eqref{prob:h_bar} when $T = 2$. Therefore, we observe that the proof of  Theorem~\ref{thm:twoperiods}  follows immediately from Propositions~\ref{prop:approx:tract} and \ref{prop:equivalence_2}. 

We next develop a general bound on the {gap} between the optimal objective values of \eqref{prob:h_bar} and \eqref{prob:bp} that holds for any fixed number of periods $T$. 
For notational convenience, we let $J^{\textnormal{\ref{prob:bp}}} $ denote the optimal objective value of \eqref{prob:bp} and $J^{\textnormal{\ref{prob:h_bar}}}$  denote the optimal objective value of \eqref{prob:h_bar}. 
\begin{proposition} \label{prop:lb_on_h2:T}
$J^{\textnormal{\ref{prob:h_bar}}} \ge \frac{1}{T} J^{\textnormal{\ref{prob:bp}}} $.
\end{proposition}
The above proposition can be viewed as valuable from a theoretical perspective,  as it shows that  the gap between the optimal objective values of \eqref{prob:h_bar} and \eqref{prob:bp} can never be arbitrarily large. Indeed, we recall from Remark~\ref{rem:f_bar} and Lemma~\ref{lem:equivalence_of_h_bar} that the optimal objective value of \eqref{prob:h_bar} is always less than or equal to the optimal objective value of \eqref{prob:bp}.  Thus, Proposition~\ref{prop:lb_on_h2:T} establishes that there exist bounds on the gap between the optimal objective values  of \eqref{prob:h_bar} and \eqref{prob:bp}  that are independent of the choice of the number of  simulated sample paths $N$,  the choice of the robustness parameter $\epsilon$ that controls the size of the uncertainty sets, the reward function, etc.

 Our third theoretical guarantee in  \S\ref{sec:approx:theory}, stated below as Proposition~\ref{prop:lb_on_h2:N}, provides a bound on the {gap} between the optimal objective values of \eqref{prob:h_bar} and \eqref{prob:bp} that does not depend on the number of periods $T$. 
 Our theoretical bound in Proposition~\ref{prop:lb_on_h2:N} will require that the reward function of the NMOS problem satisfies the following regularity condition: 

\begin{assumption} \label{ass:lipschitz_local}
There exists a constant $L > 0$ such that, for each period $t \in \{1,\ldots,T\}$, the reward function satisfies $|g(t,y) - g(t,y')| \le L \| y_t - y_t' \|_\infty$ for all $y \equiv (y_1,\ldots,y_T),y' \equiv (y'_1,\ldots,y'_T) \in \mathcal{X}^T$. 
\end{assumption}
The above assumption stipulates for each period $t \in \{1,\ldots,T\}$ that the reward function $y \mapsto g(t,y)$ is Lipschitz-continuous with respect to the state $y_t \in \mathcal{X}$. The above assumption is relatively mild in applications such as options pricing when the reward function depends only on the current state in each period.\footnote{{\color{black}For example, we observe that Assumption~\ref{ass:lipschitz_local} is satisfied by the reward function in the multi-dimensional barrier option pricing problem from \S\ref{sec:barrier} when the state space in the optimal stopping problem is augmented as $(x_t,q_t) \equiv (\max_{a \in \{1,\ldots,d\}} \xi_{t,a}, 1 - \mathbb{I}  \{ \max_{a \in \{1,\ldots,d\}} \xi_{s,a} \le B(s) \text{ for all } s \in \{1,\ldots,t\}  \})$. }} We will impose Assumption~\ref{ass:lipschitz_local} in the following  Proposition~\ref{prop:lb_on_h2:N} to eliminate pathological situations in which slight perturbations of sample paths lead to drastic changes in reward.  
Once again, for notational convenience, we let  $J^{\textnormal{\ref{prob:h_bar}}}$ and $J^{\textnormal{\ref{prob:bp}}}$ denote the optimal objective values of the optimization problems~\eqref{prob:h_bar} and \eqref{prob:bp}, respectively.
\begin{proposition} \label{prop:lb_on_h2:N} If Assumption~\ref{ass:lipschitz_local} holds, then $J^{\textnormal{\ref{prob:h_bar}}} \ge \frac{1}{\log N + 1}J^{\textnormal{\ref{prob:bp}}} -  2\epsilon L$.
\end{proposition}
Compared to the bound from Proposition~\ref{prop:lb_on_h2:T}, we observe that the bound from Proposition~\ref{prop:lb_on_h2:N} is attractive when the robustness parameter $\epsilon$ is relatively small and when the number of simulated sample paths $N$ is relatively small compared to the number of periods $T$. For example, we observe that $ \frac{1}{\log N + 1} \ll \frac{1}{T}$ will hold whenever the number of sample paths $N$ is subexponential in the number of periods $T$. 

In summary, Propositions~\ref{prop:equivalence_2}, \ref{prop:lb_on_h2:T}, and \ref{prop:lb_on_h2:N} establish under mild and verifiable conditions that the approximation gap between the proposed heuristic~\eqref{prob:h_bar} and the robust optimization problem~\eqref{prob:sro} cannot be arbitrarily large and, in some cases, is guaranteed to be equal to zero. Ultimately, the practical value of the proposed heuristic~\eqref{prob:h_bar} lies in its performance in the context of optimal stopping applications. In the following \S\ref{sec:experiments}, we provide numerical evidence that \eqref{prob:h_bar} can indeed find high-quality stopping rules for realistic stochastic optimal stopping problems in practical computation times.

}

\section{Numerical Experiments} \label{sec:experiments}

In this section, we perform numerical experiments to compare our robust optimization {\color{black}approach} and {\color{black}three} state-of-the-art benchmarks from the literature \citep{longstaff2001valuing,ciocan2020interpretable,desai2012pathwise}. The {\color{black}first two} benchmarks serve as representatives of two classes of approximation methods for stochastic optimal stopping problems (approximate dynamic programming  and parametric exercise policies) which, similarly as the robust optimization {\color{black}approach}, only require the ability to simulate sample paths of the entire sequence of random states.  {\color{black}The third benchmark is representative of state-of-the-art duality-based methods for stochastic optimal stopping problems.} 
All experiments were conducted on  a 2.6 GHz 6-Core Intel Core i7 processor with 16 GB of memory. All methods are implemented in the Julia programming language and solved using the JuMP library and Gurobi  optimization software.

\subsection{A  Simple Non-Markovian Problem} \label{sec:toy}

To demonstrate the value of our robust optimization {\color{black}approach}, we begin by investigating a simple, one-dimensional stochastic optimal stopping problem with a non-Markovian probability distribution. 
The optimal stopping problem of consideration involves a state space which is equal to the real numbers ($\mathcal{X} = \R^1$) and a reward function of $g(t,x) = x_t$ in each period $t \in \{1,\ldots,T\}$.  For any fixed duration $\Delta \in \N$, the joint probability distribution of the stochastic process is given by
\begin{align*}
x_t &\sim \text{Uniform}[0,1] +  \frac{2 \theta }{T} \mathbb{I} \left \{ \theta \le t \le \theta + \Delta \right \} \quad \text{for all }t=1,\ldots,T,
\end{align*}
where the random parameter $\theta \sim \text{Uniform}\{1,2,\ldots,T-\Delta \} $ is selected once per sample path and is unobserved.  Simulated sample paths of this non-Markovian stochastic process $x = (x_1,\ldots,x_T)$ are visualized  in Figure~\ref{fig:nonmarkov_trajectories}. 
We perform numerical experiments on the following methods:

 \begin{itemize*}
 \item \textbf{Robust Optimization (RO)}: The robust optimization {\color{black}approach} is used here to approximate \eqref{prob:main} over the non-Markovian stochastic process $x \equiv (x_1,\ldots,x_T)$, and thus  aims to find the best Markovian stopping rules for the stochastic  optimal stopping problem. The method was run with robustness parameters {\color{black}$\epsilon \in \{0,0.01,\ldots,0.2\}$ for $N < 1000$ and $\epsilon \in \{0,0.01,\ldots,0.1\}$ otherwise, and the robust optimization problem is solved approximately using our proposed heuristic \eqref{prob:h_bar} from \S\ref{sec:approx:computation}}. 
 
 \item \textbf{Least-Squares Regression (LS)}: We implement the method of Longstaff and Schwartz, which employs least-squares regression to  approximate the continuation value function (\ie the expected reward from not stopping) in each period using backwards recursion. To apply this method, we first transform each non-Markovian sample path $x^i \equiv (x^i_1,\ldots,x^i_T)$ into an augmented Markovian sample path of the form $X^i \equiv (X^i_1,\ldots,X^i_T)$ by adding the full history into the state in each period: $X^i_t \triangleq (x^i_1,\ldots,x^i_t,0,\ldots,0)$. 
The regression step requires a specification of basis functions, and we consider the following categories of basis functions.\\[-1em]

 \paragraph{Full-History:} This category of basis functions uses the entire vector $X^i_t$ of states observed up to that point. The basis functions that we consider in this category are \textsc{One} (the constant function 1), \textsc{Prices} (the states observed up to that point, $X^i_t \in \R^T$), and  \textsc{Prices2} (the product of each pair of states observed up to that point, $X^i_t (X^i_t)^\intercal \in \R^{T \times T}$). \\[-1em]
 \paragraph{Markovian:} This category of basis functions  uses only the current state $x^i_t$ in each period. The purpose of considering these basis functions is to analyze the performance when, like the robust optimization {\color{black}approach},  the method of Longstaff-Schwartz is restricted to stopping rules which depend only on the current state in each period. We consider basis functions based on the Laguerre polynomials, where \textsc{Laguerre-k}  is the polynomial $\sum_{\ell=0}^{k} \binom{k}{\ell} \frac{(-1)^\ell}{\ell!} (x^i_t)^\ell$.  
 
 \item \textbf{Tree Method (Tree)}: The method of Ciocan and Mi{\v{s}}i{\'c} approximates the stochastic optimal stopping problem \eqref{prob:main} by restricting the space of exercise policies to decision trees. Like the robust optimization {\color{black}approach}, the tree approach is used to find Markovian stopping rules for the non-Markovian optimal stopping problem \eqref{prob:main}. We apply the method with the same information as the robust optimization {\color{black}approach} at each time period (the current state and the time period) and with a splitting parameter of 0.005. 
 \end{itemize*}

In our computational experiments, we consider a time horizon of $T = 50$ stopping periods with a duration  parameter of $\Delta = 5$. 
{\color{black}A}ll methods were run using simulated training datasets of sizes  $N \in \{10^{2},10^{2.1},\ldots,10^{{\color{black}3.9}},10^{{\color{black}4}}\}$. The  robustness parameters in the ``RO" method were selected using a validation set of size $\bar{N} = 10^3$; see \S\ref{sec:params} for more details.  All methods were evaluated on a common and independent testing dataset of $\tilde{N} = 10^5$ sample paths, and experiments were repeated over 10 replications. 
 
   \afterpage{%
\null
\vfill
\begin{figure}[H]


\centering

\subfloat{%
  \includegraphics[width=0.5\textwidth]{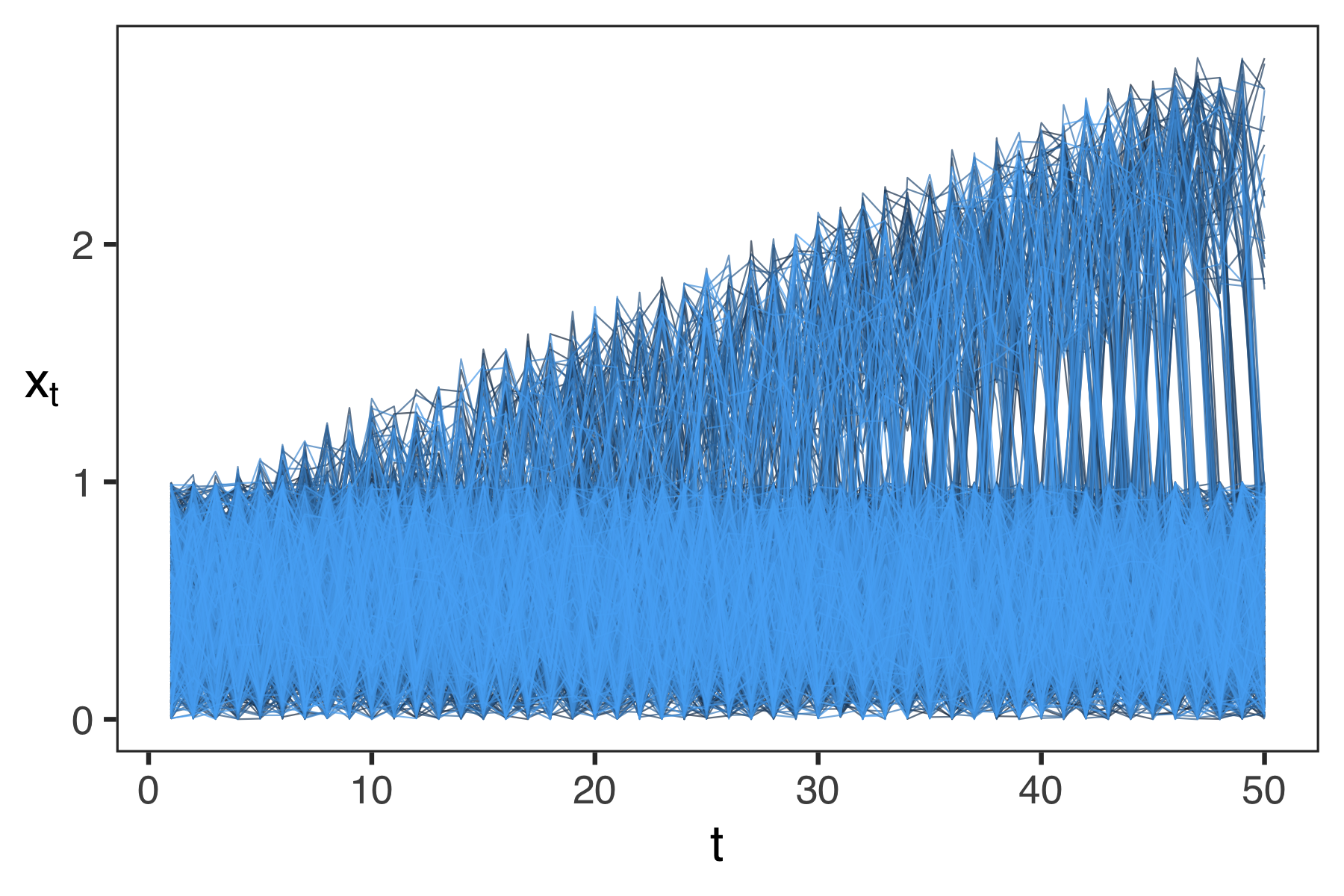}%
}
\subfloat{%
  \includegraphics[width=0.5\textwidth]{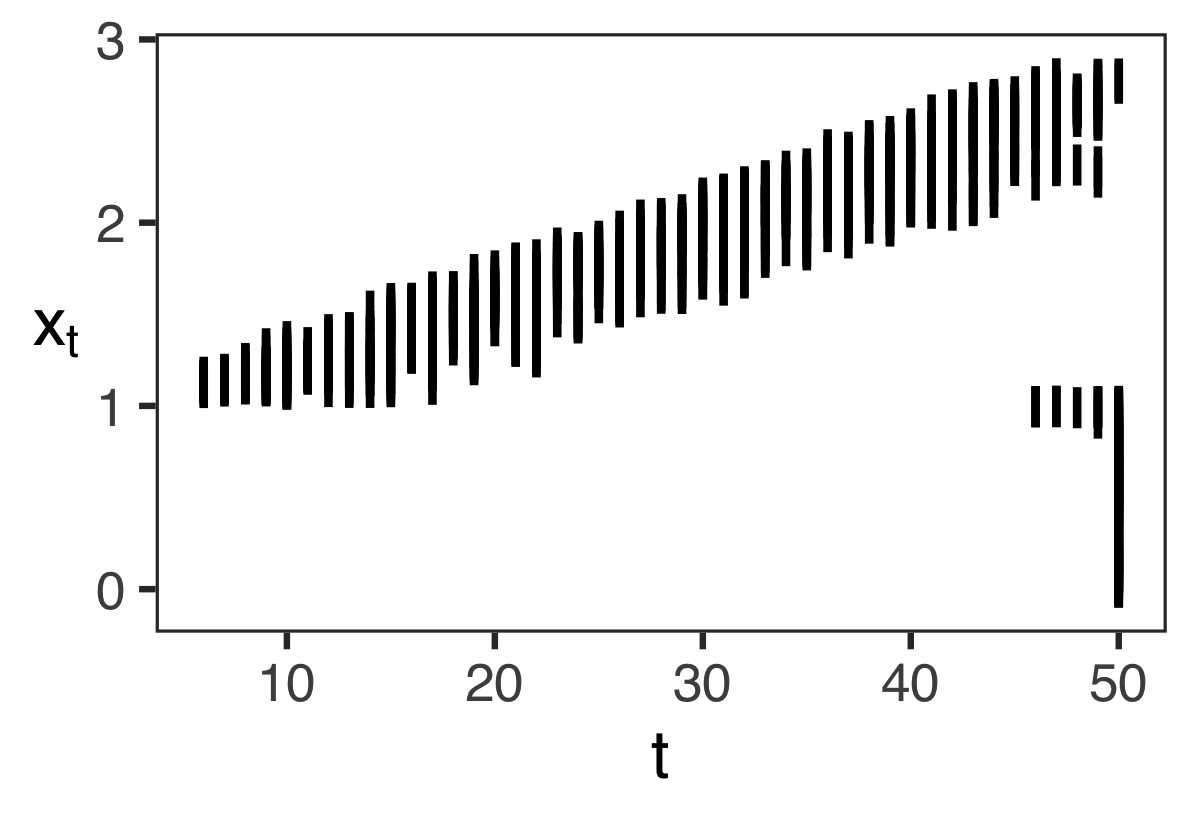}%
}

\smallskip
\caption{The left figure shows  simulated sample paths of the non-Markovian stochastic process with $T = 50$ and $\Delta = 5$.  The right figure shows the exercise policies obtained from solving the robust optimization problem constructed from a training dataset of $N = 10^3$ simulated sample paths and with robustness parameter $\epsilon = 0.1$.}\label{fig:nonmarkov_trajectories}

\subfloat{%
\includegraphics[width=1.0\textwidth]{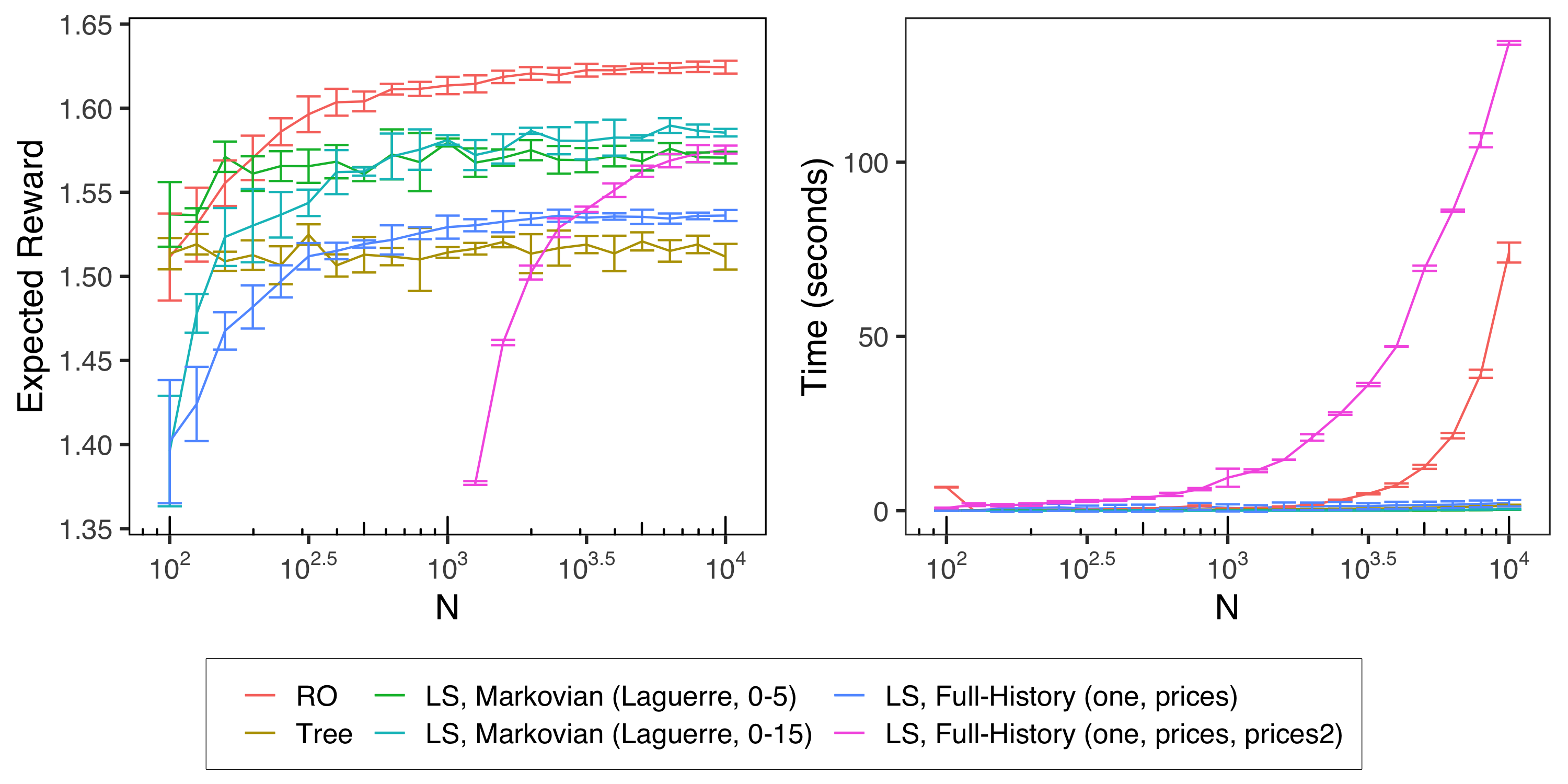}%
}

\smallskip

\caption{The left figure shows the expected reward of stopping rules obtained by the various methods, and the right figure shows the total computation time. 
The results for the robust optimization {\color{black}approach} are shown using the {\color{black}heuristic \eqref{prob:h_bar} from \S\ref{sec:approx:computation}}, the displayed computation time is the total time to construct and approximately solve the robust optimization problems over all choices of the robustness parameter, and the expected rewards are displayed for the robustness parameter that is chosen using a validation set of size $\bar{N} = 10^3$ for each choice of $N$. 
All results for all methods are averaged over 10 replications. {\color{black}The results in the left figure are truncated at 1.35 for visual clarity.} }\label{fig:nonmarkov_results}
\end{figure}%
\null
\vfill
\clearpage
}

 Figures~\ref{fig:nonmarkov_trajectories} and \ref{fig:nonmarkov_results} visualize the {\color{black}Markovian stopping rules} found by the robust optimization {\color{black}approach}, the expected rewards of the stopping rules obtained by the various methods, and the computation times of the various methods. The results of this experiment show that the robust optimization {\color{black}approach} outperforms the other approaches, producing stopping rules with an expected reward of approximately 1.62 from training datasets of $N = 10^3$ simulated sample paths {\color{black} and with a computation time of approximately 0.713 seconds}. We reflect below on the main differences between our robust optimization {\color{black}approach} and the benchmark methods.

For the least-squares regression method with full-history,  selecting a good choice of basis functions is found to be a first-order challenge. 
For example, the basis functions (\textsc{One}, \textsc{Prices}) turn out to be insufficiently rich to provide an accurate approximation of the continuation  function with the full state history. When the number of sample paths is sufficiently large, we expect that the basis functions (\textsc{One}, \textsc{Prices}, \textsc{Prices2}) should provide a better  approximation of the stochastic optimal stopping problem. However, the computational cost resulting from this rich class of basis functions precluded its practicality on sufficiently large training datasets.  In contrast, the robust optimization {\color{black}approach} only searches for Markovian stopping rules, and in doing so, has a reduction in sample complexity. In addition to achieving a significantly better expected reward in this example, the Markovian stopping rules found by the robust optimization {\color{black}approach} are considerably more interpretable than those which have full-history dependance, as illustrated in Figure~\ref{fig:nonmarkov_trajectories}.  

Compared to the tree method, the complex structure of the Markovian stopping rules found by the robust optimization {\color{black}approach} demonstrates the value of algorithms that do not impose parametric restrictions on the exercise policies. 
In theory, decision trees  with sufficient depth are capable of approximating the best Markovian stopping rules to the NMOS problem to arbitrary accuracy  \citep[Theorem 2]{ciocan2020interpretable}. 
However, the greedy heuristic that is proposed by \cite{ciocan2020interpretable} to efficiently optimize over decision trees is not able to find a good approximation of the optimal Markovian stopping rule in this example, even as the number of sample paths is large. 
We note that this issue is not unique to the decision trees \emph{per se}; the computational intractability of optimizing over parametric spaces of Markovian stopping rules has resulted in heuristics for other settings as well \cite[\S 8.2]{glasserman2013monte}. This example provides evidence that our robust optimization {\color{black}approach},  in conjunction with the proposed approximation from \S\ref{sec:approx},  can yield high-quality Markovian stopping rules to complex (low-dimensional) non-Markovian optimal stopping problems. 

Like the robust optimization and tree methods, ``LS, Markovian" aims to find Markovian stopping rules for the non-Markovian optimal stopping problem. In particular, we observe that ``LS, Markovian" was implemented with basis functions of Laguerre polynomials with a rather large degree.\footnote{We were unable to run the  ``LS, Markovian (Laguerre 0-K)" for degrees $K > 15$ due to the numerical precision required to encode the coefficients in the basis functions.} The expected reward of ``LS, Markovian (Laguerre 0-15)" may thus be interpreted as an estimate for the best Markovian stopping rule that one could achieve {using dynamic programming}. However, the optimality of backwards induction no longer holds for exercise policies that depend on the current state when the stochastic process is non-Markovian. 
This explains the inferior performance of this method in this example and motivates the use of methods for addressing non-Markovian optimal stopping problems that optimize directly over the exercise policies in all periods simultaneously. A further discussion on the limitations of dynamic programming in finding Markovian stopping rules for NMOS problems can be found in Appendix~\ref{appx:dp_limit}. 

{\color{black}In Figure~\ref{fig:nonmarkov_best_epsilon}, we present the relationship between the best choice of the robustness parameter for our heuristic \eqref{prob:h_bar} and the number of simulated sample paths $N$. The results in this figure show that the best choice of the robustness parameter $\epsilon$ decreases as the number of simulated sample paths $N$ increases. Although a formal analysis of the relationship between these two quantities is outside the scope of this paper, Figure~\ref{fig:nonmarkov_best_epsilon} provides empirical evidence that only smaller choices of the robustness parameter  need to be considered when using larger numbers of simulated sample paths.   }

In Figure~\ref{fig:nonmarkov_epsilon}, we compare the (in-sample) robust objective values and the expected rewards of stopping rules obtained by our {\color{black}heuristic \eqref{prob:h_bar}} and our exact reformulation \eqref{prob:bp} of the robust optimization problem. The comparison in Figure~\ref{fig:nonmarkov_epsilon} is performed on a smaller instance where $T=20$, $\Delta = 2$, and the robust optimization problem is constructed from a training set of $N = 100$ simulated sample paths. We observe the gap in objective values between the two algorithms is relatively small across choices of the robustness parameter, and the gap is equal to zero when the robustness parameter is set equal to zero. Moreover, the expected rewards of stopping rules obtained from these two algorithms were similarly relatively close across choices of the robustness parameter. We view these results as  promising as they suggest, at least for the present example, that the significantly more tractable optimization problem~{\color{black}\eqref{prob:h_bar}} can provide a close approximation of the robust optimization problem \eqref{prob:sro}.

In summary, the experiments from this subsection reveal our first setting in which the robust optimization {\color{black}approach} is attractive. In non-Markovian optimal stopping problems, there can exist high-quality stopping rules which are Markovian. However, finding such stopping rules is a non-trivial challenge, as the best Markovian stopping rules will not necessarily exhibit a simple structure, and cannot in general be found using dynamic programming. The robust optimization {\color{black}approach} thus provides a practical nonparametric approach for finding them. In view of these observations, the next subsection considers a benchmark optimal stopping problem from the options pricing literature.  

    \afterpage{%
\null
\vfill
   \begin{figure}[H]
\centering
\subfloat{%
\includegraphics[width=0.5\textwidth]{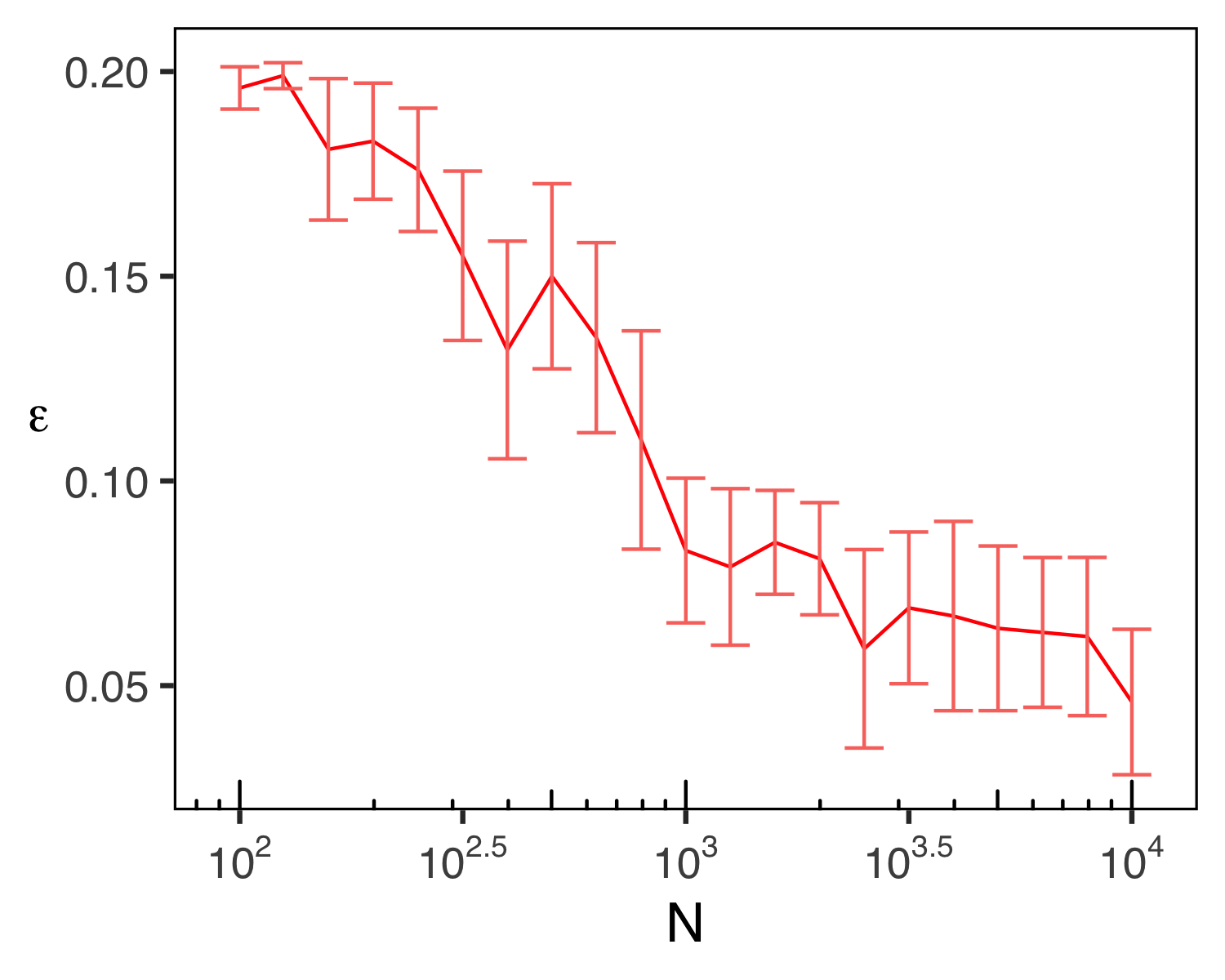}%
}
\medskip
\caption{{\color{black}The figure shows the best choice of the robustness parameter $\epsilon$ for the heuristic~\eqref{prob:h_bar} as a function of the number of training sample paths $N$.  The optimal robustness parameter is chosen using a validation set of size $\bar{N} = 10^3$ for each choice of $N$. 
All results are averaged over 10 replications.  }}\label{fig:nonmarkov_best_epsilon}
\end{figure}

  \begin{figure}[H]
\centering
\subfloat{%
\includegraphics[width=1.0\textwidth]{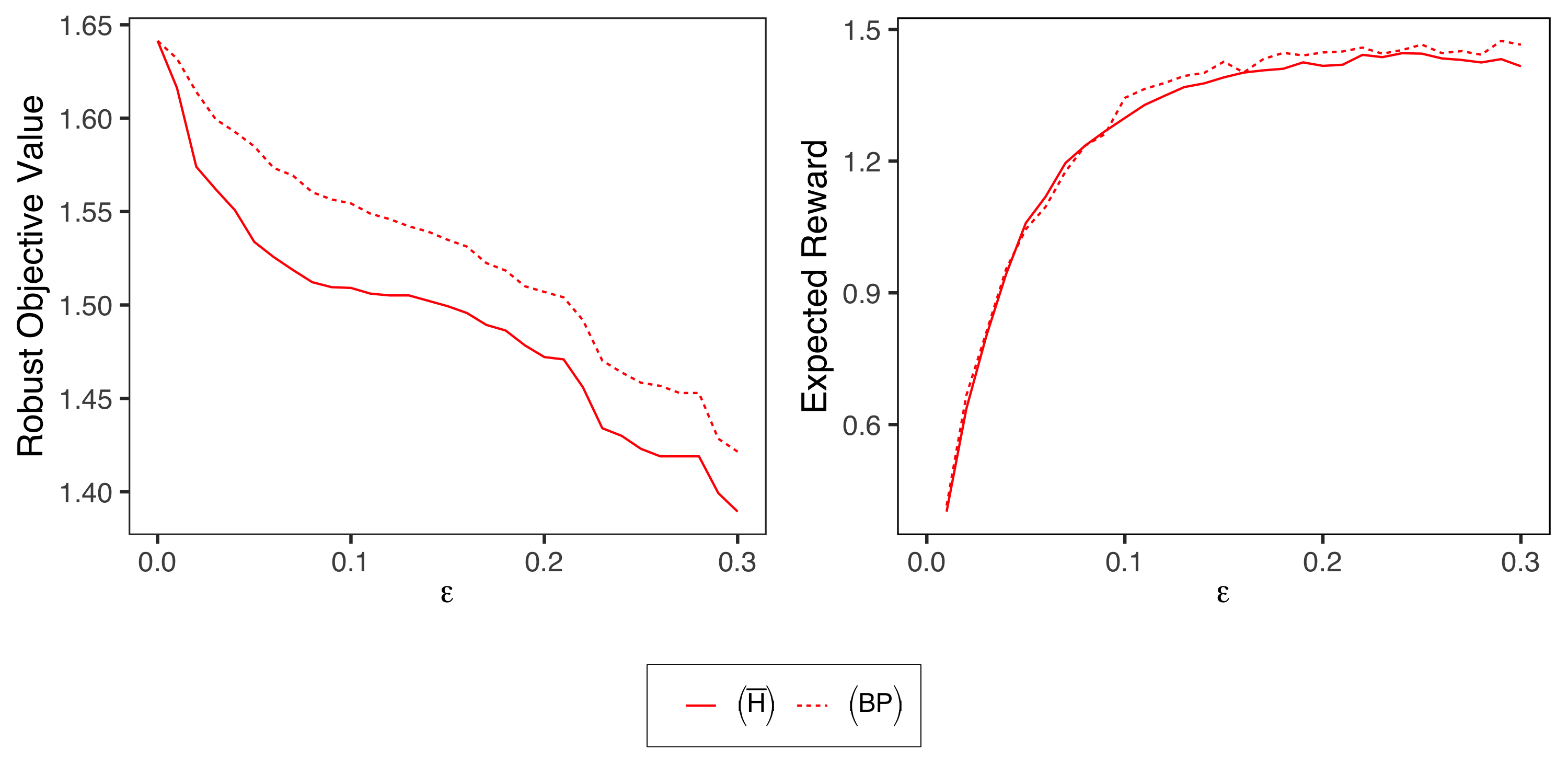}%
}
\medskip
\caption{The left figure shows the robust objective value of stopping rules obtained from solving the {\color{black}heuristic~\eqref{prob:h_bar}} and the exact reformulation~\eqref{prob:bp} of the robust optimization problem \eqref{prob:sro} {\color{black} on an instance with $T = 20$ periods}. The right figure shows the expected rewards of the stopping rules obtained by these two algorithms. The expected rewards are estimated by evaluating the stopping rules on a testing dataset of $10^5$ simulated sample paths. }\label{fig:nonmarkov_epsilon}
\end{figure}
\null
\vfill
\clearpage
}

\subsection{Pricing Multi-Dimensional Barrier Options}  \label{sec:barrier}
Building upon the previous section, we next consider the well-studied problem  of pricing discretely-monitored barrier call options over multiple assets. These are stochastic optimal stopping problems in which the random state  at each period,  $\xi_t \in \Xi \equiv \R^{d}$, is comprised of $d$  non-dividend paying assets. Provided that the option has not been `knocked-out', the reward from stopping on period $t$ is essentially an increasing function of the maximum value of the current assets, $\max_{a \in \{1,\ldots,d\}} \xi_{t,a}$. However, if the maximum value of the assets exceeds a prespecified barrier, the option becomes `knocked-out' and the reward becomes equal to zero for the remainder of the time horizon. 
In this section, we compare various methods on several instances of these stochastic optimal stopping problems in which the barrier threshold  changes over time \citep{kunitomo1992pricing}. 


To formalize the problem setting, let the components of the state vector $\xi_t \in \Xi \equiv \R^{d}$ be referenced through the notation
$\xi_t \equiv (\xi_{t,1},\ldots,\xi_{t,d}),$ 
where $\xi_{t,a} \in \R$ represents the value of asset $a$ at exercise opportunity $t \in \{1,\ldots,T\}$. The exercise opportunities are evenly spaced over a calendar of $Y$ years and thus, defining $\lambda \triangleq Y / T$, it follows that exercise opportunity $t$ occurs at the calendar time $\lambda t$. 
 Let $r \in [0,1)$ be the annualized discount rate, $K \ge 0$ be the {strike price}, $B(t) \triangleq B_0 e^{\delta \lambda t}$ be the {barrier} threshold on exercise opportunity $t$, and $\xi \equiv (\xi_1,\ldots,\xi_T)$ be the sequence of random states. With this notation, the reward function of this stochastic optimal stopping problem is defined as
\begin{align*}
g(t,\xi) & 
= \begin{cases}
e^{- r \lambda t}
 \max \left \{0,\max \limits_{a \in \{1,\ldots,d\}} \xi_{t,a} - K \right \},&\text{if } \max \limits_{a \in \{1,\ldots,d\}} \xi_{s,a} \le B(s) \text{ for all } s \in \{1,\ldots,t\},\\[10pt]
0,&\text{otherwise}.
\end{cases}
\end{align*}
Under the standard Black-Scholes setup, the sequence of random states $\xi \equiv (\xi_1,\ldots,\xi_T)$  obeys a multidimensional geometric Brownian motion where each asset $a \in \{1,\ldots,d\}$ has an initial value of $\bar{x}$, drift equal to the risk-free rate $r$, and annualized volatility equal to $\sigma_a$. The value of asset $a$ at exercise opportunity $t$ is thus given by 
$\xi_{t,a} \triangleq \bar{x} e^{\left(r - \sigma_a^2/2 \right)\lambda t + \sigma_a \lambda W_{t,a} }$,
where each $W_{t,a}$ is a standard Brownian motion  process and the instantaneous correlation of $W_{\cdot,a}$ and $W_{\cdot,a'}$ is equal to $\rho_{a,a'}$. 
We perform numerical experiments on the following methods: 
  
 \begin{itemize*}
 \item \textbf{Robust Optimization (RO)}: The robust optimization {\color{black}approach} is used here to approximate \eqref{prob:main} over the projected  stochastic process $x \equiv (x_1,\ldots,x_T)$, where $x_t \triangleq \max_{a \in \{1,\ldots,d\}} \xi_{t,a}$ denotes the maximum value of the assets in each exercise opportunity $t$. Note that this stochastic process does not include information on whether the option has been knocked-out. The robust optimization {\color{black}approach} is run with robustness parameters $\epsilon \in \{0\} \cup \{0.01,\ldots,0.09\} \cup \{0.1,\ldots,0.9\} \cup \{1,\ldots,10\}$ and solved approximately using {\color{black}our heuristic \eqref{prob:h_bar} from \S\ref{sec:approx:computation}}.

 \item \textbf{Least-Squares Regression (LS)}: The method of Longstaff and Schwartz is applied to stochastic optimal stopping problem~\eqref{prob:main} over the Markovian stochastic process $X \equiv (X_1,\ldots,X_T)$, where the state at each exercise opportunity $X_t \triangleq (\xi_{t,1},\ldots,\xi_{t,d},q_t)$ consists of the values of the assets  at the current exercise opportunity as well as an indicator variable which equals one if the option has been knocked out: $q_t \triangleq 1 - \mathbb{I}  \{ \max_{a \in \{1,\ldots,d\}} \xi_{s,a} \le B(s) \text{ for all } s \in \{1,\ldots,t\}  \}$. We consider the following basis functions:
  \begin{itemize}
\item \textsc{One}: The constant function, 1.
\item \textsc{KOind}: An indicator variable for whether the option was knocked out, $q_t $.
\item \textsc{Prices}: The values of the assets, $\xi_{t,1},\ldots,\xi_{t,d}$.
\item \textsc{PricesKO}: The assets multiplied by the indicator variable, $\xi_{t,1} (1-q_t),\ldots,\xi_{t,d}(1-q_t)$.
\item \textsc{MaxPrice}: The maximum asset value, $x_t \triangleq \max_{a \in \{1,\ldots,d\}} \xi_{t,a}$. 
\item \textsc{Payoff}: The reward of exercising the option, $(1-q_t)e^{- r \lambda t} \max \left \{0,x_t - K \right \}$. 
\end{itemize}

 {\color{black}
 \item \textbf{Pathwise Optimization (PO)}: The duality-based method of Desai et al. is applied to the stochastic optimal stopping problem~\eqref{prob:main} over the Markovian stochastic process $X \equiv (X_1,\ldots,X_T)$, where the state at each exercise opportunity $X_t \triangleq (\xi_{t,1},\ldots,\xi_{t,d},q_t)$ consists of the values of the assets  at the current exercise opportunity as well as an indicator variable which equals one if the option has been knocked out: $q_t \triangleq 1 - \mathbb{I}  \{ \max_{a \in \{1,\ldots,d\}} \xi_{s,a} \le B(s) \text{ for all } s \in \{1,\ldots,t\}  \}$. We consider the same possible basis functions for PO as were considered for LS. 
  }
 
 \item \textbf{Tree Method (Tree)}: The method of Ciocan and Mi{\v{s}}i{\'c} approximates the stochastic optimal stopping problem \eqref{prob:main} by restricting the space of exercise policies to decision trees. Like the robust optimization {\color{black}approach}, the tree approach is used to find Markovian stopping rules for the optimal stopping problem over the projected stochastic process $x \equiv (x_1,\ldots,x_T)$. We apply the method with the same information as the robust optimization {\color{black}approach} at each time period (the current state  $x_t$ and the time period $t$) and with a splitting parameter of 0.005. 
 \end{itemize*}

Our experiments and parameter settings closely parallel those of \citet[\S 5]{ciocan2020interpretable}, albeit with two primary differences.  First, our experiments consider a barrier threshold that changes as a function of time. Second, we perform experiments in which the underlying assets have symmetrical as well as  asymmetrical annualized volatilities (\ie $\sigma_a \neq \sigma_{a'}$ for $a \neq a'$). Our motivations behind these differences are to compare the various methods in settings in which the best Markovian stopping rules may not have a simple structure and in which the low-dimensional projections $x_1,\ldots,x_T$ are not sufficient statistics for the optimal stopping rules \citep[\S 7]{broadie1997valuation}. In other words, our experiments aim to evaluate the performance of the proposed robust optimization {\color{black}approach} in settings where the best Markovian stopping rules are not guaranteed to be optimal stopping rules for the optimal stopping problem.   Similarly as \citet[\S 5]{ciocan2020interpretable}, we perform experiments on various numbers of assets ($d \in \{8,16,32\}$) and various initial prices for the assets ($\bar{x} \in \{90,100,110\}$).

{\color{black}In our implementation, we experimented with a variety of combinations of basis functions, and, for the sake of brevity, we report results only for a subset of combinations with the best performance.} The expected rewards of the stopping rules obtained by the various methods were estimated using a common and independent testing dataset of $\tilde{N} = 10^5$ sample paths, and all experiments were repeated over 10 replications. {{\color{black} The RO method is run in each experiment with $N = 10^3$ training sample paths and $\bar{N} = 10^3$ validation sample paths (see \S\ref{sec:params} for additional details), the LS and Tree methods are run in each experiment with $10^5$ training sample paths, and the PO method is run in each experiment with $10^3$ outer sample paths and $500$ inner sample paths. We use smaller numbers of training and outer sample paths for RO and PO because the computation times of these methods scales relatively quickly with the number of sample paths.  }
%

\begin{figure}[t]
\centering
  \includegraphics[width=0.65\textwidth]{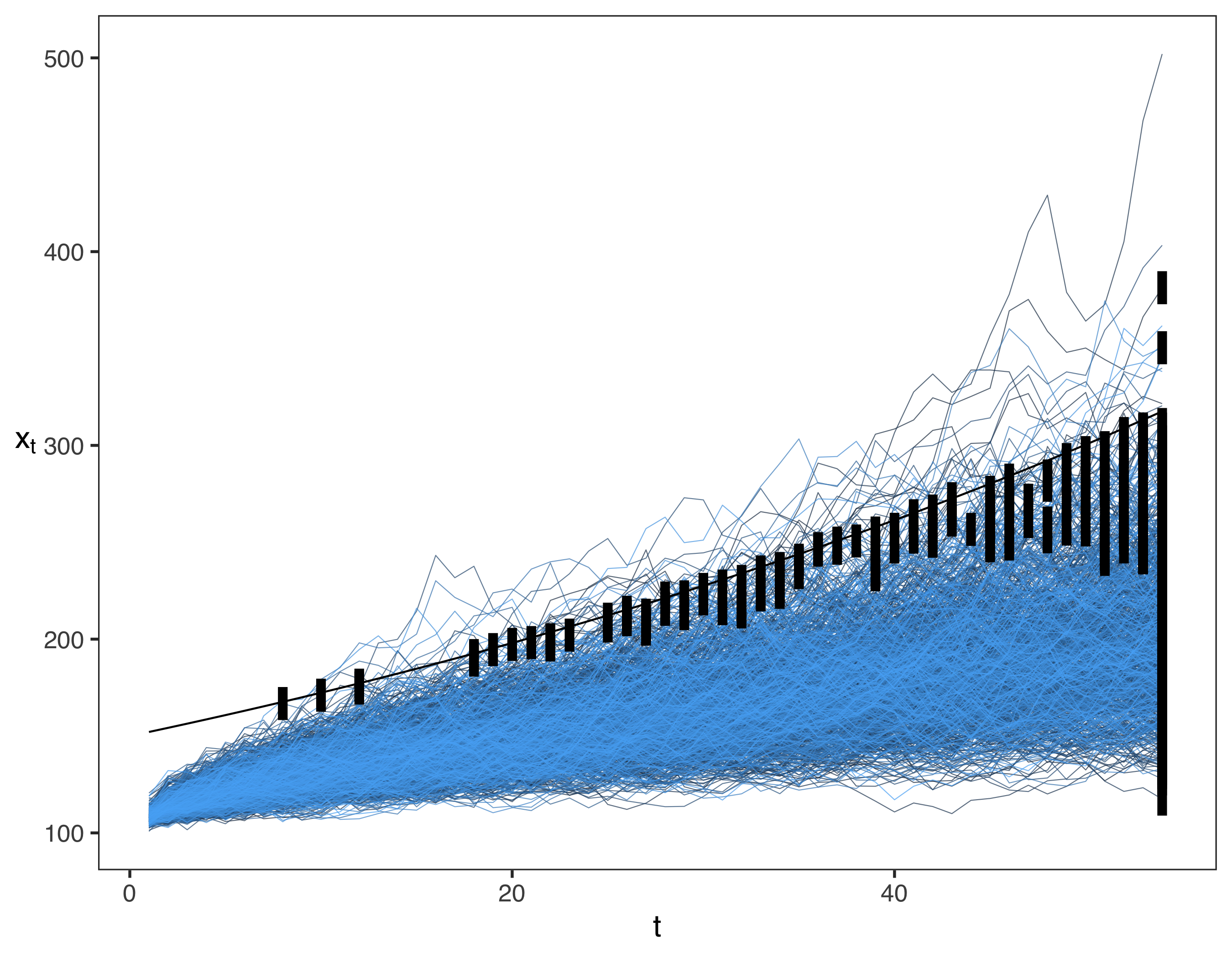}\label{fig:policies_symmetric_single}%
%

\caption{ Plot shows exercise policies obtained from solving a robust optimization problem constructed from a training dataset of  size $N = 10^3$ and with the robustness parameter selected using a validation set of size $\tilde{N} = 10^3$. The problem parameters are $\bar{x} = 100$ and $d = 16$. The remaining parameters are the same as those shown in Table~\ref{tbl:barrier_symmetric}.  The thick black rectangles are the stopping regions and the thin black line shows the barrier threshold.   }\label{fig:barrier_policies_single}
\end{figure}

    \begin{table}[t]
\TABLE{Barrier Option (Symmetric) - Expected Reward. \label{tbl:barrier_symmetric}}
{\centering \small 
\begin{tabular}{cclcccc}
  \hline
  &&& \multicolumn{3}{c}{Initial Price}& \\
 $d$ & Method & Basis functions & $\bar{x} = 90$ & $\bar{x} = 100$ & $\bar{x} = 110$ & \# of Sample Paths \\
 \hline
  8 & RO & maxprice & \textbf{54.88 (0.26)} & \textbf{68.35 (0.13)} & \textbf{75.93 (0.40)} & $10^3$ training, $10^3$ validation \\ 
    8 & LS & one, pricesKO, KOind, payoff & 54.71 (0.09) & 67.22 (0.10) & 73.52 (0.11) & $10^5$ \\ 
    8 & LS & one, pricesKO, payoff & 54.80 (0.09) & 67.29 (0.11) & 73.77 (0.12) & $10^5$ \\ 
    8 & LS & payoff, KOind, pricesKO & 54.71 (0.09) & 67.22 (0.10) & 73.52 (0.11) & $10^5$ \\ 
    8 & LS & pricesKO, payoff & 54.78 (0.09) & 67.28 (0.11) & 73.78 (0.12) & $10^5$ \\ 
    8 & LS & one, prices, payoff & 54.74 (0.09) & 65.57 (0.08) & 68.05 (0.14) & $10^5$ \\ 
    8 & LS & one & 44.96 (0.04) & 57.13 (0.10) & 60.41 (0.09) & $10^5$ \\ 
    8 & LS & one, KOind, prices & 49.38 (0.04) & 62.87 (0.10) & 69.86 (0.08) & $10^5$ \\ 
    8 & LS & one, prices & 48.90 (0.05) & 59.96 (0.12) & 59.53 (0.08) & $10^5$ \\ 
    8 & LS & one, pricesKO & 48.99 (0.05) & 62.19 (0.07) & 69.68 (0.08) & $10^5$ \\ 
    8 & LS & maxprice, KOind, pricesKO & 51.16 (0.07) & 63.37 (0.11) & 69.92 (0.08) & $10^5$ \\ 
    8 & PO & payoff, KOind, pricesKO & 54.47 (0.11) & 65.73 (0.18) & 67.90 (0.20) & $2 \times 10^3$ outer, $500$ inner \\ 
    8 & PO & prices & 52.19 (0.17) & 64.80 (0.16) & 68.11 (0.33) & $2 \times 10^3$ outer, $500$ inner \\ 
    8 & Tree & payoff, time & 54.73 (0.08) & 66.82 (0.12) & 71.30 (0.27) & $10^5$ \\ 
    8 & Tree & maxprice, time & 54.73 (0.08) & 66.82 (0.12) & 71.30 (0.27) & $10^5$ \\ 
   \hline
 16 & RO & maxprice & \textbf{71.00 (0.21)} & \textbf{83.11 (0.36)} & \textbf{82.77 (0.36)} & $10^3$ training, $10^3$ validation \\ 
   16 & LS & one, pricesKO, KOind, payoff & 70.65 (0.07) & 81.25 (0.14) & 81.38 (0.15) & $10^5$ \\ 
   16 & LS & one, pricesKO, payoff & 70.68 (0.08) & 81.40 (0.15) & 81.05 (0.16) & $10^5$ \\ 
   16 & LS & payoff, KOind, pricesKO & 70.65 (0.07) & 81.25 (0.14) & 81.38 (0.15) & $10^5$ \\ 
   16 & LS & pricesKO, payoff & 70.67 (0.08) & 81.41 (0.15) & 81.05 (0.16) & $10^5$ \\ 
   16 & LS & one, prices, payoff & 70.04 (0.10) & 76.50 (0.17) & 70.00 (0.15) & $10^5$ \\ 
   16 & LS & one & 61.48 (0.06) & 70.28 (0.07) & 57.58 (0.09) & $10^5$ \\ 
   16 & LS & one, KOind, prices & 65.28 (0.08) & 77.36 (0.09) & 74.91 (0.12) & $10^5$ \\ 
   16 & LS & one, prices & 64.03 (0.08) & 70.34 (0.08) & 59.75 (0.09) & $10^5$ \\ 
   16 & LS & one, pricesKO & 64.67 (0.07) & 77.15 (0.09) & 74.56 (0.14) & $10^5$ \\ 
   16 & LS & maxprice, KOind, pricesKO & 66.45 (0.08) & 77.62 (0.10) & 74.93 (0.14) & $10^5$ \\ 
   16 & PO & payoff, KOind, pricesKO & 69.85 (0.15) & 77.85 (0.16) & 68.30 (0.17) & $2 \times 10^3$ outer, $500$ inner \\ 
   16 & PO & prices & 68.16 (0.18) & 77.48 (0.16) & 68.44 (0.30) & $2 \times 10^3$ outer, $500$ inner \\ 
   16 & Tree & payoff, time & 70.44 (0.10) & 79.88 (0.12) & 74.71 (0.54) & $10^5$ \\ 
   16 & Tree & maxprice, time & 70.44 (0.10) & 79.88 (0.12) & 74.71 (0.54) & $10^5$ \\ 
   \hline
 32 & RO & maxprice & \textbf{86.15 (0.27)} & \textbf{93.17 (0.29)} & 78.63 (0.48) & $10^3$ training, $10^3$ validation \\ 
   32 & LS & one, pricesKO, KOind, payoff & 85.00 (0.10) & 90.89 (0.15) & \textbf{79.72 (0.16)} & $10^5$ \\ 
   32 & LS & one, pricesKO, payoff & 85.03 (0.09) & 91.03 (0.16) & 78.49 (0.14) & $10^5$ \\ 
   32 & LS & payoff, KOind, pricesKO & 85.00 (0.10) & 90.89 (0.15) & \textbf{79.72 (0.16)} & $10^5$ \\ 
   32 & LS & pricesKO, payoff & 85.03 (0.10) & 91.04 (0.15) & 78.48 (0.14) & $10^5$ \\ 
   32 & LS & one, prices, payoff & 82.21 (0.08) & 81.15 (0.15) & 60.94 (0.16) & $10^5$ \\ 
   32 & LS & one & 76.08 (0.08) & 72.07 (0.19) & 52.26 (0.08) & $10^5$ \\ 
   32 & LS & one, KOind, prices & 80.33 (0.07) & 86.28 (0.12) & 68.40 (0.13) & $10^5$ \\ 
   32 & LS & one, prices & 77.16 (0.11) & 71.90 (0.17) & 54.46 (0.05) & $10^5$ \\ 
   32 & LS & one, pricesKO & 80.00 (0.06) & 86.22 (0.12) & 65.40 (0.20) & $10^5$ \\ 
   32 & LS & maxprice, KOind, pricesKO & 80.91 (0.07) & 86.33 (0.11) & 68.81 (0.15) & $10^5$ \\ 
   32 & PO & payoff, KOind, pricesKO & 83.06 (0.20) & 82.22 (0.32) & 56.67 (0.44) & $2 \times 10^3$ outer, $500$ inner \\ 
   32 & PO & prices & 82.12 (0.29) & 82.50 (0.27) & 56.85 (0.17) & $2 \times 10^3$ outer, $500$ inner \\ 
   32 & Tree & payoff, time & 84.36 (0.16) & 86.68 (0.13) & 61.09 (0.13) & $10^5$ \\ 
   32 & Tree & maxprice, time & 84.36 (0.16) & 86.68 (0.13) & 61.09 (0.13) & $10^5$ \\ 
   \hline
\end{tabular}
}
{Optimal is indicated in bold for each number of assets $d \in \{8,16,32\}$ and initial price $\bar{x} \in \{90,100,110\}$.  
Problem parameters are $T = 54$, $Y = 3$, $r = 0.05$, $K = 100$, $B_0 = 150$, $\delta = 0.25$, $\sigma_a = 0.2$, $\rho_{a,a'} = 0$ for all $a \neq a'$.  }
\end{table}

In Table~\ref{tbl:barrier_symmetric}, we present the expected rewards of stopping rules obtained from the various methods in experiments with symmetrical annualized volatilities.\footnote{{\color{black}Expected rewards and computation times for experiments with asymmetrical annualized volatility can be found in Tables~\ref{tbl:barrier_asymmetric} and \ref{tbl:barrier_time_asymmetric} in Appendix~\ref{appx:numerics}.}} Across the parameter settings, the RO method yields stopping rules with either the best or close-to-best expected reward relative to the alternative state-of-the-art methods. 
{\color{black}The improvements of the robust optimization approach over the state-of-the-art benchmarks} are viewed as particularly encouraging given the practical significance of this class of  options pricing problems. 
In Figure~\ref{fig:barrier_policies_single}, we present visualizations of the Markovian stopping rules produced by the RO method. These visualizations provide interpretability to the stopping rules found by the RO method, which is not possible for the LS method due to the high-dimensional state space. Further visualizations of the stopping rules obtained by the RO method under additional problem parameters are provided in Appendix~\ref{appx:numerics}.

\begin{table}[t]
\TABLE{Barrier Option (Symmetric) -  Computation Times. \label{tbl:barrier_time_symmetric}}
{\centering \small
\begin{tabular}{cclcccc}
  \hline
  &&& \multicolumn{3}{c}{Initial Price}& \\
 $d$ & Method & Basis functions & $\bar{x} = 90$ & $\bar{x} = 100$ & $\bar{x} = 110$ & \# of Sample Paths \\
 \hline
  8 & RO & maxprice &   7.73 (0.33) &   8.63 (1.36) &  13.93 (2.74) & $10^3$ training, $10^3$ validation \\ 
    8 & LS & one, pricesKO, KOind, payoff &   3.35 (0.21) &   3.53 (0.5) &   3.14 (0.22) & $10^5$ \\ 
    8 & LS & one, pricesKO, payoff &   3.24 (0.2) &   2.97 (0.29) &   2.94 (0.23) & $10^5$ \\ 
    8 & LS & payoff, KOind, pricesKO &   3.52 (0.83) &   2.97 (0.28) &   3.03 (0.44) & $10^5$ \\ 
    8 & LS & pricesKO, payoff &   3.04 (0.25) &   2.81 (0.22) &   2.84 (0.19) & $10^5$ \\ 
    8 & LS & one, prices, payoff &   3.08 (0.18) &   2.97 (0.25) &   2.99 (0.1) & $10^5$ \\ 
    8 & LS & one &   0.86 (0.07) &   0.88 (0.03) &   0.93 (0.05) & $10^5$ \\ 
    8 & LS & one, KOind, prices &   3.13 (0.29) &   2.99 (0.32) &   2.92 (0.28) & $10^5$ \\ 
    8 & LS & one, prices &   2.64 (0.19) &   2.63 (0.15) &   2.62 (0.13) & $10^5$ \\ 
    8 & LS & one, pricesKO &   2.72 (0.34) &   2.77 (0.17) &   2.57 (0.17) & $10^5$ \\ 
    8 & LS & maxprice, KOind, pricesKO &   3.02 (0.15) &   2.99 (0.32) &   2.87 (0.08) & $10^5$ \\ 
    8 & PO & payoff, KOind, pricesKO &  20.40 (0.54) &  19.29 (0.69) &  17.58 (0.92) & $2 \times 10^3$ outer, $500$ inner \\ 
    8 & PO & prices &  32.61 (7.86) &  29.01 (8.21) &  23.70 (0.58) & $2 \times 10^3$ outer, $500$ inner \\ 
    8 & Tree & payoff, time &  14.18 (0.43) &  14.21 (0.18) &  19.06 (3.12) & $10^5$ \\ 
    8 & Tree & maxprice, time &  13.72 (0.3) &  13.84 (0.3) &  18.28 (3.13) & $10^5$ \\ 
   \hline
 16 & RO & maxprice &   6.82 (0.19) &  10.07 (1.17) &  28.96 (8.34) & $10^3$ training, $10^3$ validation \\ 
   16 & LS & one, pricesKO, KOind, payoff &   5.49 (0.62) &   5.15 (0.17) &   5.20 (0.24) & $10^5$ \\ 
   16 & LS & one, pricesKO, payoff &   5.29 (1.02) &   4.84 (0.43) &   5.10 (0.33) & $10^5$ \\ 
   16 & LS & payoff, KOind, pricesKO &   4.71 (0.44) &   4.64 (0.3) &   4.61 (0.32) & $10^5$ \\ 
   16 & LS & pricesKO, payoff &   5.17 (0.71) &   4.82 (0.15) &   4.74 (0.25) & $10^5$ \\ 
   16 & LS & one, prices, payoff &   4.86 (0.81) &   4.69 (0.41) &   4.97 (0.1) & $10^5$ \\ 
   16 & LS & one &   1.27 (0.15) &   1.21 (0.08) &   1.22 (0.08) & $10^5$ \\ 
   16 & LS & one, KOind, prices &   6.24 (1.34) &   5.69 (1.28) &   5.50 (0.69) & $10^5$ \\ 
   16 & LS & one, prices &   5.97 (1.17) &   5.51 (1.55) &   5.01 (0.87) & $10^5$ \\ 
   16 & LS & one, pricesKO &   4.30 (0.3) &   4.17 (0.29) &   4.45 (0.31) & $10^5$ \\ 
   16 & LS & maxprice, KOind, pricesKO &   5.12 (0.61) &   4.68 (0.27) &   4.63 (0.28) & $10^5$ \\ 
   16 & PO & payoff, KOind, pricesKO &  40.18 (1.59) &  32.82 (1.38) &  28.44 (1.41) & $2 \times 10^3$ outer, $500$ inner \\ 
   16 & PO & prices &  57.17 (24.52) &  57.76 (25.68) &  44.87 (0.49) & $2 \times 10^3$ outer, $500$ inner \\ 
   16 & Tree & payoff, time &  14.80 (0.28) &  14.19 (0.48) &  25.00 (7.02) & $10^5$ \\ 
   16 & Tree & maxprice, time &  13.88 (0.17) &  13.74 (0.21) &  23.81 (6.46) & $10^5$ \\ 
   \hline
 32 & RO & maxprice &   8.48 (1.28) &  14.93 (3.58) &  91.80 (15.54) & $10^3$ training, $10^3$ validation \\ 
   32 & LS & one, pricesKO, KOind, payoff &  13.73 (1.39) &  13.32 (2.22) &  12.00 (1.35) & $10^5$ \\ 
   32 & LS & one, pricesKO, payoff &  15.42 (3.9) &  15.33 (2.99) &  12.89 (2.17) & $10^5$ \\ 
   32 & LS & payoff, KOind, pricesKO &  13.05 (3.69) &  11.33 (1.27) &  11.66 (2.39) & $10^5$ \\ 
   32 & LS & pricesKO, payoff &  14.63 (2.31) &  14.01 (3.91) &  11.35 (1.54) & $10^5$ \\ 
   32 & LS & one, prices, payoff &  15.18 (5.56) &  14.17 (3.53) &  11.73 (1.72) & $10^5$ \\ 
   32 & LS & one &   2.13 (1.23) &   1.78 (0.93) &   1.67 (0.73) & $10^5$ \\ 
   32 & LS & one, KOind, prices &  13.25 (2.33) &  12.38 (2.44) &  11.95 (1.75) & $10^5$ \\ 
   32 & LS & one, prices &  12.23 (2.71) &  10.83 (1.02) &  10.91 (0.8) & $10^5$ \\ 
   32 & LS & one, pricesKO &  14.55 (4.76) &  11.12 (1.1) &  10.56 (0.6) & $10^5$ \\ 
   32 & LS & maxprice, KOind, pricesKO &  12.05 (1.72) &  10.92 (0.83) &  10.93 (1.42) & $10^5$ \\ 
   32 & PO & payoff, KOind, pricesKO &  91.66 (5.6) &  62.49 (5.23) &  53.57 (1.55) & $2 \times 10^3$ outer, $500$ inner \\ 
   32 & PO & prices & 120.46 (3.27) & 120.76 (1.98) & 113.86 (2.31) & $2 \times 10^3$ outer, $500$ inner \\ 
   32 & Tree & payoff, time &  17.53 (0.23) &  17.42 (0.28) &  10.42 (0.18) & $10^5$ \\ 
   32 & Tree & maxprice, time &  17.01 (0.18) &  16.97 (0.27) &  10.90 (0.31) & $10^5$ \\ 
   \hline
\end{tabular}

}
{Problem parameters are $T = 54$, $Y = 3$, $r = 0.05$, $K = 100$, $B_0 = 150$, $\delta = 0.25$, $\sigma_a = 0.2$, $\rho_{a,a'} = 0$ for all $a \neq a'$.  }
\end{table}

\begin{figure}[t]
\centering
%
%
%
\subfloat{%
  \includegraphics[width=0.5\textwidth]{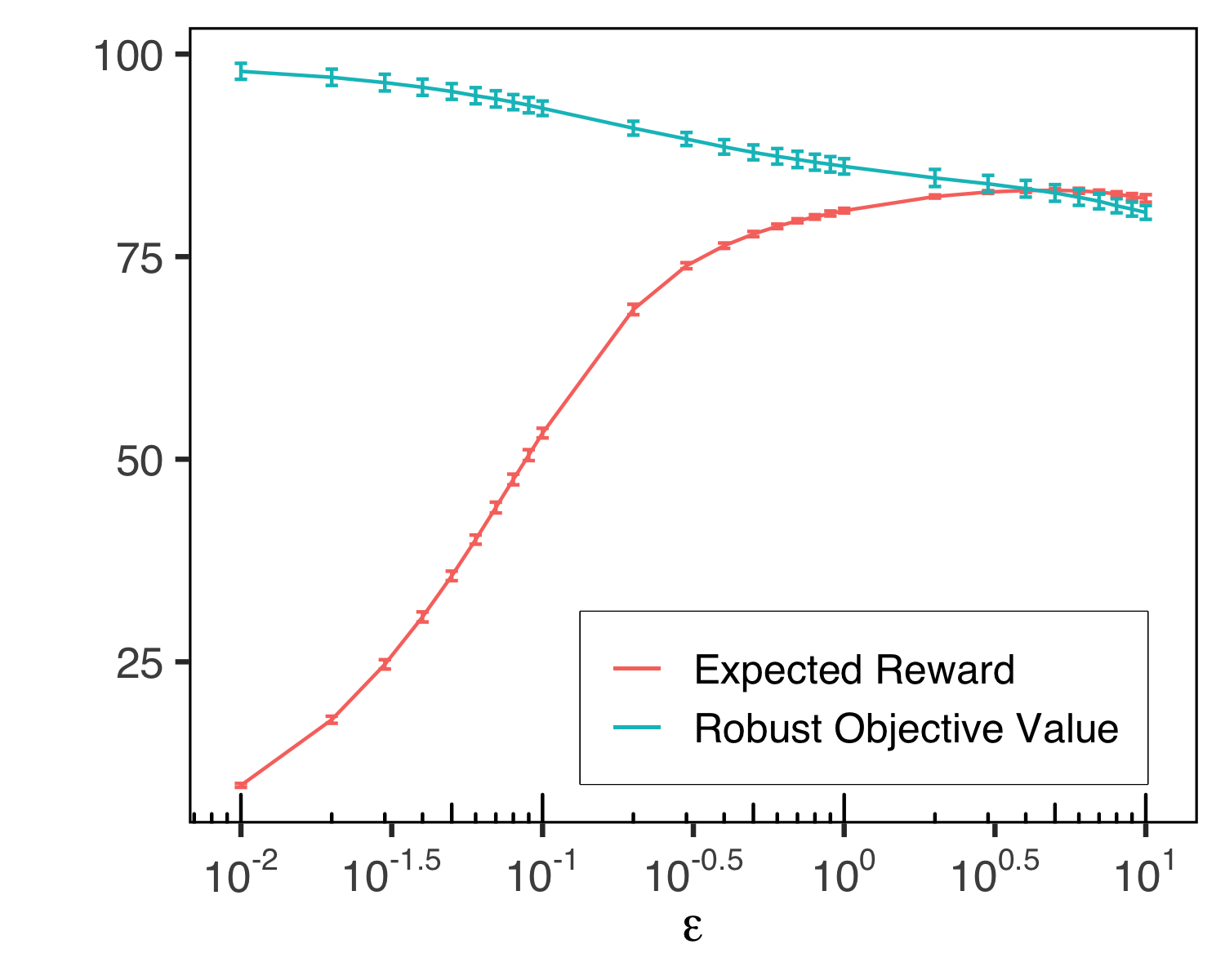}%
}
\subfloat{%
  \includegraphics[width=0.5\textwidth]{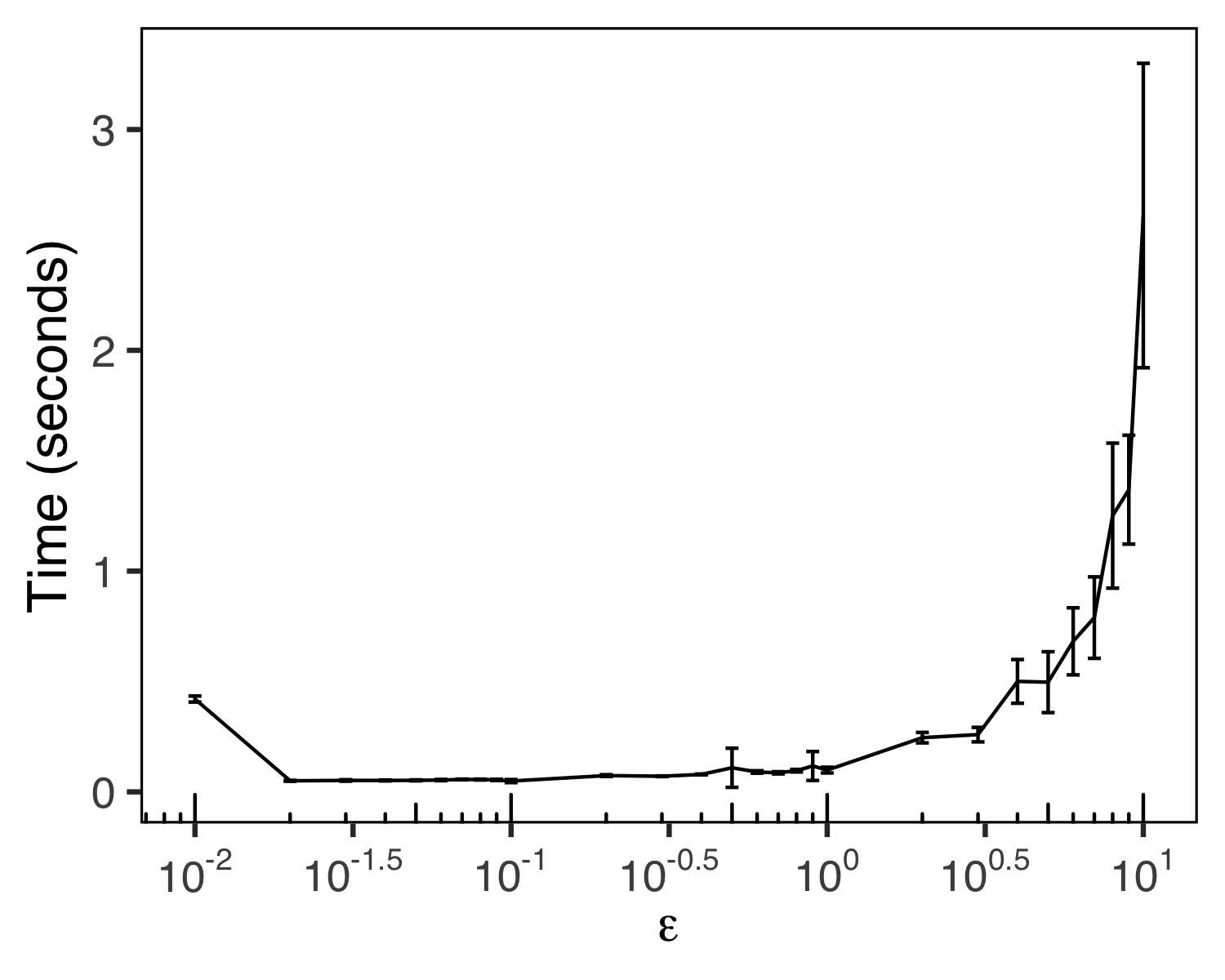}%
}

\caption{ Each plot shows performance metrics of the robust optimization problems constructed from a training dataset of  size $N = 10^3$ as a function of the robustness parameter $\epsilon$.  The left plot shows the robust objective value and expected reward of policies obtained by solving the robust optimization problem. The right plot shows the computation time for solving the robust optimization problem. In both plots, the robust optimization problem is solved approximately using {\color{black}our heuristic algorithm~\eqref{prob:h_bar} from \S\ref{sec:approx:computation}}. The problem parameters are $\bar{x} = 100$ and $d = 16$, and the remaining parameters are the same as those shown in Table~\ref{tbl:barrier_symmetric}.  }\label{fig:barrier_symmetric_epsilon_single}

\end{figure}

In Table~\ref{tbl:barrier_time_symmetric}, we present the computation times of the various methods. We observe from Table~\ref{tbl:barrier_time_symmetric} that the computation times of the RO method are highly competitive compared with the computation times of the alternative methods when the initial price is $\bar{x} = 90$ or $\bar{x} = 100$. Interestingly, the computation time of the RO method increases significantly when the initial price is $\bar{x} = 110$. To understand why this is the case, we recall that the computation time for the RO method is equal to the sum of the computation times for solving the heuristic~\eqref{prob:h_bar} over all considered choices of the robustness parameter.  Moreover,  in Figure~\ref{fig:barrier_symmetric_epsilon_single}, we show that the computation time of solving \eqref{prob:h_bar} is highly dependent on the robustness parameter; indeed, most of the computation time for the RO method comes from  solving the robust optimization problem with unnecessarily large choices of the robustness parameter. Finally, we show in Table~\ref{tbl:barrier_epsilon_symmetric} that the best choice of the robustness parameter is smallest in the numerical experiments with $\bar{x} = 110$, even though the set of robustness parameters remains constant across all of the numerical experiments. From these results, we conclude that a dynamic search over the space of robustness parameters has the potential to significantly reduce the computation times of the RO method, particularly the case where $\bar{x} = 110$. Additional numerical experiments which show the relationship between the robustness parameter and the robust optimization {\color{black}approach} are provided in Appendix~\ref{appx:numerics}.

\begin{table}[t]
\TABLE{Barrier Option (Symmetric) - Best Choice of Robustness Parameter. \label{tbl:barrier_epsilon_symmetric}}
{\centering \small 
\begin{tabular}{cccc}
  \hline
  & \multicolumn{3}{c}{Initial Price} \\
 $d$ & $\bar{x} = 90$ & $\bar{x} = 100$ & $\bar{x} = 110$\\
 \hline
  8 & 8.0 (2.45) & 6.7 (1.25) & 3.7 (1.06) \\ 
   16 & 8.3 (1.89) & 5.5 (1.27) & 2.8 (1.03) \\ 
   32 & 6.8 (1.32) & 3.4 (0.84) & 2.6 (0.70) \\ 
   \hline
\end{tabular}
}
{Best choice of robustness parameter $\epsilon$ found using the validation method from \S\ref{sec:params} in the robust optimization problems constructed from a training dataset of  size $N = 10^3$ and validation set of size $\bar{N} = 10^3$.  The remaining parameters are the same as those shown in Table~\ref{tbl:barrier_symmetric}. }
\end{table}


\section{Conclusion}
Over the past two decades, dynamic robust optimization has experienced a surge of algorithmic advances.  Until now, these advances from robust optimization have not been harnessed to develop algorithms for stochastic dynamic optimization problems with known probability distributions. In this paper, we showed the  value of bridging these traditionally separate fields of research, in application to the classical and widely-studied problem of optimal stopping. In this context, we devised new and theoretically-justified algorithms for non-Markovian optimal stopping problems and highlighted the performance of these new algorithms on stylized and well-studied problems from options pricing. Along the way, we also developed novel theoretical and computational results for solving dynamic robust optimization problems which average over multiple uncertainty sets. We believe this work takes a meaningful step towards broadening the impact of robust optimization to address stochastic dynamic optimization problems of importance to industry. 

\ACKNOWLEDGMENT{The author thanks Vivek Farias for his feedback during the early stages of this project, and Velibor Mi\v{s}i\'{c}  and Selvaprabu Nadarajah for comments on early versions of the paper.  The author also thanks the anonymous associate editor and two referees for constructive comments and suggestions
that helped improve the paper. Finally, the author is grateful to the authors of the paper \cite{ciocan2020interpretable} for making their high-quality code publicly available.     }
\bibliographystyle{informs2014} 
\bibliography{optimal_stopping_robust_optimization} 
\ECSwitch


\AppendixTitle{Technical Proofs and Additional Results}

\begin{APPENDICES}

\vspace{1em}

 \setlength{\parskip}{1em}

\SingleSpacedXI
\section{Limitations of DP for Finding Markovian Stopping Rules in Non-Markovian Optimal Stopping Problems} \label{appx:dp_limit}

In the context of non-Markovian optimal stopping problems, a subtle but important challenge is that the exercise policies  which define the best Markovian stopping rule cannot be found in general using backwards recursion. Intuitively, this problem arises because Bellman's dynamic programming equations no longer hold when the stochastic process is non-Markovian. This fact is illustrated numerically in \S\ref{sec:toy} and motivates the development of methods which optimize over the exercise policies in all time periods simultaneously.  For the sake of completeness, we provide the following example in which there is a Markovian stopping rule that is optimal for the non-Markovian stopping problem but is not obtained using backwards recursion.  
\begin{example}\label{ex:dp}
Consider a three-period optimal stopping problem with a one-dimensional non-Markovian stopping process that obeys the following probability distribution:
\begin{align*}
\Prb \left(x_1 = 3, x_2 = 2, x_3 = 1 \right) = \frac{2}{3}; \quad \Prb \left(x_1 = 1, x_2 = 2, x_3 = 3 \right) = \frac{1}{3}.
\end{align*}
Let the reward of stopping on each period $t$ be equal to the current state $x_t$, and recall that our goal is to find a stopping rule which maximizes the expected reward. We observe that there is an optimal stopping rule for this non-Markovian stopping problem that is a Markovian stopping rule, defined by exercise policies $\mu_1^*(x_1) = \textsc{Stop}$ if and only if $x_1 = 3$, $\mu_2^*(x_2) = \textsc{Continue}$ for all $x_2$, and $\mu_3^*(x_3) = \textsc{Stop}$ for all $x_3$. Applying this stopping rule yields an expected reward of $\sfrac{2}{3} \times 3 + \sfrac{1}{3} \times 3 = 3$.  

We now show that the best Markovian stopping rule which is obtained using backwards recursion will have a strictly lower expected reward. Indeed, assume that the reward from not stopping on any period is equal to zero. Then, starting on the last period, it is clear from the dynamic programming principle that the optimal exercise policy for the last period is $\mu_3^{\textnormal{DP}}(x_3) = \textsc{Stop}$. Hence, conditioned on $x_2 = 2$, the expected reward from stopping on the third period is $1 \times \Prb(x_3 = 1 \mid x_2 = 2)  + 3 \times \Prb(x_3 = 3 \mid x_2 = 2) = 1 \times \sfrac{2}{3} + 3 \times \sfrac{1}{3} = \sfrac{5}{3}$. Because the reward from stopping on the second period (2) is greater than the conditional expected reward of not stopping on the second period ($\sfrac{5}{3}$), the dynamic programming principle says that the exercise policy  in the second period should be chosen to satisfy $\mu_2^{\textnormal{DP}}(2) = \textsc{Stop}$. Finally, unfolding to the first period, we conclude that the exercise policy obtained from dynamic programming is $\mu_1^{\textnormal{DP}}(x_1) = \textsc{Stop}$ if and only if $x_1 = 3$. All together, the Markovian stopping rule for the non-Markovian optimal stopping problem that is obtained using backwards recursion yields an expected reward of $\sfrac{2}{3} \times 3  + \sfrac{1}{3} \times 2 = 2.6\bar{6}$. 
 \halmos
\end{example}

{\color{black}
\section{Comparison of Robust Optimization Formulations} \label{appx:motivation_for_sro}
As described in \S\ref{sec:rob}, our formulation of the robust optimization problem~\eqref{prob:sro} deviates from that of \eqref{prob:sro_orig}, which followed from \citetalias{bertsimasdata}. In this section, we show that all of the main results from \S\ref{sec:opt} and our characterization of optimal Markovian stopping rules in \S\ref{sec:main} also hold for formulation~\eqref{prob:sro_orig} with only minor modifications to the proofs. We will then discuss the advantages of using formulation \eqref{prob:sro} in the context of our algorithmic techniques that are developed in \S\ref{sec:algorithms}. 

\subsection{{\color{black}The Relationship Between \eqref{prob:sro} and \eqref{prob:sro_orig}}}
Before proceeding further, let us develop intuition for the relationship between the two formulations \eqref{prob:sro} and \eqref{prob:sro_orig} by comparing them in a simple example.  Speaking informally, the following Example~\ref{example:two_formulations} shows that if the reward functions $g(1,\cdot),\ldots,g(T,\cdot)$ are well behaved functions of stochastic process, and if the radius of the uncertainty sets is small, then the two robust optimization formulations are essentially equivalent.  As we will see afterwards, this intuition will extend to general classes of reward functions.

\begin{example} \label{example:two_formulations}
Consider an optimal stopping problem in which the state space is one-dimensional and the reward functions satisfy $g(t,x) = x_t$ for all periods $t$. In this case, we observe that the following equalities hold:
\begin{align}
\eqref{prob:sro_orig} &= \sup_\mu \frac{1}{N} \sum_{i=1}^N \inf_{y \in \mathcal{U}^i} g(\tau_\mu(y), y) \notag \\
 &= \sup_\mu \frac{1}{N} \sum_{i=1}^N \inf_{y \in \mathcal{U}^i} \sum_{t=1}^T  y_t \mathbb{I} \left \{ \tau_\mu(y) = t \right \}\notag \\
  &= \sup_\mu \frac{1}{N} \sum_{i=1}^N \inf_{y \in \mathcal{U}^i}\sum_{t=1}^T  \left( (y_t  - x_t^i) + x^i_t \right)  \mathbb{I} \left \{ \tau_\mu(y) = t \right \}. \label{line:dabadee}
\end{align}
Indeed, the first equality follows from the definition of \eqref{prob:sro_orig}, the second equality follows from the fact that $g(t,x) = x_t$ for all periods $t$, and the third equality follows from algebra. 

For notational convenience, let $\widehat{J}_{N,\epsilon}(\mu)$ denote the objective value of the robust optimization formulation \eqref{prob:sro} corresponding to exercise policies $\mu$. Moreover, we recall from the definition of the uncertainty sets in \S\ref{sec:rob} that the inequality $| y_t - x^i_t | \le \epsilon$ is satisfied for all sample paths $i \in \{1,\ldots,N\}$, periods $t \in \{1,\ldots,T\}$, and $y \equiv (y_1,\ldots,y_T) \in \mathcal{U}^i$. Therefore, it follows from line~\eqref{line:dabadee}  and algebra that 
\begin{align*}
\eqref{prob:sro_orig} &\le  \sup_\mu \frac{1}{N} \sum_{i=1}^N \inf_{y \in \mathcal{U}^i}\sum_{t=1}^T  \left( \epsilon + x^i_t \right)  \mathbb{I} \left \{ \tau_\mu(y) = t \right \} =  T\epsilon + \sup_\mu \widehat{J}_{N,\epsilon}(\mu) = 
 T\epsilon + \eqref{prob:sro}, \text{ and}\\
\eqref{prob:sro_orig} &\ge  \sup_\mu \frac{1}{N} \sum_{i=1}^N \inf_{y \in \mathcal{U}^i}\sum_{t=1}^T  \left( - \epsilon + x^i_t \right)  \mathbb{I} \left \{ \tau_\mu(y) = t \right \}  =  -T\epsilon + \sup_\mu \widehat{J}_{N,\epsilon}(\mu) 
= -T\epsilon + \eqref{prob:sro}. 
\end{align*}
We have thus shown that if the robustness parameter $\epsilon \ge 0$ is small, then the two formulations will be close to one another with respect to optimal objective value.  Moreover, it follows by identical reasoning that for all exercise policies $\mu$, 
\begin{align*}
\left| \widehat{J}_{N,\epsilon}(\mu) - \frac{1}{N} \sum_{i=1}^N \inf_{y \in \mathcal{U}^i} g(\tau_\mu(y), y) \right| \le T\epsilon.
\end{align*}
Therefore, we conclude that if the robustness parameter $\epsilon \ge 0$ is small, then the objective values of the two formulations will be close to one another, uniformly over the space of all exercise policies. 
\halmos
\end{example}

The above example is insightful because it reveals, at least for problems with linear reward functions, that the robust optimization formulations \eqref{prob:sro} and \eqref{prob:sro_orig} become essentially equivalent in the asymptotic regime in which the robustness parameter converges to zero. We will see shortly that the same intuition from Example~\ref{example:two_formulations} will extend to a broad class of reward functions. 

We emphasize that the above intuition, of course, does \emph{not imply} that the two formulations \eqref{prob:sro} and \eqref{prob:sro_orig} are guaranteed to have identical performance on any fixed collection of simulated sample paths. Indeed, we do not preclude the possibility that one of the two formulations \eqref{prob:sro} and \eqref{prob:sro_orig} may have better finite-sample performance than the other formulation in finding Markovian stopping rules which perform well with respect to \eqref{prob:main}.  Fortunately, in the particular setting of the present paper in which the joint probability distribution of the underlying stochastic problem is \emph{known}, any potential differences in finite-sample performance between the two formulations can be decreased arbitrarily by simulating larger number of sample paths when constructing the robust optimization problem and then choosing a smaller robustness parameter.

\subsection{{\color{black}Optimality Guarantees for \eqref{prob:sro_orig}}}

In view of the above intuition, we proceed to prove that the convergence guarantees from \S\ref{sec:opt} 
will also hold if we opted instead to use formulation~\eqref{prob:sro_orig}. Our convergence guarantees  in the following Theorem~\ref{thm:uniform_formulations} will be developed by extending the intuition from  Example~\ref{example:two_formulations}, that is, by showing that the gap between the objective values for formulations~\eqref{prob:sro_orig} and \eqref{prob:sro} converges to zero, almost surely, uniformly over the space of all exercise policies. 
We establish these convergence guarantees for formulation~\eqref{prob:sro_orig} when the reward function in the optimal stopping problem satisfies the following assumption: 
\begin{assumption}\label{ass:weird_v2}
 $\lim \limits_{\epsilon \to 0} \Omega_\epsilon(x) = 0 $ almost surely, where 
 \begin{align*}
 \Omega_\epsilon(x) \triangleq  \max_{t \in \{1,\ldots,T\}} \left \{ \sup \limits_{y \in \mathcal{X}^T: \| y - x \|_\infty \le \epsilon} g(t,y) - g(t,x)\right \}.
 \end{align*}
\end{assumption}
We readily observe that above assumption, in conjunction with Assumption~\ref{ass:weird}, is equivalent to requiring that the reward functions $g(1,\cdot),\ldots,g(T,\cdot)$ are continuous functions of the stochastic process almost surely. 
Hence, Assumption~\ref{ass:weird_v2} can be viewed as a mild assumption that is frequently satisfied in the applications of optimal stopping to options pricing. 
 In the following theorem, as in \S\ref{sec:opt}, we let $\widehat{J}_{N,\epsilon}(\mu)$  denote the objective value of the robust optimization problem~\eqref{prob:sro} corresponding to exercise policies $\mu$.

\begin{theorem}[Uniform convergence for robust optimization formulations] \label{thm:uniform_formulations}
If Assumptions~\ref{ass:bound} and \ref{ass:weird_v2} hold, then the gap in objective values between formulations \eqref{prob:sro} and \eqref{prob:sro_orig} converges to zero, almost surely, uniformly over the space of all exercise policies:
\begin{align*}
\limsup_{\epsilon \to 0} \limsup_{N \to \infty}  \sup_\mu \left| \widehat{J}_{N,\epsilon} (\mu)  - \frac{1}{N} \sum_{i=1}^N \inf_{y \in \mathcal{U}^i} g(\tau_\mu(y),y)  \right| = 0 \quad \textnormal{almost surely}.
\end{align*}
\end{theorem}
\begin{proof}{Proof.}
Consider any arbitrary sample path $x^i = (x^i_1,\ldots,x^i_T)$ and robustness parameter $\epsilon \ge 0$.  We first observe from algebra that 
\begin{align}
 \left| \inf_{y \in \mathcal{U}^i }   g(\tau_\mu(y),y) - \inf_{y \in \mathcal{U}^i }   g(\tau_\mu(y),x^i)  \right| &\le  \sup_{y \in \mathcal{U}^i}  \left|   g(\tau_\mu(y),y) -    g(\tau_\mu(y),x^i)  \right|  \notag \\
 &\le \max_{t \in \{1,\ldots,T\}}  \sup_{y \in \mathcal{U}^i}  \left|   g(t,y) -    g(t,x^i)  \right| \notag \\
 &\le \max_{t \in \{1,\ldots,T\}}   \left \{    \sup_{y \in \mathcal{X}^T:  \| y - x^i \|_\infty \le \epsilon} g(t,y) -    g(t,x^i)  \right \} \notag \\
 &= \Omega_\epsilon(x^i),\label{line:happydays}
 \end{align}
 where the last inequality follows from the definition of the uncertainty sets. 
 Therefore, 
 \begin{align*}
&\limsup_{\epsilon \to 0} \limsup_{N \to \infty}  \sup_\mu \left| \widehat{J}_{N,\epsilon} (\mu)  - \frac{1}{N} \sum_{i=1}^N \inf_{y \in \mathcal{U}^i} g(\tau_\mu(y),y)  \right| \\
 &\le  \limsup_{\epsilon \to 0} \limsup_{N \to \infty}  \sup_\mu \frac{1}{N} \sum_{i=1}^N \left| \inf_{y \in \mathcal{U}^i }   g(\tau_\mu(y),y) - \inf_{y \in \mathcal{U}^i }   g(\tau_\mu(y),x^i)  \right| \\
& \le \limsup_{\epsilon \to 0} \limsup_{N \to \infty} \frac{1}{N} \sum_{i=1}^N \Omega_\epsilon(x^i) \\
& =  \limsup_{\epsilon \to 0} \Exp \left[ \Omega_\epsilon(x) \right]  \quad \textnormal{almost surely}\\
&= 0. 
 \end{align*}
where the first inequality follows from triangle inequality, the second inequality follows from \eqref{line:happydays}, the first equality follows from the strong law of large numbers, and the second equality follows from the dominated convergence theorem and Assumption~\ref{ass:weird_v2}. We note that the strong law of large numbers and dominated convergence theorem can be applied in both cases because of the boundedness of the reward function (Assumption~\ref{ass:bound}).  This concludes the proof of Theorem~\ref{thm:uniform_formulations}. 
\halmos \end{proof}
Using standard proof techniques from stochastic programming (see, for example, \citet[\S5.1.1]{shapiro2014lectures}), Theorem~\ref{thm:uniform_formulations} readily implies that the alternative formulation \eqref{prob:sro_orig} enjoys identical convergence guarantees as those in Theorems~\ref{thm:conv:asympt}-\ref{thm:conv:ub} under Assumptions~\ref{ass:weird}-\ref{ass:bound} and \ref{ass:weird_v2}.

\subsection{{\color{black}Characterization of Optimal Markovian Stopping Rules  for \eqref{prob:sro_orig}}} \label{appx:pruning}

Next, we show that our characterization of the structure of optimal Markovian stopping rules for formulation~\eqref{prob:sro} from \S\ref{sec:main} can be readily extended to the alternative formulation~\eqref{prob:sro_orig} using similar proof techniques to those used in Theorem~\ref{thm:characterization}. 
 Our following analysis in Theorem~\ref{thm:characterization_orig} uses the following additional but relatively mild assumption on the reward functions in the optimal stopping problem:
\begin{assumption} \label{ass:reward_simple}
For each period $t$, the reward function satisfies $g(t,x) = h(t,x_t)$. 
\end{assumption}
The above assumption says that the reward function depends in each period only on the current state. This is a common assumption in the optimal stopping and options pricing literature and is often without loss of generality. With this assumption, the following theorem establishes the structure of optimal Markovian stopping rules for formulation~\eqref{prob:sro_orig}. 

\begin{theorem}
\label{thm:characterization_orig}
Under Assumption~\ref{ass:reward_simple}, there exists $\mu \in \mathcal{M}$ that is optimal for \eqref{prob:sro_orig}.
\end{theorem}

%
{\color{black}Our proof of Theorem~\ref{thm:characterization_orig} follows the same pruning technique that was used in the proof of Theorem~\ref{thm:characterization} in \S\ref{sec:main:proof}. Specifically, our proof of Theorem~\ref{thm:characterization_orig} makes use of the following  four lemmas, which are essentially restatements of  Lemmas~\ref{lem:restrict_to_completion}-\ref{lem:M_is_pruned} from \S\ref{sec:main:proof}. 
\begin{lemma}\label{lem:restrict_to_completion_orig}
The optimal objective value of \eqref{prob:sro_orig} is equal to the optimal objective value of 
\begin{align} \label{prob:sro_2_orig} \tag{RO$_T$'}
\begin{aligned}
&\sup_{\mu} &&\frac{1}{N} \sum_{i=1}^N \inf_{y \in \mathcal{U}^i}  g(\tau_\mu(y), y)\\
&\textnormal{subject to}&& \textnormal{for each } i \in \{1,\ldots,N\}, \; \textnormal{ there exists } t \in \{1,\ldots,T\} \\
&&& \textnormal{such that } \mu_t(y_t) = \textsc{Stop} \textnormal{ for all } y_t \in \mathcal{U}^i_t 
\end{aligned}
\end{align}
\end{lemma}
\begin{proof}{Proof.}The proof of Lemma~\ref{lem:restrict_to_completion_orig} is identical to  the proof of Lemma~\ref{lem:restrict_to_completion}. 
\halmos \end{proof}
 \begin{lemma} \label{lem:obj_sigma_orig}
Let Assumption~\ref{ass:reward_simple} hold, consider any $\mu \equiv (\mu_1,\ldots,\mu_T)$ that satisfies the constraints of \eqref{prob:sro_2_orig}, and define
\begin{align*}
\sigma^i \triangleq 
\min  \left \{ t \in \{1,\ldots,T\}: \; \mu_t(y_t) = \textsc{Stop} \textnormal{ for all } y_t \in \mathcal{U}^i_t \right \}  \quad \forall i \in \{1,\ldots,N\}. 
\end{align*}
Then the following equality holds for each $i \in \{1,\ldots,N\}$: 
\begin{align*}
& \inf_{y \in \mathcal{U}^i} g(\tau_{\mu}(y),x^i) =  \min \limits_{t \in \{1,\ldots,\sigma^i\}}  \inf_{y_t \in \mathcal{U}^i_t} \left \{ h(t,y_t) :  \mu_t(y_t) = \textsc{Stop} \right \}.
\end{align*}
\end{lemma}
\begin{proof}{Proof.}
Let Assumption~\ref{ass:reward_simple} hold, consider any $\mu \equiv (\mu_1,\ldots,\mu_T)$ that satisfies the constraints of \eqref{prob:sro_2_orig}, and define
\begin{align*}
\sigma^i \triangleq 
\min  \left \{ t \in \{1,\ldots,T\}: \; \mu_t(y_t) = \textsc{Stop} \textnormal{ for all } y_t \in \mathcal{U}^i_t \right \}  \quad \forall i \in \{1,\ldots,N\}. 
\end{align*}
It follows from identical reasoning as in the proof of Lemma~\ref{lem:obj_sigma} that the following equality holds for each sample path $i \in \{1,\ldots,N\}$: 
\begin{align}
\left \{ \tau_\mu(y): y \in \mathcal{U}^i \right\} &=  \left \{t \in \{1,\ldots,\sigma^i\}:  \; \textnormal{there exists } y_t \in \mathcal{U}^i_t \textnormal{ such that } \mu_{t}(y_{t}) = \textsc{Stop}  \right \}. \label{line:someday_I_will_get_good_at_coming_up_with_names_orig}
\end{align}
We thus conclude for each sample path $i \in \{1,\ldots,N\}$ that  
\begin{align*}
 \inf_{y \in \mathcal{U}^i} g(\tau_{\mu}(y),y) &= \min \limits_{t \in \{1,\ldots,\sigma^i\}}  \inf_{y_t \in \mathcal{U}^i_t} \left \{ h(t,y_t) :  \mu_t(y_t) = \textsc{Stop} \right \},
\end{align*}
where the equality follows from Assumption~\ref{ass:reward_simple} and  line~\eqref{line:someday_I_will_get_good_at_coming_up_with_names_orig}. This completes our proof of Lemma~\ref{lem:obj_sigma_orig}.\halmos \end{proof}

\begin{lemma} \label{lem:pruning_orig}
If Assumption~\ref{ass:reward_simple} holds and if $\mu'$ is a pruned version of $\mu$, then $ \inf \limits_{y \in \mathcal{U}^i} g(\tau_{\mu'}(y),y)  \ge \inf \limits_{y \in \mathcal{U}^i} g(\tau_{\mu}(y),y)$ for all $ i \in \{1,\ldots,N\}$. 
\end{lemma}
\begin{remark}
We observe that \eqref{prob:sro_2_orig} and \eqref{prob:sro_2} have identical constraints. Thus, we  use the same definition of pruning in Appendix~\ref{appx:pruning} as given by Definition~\ref{defn:pruning}. 

\end{remark}
\begin{proof}{Proof of Lemma~\ref{lem:pruning_orig}.}
Let  Assumption~\ref{ass:reward_simple} hold, and let  $\mu'$ be a pruned version of $\mu$. Let $\sigma^1,\ldots,\sigma^N$ satisfy the following equalities for each $i \in \{1,\ldots,N\}$:
\begin{align}
\sigma^i &= \min \left \{ t \in \{1,\ldots,T\}: \; \mu_t(y_t) = \textsc{Stop}\; \forall y_t \in \mathcal{U}^i_t \right \}= \min \left \{ t \in \{1,\ldots,T\}: \; \mu_t'(y_t) = \textsc{Stop}\; \forall y_t \in \mathcal{U}^i_t \right \}. \label{line:isthisthelastone}
\end{align} 
Then it follows from the fact that $\mu$ is feasible for \eqref{prob:sro_2_orig} that $\sigma^1,\ldots,\sigma^N \in \{1,\ldots,T\}$. Therefore, for each $i \in \{1,\ldots,N\}$, 
\begin{align*}
\inf_{y \in \mathcal{U}^i} g(\tau_{\mu'}(y),y) &=   \min \limits_{t \in \{1,\ldots,\sigma^i\}}  \inf_{y_t \in \mathcal{U}^i_t} \left \{ h(t,y_t) :  \mu_t'(y_t) = \textsc{Stop} \right \} \\
&\ge  \min \limits_{t \in \{1,\ldots,\sigma^i\}}  \inf_{y_t \in \mathcal{U}^i_t} \left \{ h(t,y_t) :  \mu_t(y_t) = \textsc{Stop} \right \} \\
&= \inf_{y \in \mathcal{U}^i} g(\tau_{\mu}(y),y).
\end{align*} 
Indeed, the two equalities follow from Lemma~\ref{lem:obj_sigma_orig} and line~\eqref{line:isthisthelastone}. The inequality follows from the fact that $\mu'$ is a pruned version of $\mu$, which implies that $\left \{y_t \in \mathcal{X}: \mu_t'(y_t) = \textsc{Stop} \right \} \subseteq  \left \{y_t \in \mathcal{X}: \mu_t(y_t) = \textsc{Stop} \right \}$ for all $t \in \{1,\ldots,T\}$. Our proof of Lemma~\ref{lem:pruning_orig} is thus complete. 
\halmos \end{proof}
\begin{lemma} \label{lem:M_is_pruned_orig}
Consider any $\mu \equiv (\mu_1,\ldots,\mu_T)$ that satisfies the constraints of \eqref{prob:sro_2_orig}, and define
\begin{align*}
\sigma^i \triangleq \min  \left \{ t \in \{1,\ldots,T\}: \; \mu_t(y_t) = \textsc{Stop} \textnormal{ for all } y_t \in \mathcal{U}^i_t \right \}  \quad \forall i \in \{1,\ldots,N\}. 
\end{align*}
Then $\mu^{\sigma^1 \cdots \sigma^N}$ is a pruned version of $\mu$. 
\end{lemma}
\begin{proof}{Proof.}
The proof of Lemma~\ref{lem:M_is_pruned_orig} is identical to the proof of Lemma~\ref{lem:M_is_pruned}.
\halmos \end{proof}

\vspace{1em}

In view of the above Lemmas~\ref{lem:restrict_to_completion_orig}-\ref{lem:M_is_pruned_orig}, we now present the proof of Theorem~\ref{thm:characterization_orig}. 
\begin{proof}{Proof  of Theorem~\ref{thm:characterization_orig}.}
 Lemma~\ref{lem:restrict_to_completion_orig} shows that the robust optimization problem~\eqref{prob:sro_orig} is equivalent to the robust optimization problem~\eqref{prob:sro_2_orig}. Moreover, for any arbitrary  exercise policy $\mu$ that is feasible for \eqref{prob:sro_2_orig},  Lemmas~\ref{lem:pruning_orig} and \ref{lem:M_is_pruned_orig} together show that there exists an exercise policy $\mu' \in \mathcal{M}$ such that the objective value associated with $\mu$ is less than or equal to the objective value associated with $\mu'$.   Since $\mu$ was chosen arbitrarily, our proof of Theorem~\ref{thm:characterization_orig} is complete. 
\halmos \end{proof}
}

\subsection{{\color{black}Computational Tractability of  \eqref{prob:sro_orig}}}

Finally, we now discuss our primary motivation for using formulation~\eqref{prob:sro} instead of formulation~\eqref{prob:sro_orig}. As we have shown up to this point, these two robust optimization formulations are essentially equivalent with respect to convergence guarantees and characterization of optimal Markovian stopping rules. However, as we will show momentarily, we find that formulation~\eqref{prob:sro} is significantly more amenable than \eqref{prob:sro_orig} from an algorithmic perspective. Our subsequent discussion on the computational tractability of formulation~\eqref{prob:sro_orig} makes use of the following Theorem~\ref{thm:reform_orig}, which, analogously to Theorem~\ref{thm:reform} from \S\ref{sec:algorithms:prelim}, provides a reformulation of \eqref{prob:sro_orig} as a finite-dimensional optimization problem over integer decision variables. 
\begin{theorem} \label{thm:reform_orig}
Under Assumption~\ref{ass:reward_simple}, the alternative formulation \eqref{prob:sro_orig} is equivalent to
\begin{align}
\underset{ \sigma^1,\ldots,\sigma^N \in \{1,\ldots,T\}}{\textnormal{maximize}}\;  \frac{1}{N} \sum_{i=1}^N \min_{t \in \{1,\ldots,\sigma^i \}} \min_{j: \sigma^j = t} \inf_{y_t \in \mathcal{U}^i_t \cap \mathcal{U}^j_t}   h(t,y_t). \tag{IP'}  \label{prob:mip_orig}
\end{align}
\end{theorem}
\begin{proof}{Proof.}
{\color{black}
Our proof of Theorem~\ref{thm:reform_orig} is essentially identical to the proof of Theorem~\ref{thm:reform}. Indeed, let Assumption~\ref{ass:reward_simple} hold. With this assumption, our proof of Theorem~\ref{thm:reform_orig} is split into the following two intermediary claims: 

\begin{claim}\label{claim:reformulation_1_orig}
The optimal objective value of \eqref{prob:sro_orig} is less than or equal to the optimal objective value of \eqref{prob:mip_orig}.
\end{claim}
\begin{proof}{Proof of Claim~\ref{claim:reformulation_1_orig}.}
Consider any arbitrary  exercise policy $\mu \in \mathcal{M}$, and define the integers
\begin{align}
\sigma^i \triangleq  \min \left \{ t \in \{1,\ldots,T\}: \mu_t(y_t) = \textsc{Stop} \;  \forall y_t \in \mathcal{U}^i_t \right \}, \quad \forall i \in \{1,\ldots,N\}.\label{line:argument_prime}
\end{align}
It follows from the definition of $\mathcal{M}$ that $\mu$  satisfies the constraints of \eqref{prob:sro_2_orig}, and so  Lemma~\ref{lem:M_is_pruned_orig} implies that $\mu^{\sigma^1 \cdots \sigma^N}$ is a pruned version of $\mu$. Therefore, it follows from Definition~\ref{defn:pruning} and line~\eqref{line:argument_prime} that the following equalities hold: 
\begin{align}
\sigma^i = \min  \left \{ t \in \{1,\ldots,T\}: \; \mu_t^{\sigma^1 \cdots \sigma^N}(y_t) = \textsc{Stop} \textnormal{ for all } y_t \in \mathcal{U}^i_t \right \}  \quad \forall i \in \{1,\ldots,N\}. \label{line:argument2_orig}
\end{align}
Therefore, we observe that 
\begin{align}
&\frac{1}{N} \sum_{i=1}^N  \inf_{y \in \mathcal{U}^i} g \left(\tau_{\mu}(y),y \right) \notag \\
&\le \frac{1}{N} \sum_{i=1}^N  \inf_{y \in \mathcal{U}^i} g \left(\tau_{\mu^{\sigma^1 \cdots \sigma^N}}(y),y \right)\notag  \\
  &= \frac{1}{N} \sum_{i=1}^N  \min \limits_{t \in \{1,\ldots,\sigma^i\}} \inf_{y_t \in \mathcal{U}^i_t} \left \{ h(t,y_t) :  \; \mu_t^{\sigma^1 \cdots \sigma^N}(y_t) = \textsc{Stop} \right \} \notag \\
&=  \frac{1}{N} \sum_{i=1}^N \min_{t \in \{1,\ldots,\sigma^i \}} \min_{j: \sigma^j = t} \inf_{y_t \in \mathcal{U}^i_t \cap \mathcal{U}^j_t}  h(t,y_t), \label{line:argument3_orig}
\end{align}
where the first inequality follows from Lemma~\ref{lem:pruning_orig},  the first equality follows from Lemma~\ref{lem:obj_sigma_orig} and line~\eqref{line:argument2_orig}, and the second equality  follows from the fact that $\mu^{\sigma^1 \cdots \sigma^N}_t(y_t) = \textsc{Stop}$ if and only if there exists a sample path $j$ such that $y_t \in \mathcal{U}^j_t$ and $\sigma^j = t$. Because $\mu \in \mathcal{M}$ was chosen arbitrarily, we conclude from line~\eqref{line:argument3_orig} and Theorem~\ref{thm:characterization_orig} that  the optimal objective value of \eqref{prob:sro_orig} is less than or equal to the optimal objective value of \eqref{prob:mip_orig}. Our proof of Claim~\ref{claim:reformulation_1_orig} is thus complete.  \halmos
\end{proof}

\begin{claim} \label{claim:reformulation_2_orig}
The optimal objective value of \eqref{prob:sro_orig} is greater than or equal to the optimal objective value of \eqref{prob:mip_orig}.  Furthermore, for any choice of integers $\sigma^1,\ldots,\sigma^N \in \{1,\ldots,T\}$, the corresponding exercise policy $\mu^{\sigma^1 \cdots \sigma^N} \equiv (\mu_1^{\sigma^1 \cdots \sigma^N},\ldots,\mu_T^{\sigma^1 \cdots \sigma^N})$ satisfies  %
 \begin{align*}
 &\frac{1}{N} \sum_{i=1}^N \inf_{y \in \mathcal{U}^i}  g \left(\tau_{\mu^{\sigma^1 \cdots \sigma^N}}(y), y \right) \ge \frac{1}{N} \sum_{i=1}^N \min_{t \in \{1,\ldots,\sigma^i \}} \min_{j: \sigma^j = t} \inf_{y_t \in \mathcal{U}^i_t \cap \mathcal{U}^j_t}  h(t,y_t).
 \end{align*}
\end{claim}
\begin{proof}{Proof of Claim~\ref{claim:reformulation_2_orig}.}
Consider any integers $\sigma^1,\ldots,\sigma^N \in \{1,\ldots,T\}$. For each sample path $i \in \{1,\ldots,N\}$, we observe that
\begin{align}
&\min  \left \{ t \in \{1,\ldots,T\}: \; \mu_t^{\sigma^1 \cdots \sigma^N}(y_t) = \textsc{Stop} \textnormal{ for all } y_t \in \mathcal{U}^i_t \right \}\notag  \\
&=\min  \left \{ t \in \{1,\ldots,\sigma^i\}: \; \mu_t^{\sigma^1 \cdots \sigma^N}(y_t) = \textsc{Stop} \textnormal{ for all } y_t \in \mathcal{U}^i_t \right \} \notag \\
&\le \sigma^i,\label{line:argument4_orig}
\end{align}
where  the equality follows from the fact that $\mu^{\sigma^1 \cdots \sigma^N} \equiv (\mu^{\sigma^1 \cdots \sigma^N}_1,\ldots,\mu^{\sigma^1 \cdots \sigma^N}_T)$ by construction satisfies $\mu^{\sigma^1 \cdots \sigma^N}_{\sigma^i}(y_{\sigma^i}) = \textsc{Stop}$ for all $y_{\sigma^i} \in \mathcal{U}^i_{\sigma^i}$, and the inequality follows from algebra. Therefore, 
 \begin{align}
&\frac{1}{N} \sum_{i=1}^N  \inf_{y \in \mathcal{U}^i} g \left(\tau_{\mu^{\sigma^1 \cdots \sigma^N}}(y),y \right) \notag  \\
&\ge  \frac{1}{N} \sum_{i=1}^N  \min \limits_{t \in \{1,\ldots,\sigma^i\}} \inf_{y_t \in \mathcal{U}^i_t} \left \{ h(t,y_t) :  \; \mu_t^{\sigma^1 \cdots \sigma^N}(y_t) = \textsc{Stop} \right \} \notag \\
&=  \frac{1}{N} \sum_{i=1}^N \min_{t \in \{1,\ldots,\sigma^i \}} \min_{j: \sigma^j = t} \inf_{y_t \in \mathcal{U}^i_t \cap \mathcal{U}^j_t}  h(t,y_t), \label{line:argument5_orig}
\end{align}
where the inequality follows from Lemma~\ref{lem:obj_sigma_orig} and line~\eqref{line:argument4_orig}, and the equality follows from the fact that $\mu^{\sigma^1 \cdots \sigma^N}_t(y_t) = \textsc{Stop}$ if and only if there exists a sample path $j$ such that $y_t \in \mathcal{U}^j_t$ and $\sigma^j = t$. Because $\sigma^1,\ldots,\sigma^N \in \{1,\ldots,T\}$ were chosen arbitrarily, we conclude from line~\eqref{line:argument5_orig} that the optimal objective value of \eqref{prob:sro_orig} is greater than or equal to the optimal objective value of \eqref{prob:mip_orig}, which concludes our proof of Claim~\ref{claim:reformulation_2_orig}. 
\halmos \end{proof}
Combining Claims~\ref{claim:reformulation_1_orig} and \ref{claim:reformulation_2_orig}, our proof of Theorem~\ref{thm:reform_orig} is thus complete. 
}

\halmos \end{proof}

Equipped with the above Theorem~\ref{thm:reform_orig}, we now explain why formulation~\eqref{prob:sro} is more amenable from an tractability perspective than \eqref{prob:sro_orig}. Indeed, we recall from Theorem~\ref{thm:reform} in \S\ref{sec:algorithms:prelim} that the robust optimization problem~\eqref{prob:sro} is equivalent to
\begin{align}
\underset{ \sigma^1,\ldots,\sigma^N \in \{1,\ldots,T\}}{\textnormal{maximize}}\;  \frac{1}{N} \sum_{i=1}^N \underbrace{\min_{t \in \{1,\ldots,\sigma^i \}} \left \{g(t,x^i): \textnormal{ there exists } j \textnormal{ such that } \mathcal{U}^i_t \cap \mathcal{U}^j_t \neq \emptyset \textnormal{ and } \sigma^j = t\right \}}_{\nu_i(\sigma)}, \tag{\ref{prob:mip}} 
\end{align}
where we observe for each sample path $i \in \{1,\ldots,N\}$ that the quantity $\nu_i(\sigma)$ satisfies
\begin{align*}
\nu_i(\sigma)\in \bigcup_{t=1}^T \{ g(t,x^i) \}. 
\end{align*}
In contrast, we showed in the above Theorem~\ref{thm:reform_orig} that \eqref{prob:sro_orig} is equivalent to 
\begin{align}
\underset{ \sigma^1,\ldots,\sigma^N \in \{1,\ldots,T\}}{\textnormal{maximize}}\;  \frac{1}{N} \sum_{i=1}^N \underbrace{ \min_{t \in \{1,\ldots,\sigma^i \}} \min_{j: \sigma^j = t} \inf_{y_t \in \mathcal{U}^i_t \cap \mathcal{U}^j_t}   h(t,y_t)}_{\nu_i'(\sigma)},   \tag{\ref{prob:mip_orig}}
\end{align}
where we observe for each sample path $i \in \{1,\ldots,N\}$ that the quantity $\nu_i'(\sigma)$  satisfies
\begin{align*}
\nu_i'(\sigma)\in \bigcup_{t=1}^T \bigcup_{j: \mathcal{U}^i_t  \cap \mathcal{U}^j_t \neq \emptyset} \left \{ \inf_{y_t \in \mathcal{U}^i_t \cap \mathcal{U}^j_t}   h(t,y_t) \right \}.
 \end{align*}
 Hence, for each sample path $i \in \{1,\ldots,N\}$, we observe that the number of possible values for $\nu_i'(\sigma)$ can be a factor of $\mathcal{O}(N)$ greater than the number of possible values for $\nu_i(\sigma)$. 
  This demonstrates that \eqref{prob:sro} is more amenable to compact reformulations than \eqref{prob:sro_orig} and concludes our motivation for using formulation \eqref{prob:sro} instead of formulation~\eqref{prob:sro_orig} throughout the paper. 
 }

\section{{\color{black} Proofs of Theorems~\ref{thm:conv:asympt}, \ref{thm:conv:policy}, and \ref{thm:conv:ub}}} \label{appx:proof_conv}
{\color{black}Establishing the convergence guarantees from \S\ref{sec:opt} for the robust optimization problem~\eqref{prob:sro} can be organized into two high-level steps. The first high-level step, comprised of Theorem~\ref{thm:conv:ub}, consists of showing that the objective function of the robust optimization problem~\eqref{prob:sro} converges almost surely to a (conservative) lower bound approximation of the objective function of the stochastic optimal stopping problem~\eqref{prob:main}, uniformly over the space of all exercise policies. The second high-level step, comprised of  Theorems~\ref{thm:conv:asympt} and \ref{thm:conv:policy}, consists of showing that the conservativeness of the robust optimization problem~\eqref{prob:sro} disappears almost surely as the robustness parameter tends to zero and the number of sample paths tends to infinity. We present the first high-level step in Appendix~\ref{appx:proof_conv:prelim}, and the second high-level step can be found in Appendix~\ref{appx:proof_conv:new}.

Let us reflect on the novelty of the results in the present Appendix~\ref{appx:proof_conv}. The first high-level step in Appendix~\ref{appx:proof_conv:prelim} (i.e., the proof of Theorem~\ref{thm:conv:ub}) is not novel and follows immediately from a uniform convergence result established by \citetalias{bertsimasdata}. Rather, the significant novelty of Appendix~\ref{appx:proof_conv} is found in the second high-level step in Appendix~\ref{appx:proof_conv:new} (i.e., the proofs of Theorems~\ref{thm:conv:asympt} and \ref{thm:conv:policy}). 
Stated succinctly,  the second high-level step is challenging to establish because 
it requires showing that (a) the limit of the optimal objective value of the robust optimization problem~\eqref{prob:sro} exists almost surely as $\epsilon \to \infty$ and $N \to \infty$, and showing that (b) there always exists an arbitrarily near-optimal Markovian stopping rule $\mu$ for the stochastic optimal stopping problem~\eqref{prob:main} that almost surely satisfies $\tau_\mu(y) = \tau_\mu(x)$ for all sufficiently close realizations of $y \equiv (y_1,\ldots,y_T)$ to the stochastic process $x \equiv (x_1,\ldots,x_T)$. 

For the general classes of stochastic dynamic optimization problems considered by \citetalias{bertsimasdata} and \cite{sturtthesis}, the authors were unable to identify simple and verifiable conditions under which the aforementioned properties (a) and (b) are guaranteed to hold. Consequently, the authors only presented upper bounds on the gap between the optimal objective value of the robust optimization problem and the optimal objective value of the stochastic dynamic optimization (\citetalias[Theorem 1]{bertsimasdata}). In certain cases, the authors alternatively assumed that the stochastic dynamic optimization problem happened to have arbitrarily near-optimal control policies with convenient structure \citep[Assumption 15]{sturtthesis}. The authors of \citetalias{bertsimasdata} reflect on the weaknesses of these results in the literature at the end of their \S4.3, saying that ``future work may identify subclasses of [stochastic dynamic optimization problems] where the equality of the bounds can be ensured." 

{\color{black}

In Appendix~\ref{appx:proof_conv:new} of this paper, we resolve this gap in the literature by showing that the aforementioned properties (a) and (b) are guaranteed to hold for the specific class of stochastic dynamic optimization problems~\eqref{prob:main} under the mild assumptions that the stochastic problem has a bounded objective function  (Assumption~\ref{ass:bound}) and that the random variables in the stochastic problem have a continuous joint probability distribution (Assumption~\ref{ass:continuous}). Specifically, we establish in Lemma~\ref{lem:mcdiarmid} of Appendix~\ref{appx:proof_conv:new} that property (a) holds by combining Assumption~\ref{ass:bound} with McDiarmid's inequality, and we establish in Lemma~\ref{lem:bradsmart} of Appendix~\ref{appx:proof_conv:new} that property (b) holds by combining Assumptions~\ref{ass:continuous} and \ref{ass:bound} with elementary techniques from topology and measure theory. These lemmas allow us to establish the proofs of Theorems~\ref{thm:conv:asympt} and \ref{thm:conv:policy} at the end of Appendix~\ref{appx:proof_conv:new}. 
}


 } 
\subsection{{\color{black}Proof of Theorem~\ref{thm:conv:ub}}} \label{appx:proof_conv:prelim}
{\color{black}In Appendix~\ref{appx:proof_conv:prelim}, we establish the proof of Theorem~\ref{thm:conv:ub} by using a uniform convergence result from \citetalias{bertsimasdata}.} 
We begin by presenting some preliminary notation. 
Recall that $x \equiv (x_1,\ldots,x_T)$, $x^1 \equiv (x_1^1,\ldots,x^1_T)$, $x^2 \equiv (x_1^2,\ldots,x^2_T),\ldots \in \mathcal{X}^T {\color{black}\equiv} \R^{Td}$ are sample paths drawn independently from an identical joint probability distribution. 
We will make use of the following additional notation: 
\begin{align*}
\delta_N &\triangleq N^{-\frac{1}{\max \{ 3,Td+1\}}};&\mathcal{U}^i(\delta_N) &\triangleq \left \{ y \in \mathcal{X}^T: \| y - x^i \|_\infty \le \delta_N \right \};\\
\mathcal{U}^i  &\triangleq \left \{ y \in \mathcal{X}^T: \| y - x^i \|_\infty \le \epsilon \right \};&M_N &\triangleq N^{-\frac{1}{(Td+1)(Td+2)}} \log N.
\end{align*} 
{\color{black}We now state the uniform convergence result from \citetalias{bertsimasdata}, which has been} adapted to the notation of the present paper. 
\begin{lemma}[Theorem~2 of \citetalias{bertsimasdata}] \label{lem:bertsimas}
Let Assumption~\ref{ass:lighttail} hold.  
Then there exists a finite $\bar{N} \in \N$, almost surely, such that the following inequality holds for all $N \ge \bar{N}$ and all measurable functions $f: \R^{Td} \to \R$: 
\begin{align*}
\Exp \left[f(x) \mathbb{I} \left \{  x \in \cup_{i=1}^N \mathcal{U}^i(\delta_N) \right \} \right] \ge \frac{1}{N} \sum_{i=1}^N \inf_{y \in \mathcal{U}^i(\delta_N)  } f(y) - M_N  \sup_{y \in \cup_{i=1}^N\mathcal{U}^i(\delta_N)} |f(y)|.
\end{align*}
\end{lemma}
In view of the above notation and lemma, we are now ready to present the proof {\color{black}of  Theorem~\ref{thm:conv:ub}.}   \\
\begin{proof}{Proof of Theorem~\ref{thm:conv:ub}.}
Consider any arbitrary choice of the robustness parameter $\epsilon > 0$, and 
recall  that the reward function satisfies $g(\infty,y) = 0$ (see \S\ref{sec:setting}) and $0 \le g(1, y),\ldots,g(T,y) \le U$ for all trajectories $y \in \mathcal{X}^T$ (Assumption~\ref{ass:bound}). With this notation, we observe that 
\begin{align}
&\liminf_{N \to \infty} \inf_{\mu} \left \{ J^*(\mu) - \widehat{J}_{N,\epsilon}(\mu) \right \} \notag \\
&= \liminf_{N \to \infty} \inf_{\mu} \left \{ \Exp \left[g(\tau_\mu(x), x) \right] - \widehat{J}_{N,\epsilon}(\mu) \right \}  \label{line:defn_stoch} \\
&\ge \liminf_{N \to \infty} \inf_{\mu} \left \{ \Exp \left[g(\tau_\mu(x), x)  \mathbb{I} \left \{  x \in \cup_{i=1}^N \mathcal{U}^i(\delta_N) \right \} \right] - \widehat{J}_{N,\epsilon}(\mu)\right \} \label{line:exp_to_sup} \\
& \ge \liminf_{N \to \infty} \inf_{\mu} \left \{  \frac{1}{N} \sum_{i=1}^N \inf_{y \in \mathcal{U}^i(\delta_N)  } g(\tau_\mu(y),y)- M_N U  - \widehat{J}_{N,\epsilon}(\mu)\right \}\;\;\; \text{almost surely} \label{line:hellothere} \\
&=  \liminf_{N \to \infty} \inf_{\mu} \left \{  \frac{1}{N} \sum_{i=1}^N \inf_{y \in \mathcal{U}^i(\delta_N)  } g(\tau_\mu(y),y)- \widehat{J}_{N,\epsilon}(\mu) \right \} \label{line:hellothere2}\\
&\ge\liminf_{N \to \infty} \inf_{\mu} \left \{ \frac{1}{N} \sum_{i=1}^N \inf_{y \in \mathcal{U}^i  } g(\tau_\mu(y),y)- \widehat{J}_{N,\epsilon}(\mu) \right \}.\label{line:intermediary}
\end{align} 
\eqref{line:defn_stoch} follows from applying the definition of $J^*(\mu)$; \eqref{line:exp_to_sup} holds because the reward function is nonnegative; \eqref{line:hellothere} follows from Lemma~\ref{lem:bertsimas} 
and the boundedness of the reward function; \eqref{line:hellothere2} holds because $M_N \to 0$; \eqref{line:intermediary} holds because $\delta_N \to 0$ implies, for any arbitrary $\epsilon > 0$, that the inequality $ \inf_{y \in \mathcal{U}^i } g(\tau_\mu(y),y) \le  \inf_{y \in \mathcal{U}^i(\delta_N)  } g(\tau_\mu(y),y)$ is satisfied for all $i \in \N$, for all $\mu$, and for all large $N \in \N$. Moreover: 
\begin{align}
\eqref{line:intermediary}\;&= \liminf_{N \to \infty} \inf_{\mu} \left \{ \frac{1}{N} \sum_{i=1}^N \inf_{y \in \mathcal{U}^i }   \left \{ g(\tau_\mu(y), x^i) + g(\tau_\mu(y),y)  - g(\tau_\mu(y), x^i)\right\} - \widehat{J}_{N,\epsilon}(\mu) \right \}\label{line:alg1} \\
&\ge  \liminf_{N \to \infty} \inf_{\mu} \left \{  \frac{1}{N} \sum_{i=1}^N \left( \inf_{y \in \mathcal{U}^i }    g(\tau_\mu(y), x^i) + \inf_{y \in \mathcal{U}^i  } \left \{ g(\tau_\mu(y),y)  - g(\tau_\mu(y), x^i) \right \} \right)- \widehat{J}_{N,\epsilon}(\mu) \right \} \label{line:alg2} \\
&\ge  \liminf_{N \to \infty} \inf_{\mu} \left \{  \frac{1}{N} \sum_{i=1}^N \left( \inf_{y \in \mathcal{U}^i }    g(\tau_\mu(y), x^i) +\min_{t \in \{1,\ldots,T\}} \left \{ \inf_{y \in \mathcal{U}^i  }  g(t,y)  - g(t, x^i) \right \} \right)- \widehat{J}_{N,\epsilon}(\mu) \right \} \label{line:extrastuff}\\
&\ge  \liminf_{N \to \infty} \inf_{\mu} \left \{  \frac{1}{N} \sum_{i=1}^N \left( \inf_{y \in \mathcal{U}^i }    g(\tau_\mu(y), x^i) +\Delta_{\epsilon}(x^i)  \right) - \widehat{J}_{N,\epsilon}(\mu) \right \}\label{line:defn_delta}\\
&=\liminf_{N \to \infty} \inf_{\mu} \left \{ \widehat{J}_{N,\epsilon}(\mu)  - \widehat{J}_{N,\epsilon}(\mu) \right \} +  \liminf_{N \to \infty} \frac{1}{N} \sum_{i=1}^N  \Delta_{\epsilon}(x^i)  \label{line:defn_rob} \\
&=   \Exp \left[\Delta_{\epsilon}(x) \right] \quad \text{almost surely}. \label{line:slln}
\end{align}
 \eqref{line:alg1}, \eqref{line:alg2}, and \eqref{line:extrastuff} follow from algebra; \eqref{line:defn_delta} follows from the definition of $\Delta_{\epsilon}(x^i)$ (see Assumption~\ref{ass:weird}); \eqref{line:defn_rob} follows from the definition of $\widehat{J}_{N,\epsilon}(\mu)$; \eqref{line:slln} follows from the strong law of large numbers. 
 
 Since $\epsilon > 0$ was chosen arbitrarily, we have shown that
\begin{align*}
\lim_{\epsilon \to 0}  \liminf_{N \to \infty} \inf_{\mu} \left \{ J^*(\mu) - \widehat{J}_{N,\epsilon}(\mu) \right \} \ge  \lim_{\epsilon \to 0}  \Exp \left[\Delta_{\epsilon}(x) \right]\ge 0\quad  \text{almost surely}, 
\end{align*} 
where the first inequality holds almost surely from lines~\eqref{line:defn_stoch}-\eqref{line:slln}, and the second inequality holds almost surely due to the dominated convergence theorem and  Assumption~\ref{ass:weird}. Note that the above limits exist because $\epsilon \mapsto \liminf_{N \to \infty} \inf_{\mu}  \{ J^*(\mu) - \widehat{J}_{N,\epsilon}(\mu)  \} $  and $\epsilon \mapsto \Exp[\Delta_\epsilon(x)]$  are monotonic functions. This concludes the proof of Theorem~\ref{thm:conv:ub}. 
\halmos \end{proof}

\subsection{{\color{black}Proofs of Theorems~\ref{thm:conv:asympt} and \ref{thm:conv:policy}}} \label{appx:proof_conv:new}
{\color{black}We begin Appendix~\ref{appx:proof_conv:new} by presenting the two novel intermediary results, Lemmas~\ref{lem:mcdiarmid} and \ref{lem:bradsmart}, that were discussed at the beginning of Appendix~\ref{appx:proof_conv}. In the first  novel intermediary result, denoted below by Lemma~\ref{lem:mcdiarmid}, we prove that the limit of the optimal objective value of the robust optimization problem~\eqref{prob:sro} exists almost surely as $\epsilon \to \infty$ and $N \to \infty$. The proof of the following lemma is based on McDiarmid's inequality.}

\begin{lemma} \label{lem:mcdiarmid}
Let Assumption~\ref{ass:bound} hold. Then for all $\epsilon > 0$, 
\begin{align*}
\liminf_{N \to \infty}  \sup_{\mu}\widehat{J}_{N,\epsilon}(\mu) = \limsup_{N \to \infty}  \sup_{\mu}\widehat{J}_{N,\epsilon}(\mu) \quad \textnormal{almost surely}.
\end{align*}
\end{lemma}
\begin{proof}{Proof.}
Consider any fixed $\epsilon > 0$, and,  for notational convenience, define the following function:  $$h_\epsilon(x^1,\ldots,x^N) \triangleq \sup_\mu \widehat{J}_{N,\epsilon}(\mu).$$ 
We will utilize the following intermediary claim: 
\begin{claim} \label{claim:mcdiarmid}
 The function $h_\epsilon: \mathcal{X}^T \times \cdots \times \mathcal{X}^T \to \R$ has the following `bounded differences' property: for all $\breve{x}^1,\ldots,\breve{x}^N \in \mathcal{X}^T$ and $\bar{x}^1,\ldots,\bar{x}^N \in \mathcal{X}^T$  that differ only on the $j$th coordinate ($\breve{x}^i = \bar{x}^i$ for $i \neq j$), 
\begin{align*}
\left| h_\epsilon(\bar{x}^1,\ldots,\bar{x}^N) -  h_\epsilon(\breve{x}^1,\ldots,\breve{x}^N) \right| \le \frac{U}{N}.  
\end{align*}
\end{claim}
\begin{proof}{Proof of Claim~\ref{claim:mcdiarmid}.} For any arbitrary $\eta > 0$, let the exercise policies $\breve{\mu}^\eta$ be chosen to satisfy 
\begin{align}
\frac{1}{N} \sum_{i=1}^N \inf_{y \in \mathcal{X}^T: \; \| y - \breve{x}^i \|_\infty \le \epsilon} g(\tau_{\breve{\mu}^\eta}(y),\breve{x}^i) \ge h_\epsilon(\breve{x}^1,\ldots,\breve{x}^N) - \eta. \label{line:idk}
\end{align}
We observe that
\begin{align*}
& h_\epsilon(\breve{x}^1,\ldots,\breve{x}^N) - h_\epsilon(\bar{x}^1,\ldots,\bar{x}^N) \\
&\le \left( \frac{1}{N} \sum_{i=1}^N \inf_{y \in \mathcal{X}^T: \; \| y - \breve{x}^i \|_\infty \le \epsilon} g(\tau_{\breve{\mu}^\eta}(y),\breve{x}^i) + \eta \right) -  \left(  \frac{1}{N} \sum_{i=1}^N \inf_{y \in \mathcal{X}^T: \;\| y - \bar{x}^i \|_\infty \le \epsilon} g(\tau_{\breve{\mu}^\eta}(y),\bar{x}^i) \right)  \\
&=   \frac{1}{N}  \inf_{y \in \mathcal{X}^T: \;\| y - \breve{x}^j \|_\infty \le \epsilon} g(\tau_{\breve{\mu}^\eta}(y),\breve{x}^j) -  \frac{1}{N}  \inf_{y \in \mathcal{X}^T: \;\| y - \bar{x}^j \|_\infty \le \epsilon} g(\tau_{\breve{\mu}^\eta}(y),\bar{x}^j)  + \eta \\
&\le \frac{U}{N} + \eta. 
\end{align*}
Indeed, the first inequality holds because of line~\eqref{line:idk} and because $\breve{\mu}^\eta$ is a feasible but possibly suboptimal solution to the optimization problem $\sup_\mu \frac{1}{N} \sum_{i=1}^N \inf_{y \in \mathcal{X}^T: \;\| y - \bar{x}^i \|_\infty \le \epsilon} g(\tau_{\mu}(y),\bar{x}^i) \equiv h_\epsilon(\bar{x}^1,\ldots,\bar{x}^N)$, and the second inequality follows from Assumption~\ref{ass:bound}. Because $\eta >0$ was chosen arbitrarily, we  have shown that
\begin{align*}
h_\epsilon(\breve{x}^1,\ldots,\breve{x}^N) - h_\epsilon(\bar{x}^1,\ldots,\bar{x}^N) \le \frac{U}{N}. 
\end{align*}
It follows from symmetry that
\begin{align*}
h_\epsilon(\bar{x}^1,\ldots,\bar{x}^N) - h_\epsilon(\breve{x}^1,\ldots,\breve{x}^N)  \le \frac{U}{N}, 
\end{align*}
which concludes our proof of Claim~\ref{claim:mcdiarmid}. \halmos \end{proof}

Because the above Claim~\ref{claim:mcdiarmid} holds, it follows from  McDiarmid's inequality 
 that 
  \begin{align*}
  \Prb \left( \left|  h_\epsilon(x^1,\ldots,x^N) - \Exp \left[ h_\epsilon(x^1,\ldots,x^N) \right] \right| > \eta \right)  \le 2 \textnormal{exp} \left( - \frac{2 \eta^2 N}{U^2} \right) \quad \forall \eta > 0.
  \end{align*}
It follows from the above line that
$$ \sum_{N=1}^\infty   \Prb \left( \left|  h_\epsilon(x^1,\ldots,x^N) - \Exp \left[ h_\epsilon(x^1,\ldots,x^N) \right] \right| > \eta \right) < \infty \quad \forall \eta > 0,$$
and so the Borel-Cantelli lemma implies that
\begin{align}
\lim_{N \to \infty} \left|  h_\epsilon(x^1,\ldots,x^N) - \Exp \left[h_\epsilon(x^1,\ldots,x^N) \right] \right|=0 \quad \text{almost surely}.   \label{line:nice}
\end{align}
We observe from identical reasoning as in the proof of \citet[Proposition 5.6]{shapiro2014lectures} that $ \Exp \left[h_\epsilon(x^1,\ldots,x^N) \right]$ is monotonically decreasing with respect to $N \in \N$. Since the random variables $h_\epsilon(x^1,\ldots,x^N)$ are also contained in the interval $[0,U]$ for all $N \in \N$, we conclude that $\lim_{N \to \infty}\Exp [ h_\epsilon(x^1,\ldots,x^N)] $ exists, and thus it follows from  line~\eqref{line:nice} that $\lim_{N \to \infty}h_\epsilon(x^1,\ldots,x^N)$ exists almost surely. This concludes the proof of Lemma~\ref{lem:mcdiarmid}. 
\halmos \end{proof}

{\color{black}
In the second novel intermediary result, denoted below by Lemma~\ref{lem:bradsmart}, we show that there always exists an arbitrarily near-optimal Markovian stopping rule $\mu$ for the stochastic optimal stopping problem~\eqref{prob:main} that almost surely satisfies $\tau_\mu(y) = \tau_\mu(x)$ for all sufficiently close realizations of $y \equiv (y_1,\ldots,y_T)$ to the stochastic process $x \equiv (x_1,\ldots,x_T)$. The proof of the following lemma is based on elementary techniques from topology and measure theory. 
\begin{lemma}\label{lem:bradsmart}
Under Assumptions~\ref{ass:continuous} and \ref{ass:bound},
\begin{align*}
\sup_\mu \Exp \left[ g(\tau_\mu(x),x) \right] = \sup_\mu \Exp \left[ \liminf_{y \to x} g(\tau_\mu(y),x)\right].
\end{align*}
\end{lemma}
\begin{proof}{Proof.}
Our proof of Lemma~\ref{lem:bradsmart} is organized into the following steps. In our first step, we will show that any feasible exercise policy to the stochastic optimal stopping problem~\eqref{prob:main} can be approximated to arbitrary accuracy by an exercise policy with exercise regions that are open sets. In our second step, we will show that these exercise  policies from the previous step can be approximated to arbitrary accuracy by an exercise policy with exercise regions that are the union of finitely many open balls. Finally, we will then  invoke  Assumption~\ref{ass:continuous} to conclude that the probability of a random state lying a strictly positive distance from the boundary of finitely many open balls is equal to one, which implies that the corresponding Markovian stopping rule  almost surely satisfies $\tau_\mu(y) = \tau_\mu(x)$ for all sufficiently close realizations of $y \equiv (y_1,\ldots,y_T)$ to the stochastic process $x \equiv (x_1,\ldots,x_T)$. From this, the desired lemma will  follow readily. 

In view of the above organization, we now present our proof. Choose any arbitrary exercise policy $\bar{\mu} \equiv (\bar{\mu}_1,\ldots,\bar{\mu}_T)$ which is feasible for the stochastic optimal stopping problem \eqref{prob:main}, and choose any arbitrary constant $\eta > 0$. In our first claim, we show that this exercise policy can be approximated by an exercise policy with exercise regions that are open sets. 
\begin{claim} \label{claim:2}
There exists an exercise policy $\tilde{\mu}^\eta \equiv (\tilde{\mu}_1^\eta,\ldots,\tilde{\mu}_T^\eta)$ such that:
\begin{enumerate}[label={(\alph*)},ref=(\alph*)]
\item  $\Exp[g(\tau_{\tilde{\mu}^\eta}(x),x)] \ge \Exp[g(\tau_{\bar{\mu}}(x),x)] - \eta$; \label{claim1:a}
\item For each period $t \in \{1,\ldots,T\}$, $A_t^\eta \triangleq \{y_t \in \mathcal{X}: \tilde{\mu}^\eta(y_t) = \textsc{Stop} \}$ is an open set. \label{claim1:b}
\end{enumerate}
\end{claim}
\begin{proof}{Proof of Claim~\ref{claim:2}.}  
Because  $\bar{\mu}_1,\ldots,\bar{\mu}_T: \R^d \to \{\textsc{Stop}, \textsc{Continue} \}$ are feasible for the optimization problem~\eqref{prob:main}, it follows from \S\ref{sec:setting:notation} that  $\bar{\mu}_1,\ldots,\bar{\mu}_T: \R^d \to \{\textsc{Stop}, \textsc{Continue} \}$ are measurable functions. Thus,    
it follows for each period $t \in \{1,\ldots,T\}$ that the set 
${\color{black}{A}_t} \triangleq 
 \left \{ y_t \in \R^d: {\mu}_t(y_t) = \textsc{Stop} \right \}$ is a Borel set in $\R^d$. 
Now, for each period $t \in \{1,\ldots,T\}$, let  $\Prb_{x,t}(\cdot) \triangleq \Prb(x_t \in \cdot)$ denote the marginal probability law of the stochastic process on period $t$.  Since ${A}_t$ is a Borel set, it is a well known result from measure theory \citep[Remark 11.11(b)]{rudin1964principles} that there exists an {open} set $A_t^{\eta} \subseteq \R^d$ which satisfies $A_t \subseteq A_t^\eta$ and  $\Prb_{x,t}(A_t^\eta \setminus A_t) \le \frac{\eta}{T}$. Using these sets, we define an exercise policy on each period $t$ by
\begin{align*}
\tilde{\mu}^\eta_t(y_t) \triangleq \begin{cases}
\textsc{Stop},&\text{if } y_t \in A^\eta_t,\\
\textsc{Continue},&\text{otherwise}.
\end{cases} 
\end{align*}
It follows from the above construction that the new exercise policy $\tilde{\mu}^\eta \equiv (\tilde{\mu}^\eta_1,\ldots,\tilde{\mu}^\eta_T)$ satisfies property~\ref{claim1:b} of Claim~\ref{claim:2}. To show that property~\ref{claim1:a} of Claim~\ref{claim:2} holds, let us first define the following set for notational convenience:
\begin{align*}
B \triangleq \left \{  y \equiv (y_1,\ldots,y_T): y_t \notin A_t^\eta \setminus A_t  \; \text{for all } t \in \{1,\ldots,T\} \right \} 
\end{align*}
We observe that $B$ is a Borel set in $\R^{Td}$, since $B$ is comprised of a finite number of complements and intersections of Borel sets. Then, 
\begin{align*}
&\Exp[ \left| g(\tau_{\tilde{\mu}^\eta}(x),x) - g(\tau_{\bar{\mu}}(x),x) \right| ] \\
&= \Exp \left[ \left| g(\tau_{\tilde{\mu}^\eta}(x),x) - g(\tau_{\bar{\mu}}(x),x) \right| \mathbb{I} \left \{ x \notin B \right \}   \right] + \Exp \left[ \left| g(\tau_{\tilde{\mu}^\eta}(x),x) - g(\tau_{\bar{\mu}}(x),x) \right| \mathbb{I} \left \{ x \in B \right \}   \right]  \\
&= \Exp \left[ \left| g(\tau_{\tilde{\mu}^\eta}(x),x) - g(\tau_{\bar{\mu}}(x),x) \right| \mathbb{I} \left \{ x \notin B \right \}   \right]   + 0 \\
&\le  U \Prb \left(  x \notin B \right)\\
&\le \sum_{t=1}^T \Prb \left( x_t \in A_t^\eta \setminus A_t \right)\\
&\le \eta. 
\end{align*}
The first equality  follows from the law of total expectation. The second equality holds because $g(\tau_{\tilde{\mu}^\eta}(x),x) = g(\tau_{\bar{\mu}}(x),x)$ when $x \in B$.  The first inequality follows from Assumption~\ref{ass:bound}. The second inequality follows from the union bound and the definition of the set $B$. The third and final inequality holds because $\Prb(x_t \in A_t^\eta \setminus A_t) = \Prb_{x,t}(A_t^\eta \setminus A_t) \le \frac{\eta}{T}$. This concludes our proof of Claim~\ref{claim:2}. 
\halmos \end{proof}

In the above Claim~\ref{claim:2}, we showed that we can construct a new exercise policy $\tilde{\mu}^\eta \equiv (\tilde{\mu}^\eta_1,\ldots,\tilde{\mu}^\eta_T)$  which is close to the original exercise policy  $\bar{\mu} \equiv (\bar{\mu}_1,\ldots,\bar{\mu}_T)$ with respect to expected reward. However, the new exercise policy is comprised of exercise regions in each period which are open sets. We now use this open set property, along with Assumption~\ref{ass:continuous}, to show the following second claim.
\begin{claim} \label{claim:3}
There exists an exercise policy $\breve{\mu}^\eta \equiv (\breve{\mu}_1^\eta,\ldots,\breve{\mu}_T^\eta)$ such that: 
\begin{enumerate}[label={(\alph*)},ref=(\alph*)]
\item  $\Exp[g(\tau_{\breve{\mu}^\eta}(x),x)] \ge \Exp[g(\tau_{\tilde{\mu}^\eta}(x),x)] - \eta$; \label{claim3:a}
\item $\lim_{y \to x} \tau_{\breve{\mu}^\eta}(y) = \tau_{\breve{\mu}^\eta}(x)$ almost surely. \label{claim3:b}
\end{enumerate}
\end{claim}
\begin{proof}{Proof of Claim~\ref{claim:3}.} 
For each period $t \in \{1,\ldots,T\}$, consider the exercise region $A_t^\eta \triangleq \{y_t \in \mathcal{X}: \tilde{\mu}^\eta(y_t) = \textsc{Stop} \}$ corresponding to the exercise policy $\tilde{\mu}^\eta_t$. Since $A_t^\eta $ is an open set, it is a well-known result from measure theory \citep[Remark 11.11(a)]{rudin1964principles} that $A_t^\eta$ is the union of a countable collection of open balls. That is, there exists a countable set of tuples $\{(y_t^{\eta,\ell}, \epsilon_t^{\eta,\ell})\}_{\ell \in \N}$ such that $\epsilon_t^{\eta,\ell} > 0$ for all $\ell \in \N$ and the following equality is satisfied:
\begin{align*}
A_t^\eta = \lim_{k \to \infty} \underbrace{ \bigcup_{\ell = 1}^k \left \{ y_t \in \R^d: \| y_t - y_t^{{\color{black}\eta,}\ell} \|_\infty \le \epsilon_t^{{\color{black}\eta,}\ell} \right \}}_{C^{\eta,k}_t}.
\end{align*}
We observe from the above construction that $C^{\eta,k}_t$, $k \in \N$, is an increasing sequence of sets which converges to $A_t^\eta$. Consequently, we have 
$\lim_{k \to \infty} \Prb \left( x_t \in A_t^\eta \setminus C^{\eta,k}_t \right) = 0$, 
and so it follows from the definition of a limit that there exists a $k_t^\eta \in \N$ such that  
\begin{align}
\Prb \left( x_t \in A_t^\eta \setminus C^{\eta,k_t^\eta}_t \right) \le \frac{\eta}{T}. \label{line:toolongtohandle}
\end{align}
Using these sets, we define an exercise policy on each period $t$ by
\begin{align*}
\breve{\mu}^\eta_t(y_t) \triangleq \begin{cases}
\textsc{Stop},&\text{if } y_t \in C^{\eta,k_t^\eta}_t,\\
\textsc{Continue},&\text{otherwise}.
\end{cases} 
\end{align*}
Since the inequality \eqref{line:toolongtohandle} holds for each period $t$, it follows from identical reasoning as the proof of property \ref{claim1:a} of Claim~\ref{claim:2} that the new exercise policy $\breve{\mu}^\eta \equiv (\breve{\mu}^\eta_1,\ldots,\breve{\mu}^\eta_T)$ satisfies property~\ref{claim3:a} of Claim~\ref{claim:3}.

Moreover, since the exercise regions $C^{\eta,k_t^\eta}_t$ are unions of finite numbers of open balls, and since the Lebesgue measure of the boundaries of a finite number of open balls is equal to zero, it follows from Assumption~\ref{ass:continuous} that the random state in each period $x_t \in \R^d$ will be a strictly positive distance from the boundary of  $C^{\eta,k_t^\eta}_t$ with probability one. This concludes our proof of property~\ref{claim3:b} of Claim~\ref{claim:3}, and thus concludes the proof of Claim~\ref{claim:3}. 
\halmos \end{proof}

We now combine Claims~\ref{claim:2} and \ref{claim:3} to conclude our proof of Lemma~\ref{lem:bradsmart}. Indeed, it follows from these claims that 
 \begin{align*}
  \Exp \left[ g(\tau_{\bar{\mu}}(x),x)\right] \le \Exp \left[g(\tau_{\tilde{\mu}^\eta}(x),x) \right] + \eta  \le \Exp \left[g(\tau_{\breve{\mu}^\eta}(x),x) \right]  + 2 \eta =  \Exp \left[\liminf_{y \to x} g(\tau_{\breve{\mu}^\eta}(y),x) \right]  + 2 \eta,
 \end{align*}
 where the first inequality follows from property~\ref{claim1:a} of Claim~\ref{claim:2}, the second inequality follows from property~\ref{claim3:a} of Claim~\ref{claim:3}, and the third inequality follows from property~\ref{claim3:b} of Claim~\ref{claim:3}. Since the exercise policy $\bar{\mu} \equiv (\bar{\mu}_1,\ldots,\bar{\mu}_T)$ and constant $\eta > 0$ were chosen arbitrarily, we have proven that
\begin{align*}
\sup_\mu \Exp \left[g(\tau_{\bar{\mu}}(x),x)\right] \le \sup_\mu \Exp \left[ \liminf_{y \to x} g(\tau_{{\mu}}(y),x)\right]. 
\end{align*}
The other direction of the inequality obviously holds, and so our proof of Lemma~\ref{lem:bradsmart} is complete. 
\halmos \end{proof}

We now combine the above novel intermediary lemmas to establish our proofs of Theorems~\ref{thm:conv:asympt} and \ref{thm:conv:policy}.
}

{\color{black}\begin{proof}{Proof of Theorem~\ref{thm:conv:asympt}.}
We first show that the optimal objective value of \eqref{prob:sro} is an asymptotic lower bound on the optimal objective value of \eqref{prob:main}. Indeed,
\begin{align}
0 &\ge \lim_{\epsilon \to 0}  \limsup_{N \to \infty} \sup_{\mu} \left \{  \widehat{J}_{N,\epsilon}(\mu)  - J^*(\mu)  \right \} \quad  \text{almost surely}  \label{line:idc1}\\
&\ge  \lim_{\epsilon \to 0}  \limsup_{N \to \infty} \left( \sup_{\mu} \widehat{J}_{N,\epsilon}(\mu) - \sup_{\mu}  J^*(\mu)  \right) \label{line:idc2}  \\
 &= - \sup_\mu J^*(\mu)+  \lim_{\epsilon \to 0}  \limsup_{N \to \infty}  \sup_{\mu} \widehat{J}_{N,\epsilon}(\mu) \label{line:idc3}\\
  &= - \sup_\mu J^*(\mu)+  \lim_{\epsilon \to 0}  \lim_{N \to \infty}  \sup_{\mu} \widehat{J}_{N,\epsilon}(\mu) \quad \text{almost surely}, \label{line:idc4}
\end{align}
where \eqref{line:idc1} follows from Theorem~\ref{thm:conv:ub}, \eqref{line:idc2} and \eqref{line:idc3} follow from algebra, and \eqref{line:idc4} follows from Lemma~\ref{lem:mcdiarmid}. Note that all of the above limits exist because $\epsilon \mapsto \limsup_{N \to \infty} \sup_{\mu}  \{  \widehat{J}_{N,\epsilon}(\mu)  - J^*(\mu)   \}$ and $\epsilon \mapsto  \limsup_{N \to \infty}  \sup_{\mu} \widehat{J}_{N,\epsilon}(\mu)$ are monotonic functions. 

We next show that the optimal objective value of \eqref{prob:sro} provides an asymptotic upper bound on the optimal objective value of \eqref{prob:main}. Indeed,  we observe that
\begin{align}
\lim_{\epsilon \to 0}  \lim_{N \to \infty}  \sup_{\mu} \widehat{J}_{N,\epsilon}(\mu) &= \lim_{\epsilon \to 0}   \liminf_{N \to \infty} \sup_{\mu} \frac{1}{N} \sum_{i=1}^N \inf_{y \in \mathcal{U}^i }    g(\tau_\mu(y), x^i) \notag\\
 &\ge  \lim_{\epsilon \to 0}  \sup_{\mu} \liminf_{N \to \infty}  \frac{1}{N} \sum_{i=1}^N \inf_{y \in \mathcal{U}^i }    g(\tau_{{\color{black}{\mu}}}(y), x^i) \label{line:ihatenames_again} \\
 &=  \lim_{\epsilon \to 0}  \sup_{\mu} \Exp \left[  \inf_{y \in \mathcal{X}^T: \; \| y - x \| \le \epsilon }    g(\tau_{{\color{black}{\mu}}}(y), x)\right] \quad \text{almost surely} \label{line:ihatenames}\\
  &{\color{black}=  \sup_{\mu} \Exp \left[  \liminf_{y \to x }    g(\tau_{{\mu}}(y), x)\right] \quad \text{almost surely}}. \label{line:istillhatenames}\\
 & {\color{black}= \sup_{\mu} \Exp \left[ g(\tau_\mu(x),x) \right]} \label{line:bradsosmart}
 \end{align}
 \eqref{line:ihatenames_again} follows from algebra; 
\eqref{line:ihatenames} follows from the strong law of large numbers; \eqref{line:istillhatenames} follows from the dominated convergence theorem. We note that the strong law of large numbers and dominated convergence theorem can both be applied because of Assumption~\ref{ass:bound}. Finally, \eqref{line:bradsosmart} follows from Lemma~\ref{lem:bradsmart}. Combining the above, our proof of Theorem~\ref{thm:conv:asympt} is complete. \halmos
\end{proof}}

\begin{proof}{Proof of Theorem~\ref{thm:conv:policy}.}
We observe from Theorems~\ref{thm:conv:asympt} and \ref{thm:conv:ub} that for every arbitrary $\eta > 0$, there exists a finite $\bar{\epsilon}(\eta) > 0$ almost surely such that the following statements hold for all $0 < \epsilon < \bar{\epsilon}(\eta)$: 
\begin{align}
\left| \lim_{N \to \infty} \widehat{J}_{N,\epsilon}(\hat{\mu}_{N,\epsilon}) - \sup_\mu J^*(\mu) \right|  \le \eta \quad \textnormal{almost surely}; \label{line:noname1}\\
 \liminf_{N \to \infty}  \left( J^*(\hat{\mu}_{N,\epsilon}) - \widehat{J}_{N,\epsilon}(\hat{\mu}_{N,\epsilon}) \right) \ge - \eta 
 \quad \textnormal{almost surely} .\label{line:noname2}
\end{align}
Therefore,
\begin{align*}
\sup_\mu J^*(\mu)  \ge \limsup_{N \to \infty} J^*(\hat{\mu}_{N,\epsilon})  \ge \liminf_{N \to \infty} J^*(\hat{\mu}_{N,\epsilon})  \ge  \liminf_{N \to \infty} \widehat{J}_{N,\epsilon}(\hat{\mu}_{N,\epsilon}) - \eta \ge \sup_\mu J^*(\mu) - 2\eta,
\end{align*}
where the first inequality holds because each $\hat{\mu}_{N,\epsilon}$ is a feasible but possibly suboptimal solution to \eqref{prob:main}, the second inequality is obvious, the third inequality follows from \eqref{line:noname2}, and the final inequality follows from \eqref{line:noname1}. Rearranging the above line, we have shown that the following statements hold  for all $0 < \epsilon < \bar{\epsilon}(\eta)$:
\begin{align*}
\left|  \limsup_{N \to \infty} J^*(\hat{\mu}_{N,\epsilon})  - \sup_\mu J^*(\mu) \right | &\le 2\eta \quad \textnormal{almost surely};\\
\left|  \liminf_{N \to \infty} J^*(\hat{\mu}_{N,\epsilon})  - \sup_\mu J^*(\mu) \right | &\le 2\eta \quad \textnormal{almost surely}.
\end{align*}
Since $\eta>0$ was chosen arbitrarily, our proof of Theorem~\ref{thm:conv:policy} is complete. \halmos \end{proof}

\section{\color{black}Proofs from \S\ref{sec:main}} \label{appx:characterization}

{\color{black}

\begin{proof}{Proof of Lemma~\ref{lem:restrict_to_completion}.}
We recall from \S\ref{sec:setting:notation} that a Markovian stopping rule satisfies $\tau_\mu(y) = \infty$ for a trajectory $y  \equiv (y_1,\ldots,y_T) \in \mathcal{X}^T$ if and only if $\mu_t(y_t) = \textsc{Continue}$ for each period $t \in \{1,\ldots,T\}$. Therefore, we observe the robust optimization problem~\eqref{prob:sro_2} is equivalent to the following optimization problem:
\begin{align} \label{prob:sro_intermediary} \tag{RO''}
\begin{aligned}
&\sup_{\mu} &&\frac{1}{N} \sum_{i=1}^N \inf_{y \in \mathcal{U}^i}  g(\tau_\mu(y), x^i)\\
&\textnormal{subject to}&& \tau_\mu(y) < \infty\quad  \textnormal{for all } i \in \{1,\ldots,N\} \textnormal{ and } y \in \mathcal{U}^i 
\end{aligned}
\end{align}
Now consider any arbitrary exercise policy $\mu \equiv (\mu_1,\ldots,\mu_T)$, and let $\mu' \equiv (\mu',\ldots,\mu')$ be an exercise policy that is defined for each period $t \in \{1,\ldots,T\}$ and state $y_t \in \mathcal{X}$ as
\begin{align*}
\mu'_t(y_t) \triangleq \begin{cases}
\mu_t(y_t),&\text{if } t \in \{1,\ldots,T-1\},\\
\textsc{Stop},&\text{if } t = T.
\end{cases}
\end{align*}
We readily observe that $\tau_{\mu'}(y) \le T$ for each $i \in \{1,\ldots,N\}$ and $y \in \mathcal{U}^i$, which implies that $\mu'$ satisfies the constraints of \eqref{prob:sro_intermediary}. Moreover, we observe for each $i \in \{1,\ldots,N\}$ and $y \in \mathcal{U}^i$ that
\begin{align*}
\tau_{\mu'}(y) &= \begin{cases}
\tau_{\mu}(y),&\text{if } \tau_{\mu}(y) < \infty,\\
T,&\text{if } \tau_\mu(y) = \infty.
\end{cases}
\end{align*}
Since the reward function satisfies $g(1,y),\ldots,g(T,y) \ge 0$ and $g(\infty,y) = 0$ for all $y \in \mathcal{X}^T$, we have shown for each $i \in \{1,\ldots,N\}$ and $y \in \mathcal{U}^i$ that
\begin{align*}
 g(\tau_\mu(y), x^i) &= g(\tau_{\mu'}(y),x^i) &&\textnormal{ if } \tau_\mu(y) < \infty, \textnormal{ and } \\
  g(\tau_\mu(y), x^i) &= 0 \le  g(T,x^i) = g(\tau_{\mu'}(y),x^i) &&\textnormal{ if } \tau_\mu(y) = \infty. 
\end{align*}
We thus conclude that the objective value $\frac{1}{N} \sum_{i=1}^N \inf_{y \in \mathcal{U}^i} g(\tau_{\mu'}(y),x^i)$ associated with the new exercise policy $\mu' \equiv (\mu'_1,\ldots,\mu'_T)$ is always greater than or equal to the objective value $\frac{1}{N} \sum_{i=1}^N \inf_{y \in \mathcal{U}^i} g(\tau_{\mu}(y),x^i)$ associated with the original exercise policy $\mu \equiv (\mu_1,\ldots,\mu_T)$. Since $\mu \equiv (\mu_1,\ldots,\mu_T)$ was chosen arbitrarily, our proof of Lemma~\ref{lem:restrict_to_completion} is complete. \halmos \end{proof}

\begin{proof}{Proof of Lemma~\ref{lem:obj_sigma}.}
Consider any $\mu \equiv (\mu_1,\ldots,\mu_T)$ that satisfies the constraints of \eqref{prob:sro_2}, and define
\begin{align*}
\sigma^i \triangleq 
\min  \left \{ t \in \{1,\ldots,T\}: \; \mu_t(y_t) = \textsc{Stop} \textnormal{ for all } y_t \in \mathcal{U}^i_t \right \}  \quad \forall i \in \{1,\ldots,N\}. 
\end{align*}
It follows from the fact that $\mu$ is feasible for \eqref{prob:sro_2} that $\sigma^1,\ldots,\sigma^N \in \{1,\ldots,T\}$. Moreover, for each sample path $i \in \{1,\ldots,N\}$ and trajectory $y \in \mathcal{U}^i$,  we observe that 
\begin{align*}
\tau_\mu(y) &=  \min \{ t \in \{1,\ldots,T\}: \mu_t(y_t) = \textsc{Stop} \} \le \sigma^i,
\end{align*}
where the equality is simply the definition of a Markovian stopping rule and the inequality follows from the fact that $\mu_{\sigma^i}(y_{\sigma^i}) = \textsc{Stop}$ for all $y_{\sigma^i} \in \mathcal{U}^i_{\sigma^i}$. Therefore, we observe for each sample path $i \in \{1,\ldots,N\}$ that 
\begin{align*}
\left \{ \tau_\mu(y): y \in \mathcal{U}^i \right\} &= \left \{t \in \{1,\ldots,\sigma^i\}: \begin{aligned} &\textnormal{there exists } y \in \mathcal{U}^i \textnormal{ such that } \mu_{t'}(y_{t'}) = \textsc{Continue} \\
&\textnormal{for all } t' \in \{1,\ldots,t-1\} \textnormal{ and } \mu_t(y)  = \textsc{Stop}\end{aligned} \right \} \\
&=  \left \{t \in \{1,\ldots,\sigma^i\}:  \; \textnormal{there exists } y_t \in \mathcal{U}^i_t \textnormal{ such that } \mu_{t}(y_{t}) = \textsc{Stop}  \right \}. 
\end{align*}
Indeed,  the first equality follows from the definition of a Markovian stopping rule and from the fact that $\tau_\mu(y) \le \sigma^i$ for all $y \in \mathcal{U}^i$.  The second equality follows from the fact that $\mathcal{U}^i_t$ is not a subset of $\{y_t \in \mathcal{X}: \mu_t(y_t) = \textsc{Stop} \}$ for each period $t \in \{1,\ldots,\sigma^i -1 \}$, which implies for each $t \in \{1,\ldots,\sigma^i -1 \}$ that there exists a $y_t \in \mathcal{U}^i_t$ that satisfies $\mu_t(y_t) = \textsc{Continue}$. We thus conclude for each sample path $i \in \{1,\ldots,N\}$ that  
\begin{align*}
& \inf_{y \in \mathcal{U}^i} g(\tau_{\mu}(y),x^i) = \min \limits_{t \in \{1,\ldots,\sigma^i\}} \left \{ g(t,x^i) :  \textnormal{there exists }  y_t \in \mathcal{U}^i_t \textnormal{ such that } \mu_t(y_t) = \textsc{Stop} \right \},
\end{align*}
which completes our proof of Lemma~\ref{lem:obj_sigma}.\halmos \end{proof}

\begin{proof}{Proof of Lemma~\ref{lem:pruning}.}
Let $\mu'$ be a pruned version of $\mu$, and let $\sigma^1,\ldots,\sigma^N$ satisfy the following equalities for each $i \in \{1,\ldots,N\}$:
\begin{align}
\sigma^i &= \min \left \{ t \in \{1,\ldots,T\}: \; \mu_t(y_t) = \textsc{Stop}\; \forall y_t \in \mathcal{U}^i_t \right \}= \min \left \{ t \in \{1,\ldots,T\}: \; \mu_t'(y_t) = \textsc{Stop}\; \forall y_t \in \mathcal{U}^i_t \right \}. \label{line:almostdonehurray}
\end{align} 
Then it follows from the fact that $\mu$ is feasible for \eqref{prob:sro_2} that $\sigma^1,\ldots,\sigma^N \in \{1,\ldots,T\}$. Therefore, for each $i \in \{1,\ldots,N\}$, 
\begin{align*}
\inf_{y \in \mathcal{U}^i} g(\tau_{\mu'}(y),x^i) &=  \min \limits_{t \in \{1,\ldots,\sigma^i\}} \left \{ g(t,x^i) :  \textnormal{there exists }  y_t \in \mathcal{U}^i_t \textnormal{ such that } \mu_t'(y_t) = \textsc{Stop} \right \} \\
&\ge  \min \limits_{t \in \{1,\ldots,\sigma^i\}} \left \{ g(t,x^i) :  \textnormal{there exists }  y_t \in \mathcal{U}^i_t \textnormal{ such that } \mu_t(y_t) = \textsc{Stop} \right \} \\
&= \inf_{y \in \mathcal{U}^i} g(\tau_{\mu}(y),x^i).
\end{align*} 
Indeed, the two equalities follow from Lemma~\ref{lem:obj_sigma} and line~\eqref{line:almostdonehurray}. The inequality follows from the fact that $\mu'$ is a pruned version of $\mu$, which implies that $\left \{y_t \in \mathcal{X}: \mu_t'(y_t) = \textsc{Stop} \right \} \subseteq  \left \{y_t \in \mathcal{X}: \mu_t(y_t) = \textsc{Stop} \right \}$ for all $t \in \{1,\ldots,T\}$. Our proof of Lemma~\ref{lem:pruning} is thus complete. 
\halmos \end{proof}

\begin{proof}{Proof of Lemma~\ref{lem:M_is_pruned}.}
Consider any $\mu \equiv (\mu_1,\ldots,\mu_T)$ that satisfies the constraints of \eqref{prob:sro_2}, and define
\begin{align*}
\sigma^i \triangleq \min  \left \{ t \in \{1,\ldots,T\}: \; \mu_t(y_t) = \textsc{Stop} \textnormal{ for all } y_t \in \mathcal{U}^i_t \right \}  \quad \forall i \in \{1,\ldots,N\}. 
\end{align*}
Our proof that $\mu^{\sigma^1 \cdots \sigma^N}$ is a pruned version of $\mu$ is split into the following two intermediary claims.
\begin{claim} \label{claim:M_is_pruned:1}
$ \{y_t \in \mathcal{X}: \mu_t^{\sigma^1 \cdots \sigma^N}(y_t) = \textsc{Stop}  \} \subseteq  \left \{y_t \in \mathcal{X}: \mu_t(y_t) = \textsc{Stop} \right \}$ for all $t \in \{1,\ldots,T\}$. 
 \end{claim}
 \begin{proof}{Proof of Claim~\ref{claim:M_is_pruned:1}.}
 Indeed, we observe for each period $t \in \{1,\ldots,T \}$ that 
 \begin{align*}
 &\left \{y_t \in \mathcal{X}: \mu_t^{\sigma^1 \cdots \sigma^N}(y_t) = \textsc{Stop} \right \} \\
 &= \left \{ y_t \in \mathcal{X}: y_t \in  \bigcup \limits_{i:\; \sigma^i = t} \mathcal{U}^i_t \right \} \\
 &\subseteq \left \{ y_t \in \mathcal{X}: y_t \in  \bigcup \limits_{i:\; \sigma^i = t} \left \{ y_t' \in \mathcal{X}: \mu_t(y_t') = \textsc{Stop} \right \} \right \} \\
 &= \begin{cases}
 \left \{ y_t \in \mathcal{X}: \mu_t(y_t) = \textsc{Stop} \right \},&\text{if there exists } i \in \{1,\ldots,N\} \textnormal{ such that } \sigma^i = t, \\
 \emptyset,&\text{otherwise}
 \end{cases}\\
 &\subseteq  \left \{y_t \in \mathcal{X}: \mu_t(y_t) = \textsc{Stop} \right \},
 \end{align*}
where the first equality follows from the definition of $\mu^{\sigma^1 \cdots \sigma^N}_t$, the first inclusion follows from the fact that $\mathcal{U}^i_t \subseteq \{y_t \in \mathcal{X}: \mu_t(y_t) = \textsc{Stop} \}$ for each sample path $i$ that satisfies $\sigma^i = t$,  the second equality follows from algebra, and the second inclusion follows from algebra.  This concludes our proof of Claim~\ref{claim:M_is_pruned:1}.
\halmos \end{proof}
\begin{claim}\label{claim:M_is_pruned:2}
For each $i \in \{1,\ldots,N\}$,
\begin{align*}
\min \left \{ t \in \{1,\ldots,T\}: \; \mu_t^{\sigma^1 \cdots \sigma^N}(y_t) = \textsc{Stop}\; \forall y_t \in \mathcal{U}^i_t \right \}= \min \left \{ t \in \{1,\ldots,T\}: \; \mu_t(y_t) = \textsc{Stop} \;\forall y_t \in \mathcal{U}^i_t \right \}.
\end{align*}
\end{claim}
\begin{proof}{Proof of Claim~\ref{claim:M_is_pruned:2}.}
Consider any sample path $i \in \{1,\ldots,N\}$. We first observe that 
 \begin{align*}
\min \left \{ t \in \{1,\ldots,T\}: \; \mu_t^{\sigma^1 \cdots \sigma^N}(y_t) = \textsc{Stop}\; \forall y_t \in \mathcal{U}^i_t \right \}
&\le \sigma^i \\
&= \min \left \{ t \in \{1,\ldots,T\}: \; \mu_t(y_t) = \textsc{Stop} \;\forall y_t \in \mathcal{U}^i_t \right \},
 \end{align*}
 where the inequality follows from the fact that $\mu_{\sigma^i}^{\sigma^1 \cdots \sigma^N}(y_{\sigma^i}) = \textsc{Stop}$ for all $y_{\sigma^i} \in \mathcal{U}^i_{\sigma^i}$, and the equality follows from the definition of $\sigma^i$. Moreover, it follows immediately from Claim~\ref{claim:M_is_pruned:1} that
\begin{align*}
\min \left \{ t \in \{1,\ldots,T\}: \; \mu_t^{\sigma^1 \cdots \sigma^N}(y_t) = \textsc{Stop}\; \forall y_t \in \mathcal{U}^i_t \right \}\ge  \min \left \{ t \in \{1,\ldots,T\}: \; \mu_t(y_t) = \textsc{Stop} \;\forall y_t \in \mathcal{U}^i_t \right \}.
\end{align*}
Our proof of Claim~\ref{claim:M_is_pruned:2} is thus complete. 
 \halmos \end{proof}
 Combining Claims~\ref{claim:M_is_pruned:1} and \ref{claim:M_is_pruned:2} with Definition~\ref{defn:pruning}, we conclude that $\mu^{\sigma^1 \cdots \sigma^N}$ is a pruned version of $\mu$, which completes our proof of Lemma~\ref{lem:M_is_pruned}. 
\halmos \end{proof}

The following proof follows the identical reasoning as discussed in \S\ref{sec:main:proof}, and is stated formally here for the sake of completeness. 
\begin{proof}{Proof  of Theorem~\ref{thm:characterization}.}
 Lemma~\ref{lem:restrict_to_completion} shows that the robust optimization problem~\eqref{prob:sro} is equivalent to the robust optimization problem~\eqref{prob:sro_2}. Moreover, for any arbitrary  exercise policy $\mu$ that is feasible for \eqref{prob:sro_2},  Lemmas~\ref{lem:pruning} and \ref{lem:M_is_pruned} together show that there exists an exercise policy $\mu' \in \mathcal{M}$ such that the objective value associated with $\mu$ is less than or equal to the objective value associated with $\mu'$.   Since $\mu$ was chosen arbitrarily, our proof of Theorem~\ref{thm:characterization} is complete. 
\halmos \end{proof}

\begin{proof}{Proof of Theorem~\ref{thm:reform}.}
We split our proof into the following two intermediary claims.
\begin{claim} \label{claim:reform:1}
The optimal objective value of \eqref{prob:sro} is less than or equal to the optimal objective value of \eqref{prob:mip}.
\end{claim}
\begin{proof}{Proof of Claim~\ref{claim:reform:1}.}
Consider any arbitrary $\mu \in \mathcal{M}$, and define the integers
\begin{align}
\sigma^i \triangleq \min  \left \{ t \in \{1,\ldots,T\}: \; \mu_t(y_t) = \textsc{Stop} \textnormal{ for all } y_t \in \mathcal{U}^i_t \right \}  \quad \forall i \in \{1,\ldots,N\}. \label{line:argument}
\end{align}
It follows from the definition of $\mathcal{M}$ that $\mu$  satisfies the constraints of \eqref{prob:sro_2}, and so  Lemma~\ref{lem:M_is_pruned} implies that $\mu^{\sigma^1 \cdots \sigma^N}$ is a pruned version of $\mu$. Therefore, it follows from Definition~\ref{defn:pruning} and line~\eqref{line:argument} that the following equalities hold: 
\begin{align}
\sigma^i = \min  \left \{ t \in \{1,\ldots,T\}: \; \mu_t^{\sigma^1 \cdots \sigma^N}(y_t) = \textsc{Stop} \textnormal{ for all } y_t \in \mathcal{U}^i_t \right \}  \quad \forall i \in \{1,\ldots,N\}. \label{line:argument2}
\end{align}
Therefore, we observe that 
\begin{align}
&\frac{1}{N} \sum_{i=1}^N  \inf_{y \in \mathcal{U}^i} g \left(\tau_{\mu}(y),x^i \right) \notag \\
&\le \frac{1}{N} \sum_{i=1}^N  \inf_{y \in \mathcal{U}^i} g \left(\tau_{\mu^{\sigma^1 \cdots \sigma^N}}(y),x^i \right)\notag  \\
 &= \frac{1}{N} \sum_{i=1}^N  \min \limits_{t \in \{1,\ldots,\sigma^i\}} \left \{ g(t,x^i) :  \textnormal{there exists }  y_t \in \mathcal{U}^i_t \textnormal{ such that } \mu_t^{\sigma^1 \cdots \sigma^N}(y_t) = \textsc{Stop} \right \} \notag \\
&=  \frac{1}{N} \sum_{i=1}^N \min_{t \in \{1,\ldots,\sigma^i \}} \left \{g(t,x^i): \textnormal{ there exists } j \textnormal{ such that } \mathcal{U}^i_t \cap \mathcal{U}^j_t \neq \emptyset \textnormal{ and } \sigma^j = t\right \}, \label{line:argument3}
\end{align}
where the first inequality follows from Lemma~\ref{lem:pruning},  the first equality follows from Lemma~\ref{lem:obj_sigma} and line~\eqref{line:argument2}, and the second equality  follows from the fact that $\mu^{\sigma^1 \cdots \sigma^N}_t(y_t) = \textsc{Stop}$ if and only if there exists a sample path $j$ such that $y_t \in \mathcal{U}^j_t$ and $\sigma^j = t$. Because $\mu \in \mathcal{M}$ was chosen arbitrarily, we conclude from line~\eqref{line:argument3} and Theorem~\ref{thm:characterization} that  the optimal objective value of \eqref{prob:sro} is less than or equal to the optimal objective value of \eqref{prob:mip}. Our proof of Claim~\ref{claim:reform:1} is thus complete.  \halmos
\end{proof}

\begin{claim} \label{claim:reform:2}
The optimal objective value of \eqref{prob:sro} is greater than or equal to the optimal objective value of \eqref{prob:mip}.  Furthermore, for any choice of integers $\sigma^1,\ldots,\sigma^N \in \{1,\ldots,T\}$, the corresponding exercise policy $\mu^{\sigma^1 \cdots \sigma^N} \equiv (\mu_1^{\sigma^1 \cdots \sigma^N},\ldots,\mu_T^{\sigma^1 \cdots \sigma^N})$ satisfies  %
 \begin{align*}
 &\frac{1}{N} \sum_{i=1}^N \inf_{y \in \mathcal{U}^i}  g \left(\tau_{\mu^{\sigma^1 \cdots \sigma^N}}(y), x^i \right) \\
 &\ge \frac{1}{N} \sum_{i=1}^N \min_{t \in \{1,\ldots,\sigma^i \}} \left \{g(t,x^i): \textnormal{ there exists } j \textnormal{ such that } \mathcal{U}^i_t \cap \mathcal{U}^j_t \neq \emptyset \textnormal{ and } \sigma^j = t\right \}.
 \end{align*}
\end{claim}
\begin{proof}{Proof of Claim~\ref{claim:reform:2}.}
Consider any arbitrary integers $\sigma^1,\ldots,\sigma^N \in \{1,\ldots,T\}$. For each sample path $i \in \{1,\ldots,N\}$, we observe that
\begin{align}
&\min  \left \{ t \in \{1,\ldots,T\}: \; \mu_t^{\sigma^1 \cdots \sigma^N}(y_t) = \textsc{Stop} \textnormal{ for all } y_t \in \mathcal{U}^i_t \right \}\notag  \\
&=\min  \left \{ t \in \{1,\ldots,\sigma^i\}: \; \mu_t^{\sigma^1 \cdots \sigma^N}(y_t) = \textsc{Stop} \textnormal{ for all } y_t \in \mathcal{U}^i_t \right \} \notag \\
&\le \sigma^i,\label{line:argument4}
\end{align}
where  the equality follows from the fact that $\mu^{\sigma^1 \cdots \sigma^N} \equiv (\mu^{\sigma^1 \cdots \sigma^N}_1,\ldots,\mu^{\sigma^1 \cdots \sigma^N}_T)$ by construction satisfies $\mu^{\sigma^1 \cdots \sigma^N}_{\sigma^i}(y_{\sigma^i}) = \textsc{Stop}$ for all $y_{\sigma^i} \in \mathcal{U}^i_{\sigma^i}$, and the inequality follows from algebra. Therefore, 
 \begin{align}
&\frac{1}{N} \sum_{i=1}^N  \inf_{y \in \mathcal{U}^i} g \left(\tau_{\mu^{\sigma^1 \cdots \sigma^N}}(y),x^i \right) \notag  \\
&\ge \frac{1}{N} \sum_{i=1}^N  \min \limits_{t \in \{1,\ldots,\sigma^i\}} \left \{ g(t,x^i) :  \textnormal{there exists }  y_t \in \mathcal{U}^i_t \textnormal{ such that } \mu_t^{\sigma^1 \cdots \sigma^N}(y_t) = \textsc{Stop} \right \} \notag \\
&=  \frac{1}{N} \sum_{i=1}^N \min_{t \in \{1,\ldots,\sigma^i \}} \left \{g(t,x^i): \textnormal{ there exists } j \textnormal{ such that } \mathcal{U}^i_t \cap \mathcal{U}^j_t \neq \emptyset \textnormal{ and } \sigma^j = t\right \},  \label{line:argument5}
\end{align}
where the inequality follows from Lemma~\ref{lem:obj_sigma} and line~\eqref{line:argument4}, and the equality follows from the fact that $\mu^{\sigma^1 \cdots \sigma^N}_t(y_t) = \textsc{Stop}$ if and only if there exists a sample path $j$ such that $y_t \in \mathcal{U}^j_t$ and $\sigma^j = t$. Because $\sigma^1,\ldots,\sigma^N \in \{1,\ldots,T\}$ were chosen arbitrarily, we conclude from line~\eqref{line:argument5} and Theorem~\ref{thm:characterization} that the optimal objective value of \eqref{prob:sro} is greater than or equal to the optimal objective value of \eqref{prob:mip}, which concludes our proof  of Claim~\ref{claim:reform:2} is complete.
%
\halmos \end{proof}
Combining Claims~\ref{claim:reform:1} and \ref{claim:reform:2}, our proof of Theorem~\ref{thm:reform} is thus complete. 
\halmos \end{proof}

}

\section{Proof of Theorem~\ref{thm:hard}} \label{appx:hard}
Our proof of the computational complexity of \eqref{prob:mip} consists of a reduction from MIN-2-SAT, which is shown to be strongly NP-hard by \cite{ kohli1994minimum}:
%

\vspace{1em}

\begin{center}
\fbox{\begin{minipage}{0.9\linewidth}
\begin{center}
\textbf{\underline{MIN-2-SAT}}
\end{center}
\vspace{1em}
The optimization version of MIN-2-SAT is to compute the optimal objective value of the binary linear optimization problem 
\begin{equation*}
\begin{aligned}
v^{\textnormal{MIN-2-SAT}}\triangleq \quad &\underset{b,z}{\textnormal{minimize}}&& \sum_{k=1}^K z_k\\
&\textnormal{subject to}&& z_k \ge b_\ell && \forall k \in \{1,\ldots,K \}, \forall \ell \in I^+_k\\
&&& z_k \ge 1-b_\ell && \forall k \in \{1,\ldots,K \}, \forall \ell \in I^-_k\\
&&& b_\ell \in \{0,1\}&& \forall \ell\in \{1,\ldots,L\}, 
\end{aligned}
\end{equation*}
where the given sets $I^+_k, I^-_k \subseteq \{1,\ldots,L\}$ satisfy $| I^+_k|+|I^-_k| = 2$ for each $k \in \{1,\ldots,K\}$. 
\end{minipage}}
\end{center}
\vspace{1em}

\noindent Note that the following equality is obtained by replacing each decision variable $z_k$ with $1- z_k$:
\begin{equation*}
\begin{aligned}
v^{\textnormal{MIN-2-SAT}}= \quad  N \;\; - \;\; &\underset{b,z}{\textnormal{maximize}}&& \sum_{k=1}^K z_k\\
&\textnormal{subject to}&& z_k \le 1- b_\ell && \forall k \in \{1,\ldots,K \}, \forall \ell \in I^+_k\\
&&& z_k \le b_\ell && \forall k \in \{1,\ldots,K \}, \forall \ell \in I^-_k\\
&&& b_\ell \in \{0,1\}&& \forall \ell\in \{1,\ldots,L\}. 
\end{aligned}
\end{equation*} 
We now show that any instance of the above maximization problem can be equivalently reformulated as polynomially-size instance of   \eqref{prob:mip} with $T=3$ periods. 

\begin{proof}{Proof of Theorem~\ref{thm:hard}.}
Consider any arbitrary instance of the binary linear optimization problem 
\begin{equation} \label{prob:notmin2sat} \tag{$\lnot$MIN-2-SAT}
\begin{aligned}
 &\underset{b,z}{\textnormal{maximize}}&& \sum_{k=1}^K z_k\\
&\textnormal{subject to}&& z_k \le 1- b_\ell && \forall k \in \{1,\ldots,K \}, \forall \ell \in I^+_k\\
&&& z_k \le b_\ell && \forall k \in \{1,\ldots,K \}, \forall \ell \in I^-_k\\
&&& b_\ell \in \{0,1\}&& \forall \ell\in \{1,\ldots,L\}, 
\end{aligned}
\end{equation} 
and let $e_\ell \in \R^{L+1}$ denote the $\ell$-th column vector of the identity matrix. We construct an instance of \eqref{prob:mip} defined as follows: 
\begin{itemize}
\item The number of periods is $T=3$. 
\item The state space is $\mathcal{X} = \R^{L+1}$.
\item The reward function for each period $t \in \{1,2,3\}$ is $g(t,y) = y_t \cdot e_{{\color{black}L}+1} + K$.
\item The robustness parameter in the uncertainty sets is $\epsilon = \frac{2}{3}$. 
\item The number of sample paths is $N \triangleq L+K$, and the sample paths are defined as follows: 
\begin{itemize}
\item For each $\ell  \in \{1,\ldots,L\}$, let $x^\ell_1 =x^\ell_2 = e_\ell$ and  $x^\ell_3 = -K e_{L+1}$.
\item  For each $k \in \{1,\ldots,K\}$, let $x^{L+k}_1 =\frac{1}{2} \sum_{\ell \in I^+_k} e_{\ell} $, $x^{L+k}_2 = \frac{1}{2} \sum_{\ell \in I^-_k} e_{\ell} $, and $x_3^{L+k} = e_{L+1}$.
\end{itemize}
\end{itemize}
In the remainder of the proof, we show that the above instance of \eqref{prob:mip} is equivalent to \eqref{prob:notmin2sat}. Indeed, it follows immediately from the above construction that the values of {\color{black}$g(t,x^i)$} for each sample path $i$ and period $t$ are:
{\color{black}
\begin{align*}
\begin{aligned}
g(1,x^\ell) &= K; & g(2,x^\ell) &= K ; &  g(3,x^\ell) &= 0, && \forall \ell \in \{1,\ldots,L\},\\
g(1,x^{L+k}) &= K; &  g(2,x^{L+k}) &= K; & g(3,x^{L+k}) &= K+1, && \forall k \in \{1,\ldots,K\}.
\end{aligned}
\end{align*}
}
We require two intermediary claims: 
\begin{claim}\label{claim:hardness:2} {There exists an optimal solution for \eqref{prob:mip} which satisfies}
  $$\sigma^1,\ldots,\sigma^L \in \{1,2\} \text{ and }  \sigma^{L+1} = \cdots = \sigma^{L+K} =3.$$
  \end{claim}
\begin{proof}{Proof of Claim~\ref{claim:hardness:2}.} We observe that the optimal objective value of \eqref{prob:mip} is greater than or equal to $K$, since this objective value would be achieved by  setting $\sigma^1 = \cdots = \sigma^N = 1$. Now consider any solution to \eqref{prob:mip} where $\sigma^{\ell'} = 3$ for some $\ell' \in \{1,\ldots,L\}$. For that solution, 
{\color{black}
\begin{align*}
& \frac{1}{N} \sum_{i=1}^N \min_{t \in \{1,\ldots,\sigma^i \}} \left \{g(t,x^i): \textnormal{ there exists } j \textnormal{ such that } \mathcal{U}^i_t \cap \mathcal{U}^j_t \neq \emptyset \textnormal{ and } \sigma^j = t\right \}\\
&= \frac{1}{N} \sum_{\ell \in \{1,\ldots,L\}: \ell \neq \ell'} \min_{t \in \{1,\ldots,\sigma^\ell \}} \left \{g(t,x^\ell): \textnormal{ there exists } j \textnormal{ such that } \mathcal{U}^\ell_t \cap \mathcal{U}^j_t \neq \emptyset \textnormal{ and } \sigma^j = t\right \}\\
&\quad + \frac{1}{N}  \min_{t \in \{1,\ldots,\sigma^{\ell'} \}} \left \{g(t,x^{\ell'}): \textnormal{ there exists } j \textnormal{ such that } \mathcal{U}^{\ell'}_t \cap \mathcal{U}^j_t \neq \emptyset \textnormal{ and } \sigma^j = t\right \}\\
&\quad + \frac{1}{N} \sum_{k=1}^K \min_{t \in \{1,\ldots,\sigma^{L+k} \}} \left \{g(t,x^{L+k}): \textnormal{ there exists } j \textnormal{ such that } \mathcal{U}^{L+k}_t \cap \mathcal{U}^j_t \neq \emptyset \textnormal{ and } \sigma^j = t\right \}\\
&\le \frac{1}{N}  \sum_{\ell \in \{1,\ldots,L\}: \ell \neq \ell'} K    + 0 + \sum_{k=1}^K (K+1) \\
 &= K.
 \end{align*}}
Because the objective value associated with this solution is never better than the objective value obtained by the solution $\sigma^1 = \cdots = \sigma^{N} = 1$,  we have shown that there exists an optimal solution for  \eqref{prob:mip} that satisfies $\sigma^1,\ldots,\sigma^L \in \{1,2\}$.  

Consider any arbitrary solution $\sigma^1,\ldots,\sigma^N \in \{1,2,3\}$ that satisfies  $\sigma^1,\ldots,\sigma^L \in \{1,2\}$, and suppose that $\sigma^{L+k} \in \{1,2\}$ for some $k \in \{1,\ldots,K\}$. To perform an exchange argument, we construct an alternative solution  $\bar{\sigma}^1,\ldots,\bar{\sigma}^N \in \{1,2,3\}$ defined as
$$
\bar{\sigma}^i \triangleq \begin{cases}
\sigma^i,&\text{if } i \neq L+k,\\
3,&\text{if } i = L+k.
\end{cases}
$$
We observe that {\color{black}the inclusion} $\{j: \sigma^j \le \sigma^i \} \supseteq \{ j: \bar{\sigma}^j \le \bar{\sigma}^i \}$ holds for all $i \in \{1,\ldots,N\} \setminus \{ L+k\}$. Moreover, {\color{black}we observe that
\begin{align*}
 \min_{t \in \{1,\ldots,\sigma^{L+k} \}} \left \{g(t,x^{L+k}): \textnormal{ there exists } j \textnormal{ such that } \mathcal{U}^{L+k}_t \cap \mathcal{U}^j_t \neq \emptyset \textnormal{ and } \sigma^j = t\right \} &= K, \; \text{and}\\
  \min_{t \in \{1,\ldots,\bar{\sigma}^{L+k} \}} \left \{g(t,x^{L+k}): \textnormal{ there exists } j \textnormal{ such that } \mathcal{U}^{L+k}_t \cap \mathcal{U}^j_t \neq \emptyset \textnormal{ and } \bar{\sigma}^j = t\right \} & \in \{K,K+1\}.
\end{align*}
}
Therefore, 
{\color{black}
\begin{align*}
& \frac{1}{N} \sum_{i=1}^N \min_{t \in \{1,\ldots,\sigma^i \}} \left \{g(t,x^i): \textnormal{ there exists } j \textnormal{ such that } \mathcal{U}^i_t \cap \mathcal{U}^j_t \neq \emptyset \textnormal{ and } \sigma^j = t\right \}\\
&\le \frac{1}{N} \sum_{i=1}^N \min_{t \in \{1,\ldots,\bar{\sigma}^i \}} \left \{g(t,x^i): \textnormal{ there exists } j \textnormal{ such that } \mathcal{U}^i_t \cap \mathcal{U}^j_t \neq \emptyset \textnormal{ and } \bar{\sigma}^j = t\right \}
\end{align*}
}
 Because $\sigma^1,\ldots,\sigma^N$ was chosen arbitrarily, we conclude that there exists an optimal solution  for  \eqref{prob:mip} which satisfies $\sigma^1,\ldots,\sigma^L \in \{1,2\}$ and $\sigma^{L+1} = \cdots = \sigma^{L+K} =3$. This concludes our proof of Claim~\ref{claim:hardness:2}. 
 \halmos \end{proof}

\begin{claim}\label{claim:hardness:3}  
{If  $\sigma^{L+1} = \cdots = \sigma^{L+K} = 3$, then the following equality holds for each $k \in \{1,\ldots,K\}$:}
{\color{black}
\begin{align*}
& \min_{t \in \{1,\ldots,\sigma^{L+k} \}} \left \{g(t,x^{L+k}): \textnormal{ there exists } j \textnormal{ such that } \mathcal{U}^{L+k}_t \cap \mathcal{U}^j_t \neq \emptyset \textnormal{ and } \sigma^j = t\right \} \\
 &=K + \mathbb{I} \left \{\sigma^{\ell} = 2 \textnormal{ for all } \ell \in I^+_k \textnormal{ and } \sigma^{\ell} = 1 \textnormal{ for all } \ell \in I^-_k\right \}.
 \end{align*}
}
  \end{claim}
\begin{proof}{Proof of Claim~\ref{claim:hardness:3}.} 
 For each $\ell  \in \{1,\ldots,L\}$ and $k \in \{1,\ldots,K\}$, we observe that
 \begin{align*}
 \left \| x^\ell_1 - x^{L+k}_1 \right \|_\infty = \left \|  e_\ell - \frac{1}{2} \sum_{\ell' \in I^+_k} e_{\ell'}  \right \|_\infty = \begin{cases}
 \frac{1}{2},&\text{if } \ell \in {I}^+_k,\\
 1,&\text{otherwise}. 
 \end{cases}\\
  \left \| x^\ell_2 - x^{L+k}_2 \right \|_\infty = \left \|  e_\ell - \frac{1}{2} \sum_{\ell' \in I^-_k} e_{\ell'}  \right \|_\infty = \begin{cases}
 \frac{1}{2},&\text{if } \ell \in {I}^-_k,\\
 1,&\text{otherwise}. 
 \end{cases}
 \end{align*}
This implies that the set
 \begin{align*}
 \mathcal{U}^\ell_{1} \cap  \mathcal{U}^{L+k}_{1} = \left \{ y_1 \in \R^{L+1}: \| y_1 - x^{\ell}_1 \|_\infty \le \frac{2}{3}\right \} \cap  \left \{ y_1 \in \R^{L+1}: \| y_1 - x^{L+k}_1 \|_\infty \le \frac{2}{3}\right \} 
 \end{align*}
 is nonempty if and only if $\ell \in {I}^+_k$, and the set
  \begin{align*}
 \mathcal{U}^\ell_{2} \cap  \mathcal{U}^{L+k}_{2} = \left \{ y_2 \in \R^{L+1}: \| y_2 - x^{\ell}_2 \|_\infty \le \frac{2}{3}\right \} \cap  \left \{ y_2 \in \R^{L+1}: \| y_2 - x^{L+k}_2 \|_\infty \le \frac{2}{3}\right \} 
 \end{align*}
  is nonempty if and only if $\ell \in {I}^-_k$. Consequently,
since $\sigma^{L+1} = \cdots = \sigma^{L+K} = 3$, we conclude that the following equalities hold for all $k \in \{1,\ldots,K\}$:
\begin{align*}
&{\color{black} \min_{t \in \{1,\ldots,\sigma^{L+k} \}} \left \{g(t,x^{L+k}): \textnormal{ there exists } j \textnormal{ such that } \mathcal{U}^{L+k}_t \cap \mathcal{U}^j_t \neq \emptyset \textnormal{ and } \sigma^j = t\right \}} \\
&=\min \left \{ \min_{\ell \in {I}^+_k: \;\sigma^\ell = 1} K ,\min_{\ell \in {I}^-_k: \; \sigma^\ell = 2} K, K+1\right \} \\
&=  K + \mathbb{I} \left \{\sigma^{\ell} = 2 \textnormal{ for all } \ell \in I^+_k \textnormal{ and } \sigma^{\ell} = 1 \textnormal{ for all } \ell \in I^-_k\right \}
\end{align*}
This concludes the proof of Claim~\ref{claim:hardness:3}. \halmos \end{proof}

  We now combine Claims~\ref{claim:hardness:2} and \ref{claim:hardness:3} to complete our proof of Theorem~\ref{thm:hard}:

  \begin{align}
&\eqref{prob:mip}   \\
&= {\color{black}  \max_{\substack{\sigma^1,\ldots,\sigma^{L} \in \{1,2\}\\\sigma^{L+1} = \cdots = \sigma^{L+K} = 3}} \frac{1}{L+K} \sum_{i=1}^{L+K} \min_{t \in \{1,\ldots,\sigma^i \}} \left \{g(t,x^i): \textnormal{ there exists } j \textnormal{ such that } \mathcal{U}^i_t \cap \mathcal{U}^j_t \neq \emptyset \textnormal{ and } \sigma^j = t\right \}} \label{line:applying_da_claims} \\
&={\color{black} \max_{\substack{\sigma^1,\ldots,\sigma^{L} \in \{1,2\}\\\sigma^{L+1} = \cdots = \sigma^{L+K} = 3}} \left \{  \frac{1}{L+K} \sum_{\ell=1}^{L} K \right.} \notag \\
&\quad  {\color{black}+ \left.  \frac{1}{L+K}   \sum_{k=1}^{K}  \min_{t \in \{1,\ldots,\sigma^{L+k} \}} \left \{g(t,x^{L+k}): \textnormal{ there exists } j \textnormal{ such that } \mathcal{U}^{L+k}_t \cap \mathcal{U}^j_t \neq \emptyset \textnormal{ and } \sigma^j = t\right \}   \right \} }\label{line:applying_da_claims2} \\
&= \max_{\sigma^1,\ldots,\sigma^{L} \in \{1,2\}} \left \{  \frac{LK}{L+K}  + \frac{1}{L+K}  \sum_{k=1}^{K} \left( K +  \mathbb{I} \left \{\sigma^{\ell} = 2 \text{ for all } \ell \in I^+_k \text{ and } \sigma^{\ell} = 1 \text{ for all } \ell \in I^-_k\right \} \right)  \right \} \label{line:applying_da_claims3} \\
&=K +  \left(\frac{1}{L+K} \right)\max_{\sigma^1,\ldots,\sigma^{L} \in \{1,2\}}  \left \{  \sum_{k=1}^{K} \mathbb{I} \left \{\sigma^{\ell} = 2 \text{ for all } \ell \in I^+_k \text{ and } \sigma^{\ell} = 1 \text{ for all } \ell \in I^-_k\right \}  \right \} \notag \\
&= K +  \left(\frac{1}{L+K} \right) * \eqref{prob:notmin2sat}. \label{line:applying_da_claims4}
\end{align}
Indeed, \eqref{line:applying_da_claims} follows from Claim~\ref{claim:hardness:2}; \eqref{line:applying_da_claims2} holds because {\color{black}$g(1,x^\ell) = g(2,x^\ell) = K$} for all $\ell \in \{1,\ldots,L\}$;  \eqref{line:applying_da_claims3} follows from Claim~\ref{claim:hardness:3}; \eqref{line:applying_da_claims4} follows from algebra and setting $b^\ell = 0$ if and only if $\sigma^\ell = 2$. We have thus shown that any instance of MIN-2-SAT can be reduced to solving a polynomially-sized instance of \eqref{prob:mip} with $T=3$, which concludes our proof of Theorem~\ref{thm:hard}.  
\halmos \end{proof}

\section{{\color{black}Reformulation of \eqref{prob:bp} as Mixed-Integer Linear Optimization Problem}} \label{appx:reform}
Zero-one bilinear programs can be transformed into equivalent mixed-integer linear optimization problems by introducing auxiliary decision variables  \citep{adams1986tight}. 
In numerical experiments in \S\ref{sec:experiments}, we perform such a linearization of the bilinear program~\eqref{prob:bp} by introducing auxiliary continuous decision variables $f^i_{t\ell}$ which obey the constraints
\begin{align*}
f^i_{t\ell} \le b^i_t && \textnormal{for all } i \in \{1,\ldots,N\}, \; t \in \{1,\ldots,T\}, \; \ell \in \{1,\ldots,L^i_t-1\}\\
f^i_{t\ell} \le {\color{black}1}-w^i_{t\ell} && \textnormal{for all } i \in \{1,\ldots,N\}, \; t \in \{1,\ldots,T\}, \; \ell \in \{1,\ldots,L^i_t-1\}
\end{align*}
and replacing the objective function of  \eqref{prob:bp}  with
\begin{align*}
\begin{aligned}
&\underset{b,w,f}{\textnormal{maximize}}&&\frac{1}{N} \sum_{i=1}^N \sum_{t=1}^T \sum_{\ell=1}^{L^i_t-1}( \kappa^i_{\ell+1} - \kappa^i_\ell)f^i_{t\ell}.
\end{aligned}
\end{align*}
To strengthen this linear relaxation of \eqref{prob:bp}, we also add the valid constraints:
\begin{align*}
\begin{aligned}
& w^i_{1,1} = 0 && \textnormal{for all }i \in \{1,\ldots,N\} \\
&b^i_t + w^i_{t,1} = w^i_{t+1,1}&&\textnormal{for all }i \in \{1,\ldots,N\},\; t \in \{1,\ldots,T-1\}\\
&b^i_T + w^i_{T,1} = 1 && \textnormal{for all }i \in \{1,\ldots,N\}.
\end{aligned}
\end{align*}
Indeed, the validity of the above constraints for \eqref{prob:bp} follows from the fact that there is an optimal solution to this zero-one bilinear program which satisfies  $\sum_{t=1}^T b^i_t = 1$ for each sample path $i$. In summary,  this linearization procedure transforms \eqref{prob:bp} into the following equivalent mixed-integer linear optimization problem: 
\begin{equation*}
\begin{aligned}
&\underset{b,w,f}{\textnormal{maximize}}&&\frac{1}{N} \sum_{i=1}^N \sum_{t=1}^T \sum_{\ell=1}^{L^i_t-1}( \kappa^i_{\ell+1} - \kappa^i_\ell)f^i_{t\ell}\\
&\textnormal{subject to}&& \begin{aligned}[t]
&f^i_{t\ell} \le b^i_t && \textnormal{for all } i \in \{1,\ldots,N\}, \; t \in \{1,\ldots,T\}, \; \ell \in \{1,\ldots,L^i_t-1\}\\
&f^i_{t\ell} \le {\color{black}1}-w^i_{t\ell} && \textnormal{for all } i \in \{1,\ldots,N\}, \; t \in \{1,\ldots,T\}, \; \ell \in \{1,\ldots,L^i_t-1\}\\
&w^{i}_{t, \ell} \le  w^i_{t+1,\ell} && \textnormal{for all } i \in \{1,\ldots,N\}, \;t \in \{1,\ldots,T-1\},\; \ell \in \{1,\ldots,|\mathcal{K}^i|\}\\
&w^{i}_{t\ell} \le w^i_{t,\ell+1} && \textnormal{for all } i \in \{1,\ldots,N\},\; t \in \{1,\ldots,T\},\;\ell \in \{1,\ldots,|\mathcal{K}^i|-1\} \\
& w^i_{1,1} = 0&&\textnormal{for all }i \in \{1,\ldots,N\}\\
&b^i_t + w^i_{t,1} = w^i_{t+1,1}&&\textnormal{for all }i \in \{1,\ldots,N\},\; t \in \{1,\ldots,T-1\}\\
&b^i_T + w^i_{T,1} = 1&&\textnormal{for all }i \in \{1,\ldots,N\}\\
&{\color{black}b^i_t \le w^i_{t+1,1}}&&{\color{black}\textnormal{for all }i \in \{1,\ldots,N\},\; t \in \{1,\ldots,T-1\}}\\
&{\color{black}b^j_t \le w^i_{t\ell}}&&{\color{black} \textnormal{for all } i, j \in \{1,\ldots,N\} \textnormal{ and } t \in \{1,\dots,T\} \textnormal{ such that } g(t,x^i) = \kappa^i_\ell } \\
&&&{\color{black} \textnormal{and } \mathcal{U}^i_t \cap \mathcal{U}^j_t \neq \emptyset}\\
&b^i_t \in \{0,1\} && \textnormal{for all } i \in \{1,\ldots,N\}, \; t \in \{1,\dots,T\}\\
&w^i_{t\ell} \in \R  && \textnormal{for all } i \in \{1,\ldots,N\}, \; t \in \{1,\ldots,T\}, \; \ell \in \{1,\ldots,| \mathcal{K}^i|\}.
\end{aligned}
\end{aligned}
\end{equation*}

\section{{\color{black}Proofs of  Theorem~\ref{thm:bp_main},  Lemma~\ref{lem:exact:bp_1}, and Lemma~\ref{lem:submodular_reform}}}\label{appx:exact}
{\color{black}
\begin{proof}{Proof of Lemma~\ref{lem:exact:bp_1}.}
Our proof of Lemma~\ref{lem:exact:bp_1} is split into two intermediary steps. 

In the first intermediary step of our proof of Lemma~\ref{lem:exact:bp_1}, we show that every feasible solution for \eqref{prob:mip} can be transformed into a feasible solution for the \eqref{prob:bp_1} with the same objective value. Indeed, 
consider any feasible solution $\sigma^1,\ldots,\sigma^N \in \{1,\ldots,T\}$ for the optimization problem~\eqref{prob:mip}. From these integers, we can define binary variables  $b^i_t \triangleq \mathbb{I} \left \{ \sigma^i = t \right \}$ for each sample path $i \in \{1,\ldots,N\}$ and period $t \in \{1,\ldots,T\}$. With this definition of a binary vector $b$, we observe  for  each sample path $i \in \{1,\ldots,N\}$ that  
\begin{align}
&\psi^i(b) \notag \\
&= \sum_{t=1}^T b^i_t \left( \prod_{s=1}^{t-1}(1 - b^i_s)  \right) \min_{s \in \{1,\ldots,t \}} \left \{g(s,x^i): \textnormal{there exists } j \textnormal{ such that } \mathcal{U}^i_s \cap \mathcal{U}^j_s \neq \emptyset \textnormal{ and } b^j_s = 1\right \} \label{line:placeholder1}\\
&= \min_{s \in \{1,\ldots,\sigma^i \}} \left \{g(s,x^i): \textnormal{there exists } j \textnormal{ such that } \mathcal{U}^i_s \cap \mathcal{U}^j_s \neq \emptyset \textnormal{ and } b^j_s = 1\right \} \label{line:placeholder2}\\
&=  \min \limits_{s \in \{1,\ldots,\sigma^i \}} \left \{g(s,x^i): \textnormal{there exists } j \textnormal{ such that } \mathcal{U}^i_s \cap \mathcal{U}^j_s \neq \emptyset \textnormal{ and } \sigma^j = s\right \}, \label{line:placeholder3}
\end{align}
where line~\eqref{line:placeholder1} follows from the definition of $\psi^i(b)$, and lines~\eqref{line:placeholder2} and \eqref{line:placeholder3} follow from the fact that we have defined the binary variables to satisfy $b^j_t = \mathbb{I} \left \{ \sigma^j = t \right \}$ for each sample path $j \in \{1,\ldots,N\}$ and period $t \in \{1,\ldots,T\}$. Because the above reasoning holds for each sample path $i \in \{1,\ldots,N\}$, we have shown that 
\begin{align}
 \frac{1}{N} \sum_{i=1}^N \psi^i \left(b \right)  &= \frac{1}{N} \sum_{i=1}^N \min_{t \in \{1,\ldots,\sigma^i \}} \left \{g(t,x^i): \textnormal{ there exists } j \textnormal{ such that } \mathcal{U}^i_t \cap \mathcal{U}^j_t \neq \emptyset \textnormal{ and } \sigma^j = t\right \}. \label{line:placeholder4}
 \end{align}
Therefore, we conclude that every feasible solution for \eqref{prob:mip} can be transformed into a feasible solution for \eqref{prob:bp_1} with the same objective value. 

In the second intermediary step  of the proof of Lemma~\ref{lem:exact:bp_1}, we show that if   $b$ is a feasible solution for \eqref{prob:bp_1} and if $\sigma^i \triangleq \min \{\min \{t \in \{1,\ldots,T\}:  b^i_t = 1 \}, T \}$ for each $i \in \{1,\ldots,N\}$, then
  \begin{align}
&\frac{1}{N} \sum_{i=1}^N \min_{t \in \{1,\ldots,\sigma^i \}} \left \{g(t,x^i): \textnormal{ there exists } j \textnormal{ such that } \mathcal{U}^i_t \cap \mathcal{U}^j_t \neq \emptyset \textnormal{ and } \sigma^j = t\right \} \ge  \frac{1}{N} \sum_{i=1}^N\psi^i(b).  \label{line:idkanymorek}
  \end{align}
 Indeed, consider any feasible solution $b$ for \eqref{prob:bp_1}, and let us define a new binary vector $\bar{b}$ as
 \begin{align*}
 \bar{b}^i_t \triangleq \begin{cases}
 b^i_t,&\text{if } t \in \{1,\ldots,T-1\},\\
 1,&\text{if } t = T.
 \end{cases}
 \end{align*}
It follows immediately from the above definition of $\bar{b}$ and from the definition of the function $\psi^i(\cdot)$ for each sample path $i \in \{1,\ldots,N\}$ that
\begin{align*}
& \psi^i \left(\bar{b} \right) -  \psi^i \left({b} \right)   \\
 &= \begin{cases}
\min \limits_{s \in \{1,\ldots,T\}} \left \{g(s,x^i): \textnormal{there exists } j \textnormal{ such that } \mathcal{U}^i_s \cap \mathcal{U}^j_s \neq \emptyset \textnormal{ and } b^j_s = 1\right \} - 0,&\text{if } b^i_1 = \cdots = b^i_T = 0,\\
0,&\text{otherwise}\\
 \end{cases}\\
 &\ge 0.
\end{align*}
where the inequality holds because the reward function is nonnegative, that is, $g(t,x^i) \ge 0$ for each sample path $i \in \{1,\ldots,N\}$ and period $s \in \{1,\ldots,T\}$  (see \S\ref{sec:setting:notation}). 
Therefore, we have proven that $\bar{b}$ has the same or better objective value as $b$ in the optimization problem~\eqref{prob:bp_1}, i.e.,
\begin{align}
 \frac{1}{N} \sum_{i=1}^N \psi^i \left(b \right) \le  \frac{1}{N} \sum_{i=1}^N  \psi^i \left(\bar{b} \right). \label{line:boring_stuff_1}
\end{align}
Now, for each sample path $i \in \{1,\ldots,N\}$, let us  define the following integer:
\begin{align*}
\sigma^i \triangleq \min \left \{ \min_{t \in \{1,\ldots,T\}} \left \{ t: b^i_t = 1 \right \}, T \right \} = \min_{t \in \{1,\ldots,T\}} \left \{ t: \bar{b}^i_t = 1 \right \}.
\end{align*}
Given the integers $\sigma^1,\ldots,\sigma^N \in \{1,\ldots,T\}$ defined above, it follows from identical reasoning as in lines~\eqref{line:placeholder1}, \eqref{line:placeholder2}, \eqref{line:placeholder3},  and \eqref{line:placeholder4} that 
\begin{align*}
 \frac{1}{N} \sum_{i=1}^N  \psi^i \left(\bar{b} \right)  &= \frac{1}{N} \sum_{i=1}^N \min_{t \in \{1,\ldots,\sigma^i \}} \left \{g(t,x^i): \textnormal{ there exists } j \textnormal{ such that } \mathcal{U}^i_t \cap \mathcal{U}^j_t \neq \emptyset \textnormal{ and } \sigma^j = t\right \}. 
 \end{align*}
 Combining the above equality with \eqref{line:boring_stuff_1}, we have proven that line~\eqref{line:idkanymorek} holds, thereby completing our proof of the second intermediary step.
 
 In conclusion, we  have shown through the above two intermediary steps  that every feasible solution for either \eqref{prob:mip} or \eqref{prob:bp_1} can be transformed into a feasible solution for the other problem with the same or better objective value, which implies that the optimal objective values of \eqref{prob:mip} and \eqref{prob:bp_1} are equal. Moreover, the inequality in Lemma~\ref{lem:exact:bp_1} follows immediately from the second intermediary step. Thus, our proof of Lemma~\ref{lem:exact:bp_1} is complete. 
\halmos \end{proof}
 \begin{proof}{Proof of Lemma~\ref{lem:submodular_reform}.}
  Consider any sample path $i \in \{1,\ldots,N\}$ and binary vector $b$.  For the sake of convenience,  let us repeat below the linear optimization problem that is found in the statement of  Lemma~\ref{lem:submodular_reform} for the given sample path $i$ and binary vector $b$: 
   \begin{equation} \tag{\ref{line:partial_reform}}
 \begin{aligned}
  &\underset{w^i}{\textnormal{maximize}}&& \sum_{t=1}^T  \sum_{\ell=1}^{L^i_t-1}( \kappa^i_{\ell+1} - \kappa^i_\ell) b^i_t (1 - w^i_{t\ell}) \\
  &\textnormal{subject to}&& \begin{aligned}[t]
  &w^{i}_{t \ell} \le  w^i_{t+1,\ell} && \textnormal{for all } t \in \{1,\ldots,T-1\},\; \ell \in \{1,\ldots,|\mathcal{K}^i|\}\\
&w^{i}_{t\ell} \le w^i_{t,\ell+1} && \textnormal{for all } t \in \{1,\ldots,T\},\; \ell \in \{1,\ldots,|\mathcal{K}^i|-1\} \\
&b^j_t \le w^i_{t\ell}&& \textnormal{for all } j \in \{1,\ldots,N\} \textnormal{ and } t \in \{1,\dots,T\} \\
&&& \textnormal{such that }g(t,x^i)  = \kappa^i_\ell \textnormal{ and }  \mathcal{U}^i_t \cap \mathcal{U}^j_t \neq \emptyset\\
& b^i_t \le w^i_{t+1,1}  && \textnormal{for all }  \; t \in \{1,\ldots,T-1\}\\
&w^i_{t \ell} \in \R&& \textnormal{for all }\; t \in \{1,\ldots,T\}, \; \ell \in \{1,\ldots,|\mathcal{K}^i| \}.
  \end{aligned}
  \end{aligned} 
 \end{equation}
In the remainder of this proof, we show that the optimal objective value of \eqref{line:partial_reform} is equal to $\psi^i(b)$. 

To this end, we begin by characterizing the decision variables $w^i$ in \eqref{line:partial_reform} at optimality. Indeed,  it follows from the construction of the constants $\kappa^i_\ell$ that each quantity $\kappa^i_{\ell+1} - \kappa^i_\ell$ is strictly positive. Therefore, since each $b^j_t$ is binary, we readily observe that there always exists an optimal solution $w^i$ for \eqref{line:partial_reform}  where the following equality holds for each period $t \in \{1,\ldots,T\}$ and each $\ell \in \{1,\ldots, | \mathcal{K}^i| \}$: 
\begin{align}
w^{i}_{t \ell} &= \begin{cases}
0, &\text{if } \left[ b^j_s = 0  \textnormal{ for all } j \in \{1,\ldots,N\} \textnormal{ and } s \in \{1,\dots,t\} \textnormal{ such that }g(s,x^i)  \le \kappa^i_\ell \textnormal{ and }  \mathcal{U}^i_s \cap \mathcal{U}^j_s \neq \emptyset\right] \\
&\text{and } \left[ b^i_s = 0 \textnormal{ for all } s \in \{1,\ldots,t-1\} \right], \\
1, &\text{otherwise}. 
\end{cases} \label{line:closed_form_w}
\end{align}

Now consider any optimal solution $w^i$ for the optimization problem \eqref{line:partial_reform}  for which the equality in \eqref{line:closed_form_w} is satisfied for each  period $t \in \{1,\ldots,T\}$ and each $\ell \in \{1,\ldots, | \mathcal{K}^i| \}$. For each period $t \in \{1,\ldots,T\}$, let $\tilde{\ell}^i_t$ be defined as the smallest integer such that there exists a sample path $j \in \{1,\ldots,N\}$ and a period $s \in \{1,\ldots,t\}$ that satisfy $b^j_s = 1$,  $g(s,x^i)  = \kappa^i_{\tilde{\ell}^i_t}$, and $\mathcal{U}^i_s \cap \mathcal{U}^j_s \neq \emptyset$.\footnote{{\color{black}If no such integer exists, then we assign $\tilde{\ell}^i_t$ to be equal to $L^i_t$.}} 
In particular, we observe from \eqref{line:closed_form_w} that the quantity $\tilde{\ell}^i_t$ is equal to the smallest integer such that $w^i_{t \tilde{\ell}^i_t} = 1$. Hence,  the optimal objective value of the optimization problem \eqref{line:partial_reform} is equal to:
\begin{align}
&\sum_{t=1}^T \sum_{\ell=1}^{L^i_t-1}( \kappa^i_{\ell+1} - \kappa^i_\ell) b^i_t(1 - w^i_{t\ell}) \notag  \\
&= \sum_{t \in \{1,\ldots,T\}: \; b^i_t = 1} \sum_{\ell=1}^{L^i_t-1}( \kappa^i_{\ell+1} - \kappa^i_\ell) (1 - w^i_{t\ell}) \label{line:using_algebra_for_b}\\
&= \begin{cases}
\sum_{\ell=1}^{L^i_t-1}( \kappa^i_{\ell+1} - \kappa^i_\ell) (1 - w^i_{t\ell}),&\textnormal{if } b^i_1 = \cdots =  b^i_{t-1} = 0 \textnormal{ and } b^i_t = 1 \textnormal{ for } t \in \{1,\ldots,T\},\\
0,&\text{otherwise}
\end{cases} \label{line:using_w_property}\\
&= \begin{cases}
\sum_{\ell=1}^{\tilde{\ell}^i_t-1}( \kappa^i_{\ell+1} - \kappa^i_\ell),&\textnormal{if } b^i_1 = \cdots =  b^i_{t-1} = 0 \textnormal{ and } b^i_t = 1 \textnormal{ for } t \in \{1,\ldots,T\},\\
0,&\text{otherwise}
\end{cases} \label{line:using_ell_property}\\
&= \begin{cases}
 \kappa^i_{\tilde{\ell}^i_t},&\textnormal{if } b^i_1 = \cdots =  b^i_{t-1} = 0 \textnormal{ and } b^i_t = 1 \textnormal{ for } t \in \{1,\ldots,T\},\\
0,&\text{otherwise}
\end{cases} \label{line:using_kappa_property}\\
&= \sum_{t=1}^T b^i_t \left(\prod_{s=1}^{t-1} \left(1 - b^i_s \right) \right) \kappa^i_{\tilde{\ell}^i_t}  \label{line:using_whatev_property}  \\
 &=  \sum_{t=1}^T b^i_t \left(\prod_{s=1}^{t-1} \left(1 - b^i_s \right) \right) \min_{s \in \{1,\ldots,t \}} \left \{g(s,x^i): \textnormal{ there exists } j \textnormal{ such that } \mathcal{U}^i_s \cap \mathcal{U}^j_s \neq \emptyset \textnormal{ and } b^j_s = 1\right \} \label{line:using_whatev_property2} \\
 &= \psi^i(b). \notag
\end{align}
Indeed, we observe that the optimal objective value of \eqref{line:partial_reform} is equal to $\sum_{t=1}^T \sum_{\ell=1}^{L^i_t-1}( \kappa^i_{\ell+1} - \kappa^i_\ell) b^i_t(1 - w^i_{t\ell})$ because $w^i$ is an optimal solution for \eqref{line:partial_reform}. Line~\eqref{line:using_algebra_for_b} follows from algebra. Line~\eqref{line:using_w_property} follows from    \eqref{line:closed_form_w}, which implies that if $b^i_t = 1$, then $w^i_{s\ell} = 1$ for all $s \in \{t+1,\ldots,T\}$ and all $\ell \in \{1,\ldots,| \mathcal{K}^i|\}$.  Line~\eqref{line:using_ell_property} follows from the definition of $\tilde{\ell}^{i}_t$ and from \eqref{line:closed_form_w}. Line~\eqref{line:using_kappa_property} follows from the fact that $\kappa^i_1 = 0$. Line~\eqref{line:using_whatev_property} follows from algebra. Line~\eqref{line:using_whatev_property2} follows from the definition of $\tilde{\ell}^i_t$ and the definition of the constants $\kappa^i_1,\ldots,\kappa^i_{| \mathcal{K}^i|}$. The final equality follows from the definition of $\psi^i(b)$. This completes our proof of Lemma~\ref{lem:submodular_reform}.  \halmos \end{proof}

\begin{proof}{Proof of Theorem~\ref{thm:bp_main}.}
The proof of Theorem~\ref{thm:bp_main}  follows immediately from Lemmas~\ref{lem:exact:bp_1} and \ref{lem:submodular_reform}. 
\halmos \end{proof}

\section{Proofs from \S\ref{sec:approx}} \label{appx:approx}

\subsection{Proofs from \S\ref{sec:approx:prelim}} 
\begin{proof}{Proof of Proposition~\ref{prop:similarity_of_formulations}.}
Let $\hat{b},\hat{w}$ denote an optimal solution for \eqref{prob:bp}, and let the set $\widehat{\mathcal{T}}^i \triangleq \{t: \; \hat{b}^i_t = 1 \}$ be defined for each sample path $i$. We will henceforth consider the optimization problem~\eqref{prob:h} in the case where the linear objective function $f(b,w)$ is defined equal to the following function:
\begin{align*}
\hat{f}(b,w) \triangleq \frac{1}{N} \sum_{i=1}^N \sum_{t \in \widehat{\mathcal{T}}^i} \sum_{\ell=1}^{L^i_t-1}( \kappa^i_{\ell+1} - \kappa^i_\ell) (b^i_t - w^i_{t\ell}).
\end{align*}
 We readily observe that the function $\hat{f}(b,w)$ is linear in $b$ and $w$. Moreover, it follows from algebra that $\hat{f}(b,w)$ is less than or equal to the objective function of \eqref{prob:bp} for all feasible solutions of \eqref{prob:bp}; that is, the linear function $\hat{f}(b,w)$ satisfies the condition from line~\eqref{line:f_lb}.\footnote{{\color{black}To see why $\hat{f}(b,w)$ is a lower bound on the objective function of \eqref{prob:bp}, consider any arbitrary vectors $b,w$ that satisfy the constraints of \eqref{prob:bp}. Since feasibility for the optimization problem~\eqref{prob:bp} implies that $b$ is a binary vector, we observe that the equality  $b^i_t(1 - w^i_{t \ell}) = \max \left \{ b^i_t - w^i_{t\ell}, 0 \right \}$ holds for each $i \in \{1,\ldots,N\}$, $t \in \{1,\ldots,T\}$, and $\ell \in \{1,\ldots,| \mathcal{K}^i|\}$.}} Thus, we conclude that solving the optimization problem~\eqref{prob:h} with objective function $f(b,w) = \hat{f}(b,w)$ will provide a lower-bound approximation of the optimization problem~\eqref{prob:bp}, and any optimal solution for \eqref{prob:h} will be a feasible solution for \eqref{prob:bp}. For notational convenience, we henceforth let $J^{\textnormal{\ref{prob:h}}}$ denote the optimal objective value of \eqref{prob:h} with objective function $f(b,w) = \hat{f}(b,w)$, and we let $J^{\textnormal{\ref{prob:bp}}}$ denote the optimal objective value of \eqref{prob:bp}.  

We first show  that the optimal objective value of \eqref{prob:h}  with objective function $f(b,w) = \hat{f}(b,w)$ is equal to the optimal objective value of \eqref{prob:bp}. Indeed,  it follows from the fact that \eqref{prob:h} and \eqref{prob:bp} have the same constraints that the optimal solution $\hat{b},\hat{w}$ for \eqref{prob:bp} is a feasible solution for \eqref{prob:h}. Therefore, 
\begin{align*}
J^{\textnormal{\ref{prob:h}}} &\ge \frac{1}{N} \sum_{i=1}^N \sum_{t \in \widehat{\mathcal{T}}^i} \sum_{\ell=1}^{L^i_t-1}( \kappa^i_{\ell+1} - \kappa^i_\ell) \left(\hat{b}^i_t - \hat{w}^i_{t\ell} \right)  \\
&= \frac{1}{N} \sum_{i=1}^N \sum_{t: \hat{b}_t^i = 1} \sum_{\ell=1}^{L^i_t-1}( \kappa^i_{\ell+1} - \kappa^i_\ell) (1 - \hat{w}^i_{t\ell})  + \frac{1}{N} \sum_{i=1}^N \sum_{t: \hat{b}_t^i = 0} \sum_{\ell=1}^{L^i_t-1}( \kappa^i_{\ell+1} - \kappa^i_\ell) 0 \\
&= \frac{1}{N} \sum_{i=1}^N \sum_{t=1}^T \sum_{\ell=1}^{L^i_t-1}( \kappa^i_{\ell+1} - \kappa^i_\ell)\hat{b}^i_t (1 - \hat{w}^i_{t\ell})  \\
&= J^{\textnormal{\ref{prob:bp}}},
\end{align*}
where the first inequality holds because $\hat{b},\hat{w}$ is a feasible but possibly suboptimal solution for \eqref{prob:h}, the first equality follows from the construction of the sets  $\widehat{\mathcal{T}}^1,\ldots,\widehat{\mathcal{T}}^N \subseteq \{1,\ldots,T\}$, the second equality follows from algebra, and the final equality follows from the fact that $\hat{b},\hat{w}$ is an optimal solution for \eqref{prob:bp}. Since the optimal objective value of \eqref{prob:h} is always less than or equal to the optimal objective value of \eqref{prob:bp}, we have proved that  the optimal objective value of \eqref{prob:h} is equal to the optimal objective value of \eqref{prob:bp}.

We conclude the proof of Proposition~\ref{prop:similarity_of_formulations} by showing that every optimal solution for \eqref{prob:h}  with objective function $f(b,w) = \hat{f}(b,w)$ is an optimal solution for \eqref{prob:bp}. Indeed, let $\bar{b},\bar{w}$ denote an optimal solution for \eqref{prob:h}. Then, 
\begin{align*}
J^{\textnormal{\ref{prob:h}}} =  \hat{f}(\bar{b},\bar{w}) \le \frac{1}{N} \sum_{i=1}^N \sum_{t=1}^T \sum_{\ell=1}^{L^i_t-1}( \kappa^i_{\ell+1} - \kappa^i_\ell)\bar{b}^i_t\left(1 - \bar{w}^i_{t\ell} \right) \le J^{\textnormal{\ref{prob:bp}}},
\end{align*}
where the first equality follows from the fact that $\bar{b},\bar{w}$ is an optimal solution for \eqref{prob:h} and from the fact that the objective function of \eqref{prob:h} is $f(b,w) = \hat{f}(b,w)$, the first inequality follows from the fact that the linear function $\hat{f}(b,w)$ satisfies the condition from line~\eqref{line:f_lb}, and the final inequality follows from the fact that $\bar{b},\bar{w}$ is a feasible but possibly suboptimal solution for \eqref{prob:bp}. Since we have previously shown that the optimal objective value $J^{\textnormal{\ref{prob:h}}}$ of \eqref{prob:h} is equal to the optimal objective value $J^{\textnormal{\ref{prob:bp}}}$ of \eqref{prob:bp}, we have thus proven that $\bar{b},\bar{w}$ is an optimal solution for \eqref{prob:bp}. This concludes our proof of Proposition~\ref{prop:similarity_of_formulations}.
\halmos \end{proof}

\subsection{Proofs from \S\ref{sec:approx:computation}} 
 \begin{proof}{Proof of Lemma~\ref{lem:optimal_h_bar}.}
Let $f(b,w) = \bar{f}(b,w)$,  and consider any binary vector $b$ that is optimal for the optimization problem \eqref{prob:h}. Since $b$ is binary, we readily observe that there exists an optimal choice for the remaining decision variables $w$ in the optimization problem~\eqref{prob:h}  in which the following equality holds for each sample path $i \in \{1,\ldots,N\}$, period $t \in \{1,\ldots,T\}$, and $\ell \in \{1,\ldots, | \mathcal{K}^i| \}$: 
\begin{align}
w^{i}_{t \ell} &= \begin{cases}
0, &\text{if } \left[ b^j_s = 0  \textnormal{ for all } j \in \{1,\ldots,N\} \textnormal{ and } s \in \{1,\dots,t\} \textnormal{ such that }g(s,x^i)  \le \kappa^i_\ell \textnormal{ and }  \mathcal{U}^i_s \cap \mathcal{U}^j_s \neq \emptyset\right] \\
&\text{and } \left[ b^i_s = 0 \textnormal{ for all } s \in \{1,\ldots,t-1\} \right], \\
1, &\text{otherwise}. 
\end{cases} \label{line:closed_form_w_argg}
\end{align}

We will now construct a new solution for  \eqref{prob:h} that has the same or greater objective value than $b,w$. Indeed, let $\bar{b},\bar{w}$ be a solution for \eqref{prob:h} defined by the following equalities for each sample path $i \in \{1,\ldots,N\}$, period $t \in \{1,\ldots,T\}$, and $\ell \in \{1,\ldots, | \mathcal{K}^i| \}$: 
\begin{align*}
\bar{b}^i_{t} &\triangleq \begin{cases}
1,&\text{if } \left[b^i_{t} = 1 \text{ and } t = T^i \right] \text{ or } \left[t = T\right],\\
0,&\text{otherwise};
\end{cases}\\
\bar{w}^{i}_{t \ell} &\triangleq \begin{cases}
0, &\text{if } \left[ \bar{b}^j_s = 0  \textnormal{ for all } j \in \{1,\ldots,N\} \textnormal{ and } s \in \{1,\dots,t\} \textnormal{ such that }g(s,x^i)  \le \kappa^i_\ell \textnormal{ and }  \mathcal{U}^i_s \cap \mathcal{U}^j_s \neq \emptyset\right] \\
&\text{and } \left[ \bar{b}^i_s = 0 \textnormal{ for all } s \in \{1,\ldots,t-1\} \right], \\
1, &\text{otherwise}.
\end{cases}
\end{align*}

We observe for each sample path $i \in \{1,\ldots,N\}$ that $\bar{b}^i_T = 1$ and $\bar{b}^i_t = 0$ for all $t \in \{1,\ldots,T\} \setminus \mathcal{T}^i$. Moreover, we readily observe from inspection that the solution $\bar{b},\bar{w}$ is feasible for \eqref{prob:h}. Therefore, it remains for us to prove that $\bar{b},\bar{w}$ is an optimal solution for \eqref{prob:h}. 

To show that $\bar{b},\bar{w}$ is an optimal solution for \eqref{prob:h}, we observe  
for each sample path $i \in \{1,\ldots,N\}$, period $t \in \{1,\ldots,T\}$, and $\ell \in \{1,\ldots, L^i_t - 1\}$, 
\begin{align}
\bar{w}^{i}_{t \ell} &= \begin{cases}
0, &\text{if } \left[ {b}^j_{T^j} = 0  \textnormal{ for all } j \in \{1,\ldots,N\} \textnormal{ such that } T^j \le t, \; g(T^j,x^i)  \le \kappa^i_\ell, \textnormal{ and }  \mathcal{U}^i_{T^j} \cap \mathcal{U}^j_{T^j} \neq \emptyset\right] \\
&\text{and } \left[b^i_{T^i} = 0  \text{ if } T^i  \in \{1,\ldots,t-1\} \right], \\
1, &\text{otherwise}
\end{cases} \notag \\
&\le w^i_{t \ell}, \label{line:improving_fun_times}
\end{align}
where the equality follows from the definition of $\bar{b},\bar{w}$, and the inequality follows from line~\eqref{line:closed_form_w_argg}. Therefore,
\begin{align*}
\bar{f}(\bar{b},\bar{w}) &= \frac{1}{N} \sum_{i=1}^N \sum_{t \in \mathcal{T}^i} \sum_{\ell=1}^{L^i_t-1}( \kappa^i_{\ell+1} - \kappa^i_\ell) \left( \bar{b}^i_t  - \bar{w}^i_{t\ell} \right)\\
&\ge \frac{1}{N} \sum_{i=1}^N \sum_{t \in \mathcal{T}^i} \sum_{\ell=1}^{L^i_t-1}( \kappa^i_{\ell+1} - \kappa^i_\ell) \left( {b}^i_t  - {w}^i_{t\ell} \right)\\
&= \bar{f}(b,w).
\end{align*}
Indeed, the first equality follows from the definition of $\bar{f}(\cdot,\cdot)$. The inequality follows from line~\eqref{line:improving_fun_times}, from the facts that $\bar{b}^i_{T^i} \ge b^i_{T^i}$ and $\bar{b}^i_T \ge b^i_T$, and from the fact that each quantity $\kappa^i_{\ell+1} - \kappa^i_\ell$ is strictly positive. The final equality follows from the definition of $\bar{f}(\cdot,\cdot)$. Since $b,w$ was an optimal solution for \eqref{prob:h}, our proof of Lemma~\ref{lem:optimal_h_bar} is complete.  \halmos \end{proof}

\begin{proof}{Proof of Lemma~\ref{lem:equivalence_of_h_bar}.}
 Let $f(b,w) = \bar{f}(b,w)$. In this case, we recall from Lemma~\ref{lem:optimal_h_bar} that there exists an optimal solution $b,w$ for \eqref{prob:h} that satisfies  $b^1_T = \cdots = b^N_T = 1$ and $b^i_t = 0$ for each sample path $i \in \{1,\ldots,N\}$ and period $t \in \{1,\ldots,T\} \setminus \mathcal{T}^i$. Therefore, we can without loss of generality impose those equality constraints into the optimization problem \eqref{prob:h}. That is, \eqref{prob:h} can be rewritten equivalently as 
  \begin{equation} \label{prob:h1} \tag{H-1}
\begin{aligned}
 &\underset{ b,w}{\textnormal{maximize}}&& \frac{1}{N} \sum_{i=1}^N \sum_{t \in \mathcal{T}^i} \sum_{\ell=1}^{L^i_t-1}( \kappa^i_{\ell+1} - \kappa^i_\ell) \left( b^i_t  - w^i_{t\ell} \right) \\
&\textnormal{subject to}&& \begin{aligned}[t]
&w^{i}_{t\ell} \le  w^i_{t+1,\ell} && \textnormal{for all } i \in \{1,\ldots,N\}, \;t \in \{1,\ldots,T-1\},\; \ell \in \{1,\ldots,|\mathcal{K}^i|\}\\
&w^{i}_{t\ell} \le w^i_{t,\ell+1} && \textnormal{for all } i \in \{1,\ldots,N\},\; t \in \{1,\ldots,T\},\;\ell \in \{1,\ldots,|\mathcal{K}^i|-1\} \\
&b^i_t \le w^i_{t+1,1}&&\textnormal{for all }i \in \{1,\ldots,N\},\; t \in \{1,\ldots,T-1\}\\
&b^j_t \le w^i_{t\ell}&& \textnormal{for all } i, j \in \{1,\ldots,N\} \textnormal{ and } t \in \{1,\dots,T\} \textnormal{ such that } g(t,x^i) = \kappa^i_\ell\\
&&& \textnormal{and } \mathcal{U}^i_t \cap \mathcal{U}^j_t \neq \emptyset\\
&b^i_T = 1 && \textnormal{for all } i \in \{1,\ldots,N\}\\
&b^i_t = 0 && \textnormal{for all } i \in \{1,\ldots,N\}, \; t \in \{1,\ldots,T\} \setminus \mathcal{T}^i\\
&b^i_t \in \{0,1\} && \textnormal{for all } i \in \{1,\ldots,N\}, \; t \in \{1,\dots,T\}\\
&w^i_{t\ell} \in \R  && \textnormal{for all } i \in \{1,\ldots,N\}, \; t \in \{1,\ldots,T\}, \; \ell \in \{1,\ldots,| \mathcal{K}^i|\}.
\end{aligned}
\end{aligned}
\end{equation}
After substituting out the decision variables $b^i_t$ that have been set to zero or one, it follows from algebra that the above optimization problem can be rewritten equivalently as 
   \begin{equation} \label{prob:h2} \tag{H-2}
\begin{aligned}
 &\underset{ b,w}{\textnormal{maximize}}&& \frac{1}{N} \sum_{i \in \{1,\ldots,N\}: \; T^i < T} \sum_{\ell=1}^{L^i_{T^i}-1}( \kappa^i_{\ell+1} - \kappa^i_\ell) \left( b^i_{T^i}  - w^i_{T^i \ell} \right) +  \frac{1}{N} \sum_{i =1}^N \sum_{\ell=1}^{L^i_{T}-1}( \kappa^i_{\ell+1} - \kappa^i_\ell) \left(1  - w^i_{T \ell} \right) \\
&\textnormal{subject to}&& \begin{aligned}[t]
&w^{i}_{t\ell} \le  w^i_{t+1,\ell} && \textnormal{for all } i \in \{1,\ldots,N\}, \;t \in \{1,\ldots,T-1\},\; \ell \in \{1,\ldots,|\mathcal{K}^i|\}\\
&w^{i}_{t\ell} \le w^i_{t,\ell+1} && \textnormal{for all } i \in \{1,\ldots,N\},\; t \in \{1,\ldots,T\},\;\ell \in \{1,\ldots,|\mathcal{K}^i|-1\} \\
&b^i_{T^i} \le w^i_{T^i+1,1}&&\textnormal{for all }i \in \{1,\ldots,N\} \textnormal{ such that } T^i < T\\
&b^j_{T^j} \le w^i_{T^j \ell}&& \textnormal{for all } i, j \in \{1,\ldots,N\} \textnormal{ such that } T^j < T, \; g(T^j,x^i) = \kappa^i_\ell,\; \textnormal{and } \mathcal{U}^i_{T^j} \cap \mathcal{U}^j_{T^j} \neq \emptyset\\
&w^i_{T,L^i_T} = 1 && \textnormal{for all } i \in \{1,\ldots,N\}\\
&b^i_{T^i} \in \{0,1\} && \textnormal{for all } i \in \{1,\ldots,N\} \textnormal{ such that } T^i < T\\
&w^i_{t\ell} \in \R  && \textnormal{for all } i \in \{1,\ldots,N\}, \; t \in \{1,\ldots,T\}, \; \ell \in \{1,\ldots,| \mathcal{K}^i|\}.
\end{aligned}
\end{aligned}
\end{equation}
Let us make three observations about the above optimization problem. First, we observe from inspection that the constraints in the above optimization problem of the form $w^i_{T,L^i_T} = 1$ for all $i \in \{1,\ldots,N\}$ can be removed from the above optimization problem without affecting its optimal objective value. Second, we observe that the decision variables $w^i_{t \ell}$ for each $t \in \{1,\ldots,T\} \setminus \mathcal{T}^i$ do not appear in the objective function of \eqref{prob:h2}. Third, since each term $\kappa^i_{\ell+1} - \kappa^i_\ell$ is strictly positive, we observe that there exists an optimal solution for \eqref{prob:h2} that satisfies the following equality  for each sample path $i \in \{1,\ldots,N\}$, period $t \in \mathcal{T}^i$, and $\ell \in \{1,\ldots,| \mathcal{K}^i|\}$:
\begin{align*}
w^{i}_{t \ell} &= \begin{cases}
0, &\text{if } \left[ b^j_{T^j} = 0  \textnormal{ for all } j \in \{1,\ldots,N\} \textnormal{ such that } T^j \le t, \;T^j < T, \; g(T^j,x^i)  \le \kappa^i_\ell,\; \textnormal{and }  \mathcal{U}^i_{T^j} \cap \mathcal{U}^j_{T^j} \neq \emptyset\right] \\
&\text{and } \left[  \textnormal{if } T^i < t, \textnormal{ then } b^i_{T^i} = 0\right], \\
1, &\text{otherwise}. 
\end{cases}
\end{align*}
It follows from the aforementioned three observations that \eqref{prob:h2} can be rewritten equivalently as
 \begin{equation} \tag{\ref{prob:h_bar}}
\begin{aligned}
  &\underset{ b,w}{\textnormal{maximize}}&& \frac{1}{N} \sum_{i \in \{1,\ldots,N\}: T^i < T} \sum_{\ell=1}^{L^i_{T^i}-1}( \kappa^i_{\ell+1} - \kappa^i_\ell) \left( b^i_{T^i}  - w^i_{T^i \ell} \right)  + \frac{1}{N} \sum_{i=1}^N \sum_{\ell=1}^{L^i_T - 1} ( \kappa^i_{\ell+1} - \kappa^i_\ell)  \left(1 - w^i_{T \ell} \right)  \\
&\textnormal{subject to}&& \begin{aligned}[t]
&w^{i}_{t\ell} \le w^i_{t,\ell+1} && \textnormal{for all } i \in \{1,\ldots,N\},\; t \in \mathcal{T}^i,\;\ell \in \{1,\ldots,|\mathcal{K}^i|-1\} \\
&b^i_{T^i} \le w^i_{T1}&&\textnormal{for all }i \in \{1,\ldots,N\} \textnormal{ such that }  T^i < T\\
&b^j_{T^j} \le w^i_{t \ell}&& \textnormal{for all } i, j \in \{1,\ldots,N\} \textnormal{ and } t \in \mathcal{T}^i \textnormal{ such that } g(T^j,x^i) = \kappa^i_\ell,\\
&&&T^j \le t, \;  T^j < T, \; \textnormal{and } \mathcal{U}^i_{T^j} \cap \mathcal{U}^j_{T^j} \neq \emptyset\\
&b^i_{T^i} \in \{0,1\} && \textnormal{for all } i \in \{1,\ldots,N\} \textnormal{ such that } T^i < T\\
&w^i_{t\ell} \in \R  && \textnormal{for all } i \in \{1,\ldots,N\}, \; t \in \mathcal{T}^i, \; \ell \in \{1,\ldots,| \mathcal{K}^i|\}.
\end{aligned}
\end{aligned}
\end{equation}
We have thus shown that the optimal objective value of \eqref{prob:h_bar} is equal to the optimal objective value of \eqref{prob:h} when $f(b,w) = \bar{f}(b,w)$. Moreover, it follows from the above reasoning that any optimal solution $\bar{b},\bar{w}$ for \eqref{prob:h_bar} can be transformed into an optimal solution for \eqref{prob:h} using the following equality:
 \begin{align*}
{b}^i_t &= \begin{cases}
\bar{b}^i_t,&\textnormal{if } t = T^i \textnormal{ and } T^i < T,\\
1,&\textnormal{if } t = T,\\
0,&\textnormal{otherwise}. 
\end{cases}
\end{align*}
This concludes our proof of Lemma~\ref{lem:equivalence_of_h_bar}. 
\halmos \end{proof}

{\color{black}
\begin{proof}{Proof of Proposition~\ref{prop:approx:tract}.}

We observe that {\color{black}\eqref{prob:h_bar}} is equivalent to a binary linear optimization problem with $\mathcal{O}(NT)$ binary decision variables and $\mathcal{O}(N^2 + NT)$ constraints, with each constraint of the form $\lambda_i \ge \lambda_j$.  
 It thus follows from \citet[\S 3]{picard1976maximal} that the optimization problem {\color{black}\eqref{prob:h_bar}} is equivalent to a problem of computing the maximal closure of a directed graph with $\mathcal{O}(NT)$ nodes and $\mathcal{O}(N^2 + NT)$ edges. Furthermore, \citet[\S 4]{picard1976maximal} shows that any maximal closure problem can be solved by computing the maximum flow in an augmented graph of identical size. Applying the algorithm of \cite{orlin2013max} to compute the maximum flow in this augmented graph, we obtain an $\mathcal{O}(N^2 T(N + T))$ algorithm for solving {\color{black}\eqref{prob:h_bar}}. {\color{black}Finally, the output of the maximal closure problem can be transformed into an optimal solution for \eqref{prob:h} using Lemma~\ref{lem:equivalence_of_h_bar}.  This concludes our proof of Proposition~\ref{prop:approx:tract}.}
 \halmos \end{proof}
 }

\subsection{Proofs from \S\ref{sec:approx:theory}}

\begin{proof}{Proof of Proposition~\ref{prop:equivalence_2}.}
Suppose that the number of periods is $T = 2$. We recall from Remark~\ref{rem:f_bar} that the optimal objective value of \eqref{prob:h_bar} is less than or equal to the optimal objective value of \eqref{prob:bp}. Therefore, it remains for us to show that the optimal objective value of \eqref{prob:bp} is less than or equal to the optimal objective value of \eqref{prob:h_bar}.

We begin by restating the optimization problem~\eqref{prob:bp} when $T=2$ for the sake of convenience:
 \begin{equation} \tag{\ref{prob:bp}}
\begin{aligned}
 &\underset{ b,w}{\textnormal{maximize}}&& \frac{1}{N} \sum_{i=1}^N  \sum_{\ell=1}^{L^i_1-1}( \kappa^i_{\ell+1} - \kappa^i_\ell)b^i_1 (1 - w^i_{1\ell}) + \frac{1}{N} \sum_{i =1}^N \sum_{\ell=1}^{L^i_2-1}( \kappa^i_{\ell+1} - \kappa^i_\ell) b^i_2(1 - w^i_{2\ell})  \\
&\textnormal{subject to}&& \begin{aligned}[t]
&w^{i}_{1 \ell} \le  w^i_{2 \ell} && \textnormal{for all } i \in \{1,\ldots,N\}, \;\ell \in \{1,\ldots,|\mathcal{K}^i|\}\\
&w^{i}_{t\ell} \le w^i_{t,\ell+1} && \textnormal{for all } i \in \{1,\ldots,N\},\; t \in \{1,2\},\;\ell \in \{1,\ldots,|\mathcal{K}^i|-1\} \\
&b^i_1 \le w^i_{21}&&\textnormal{for all }i \in \{1,\ldots,N\}\\
&b^j_1 \le w^i_{1 \ell}&& \textnormal{for all } i, j \in \{1,\ldots,N\}  \textnormal{ such that } g(1,x^i) = \kappa^i_\ell\; \textnormal{and } \mathcal{U}^i_1 \cap \mathcal{U}^j_1 \neq \emptyset\\
&b^j_2 \le w^i_{2 \ell}&& \textnormal{for all } i, j \in \{1,\ldots,N\}  \textnormal{ such that } g(2,x^i) = \kappa^i_\ell\; \textnormal{and } \mathcal{U}^i_2 \cap \mathcal{U}^j_2 \neq \emptyset\\
&b^i_t \in \{0,1\} && \textnormal{for all } i \in \{1,\ldots,N\}, \; t \in \{1,2\} \\
&w^i_{t\ell} \in \R  && \textnormal{for all } i \in \{1,\ldots,N\}, \; t \in \{1,2\}, \; \ell \in \{1,\ldots,| \mathcal{K}^i|\}.
\end{aligned}
\end{aligned}
\end{equation}
We readily observe from inspection of the optimization problem~\eqref{prob:bp} that there exists an optimal solution $b,w$ for \eqref{prob:bp} that satisfies the equalities $b^1_2 = \cdots = b^N_2 = 1$ and $b^i_1 = 0$ for all $i \in \{1,\ldots,N\}$ such that $T^i = 2$. Therefore, we observe that those equality constraints can be added to the optimization problem~\eqref{prob:bp}  without changing its optimal objective value. That is, the optimization problem~\eqref{prob:bp} can be rewritten equivalently as
 \begin{equation*} 
\begin{aligned}
 &\underset{ b,w}{\textnormal{maximize}}&& \frac{1}{N} \sum_{i=1}^N  \sum_{\ell=1}^{L^i_1-1}( \kappa^i_{\ell+1} - \kappa^i_\ell)b^i_1 (1 - w^i_{1\ell}) + \frac{1}{N} \sum_{i =1}^N \sum_{\ell=1}^{L^i_2-1}( \kappa^i_{\ell+1} - \kappa^i_\ell) b^i_2(1 - w^i_{2\ell})  \\
&\textnormal{subject to}&& \begin{aligned}[t]
&w^{i}_{1 \ell} \le  w^i_{2 \ell} && \textnormal{for all } i \in \{1,\ldots,N\}, \;\ell \in \{1,\ldots,|\mathcal{K}^i|\}\\
&w^{i}_{t\ell} \le w^i_{t,\ell+1} && \textnormal{for all } i \in \{1,\ldots,N\},\; t \in \{1,2\},\;\ell \in \{1,\ldots,|\mathcal{K}^i|-1\} \\
&b^i_1 \le w^i_{21}&&\textnormal{for all }i \in \{1,\ldots,N\}\\
&b^j_1 \le w^i_{1 \ell}&& \textnormal{for all } i, j \in \{1,\ldots,N\}  \textnormal{ such that } g(1,x^i) = \kappa^i_\ell\; \textnormal{and } \mathcal{U}^i_1 \cap \mathcal{U}^j_1 \neq \emptyset\\
&b^j_2 \le w^i_{2 \ell}&& \textnormal{for all } i, j \in \{1,\ldots,N\}  \textnormal{ such that } g(2,x^i) = \kappa^i_\ell\; \textnormal{and } \mathcal{U}^i_2 \cap \mathcal{U}^j_2 \neq \emptyset\\
&b^i_2 = 1 && \textnormal{for all } i \in \{1,\ldots,N\}\\
&b^i_1 = 0 && \textnormal{for all } i \in \{1,\ldots,N\} \textnormal{ such that } T^i = 2\\
&b^i_t \in \{0,1\} && \textnormal{for all } i \in \{1,\ldots,N\}, \; t \in \{1,2\} \\
&w^i_{t\ell} \in \R  && \textnormal{for all } i \in \{1,\ldots,N\}, \; t \in \{1,2\}, \; \ell \in \{1,\ldots,| \mathcal{K}^i|\}.
\end{aligned}
\end{aligned}
\end{equation*}
After eliminating the decision variables $b^i_t$ that have been constrained to be equal to zero or one, the above optimization problem can be rewritten equivalently as
 \begin{equation*} 
\begin{aligned}
 &\underset{ b,w}{\textnormal{maximize}}&& \frac{1}{N} \sum_{i \in \{1,\ldots,N\}: T^i = 1}  \sum_{\ell=1}^{L^i_1-1}( \kappa^i_{\ell+1} - \kappa^i_\ell)b^i_1 (1 - w^i_{1\ell}) + \frac{1}{N} \sum_{i =1}^N \sum_{\ell=1}^{L^i_2-1}( \kappa^i_{\ell+1} - \kappa^i_\ell) (1 - w^i_{2\ell})  \\
&\textnormal{subject to}&& \begin{aligned}[t]
&w^{i}_{1 \ell} \le  w^i_{2 \ell} && \textnormal{for all } i \in \{1,\ldots,N\}, \ell \in \{1,\ldots,|\mathcal{K}^i|\}\\
&w^{i}_{t\ell} \le w^i_{t,\ell+1} && \textnormal{for all } i \in \{1,\ldots,N\},\; t \in \{1,2\},\;\ell \in \{1,\ldots,|\mathcal{K}^i|-1\} \\
&b^i_1 \le w^i_{21}&&\textnormal{for all }i \in \{1,\ldots,N\} \textnormal{ such that }T^i = 1\\
&b^j_1 \le w^i_{1 \ell}&& \textnormal{for all } i, j \in \{1,\ldots,N\}  \textnormal{ such that } g(1,x^i) = \kappa^i_\ell,\; T^j = 1, \; \textnormal{and } \mathcal{U}^i_1 \cap \mathcal{U}^j_1 \neq \emptyset\\
&w^i_{2, L^i_2} = 1 && \textnormal{for all } i \in \{1,\ldots,N\} \\
&b^i_1 \in \{0,1\} && \textnormal{for all } i \in \{1,\ldots,N\} \textnormal{ such that } T^i = 1\\
&w^i_{t\ell} \in \R  && \textnormal{for all } i \in \{1,\ldots,N\}, \; t \in \{1,2\}, \; \ell \in \{1,\ldots,| \mathcal{K}^i|\}.
\end{aligned}
\end{aligned}
\end{equation*}
Let us make two observations about the above optimization problem. First, we observe from inspection that the constraints in the above optimization problem of the form $w^i_{2,L^i_2} = 1$ for all $i \in \{1,\ldots,N\}$ can be removed from the above optimization problem without affecting its optimal objective value. Second,  since each term $\kappa^i_{\ell+1} - \kappa^i_\ell$ is strictly positive, we observe that there exists an optimal solution for the above optimization problem in which the equalities $w^i_{1 1} = \cdots = w^i_{1, L^i_1 - 1} = 0$ are satisfied for each sample path $i \in \{1,\ldots,N\}$. 
 Therefore, the above optimization problem can be rewritten equivalently as 
  \begin{equation*} 
\begin{aligned}
 &\underset{ b,w}{\textnormal{maximize}}&& \frac{1}{N} \sum_{i \in \{1,\ldots,N\}: T^i = 1}  \sum_{\ell=1}^{L^i_1-1}( \kappa^i_{\ell+1} - \kappa^i_\ell)b^i_1 (1 - w^i_{1\ell}) + \frac{1}{N} \sum_{i =1}^N \sum_{\ell=1}^{L^i_2-1}( \kappa^i_{\ell+1} - \kappa^i_\ell) (1 - w^i_{2\ell})  \\
&\textnormal{subject to}&& \begin{aligned}[t]
&w^{i}_{1 \ell} \le  w^i_{2 \ell} && \textnormal{for all } i \in \{1,\ldots,N\}, \ell \in \{1,\ldots,|\mathcal{K}^i|\}\\
&w^{i}_{t\ell} \le w^i_{t,\ell+1} && \textnormal{for all } i \in \{1,\ldots,N\},\; t \in \{1,2\},\;\ell \in \{1,\ldots,|\mathcal{K}^i|-1\} \\
&b^i_1 \le w^i_{21}&&\textnormal{for all }i \in \{1,\ldots,N\} \textnormal{ such that }T^i = 1\\
&b^j_1 \le w^i_{1 \ell}&& \textnormal{for all } i, j \in \{1,\ldots,N\}  \textnormal{ such that } g(1,x^i) = \kappa^i_\ell,\; T^j = 1, \; \textnormal{and } \mathcal{U}^i_1 \cap \mathcal{U}^j_1 \neq \emptyset\\
&w^i_{1 \ell} = 0 &&  \textnormal{for all } i \in \{1,\ldots,N\} \textnormal{ and } \ell \in \{1,\ldots,L^i_1 - 1\} \\  
&b^i_1 \in \{0,1\} && \textnormal{for all } i \in \{1,\ldots,N\} \textnormal{ such that } T^i = 1\\
&w^i_{t\ell} \in \R  && \textnormal{for all } i \in \{1,\ldots,N\}, \; t \in \{1,2\}, \; \ell \in \{1,\ldots,| \mathcal{K}^i|\}.
\end{aligned}
\end{aligned}
\end{equation*}
Because the above optimization problem has the constraint that $w^i_{1 \ell} = 0$ for all  $i \in \{1,\ldots,N\}$ and $\ell \in \{1,\ldots,L^i_1 - 1\}$, we can without loss of generality replace each term $b^i_1 (1 - w^i_{1\ell})$ in the objective function of the above optimization problem with $b^i_1 - w^i_{1\ell}$. That is, the above optimization problem is equivalent to 
\begin{equation*} 
\begin{aligned}
 &\underset{ b,w}{\textnormal{maximize}}&& \frac{1}{N} \sum_{i \in \{1,\ldots,N\}: T^i = 1}  \sum_{\ell=1}^{L^i_1-1}( \kappa^i_{\ell+1} - \kappa^i_\ell)(b^i_1 - w^i_{1\ell}) + \frac{1}{N} \sum_{i =1}^N \sum_{\ell=1}^{L^i_2-1}( \kappa^i_{\ell+1} - \kappa^i_\ell) (1 - w^i_{2\ell})  \\
&\textnormal{subject to}&& \begin{aligned}[t]
&w^{i}_{1 \ell} \le  w^i_{2 \ell} && \textnormal{for all } i \in \{1,\ldots,N\}, \ell \in \{1,\ldots,|\mathcal{K}^i|\}\\
&w^{i}_{t\ell} \le w^i_{t,\ell+1} && \textnormal{for all } i \in \{1,\ldots,N\},\; t \in \{1,2\},\;\ell \in \{1,\ldots,|\mathcal{K}^i|-1\} \\
&b^i_1 \le w^i_{21}&&\textnormal{for all }i \in \{1,\ldots,N\} \textnormal{ such that }T^i = 1\\
&b^j_1 \le w^i_{1 \ell}&& \textnormal{for all } i, j \in \{1,\ldots,N\}  \textnormal{ such that } g(1,x^i) = \kappa^i_\ell,\; T^j = 1, \; \textnormal{and } \mathcal{U}^i_1 \cap \mathcal{U}^j_1 \neq \emptyset\\
&w^i_{1 \ell} = 0 &&  \textnormal{for all } i \in \{1,\ldots,N\} \textnormal{ and } \ell \in \{1,\ldots,L^i_1 - 1\} \\  
&b^i_1 \in \{0,1\} && \textnormal{for all } i \in \{1,\ldots,N\} \textnormal{ such that } T^i = 1\\
&w^i_{t\ell} \in \R  && \textnormal{for all } i \in \{1,\ldots,N\}, \; t \in \{1,2\}, \; \ell \in \{1,\ldots,| \mathcal{K}^i|\}.
\end{aligned}
\end{aligned}
\end{equation*}
Finally, we can relax the above optimization problem by removing the constraint that $w^i_{1 \ell} = 0$ for all  $i \in \{1,\ldots,N\}$ and $\ell \in \{1,\ldots,L^i_1 - 1\}$. That is, the optimal objective value of the above optimization problem is less than or equal to the optimal objective value of the following optimization problem:
\begin{equation*} 
\begin{aligned}
 &\underset{ b,w}{\textnormal{maximize}}&& \frac{1}{N} \sum_{i \in \{1,\ldots,N\}: T^i = 1}  \sum_{\ell=1}^{L^i_1-1}( \kappa^i_{\ell+1} - \kappa^i_\ell)(b^i_1 - w^i_{1\ell}) + \frac{1}{N} \sum_{i =1}^N \sum_{\ell=1}^{L^i_2-1}( \kappa^i_{\ell+1} - \kappa^i_\ell) (1 - w^i_{2\ell})  \\
&\textnormal{subject to}&& \begin{aligned}[t]
&w^{i}_{1 \ell} \le  w^i_{2 \ell} && \textnormal{for all } i \in \{1,\ldots,N\}, \ell \in \{1,\ldots,|\mathcal{K}^i|\}\\
&w^{i}_{t\ell} \le w^i_{t,\ell+1} && \textnormal{for all } i \in \{1,\ldots,N\},\; t \in \{1,2\},\;\ell \in \{1,\ldots,|\mathcal{K}^i|-1\} \\
&b^i_1 \le w^i_{21}&&\textnormal{for all }i \in \{1,\ldots,N\} \textnormal{ such that }T^i = 1\\
&b^j_1 \le w^i_{1 \ell}&& \textnormal{for all } i, j \in \{1,\ldots,N\}  \textnormal{ such that } g(1,x^i) = \kappa^i_\ell,\; T^j = 1, \; \textnormal{and } \mathcal{U}^i_1 \cap \mathcal{U}^j_1 \neq \emptyset\\
&b^i_1 \in \{0,1\} && \textnormal{for all } i \in \{1,\ldots,N\} \textnormal{ such that } T^i = 1\\
&w^i_{t\ell} \in \R  && \textnormal{for all } i \in \{1,\ldots,N\}, \; t \in \{1,2\}, \; \ell \in \{1,\ldots,| \mathcal{K}^i|\}.
\end{aligned}
\end{aligned}
\end{equation*}
We observe that the above optimization problem is equivalent to \eqref{prob:h_bar} for the case where $T = 2$. Therefore, we have shown that the optimal objective value of \eqref{prob:bp} is less than or equal to the optimal objective value of \eqref{prob:h_bar}, which concludes our proof of Proposition~\ref{prop:equivalence_2}. 
\halmos \end{proof}


Our proofs of Propositions~\ref{prop:lb_on_h2:T} and \ref{prop:lb_on_h2:N} will make use of two intermediary results, denoted below by Lemmas~\ref{lem:lb_on_h2:ub} and \ref{lem:lb_on_h2:lb}.  These intermediary lemmas establish an upper bound on the optimal objective value of \eqref{prob:bp} and  a lower bound on the optimal objective value of \eqref{prob:h_bar}, respectively. Throughout the proofs, we let $J^{\textnormal{\ref{prob:bp}}} $ denote the optimal objective value of \eqref{prob:bp} and $J^{\textnormal{\ref{prob:h_bar}}}$  denote the optimal objective value of \eqref{prob:h_bar}. 
\begin{lemma} \label{lem:lb_on_h2:ub}
$J^{\textnormal{\ref{prob:bp}}} \le \frac{1}{N} \sum_{i=1}^N g(T^i,x^i)$.
\end{lemma}

\begin{proof}{Proof of Lemma~\ref{lem:lb_on_h2:ub}.}
By removing constraints from the optimization problem~\eqref{prob:bp}, we obtain the following optimization problem:
 \begin{equation*} 
\begin{aligned}
 &\underset{ b,w}{\textnormal{maximize}}&& \frac{1}{N} \sum_{i=1}^N \sum_{t=1}^T \sum_{\ell=1}^{L^i_t-1}( \kappa^i_{\ell+1} - \kappa^i_\ell)b^i_t (1 - w^i_{t\ell}) \\
&\textnormal{subject to}&& \begin{aligned}[t]
&w^{i}_{t\ell} \le  w^i_{t+1,\ell} && \textnormal{for all } i \in \{1,\ldots,N\}, \;t \in \{1,\ldots,T-1\},\; \ell \in \{1,\ldots,|\mathcal{K}^i|\}\\
&w^{i}_{t\ell} \le w^i_{t,\ell+1} && \textnormal{for all } i \in \{1,\ldots,N\},\; t \in \{1,\ldots,T\},\;\ell \in \{1,\ldots,|\mathcal{K}^i|-1\} \\
&b^i_t \le w^i_{t+1,1}&&\textnormal{for all }i \in \{1,\ldots,N\},\; t \in \{1,\ldots,T-1\}\\
&b^i_t \in \{0,1\} && \textnormal{for all } i \in \{1,\ldots,N\}, \; t \in \{1,\dots,T\}\\
&w^i_{t\ell} \in \R  && \textnormal{for all } i \in \{1,\ldots,N\}, \; t \in \{1,\ldots,T\}, \; \ell \in \{1,\ldots,| \mathcal{K}^i|\}.
\end{aligned}
\end{aligned}
\end{equation*}
In greater detail, we observe that the above optimization problem is identical to \eqref{prob:bp}, with the exception that the above optimization problem does not include the constraints of the form $b^j_t \le w^i_{t\ell}$ for all $i, j \in \{1,\ldots,N\}$ and $t \in \{1,\dots,T\}$ such that $g(t,x^i) = \kappa^i_\ell$ and $\mathcal{U}^i_t \cap \mathcal{U}^j_t \neq \emptyset$. Consequently, the optimal objective value of the above optimization problem is greater than or equal to the optimal objective value of \eqref{prob:bp}. Moreover,  it follows from  inspection and from the fact that each term $\kappa^i_{\ell+1} - \kappa^i_\ell$ is strictly positive  that there exists an optimal solution for the above optimization problem that satisfies $b^i_t = \mathbb{I} \left \{ t = T^i \right \}$ for each sample path $i \in \{1,\ldots,N\}$. Therefore, the optimal objective value of the above optimization problem is equal to
 \begin{align*}
  \frac{1}{N} \sum_{i=1}^N \sum_{t=1}^T \sum_{\ell=1}^{L^i_t-1}( \kappa^i_{\ell+1} - \kappa^i_\ell)\mathbb{I} \left \{ t = T^i \right \}(1 - w^i_{t\ell}) &=   \frac{1}{N} \sum_{i=1}^N \sum_{\ell=1}^{L^i_{T^i}-1}( \kappa^i_{\ell+1} - \kappa^i_\ell) (1 - w^i_{T^i \ell})\\
&\le    \frac{1}{N} \sum_{i=1}^N \sum_{\ell=1}^{L^i_{T^i}-1}( \kappa^i_{\ell+1} - \kappa^i_\ell) \\
&= \frac{1}{N} \sum_{i=1}^N g(T^i,x^i),
 \end{align*}
 where the first equality follows from algebra, the inequality follows form the fact that each decision variable $w^i_{T^i \ell}$ must be greater than or equal to zero, and the final equality follows from the definitions of the constants $\kappa^i_{\ell}$ and $L^i_t$. This concludes our proof of Lemma~\ref{lem:lb_on_h2:ub}. 
\halmos
\end{proof}

\begin{lemma} \label{lem:lb_on_h2:lb}
 $J^{\textnormal{\ref{prob:h_bar}}}$ is greater than or equal to the optimal objective value of the following  optimization problem: 
 \begin{equation} \label{prob:h_bar_lb} \tag{$\bar{\textnormal{H}}$-LB}
\begin{aligned}
  &\underset{ b,w}{\textnormal{maximize}}&& \frac{1}{N} \sum_{i=1}^N \sum_{\ell=1}^{L^i_{T^i}-1}( \kappa^i_{\ell+1} - \kappa^i_\ell) \left( b^i_{T^i}  - w^i_{T^i \ell} \right)    \\
&\textnormal{subject to}&& \begin{aligned}[t]
&w^{i}_{T^i \ell} \le w^i_{T^i,\ell+1} && \textnormal{for all } i \in \{1,\ldots,N\},\;\ell \in \{1,\ldots,|\mathcal{K}^i|-1\} \\
&b^j_{T^j} \le w^i_{T^i \ell}&& \textnormal{for all } i, j \in \{1,\ldots,N\}  \textnormal{ such that } g(T^j,x^i) = \kappa^i_\ell,\\
&&& \kappa^i_\ell < g(T^i,x^i),\; T^j < T^i, \; \textnormal{and } \mathcal{U}^i_{T^j} \cap \mathcal{U}^j_{T^j} \neq \emptyset\\
&b^i_{T^i} \in \{0,1\} && \textnormal{for all } i \in \{1,\ldots,N\} \\
&w^i_{T^i \ell} \in \R  && \textnormal{for all } i \in \{1,\ldots,N\},  \; \ell \in \{1,\ldots,| \mathcal{K}^i| - 1\}.
\end{aligned}
\end{aligned}
\end{equation}
\end{lemma}

\begin{proof}{Proof of Lemma~\ref{lem:lb_on_h2:lb}.}
We begin by rewriting \eqref{prob:h_bar} as the following equivalent optimization problem: 
 \begin{equation} \label{prob:h_bar_prime} \tag{$\bar{\textnormal{H}}'$}
\begin{aligned}
  &\underset{ b,w}{\textnormal{maximize}}&& \frac{1}{N} \sum_{i=1}^N \sum_{\ell=1}^{L^i_{T^i}-1}( \kappa^i_{\ell+1} - \kappa^i_\ell) \left( b^i_{T^i}  - w^i_{T^i \ell} \right)  + \frac{1}{N} \sum_{i \in \{1,\ldots,N\}: T^i < T} \sum_{\ell=1}^{L^i_T - 1} ( \kappa^i_{\ell+1} - \kappa^i_\ell)  \left(1 - w^i_{T \ell} \right)  \\
&\textnormal{subject to}&& \begin{aligned}[t]
&w^{i}_{t\ell} \le w^i_{t,\ell+1} && \textnormal{for all } i \in \{1,\ldots,N\},\; t \in \mathcal{T}^i,\;\ell \in \{1,\ldots,|\mathcal{K}^i|-1\} \\
&b^i_{T^i} \le w^i_{T1}&&\textnormal{for all }i \in \{1,\ldots,N\} \textnormal{ such that }  T^i < T\\
&b^j_{T^j} \le w^i_{t \ell}&& \textnormal{for all } i, j \in \{1,\ldots,N\} \textnormal{ and } t \in \mathcal{T}^i \textnormal{ such that } g(T^j,x^i) = \kappa^i_\ell,\\
&&&T^j \le t, \;  T^j < T, \; \textnormal{and } \mathcal{U}^i_{T^j} \cap \mathcal{U}^j_{T^j} \neq \emptyset\\
&b^i_{T^i} \in \{0,1\} && \textnormal{for all } i \in \{1,\ldots,N\}\\
&w^i_{t\ell} \in \R  && \textnormal{for all } i \in \{1,\ldots,N\}, \; t \in \mathcal{T}^i, \; \ell \in \{1,\ldots,| \mathcal{K}^i|\}.
\end{aligned}
\end{aligned}
\end{equation}
We claim that the optimal objective value of the above optimization problem~\eqref{prob:h_bar_prime} is equal to the optimal objective value of \eqref{prob:h_bar}. To see why this is true, we first observe that the above optimization problem is identical to \eqref{prob:h_bar}, except for the fact that the term $  \left(1 - w^i_{T \ell} \right) $ in the objective function of \eqref{prob:h_bar} has been replaced with $  \left(b^i_T - w^i_{T \ell} \right) $ in the objective function of \eqref{prob:h_bar_prime}  for each $i \in \{1,\ldots,N\}$ such that $T^i = T$. Because the new decision variables $b^i_{T^i} \in \{0,1\}$ for each $i \in \{1,\ldots,N\}$ such that $T^i = T$ do not appear in any inequality constraints, and because  each term $\kappa^i_{\ell+1} - \kappa^i_\ell$ is strictly positive,  we observe that there exists an optimal solution for the above optimization problem in which $b^i_{T^i} = 1$ for each $ i \in \{1,\ldots,N\}$ such that $T^i = T$. Hence, we conclude that the optimal objective value of \eqref{prob:h_bar_prime} is equal to the optimal objective value of \eqref{prob:h_bar}.

We now construct a lower bound  approximation of \eqref{prob:h_bar_prime} by modifying its objective function and adding constraints. Indeed, we first observe for every feasible solution for \eqref{prob:h_bar_prime} that the objective function of \eqref{prob:h_bar_prime} satisfies 
\begin{align*}
&\frac{1}{N} \sum_{i=1}^N \sum_{\ell=1}^{L^i_{T^i}-1}( \kappa^i_{\ell+1} - \kappa^i_\ell) \left( b^i_{T^i}  - w^i_{T^i \ell} \right)  + \frac{1}{N} \sum_{i \in \{1,\ldots,N\}: T^i < T} \sum_{\ell=1}^{L^i_T - 1} ( \kappa^i_{\ell+1} - \kappa^i_\ell)  \left(1 - w^i_{T \ell} \right)\\
\ge& \frac{1}{N} \sum_{i=1}^N \sum_{\ell=1}^{L^i_{T^i}-1}( \kappa^i_{\ell+1} - \kappa^i_\ell) \left( b^i_{T^i}  - w^i_{T^i \ell} \right),
\end{align*}
where the inequality holds because  each term $\kappa^i_{\ell+1} - \kappa^i_\ell$ is strictly positive and because the decision variables $w^i_{T \ell}$ are always nonnegative. Therefore, a lower bound on the optimal objective value of \eqref{prob:h_bar_prime} is given by the optimal objective value of the following optimization problem:
 \begin{equation}  \label{prob:h_bar_lb_1} \tag{$\bar{\textnormal{H}}$-LB-1} 
\begin{aligned}
  &\underset{ b,w}{\textnormal{maximize}}&& \frac{1}{N} \sum_{i=1}^N \sum_{\ell=1}^{L^i_{T^i}-1}( \kappa^i_{\ell+1} - \kappa^i_\ell) \left( b^i_{T^i}  - w^i_{T^i \ell} \right)   \\
&\textnormal{subject to}&& \begin{aligned}[t]
&w^{i}_{t\ell} \le w^i_{t,\ell+1} && \textnormal{for all } i \in \{1,\ldots,N\},\; t \in \mathcal{T}^i,\;\ell \in \{1,\ldots,|\mathcal{K}^i|-1\} \\
&b^i_{T^i} \le w^i_{T1}&&\textnormal{for all }i \in \{1,\ldots,N\} \textnormal{ such that }  T^i < T\\
&b^j_{T^j} \le w^i_{t \ell}&& \textnormal{for all } i, j \in \{1,\ldots,N\} \textnormal{ and } t \in \mathcal{T}^i \textnormal{ such that } g(T^j,x^i) = \kappa^i_\ell,\\
&&&T^j \le t, \;  T^j < T, \; \textnormal{and } \mathcal{U}^i_{T^j} \cap \mathcal{U}^j_{T^j} \neq \emptyset\\
&b^i_{T^i} \in \{0,1\} && \textnormal{for all } i \in \{1,\ldots,N\}\\
&w^i_{t\ell} \in \R  && \textnormal{for all } i \in \{1,\ldots,N\}, \; t \in \mathcal{T}^i, \; \ell \in \{1,\ldots,| \mathcal{K}^i|\}.
\end{aligned}
\end{aligned}
\end{equation}
Moreover, we observe from inspection that the optimal objective value of the above optimization problem would not change if we added the constraints $w^i_{T\ell} = 1$ to the above optimization problem for each sample path $i \in \{1,\ldots,N\}$ and $\ell \in \{1,\ldots,| \mathcal{K}^i| \}$ that satisfies $T^i < T$. 
By eliminating these decision variables $w^i_{T \ell}$ that can be constrained without loss of generality to be equal to one, 
 it follows from algebra that the optimal objective value of the optimization problem~\eqref{prob:h_bar_lb_1} is equal to 
 \begin{equation} \label{prob:h_bar_lb_2} \tag{$\bar{\textnormal{H}}$-LB-2}
\begin{aligned}
  &\underset{ b,w}{\textnormal{maximize}}&& \frac{1}{N} \sum_{i=1}^N \sum_{\ell=1}^{L^i_{T^i}-1}( \kappa^i_{\ell+1} - \kappa^i_\ell) \left( b^i_{T^i}  - w^i_{T^i \ell} \right)    \\
&\textnormal{subject to}&& \begin{aligned}[t]
&w^{i}_{T^i \ell} \le w^i_{T^i,\ell+1} && \textnormal{for all } i \in \{1,\ldots,N\},\;\ell \in \{1,\ldots,|\mathcal{K}^i|-1\} \\
&b^j_{T^j} \le w^i_{T^i \ell}&& \textnormal{for all } i, j \in \{1,\ldots,N\}  \textnormal{ such that } g(T^j,x^i) = \kappa^i_\ell,\\
&&&T^j \le T^i, \; T^j < T, \; \textnormal{and } \mathcal{U}^i_{T^j} \cap \mathcal{U}^j_{T^j} \neq \emptyset\\
&b^i_{T^i} \in \{0,1\} && \textnormal{for all } i \in \{1,\ldots,N\} \\
&w^i_{T^i \ell} \in \R  && \textnormal{for all } i \in \{1,\ldots,N\},  \; \ell \in \{1,\ldots,| \mathcal{K}^i|\}.
\end{aligned}
\end{aligned}
\end{equation}
We further observe from inspection that the optimal objective value of the above optimization problem would not change if we added the constraints $w^i_{T^i,| \mathcal{K}^i|} = 1$ to the above optimization problem for each sample path $i \in \{1,\ldots,N\}$.\footnote{{\color{black}We note  that the equality $L^i_{T^i} = | \mathcal{K}^i|$ holds for each sample path $i \in \{1,\ldots,N\}$. This equality can be verified by simply applying the definitions of $L^i_t$, $\mathcal{K}^i$, and $T^i$.}} By eliminating these decision variables $w^i_{T^i,| \mathcal{K}^i|}$ that can be constrained without loss of generality to be equal to one, and by observing that the above optimization problem has a constraint of the form $b^j_{t} \le w^i_{t \ell}$ if and only if  $\ell = | \mathcal{K}^i|$ and $t = T^i = T^j$,  we conclude that the optimal objective value of the optimization problem~\eqref{prob:h_bar_lb_2} is equal to 
 \begin{equation} \tag{\ref{prob:h_bar_lb}}
\begin{aligned}
  &\underset{ b,w}{\textnormal{maximize}}&& \frac{1}{N} \sum_{i=1}^N \sum_{\ell=1}^{L^i_{T^i}-1}( \kappa^i_{\ell+1} - \kappa^i_\ell) \left( b^i_{T^i}  - w^i_{T^i \ell} \right)    \\
&\textnormal{subject to}&& \begin{aligned}[t]
&w^{i}_{T^i \ell} \le w^i_{T^i,\ell+1} && \textnormal{for all } i \in \{1,\ldots,N\},\;\ell \in \{1,\ldots,|\mathcal{K}^i|-1\} \\
&b^j_{T^j} \le w^i_{T^i \ell}&& \textnormal{for all } i, j \in \{1,\ldots,N\}  \textnormal{ such that } g(T^j,x^i) = \kappa^i_\ell,\\
&&& \kappa^i_\ell < g(T^i,x^i),\; T^j < T^i, \; \textnormal{and } \mathcal{U}^i_{T^j} \cap \mathcal{U}^j_{T^j} \neq \emptyset\\
&b^i_{T^i} \in \{0,1\} && \textnormal{for all } i \in \{1,\ldots,N\} \\
&w^i_{T^i \ell} \in \R  && \textnormal{for all } i \in \{1,\ldots,N\},  \; \ell \in \{1,\ldots,| \mathcal{K}^i| - 1\}.
\end{aligned}%
\end{aligned}
\end{equation}
This concludes our proof of Lemma~\ref{lem:lb_on_h2:lb}.
\halmos \end{proof}
In view of the above Lemmas~\ref{lem:lb_on_h2:ub} and \ref{lem:lb_on_h2:lb}, we now present the proofs of Propositions~\ref{prop:lb_on_h2:T} and \ref{prop:lb_on_h2:N}.

\begin{proof}{Proof of Proposition~\ref{prop:lb_on_h2:T}.}
In the above Lemma~\ref{lem:lb_on_h2:ub}, we showed that $J^{\textnormal{\ref{prob:bp}}} \le \frac{1}{N} \sum_{i=1}^N g(T^i,x^i)$. Therefore, it remains for us to show that $J^{\textnormal{\ref{prob:h_bar}}} \ge \frac{1}{T} ( \frac{1}{N} \sum_{i=1}^N g(T^i,x^i) )$. Indeed, 
let 
$$t^* \triangleq \argmax_{t \in \{1,\ldots,T\}} \sum_{i \in \{1,\ldots,N\}: T^i = t}  g(T^i,x^i) $$ be defined as any period that maximizes the sum  of  the rewards $g(T^i,x^i)$ over all of the sample paths $i \in \{1,\ldots,N\}$ that satisfy $T^i = t^*$. We observe from algebra and from our construction of $t^*$ that the following inequality must hold:
\begin{align}
\frac{1}{N}  \sum_{i \in \{1,\ldots,N\}: T^i = t^*}  g(T^i,x^i) \ge  \frac{1}{T} \left( \frac{1}{N} \sum_{i=1}^N g(T^i,x^i) \right).  \label{line:benji_laughing}
 \end{align} 
Moreover, we recall from the above Lemma~\ref{lem:lb_on_h2:lb} that $J^{\textnormal{\ref{prob:h_bar}}}$ is greater than or equal to the optimal objective value of the optimization problem~\eqref{prob:h_bar_lb}. Consider the solution $b,w$ for the optimization problem~\eqref{prob:h_bar_lb} that satisfies the following equalities for each sample path $i \in \{1,\ldots,N\}$, period $t \in \{1,\ldots,T\}$, and $\ell \in \{1,\ldots,| \mathcal{K}^i| - 1 \}$:
 \begin{align*}
 b^i_{T^i} &= \begin{cases}
 1,&\text{if } T^i \in \{t^*,\ldots,T\},\\
 0,&\text{if } T^i \in \{1,\ldots,t^*-1\};
 \end{cases}
  &w^i_{T^i \ell} &= \begin{cases}
 1,&\text{if } T^i \in \{t^* + 1,\ldots,T\},\\
 0,&\text{if } T^i \in \{1,\ldots,t^*\}.
 \end{cases}
 \end{align*}
 We observe from inspection that the solution $b,w$ defined by the above equalities is a feasible solution for the optimization problem~\eqref{prob:h_bar_lb}. Therefore, the optimal objective value of \eqref{prob:h_bar_lb} is greater than or equal to
 \begin{align*}
 & \frac{1}{N} \sum_{i=1}^N \sum_{\ell=1}^{L^i_{T^i}-1}( \kappa^i_{\ell+1} - \kappa^i_\ell) \left( \mathbb{I} \left \{ T^i \in \{t^*,\ldots,T\} \right \}  - \mathbb{I} \left \{ T^i \in \{t^* + 1,\ldots,T\} \right \}  \right)    \\
  &= \frac{1}{N} \sum_{i=1}^N \sum_{\ell=1}^{L^i_{T^i}-1}( \kappa^i_{\ell+1} - \kappa^i_\ell)  \mathbb{I} \left \{ T^i = t^* \right \}    \\
  &=  \frac{1}{N} \sum_{i \in \{1,\ldots,N\}: T^i = t^*} \sum_{\ell=1}^{L^i_{T^i}-1}( \kappa^i_{\ell+1} - \kappa^i_\ell) \\
  &=  \frac{1}{N} \sum_{i \in \{1,\ldots,N\}: T^i = t^*} g(T^i,x^i),
  \end{align*}
  where the last equality follows from the definition of the constants $\kappa^i_\ell$. Combining the above equalities with line~\eqref{line:benji_laughing}, our proof of Proposition~\ref{prop:lb_on_h2:T} is complete. 
\halmos \end{proof}

\begin{proof}{Proof of Proposition~\ref{prop:lb_on_h2:N}.}
Let Assumption~\ref{ass:lipschitz_local} hold. Our proof of Proposition~\ref{prop:lb_on_h2:N} is split into the following two intermediary claims. 

\begin{claim} \label{claim:lb_on_h2:N}
Assume without loss of generality that $g(T^1,x^1) \ge \cdots \ge g(T^N,x^N)$. Then,
\begin{align*}
J^{\textnormal{\ref{prob:h_bar}}} \ge \max_{i \in \{1,\ldots,N\}} \frac{i}{N} g(T^i,x^i) - 2 \epsilon L,
\end{align*}
where $\epsilon$ is the radius of the uncertainty sets (see \S\ref{sec:rob}) and $L$ is the constant from Assumption~\ref{ass:lipschitz_local}. 
\end{claim}
\begin{proof}{Proof of Claim~\ref{claim:lb_on_h2:N}.}
Assume without loss of generality that $g(T^1,x^1) \ge \cdots \ge g(T^N,x^N)$, and consider any arbitrary $i^* \in \{1,\ldots,N\}$. In the remainder of the proof, we will show that 
\begin{align*}
J^{\textnormal{\ref{prob:h_bar}}} \ge \frac{i^*}{N} g(T^{i^*},x^{i^*}) - 2 \epsilon L. 
\end{align*}
To begin, we recall from the above Lemma~\ref{lem:lb_on_h2:lb} that $J^{\textnormal{\ref{prob:h_bar}}}$ is greater than or equal to the optimal objective value of the optimization problem~\eqref{prob:h_bar_lb}. Next, consider the solution $b,w$ for the optimization problem~\eqref{prob:h_bar_lb} that satisfies the following equalities for each sample path $i \in \{1,\ldots,N\}$, period $t \in \{1,\ldots,T\}$, and $\ell \in \{1,\ldots,| \mathcal{K}^i| - 1 \}$:
 \begin{align*}
 b^i_{T^i} &= \begin{cases}
 1,&\text{if } i \in \{1,\ldots,i^*\},\\
 0,&\text{if } i \in \{i^*+1,\ldots,N\};
 \end{cases}
  &w^i_{T^i \ell} &= \begin{cases}
 1,&\text{if } \kappa^i_\ell \ge g(T^{i^*},x^{i^*}) - 2 L \epsilon, \\
 0,&\text{if } T^i \in \{1,\ldots,t^*\}.
 \end{cases}
 \end{align*}
 We will now prove that the solution $b,w$ defined by the above equalities is a feasible solution for the optimization problem~\eqref{prob:h_bar_lb}. Indeed, we readily observe from inspection that the solution $b,w$ defined by the above equalities satisfies the first set of constraints in the optimization problem~\eqref{prob:h_bar_lb}, i.e.,
 \begin{align*}
 \begin{aligned}
 &w^{i}_{T^i \ell} \le w^i_{T^i,\ell+1} && \textnormal{for all } i \in \{1,\ldots,N\},\;\ell \in \{1,\ldots,|\mathcal{K}^i|-1\}.
 \end{aligned}
\end{align*}
To show that the solution $b,w$ satisfies the second set of constraints in the optimization problem~\eqref{prob:h_bar_lb}, consider any sample paths $ i, j \in \{1,\ldots,N\}$ such that $g(T^j,x^i) = \kappa^i_\ell$, $\kappa^i_\ell < g(T^i,x^i)$, $T^j < T^i$, and $\mathcal{U}^i_{T^j} \cap \mathcal{U}^j_{T^j} \neq \emptyset$. We have two cases to consider. First, suppose that the sample path $j$ satisfies $j \in \{i^*+1,\ldots,N\}$. In this case, it follows from the fact that $b^j_{T^j} = 0$ that the inequality $b^j_{T^j} \le w^i_{T^i \ell}$ is satisfied. Second, suppose that sample path $j$ satisfies $j \in \{1,\ldots,i^*\}$. In this case, we observe that 
\begin{align*}
g(T^j,x^i)  &= \kappa^i_\ell \ge g(T^j,x^j) - 2 \epsilon L \ge g(T^{i^*},x^{i^*}) - 2 \epsilon L,
\end{align*}
where the first inequality follows from Assumption~\ref{ass:lipschitz_local} and from the fact that $\mathcal{U}^i_{T^j} \cap \mathcal{U}^j_{T^j} \neq \emptyset$, and the second inequality follows from the fact that $j \in \{1,\ldots,i^*\}$. Therefore, we conclude that from our construction of the solution $b,w$ that $w^i_{T^i \ell} = 1$, which implies that the inequality $b^j_{T^j} \le w^i_{T^i \ell}$ is satisfied. We have thus shown that the solution $b,w$ satisfies the second set of constraints in the optimization problem~\eqref{prob:h_bar_lb}, i.e., 
\begin{align*}
\begin{aligned}
&b^j_{T^j} \le w^i_{T^i \ell}&& \textnormal{for all } i, j \in \{1,\ldots,N\}  \textnormal{ such that } g(T^j,x^i) = \kappa^i_\ell,\\
&&&\kappa^i_\ell < g(T^i,x^i),\; T^j < T^i, \; \textnormal{and } \mathcal{U}^i_{T^j} \cap \mathcal{U}^j_{T^j} \neq \emptyset.
\end{aligned}
\end{align*}
This concludes our proof that the solution $b,w$ is a feasible solution for the optimization problem~\eqref{prob:h_bar_lb}.

Since $b,w$ is a feasible solution for the optimization problem~\eqref{prob:h_bar_lb}, we have the following:
\begin{align}
J^{\textnormal{\ref{prob:h_bar}}} &\ge  \frac{1}{N} \sum_{i=1}^N \sum_{\ell=1}^{L^i_{T^i}-1}( \kappa^i_{\ell+1} - \kappa^i_\ell) \left( \mathbb{I} \left \{ i \in \{1,\ldots,i^*\} \right \}  - \mathbb{I} \left \{ \kappa^i_\ell \ge g(T^{i^*},x^{i^*}) - 2 L \epsilon \right \} \right) \label{line:silly:1}\\
&= \frac{1}{N} \sum_{i=1}^{i^*} \sum_{\ell =1}^{L^i_{T^i}-1}( \kappa^i_{\ell+1} - \kappa^i_\ell) -  \frac{1}{N} \sum_{i=1}^{N}\sum_{\ell=1}^{L^i_{T^i}-1}   \mathbb{I} \left \{ \kappa^i_\ell \ge g(T^{i^*},x^{i^*}) - 2 L \epsilon \right \} ( \kappa^i_{\ell+1} - \kappa^i_\ell)  \label{line:silly:2} \\
&=   \frac{1}{N} \sum_{i=1}^{i^*} g(T^i,x^i)  -  \frac{1}{N} \sum_{i=1}^{N}\sum_{\ell=1}^{L^i_{T^i}-1}   \mathbb{I} \left \{ \kappa^i_\ell \ge g(T^{i^*},x^{i^*}) - 2 L \epsilon \right \} ( \kappa^i_{\ell+1} - \kappa^i_\ell)  \label{line:silly:2a} \\
&\ge  \frac{1}{N} \sum_{i=1}^{i^*} g(T^i,x^i) \notag \\
&\quad -  \frac{1}{N} \sum_{i=1}^{N} \mathbb{I} \left \{  \exists \ell \in \{1,\ldots,L^i_{T^i} - 1 \} \textnormal{ such that } \kappa^i_{ \ell} \ge g(T^{i^*},x^{i^*}) - 2 L \epsilon \right \}\left( g(T^i,x^i) - \left( g(T^{i^*},x^{i^*}) - 2 L \epsilon \right ) \right) \label{line:silly:3} \\
&\ge  \frac{1}{N} \sum_{i=1}^{i^*} g(T^i,x^i) -  \frac{1}{N} \sum_{i=1}^{N} \mathbb{I} \left \{  g(T^i,x^i) \ge g(T^{i^*},x^{i^*}) - 2 L \epsilon \right \}\left( g(T^i,x^i) - g(T^{i^*},x^{i^*}) + 2 L \epsilon \right )  \label{line:silly:3a} \\
&= \frac{1}{N} \sum_{i=1}^{i^*} \left( g(T^{i^*},x^{i^*}) - 2 L \epsilon \right)  - \frac{1}{N} \sum_{i \in \{i^* + 1,\ldots,N\}:  \;g(T^i,x^i) \ge g(T^{i^*},x^{i^*}) - 2 L \epsilon} \left( g(T^i,x^i) - g(T^{i^*},x^{i^*}) + 2 L \epsilon \right )  \label{line:silly:4}\\
&\ge \frac{1}{N} \sum_{i=1}^{i^*} \left( g(T^{i^*},x^{i^*}) - 2 L \epsilon \right)  - \frac{1}{N} \sum_{i \in \{i^* + 1,\ldots,N\}:  \;g(T^i,x^i) \ge g(T^{i^*},x^{i^*}) - 2 L \epsilon}  2 L \epsilon  \label{line:silly:5}  \\
&\ge \frac{1}{N} \sum_{i=1}^{i^*} \left( g(T^{i^*},x^{i^*}) - 2 L \epsilon \right)  - \frac{ N - i^*}{N}  2 L \epsilon   \label{line:silly:6} \\
&= \frac{i^*}{N} g(T^{i^*},x^{i^*}) - 2 \epsilon L, \label{line:silly:7}
\end{align}
Indeed, \eqref{line:silly:1} holds because $J^{\textnormal{\ref{prob:h_bar}}}$ is greater than or equal to the optimal objective value of the optimization problem~\eqref{prob:h_bar_lb} and because $b,w$ is a feasible but possibly suboptimal solution for the optimization problem~\eqref{prob:h_bar_lb}. Lines~\eqref{line:silly:2}, \eqref{line:silly:2a}, and \eqref{line:silly:3} follow from algebra. Line~\eqref{line:silly:3a} follows from the fact that $\kappa_0^i < \cdots < \kappa^i_{L^i_{T^i}}$. Line~\eqref{line:silly:4} follows from rearranging terms. Line~\eqref{line:silly:5} follows from the fact that $g(T^i,x^i) - g(T^{i^*},x^{i^*}) \le 0$ for all $i \in \{i^*+1,\ldots,N\}$. Lines~\eqref{line:silly:6} and \eqref{line:silly:7} follow from algebra. 

Since $i^* \in \{1,\ldots,N\}$ was chosen arbitrarily, our proof of Claim~\ref{claim:lb_on_h2:N} is complete.  \halmos \end{proof}

\begin{claim} \label{claim:lb_on_h2:N2}
$J^{\textnormal{\ref{prob:h_bar}}} \ge \frac{1}{\log N + 1} \left( \frac{1}{N} \sum_{i=1}^N g(T^i,x^i) \right) -  2\epsilon L$. 
\end{claim}
\begin{proof}{Proof of Claim~\ref{claim:lb_on_h2:N2}.}
Assume without loss of generality that $g(T^1,x^1) \ge \cdots \ge g(T^N,x^N)$. In Claim~\ref{claim:lb_on_h2:N}, we showed that 
\begin{align*}
J^{\textnormal{\ref{prob:h_bar}}} \ge \max_{i \in \{1,\ldots,N\}} \frac{i}{N} g(T^i,x^i) - 2 \epsilon L.
\end{align*}
For notational convenience, let us define $\gamma^i \triangleq  \frac{i}{N} g(T^i,x^i) $ for each sample path $i \in \{1,\ldots,N\}$. With this notation, we observe from algebra that
\begin{align*}
\frac{1}{\log N + 1} \left( \frac{1}{N} \sum_{i=1}^N g(T^i,x^i) \right)  - 2 \epsilon L &= \frac{1}{\log N + 1} \left(  \sum_{i=1}^N \frac{1}{i} \gamma^i \right)  - 2 \epsilon L\\
&\le  \frac{1}{\log N + 1} \left( \max_{i \in \{1,\ldots,N\}} \gamma^{i} \right)  \sum_{i=1}^N \frac{1}{i}  - 2 \epsilon L\\
& \le   \max_{i \in \{1,\ldots,N\}} \gamma^{i}  - 2 \epsilon L \\
&=  \max_{i \in \{1,\ldots,N\}}  \frac{i}{N} g(T^i,x^i)    - 2 \epsilon L \\
&\le J^{\textnormal{\ref{prob:h_bar}}}, 
\end{align*}
where the last line follows from Claim~\ref{claim:lb_on_h2:N}. This completes the proof of Claim~\ref{claim:lb_on_h2:N2}. 
\halmos \end{proof}

Combining Claim~\ref{claim:lb_on_h2:N2} with Lemma~\ref{lem:lb_on_h2:ub}  completes our proof of Proposition~\ref{prop:lb_on_h2:N}. 
\halmos
\end{proof}
}

\section{Additional Numerical Results} \label{appx:numerics}
In this appendix, we present numerical results for additional parameter settings which were omitted from \S\ref{sec:barrier} due to length considerations.

%

    \afterpage{%
\null
\vfill
\begin{table}[H]
\TABLE{Barrier Option (Asymmetric) - Expected Reward. \label{tbl:barrier_asymmetric}}
{\centering \small
\begin{tabular}{cclcccc}
  \hline
  &&& \multicolumn{3}{c}{Initial Price}& \\
 $d$ & Method & Basis functions & $\bar{x} = 90$ & $\bar{x} = 100$ & $\bar{x} = 110$ & \# of Sample Paths \\
 \hline
  8 & RO & maxprice & \textbf{57.53 (0.16)} & \textbf{67.25 (0.27)} & \textbf{69.10 (0.42)} & $10^3$ training, $10^3$ validation \\ 
    8 & LS & one, pricesKO, KOind, payoff & 56.44 (0.11) & 65.52 (0.16) & 68.34 (0.09) & $10^5$ \\ 
    8 & LS & one, pricesKO, payoff & 56.54 (0.12) & 65.74 (0.17) & 68.31 (0.08) & $10^5$ \\ 
    8 & LS & payoff, KOind, pricesKO & 56.44 (0.11) & 65.52 (0.16) & 68.34 (0.09) & $10^5$ \\ 
    8 & LS & pricesKO, payoff & 56.51 (0.12) & 65.75 (0.18) & 68.31 (0.08) & $10^5$ \\ 
    8 & LS & one, prices, payoff & 55.56 (0.09) & 61.92 (0.09) & 60.07 (0.12) & $10^5$ \\ 
    8 & LS & one & 45.80 (0.06) & 52.84 (0.07) & 48.56 (0.08) & $10^5$ \\ 
    8 & LS & one, KOind, prices & 52.33 (0.09) & 61.85 (0.10) & 63.88 (0.08) & $10^5$ \\ 
    8 & LS & one, prices & 48.49 (0.11) & 53.03 (0.09) & 50.83 (0.06) & $10^5$ \\ 
    8 & LS & one, pricesKO & 51.67 (0.08) & 61.49 (0.10) & 63.14 (0.09) & $10^5$ \\ 
    8 & LS & maxprice, KOind, pricesKO & 52.63 (0.09) & 61.97 (0.10) & 64.14 (0.12) & $10^5$ \\ 
    8 & PO & payoff, KOind, pricesKO & 53.61 (0.18) & 58.29 (0.36) & 52.96 (0.30) & $2 \times 10^3$ outer, $500$ inner \\ 
    8 & PO & prices & 53.90 (0.16) & 60.66 (0.28) & 58.14 (0.12) & $2 \times 10^3$ outer, $500$ inner \\ 
    8 & Tree & payoff, time & 55.75 (0.13) & 62.78 (0.31) & 63.95 (0.52) & $10^5$ \\ 
    8 & Tree & maxprice, time & 55.75 (0.13) & 62.78 (0.31) & 63.95 (0.52) & $10^5$ \\ 
   \hline
 16 & RO & maxprice & \textbf{72.56 (0.15)} & \textbf{77.44 (0.29)} & 69.74 (0.30) & $10^3$ training, $10^3$ validation \\ 
   16 & LS & one, pricesKO, KOind, payoff & 70.94 (0.08) & 76.10 (0.10) & \textbf{70.59 (0.13)} & $10^5$ \\ 
   16 & LS & one, pricesKO, payoff & 71.06 (0.08) & 76.14 (0.08) & 70.05 (0.12) & $10^5$ \\ 
   16 & LS & payoff, KOind, pricesKO & 70.94 (0.08) & 76.10 (0.10) & \textbf{70.59 (0.13)} & $10^5$ \\ 
   16 & LS & pricesKO, payoff & 71.06 (0.08) & 76.14 (0.08) & 70.04 (0.12) & $10^5$ \\ 
   16 & LS & one, prices, payoff & 67.57 (0.08) & 67.07 (0.16) & 56.36 (0.14) & $10^5$ \\ 
   16 & LS & one & 59.09 (0.05) & 56.20 (0.11) & 45.55 (0.05) & $10^5$ \\ 
   16 & LS & one, KOind, prices & 66.78 (0.06) & 71.06 (0.14) & 63.54 (0.20) & $10^5$ \\ 
   16 & LS & one, prices & 59.35 (0.07) & 57.68 (0.08) & 49.67 (0.07) & $10^5$ \\ 
   16 & LS & one, pricesKO & 66.51 (0.05) & 70.70 (0.13) & 62.53 (0.16) & $10^5$ \\ 
   16 & LS & maxprice, KOind, pricesKO & 67.00 (0.06) & 71.27 (0.13) & 64.04 (0.19) & $10^5$ \\ 
   16 & PO & payoff, KOind, pricesKO & 65.14 (0.31) & 62.12 (0.34) & 47.16 (0.37) & $2 \times 10^3$ outer, $500$ inner \\ 
   16 & PO & prices & 66.69 (0.18) & 66.23 (0.23) & 53.84 (0.26) & $2 \times 10^3$ outer, $500$ inner \\ 
   16 & Tree & payoff, time & 68.88 (0.15) & 71.11 (0.19) & 55.05 (0.11) & $10^5$ \\ 
   16 & Tree & maxprice, time & 68.88 (0.15) & 71.11 (0.19) & 55.05 (0.11) & $10^5$ \\ 
   \hline
 32 & RO & maxprice & \textbf{84.12 (0.28)} & 79.14 (0.40) & 60.66 (0.57) & $10^3$ training, $10^3$ validation \\ 
   32 & LS & one, pricesKO, KOind, payoff & 82.38 (0.09) & \textbf{79.17 (0.10)} & \textbf{62.92 (0.15)} & $10^5$ \\ 
   32 & LS & one, pricesKO, payoff & 82.45 (0.10) & 78.91 (0.10) & 62.27 (0.14) & $10^5$ \\ 
   32 & LS & payoff, KOind, pricesKO & 82.38 (0.09) & \textbf{79.17 (0.10)} & \textbf{62.92 (0.15)} & $10^5$ \\ 
   32 & LS & pricesKO, payoff & 82.45 (0.10) & 78.91 (0.10) & 62.27 (0.14) & $10^5$ \\ 
   32 & LS & one, prices, payoff & 73.46 (0.19) & 62.81 (0.17) & 48.79 (0.17) & $10^5$ \\ 
   32 & LS & one & 64.26 (0.12) & 50.78 (0.09) & 44.32 (0.05) & $10^5$ \\ 
   32 & LS & one, KOind, prices & 77.53 (0.11) & 71.06 (0.12) & 56.06 (0.21) & $10^5$ \\ 
   32 & LS & one, prices & 64.63 (0.11) & 55.21 (0.12) & 46.58 (0.05) & $10^5$ \\ 
   32 & LS & one, pricesKO & 77.36 (0.10) & 70.44 (0.16) & 55.11 (0.24) & $10^5$ \\ 
   32 & LS & maxprice, KOind, pricesKO & 77.67 (0.11) & 71.53 (0.13) & 56.30 (0.22) & $10^5$ \\ 
   32 & PO & payoff, KOind, pricesKO & 70.79 (0.49) & 55.67 (0.31) & 35.36 (1.21) & $2 \times 10^3$ outer, $500$ inner \\ 
   32 & PO & prices & 73.86 (0.32) & 62.09 (0.22) & 44.23 (0.29) & $2 \times 10^3$ outer, $500$ inner \\ 
   32 & Tree & payoff, time & 77.64 (0.44) & 63.91 (0.13) & 50.95 (0.04) & $10^5$ \\ 
   32 & Tree & maxprice, time & 77.64 (0.44) & 63.91 (0.13) & 50.95 (0.04) & $10^5$ \\ 
   \hline
\end{tabular}

}
{Optimal is indicated in bold for each number of assets $d \in \{8,16,32\}$ and initial price $\bar{x} \in \{90,100,110\}$.  
Problem parameters are $T = 54$, $Y = 3$, $r = 0.05$, $K = 100$, $B_0 = 150$, $\delta = 0.25$, $\sigma_a = 0.1 + s\frac{a \times d}{5}$, $\rho_{a,a'} = 0$ for all $a \neq a'$.  }
\end{table}
\null
\vfill
\clearpage
}

    \afterpage{%
\null
\vfill
\begin{table}[H]
\TABLE{Barrier Option (Asymmetric) - Computation Times. \label{tbl:barrier_time_asymmetric}}
{\centering \small
\begin{tabular}{cclcccc}
  \hline
  &&& \multicolumn{3}{c}{Initial Price}& \\
 $d$ & Method & Basis functions & $\bar{x} = 90$ & $\bar{x} = 100$ & $\bar{x} = 110$ & \# of Sample Paths \\
 \hline
  8 & RO & maxprice &   7.19 (0.37) &   8.62 (1.09) &  18.38 (4.59) & $10^3$ training, $10^3$ validation \\ 
    8 & LS & one, pricesKO, KOind, payoff &   3.29 (0.2) &   3.20 (0.31) &   3.24 (0.15) & $10^5$ \\ 
    8 & LS & one, pricesKO, payoff &   3.16 (0.2) &   3.05 (0.25) &   3.13 (0.21) & $10^5$ \\ 
    8 & LS & payoff, KOind, pricesKO &   3.76 (0.63) &   3.29 (0.58) &   3.05 (0.44) & $10^5$ \\ 
    8 & LS & pricesKO, payoff &   3.03 (0.14) &   2.89 (0.23) &   2.87 (0.12) & $10^5$ \\ 
    8 & LS & one, prices, payoff &   3.31 (0.56) &   2.97 (0.29) &   2.90 (0.24) & $10^5$ \\ 
    8 & LS & one &   0.89 (0.07) &   0.93 (0.06) &   0.94 (0.08) & $10^5$ \\ 
    8 & LS & one, KOind, prices &   3.39 (0.55) &   2.99 (0.26) &   2.98 (0.21) & $10^5$ \\ 
    8 & LS & one, prices &   2.63 (0.19) &   2.71 (0.22) &   2.72 (0.19) & $10^5$ \\ 
    8 & LS & one, pricesKO &   2.78 (0.35) &   2.62 (0.18) &   2.53 (0.21) & $10^5$ \\ 
    8 & LS & maxprice, KOind, pricesKO &   3.01 (0.13) &   2.85 (0.26) &   2.87 (0.08) & $10^5$ \\ 
    8 & PO & payoff, KOind, pricesKO &  18.75 (0.71) &  18.12 (0.41) &  16.23 (0.71) & $2 \times 10^3$ outer, $500$ inner \\ 
    8 & PO & prices &  23.31 (0.45) &  23.57 (0.36) &  22.69 (0.29) & $2 \times 10^3$ outer, $500$ inner \\ 
    8 & Tree & payoff, time &  13.80 (0.28) &  15.72 (5.51) &  57.70 (10.81) & $10^5$ \\ 
    8 & Tree & maxprice, time &  13.47 (0.19) &  15.04 (5.42) &  55.45 (10.47) & $10^5$ \\ 
   \hline
 16 & RO & maxprice &   6.80 (1) &  10.15 (1.4) &  51.58 (9.06) & $10^3$ training, $10^3$ validation \\ 
   16 & LS & one, pricesKO, KOind, payoff &   5.55 (0.61) &   5.30 (0.34) &   5.01 (0.42) & $10^5$ \\ 
   16 & LS & one, pricesKO, payoff &   5.50 (0.78) &   5.23 (0.33) &   5.24 (0.33) & $10^5$ \\ 
   16 & LS & payoff, KOind, pricesKO &   4.74 (0.72) &   4.62 (0.54) &   4.55 (0.41) & $10^5$ \\ 
   16 & LS & pricesKO, payoff &   5.29 (0.37) &   5.18 (0.38) &   4.87 (0.26) & $10^5$ \\ 
   16 & LS & one, prices, payoff &   5.57 (0.55) &   4.93 (0.29) &   5.10 (0.32) & $10^5$ \\ 
   16 & LS & one &   1.17 (0.11) &   1.26 (0.08) &   1.24 (0.05) & $10^5$ \\ 
   16 & LS & one, KOind, prices &   6.21 (1.46) &   5.72 (0.76) &   5.16 (0.51) & $10^5$ \\ 
   16 & LS & one, prices &   6.02 (1.63) &   5.11 (1.25) &   4.86 (0.73) & $10^5$ \\ 
   16 & LS & one, pricesKO &   4.61 (0.62) &   4.58 (0.52) &   4.67 (0.58) & $10^5$ \\ 
   16 & LS & maxprice, KOind, pricesKO &   5.11 (0.64) &   4.90 (0.6) &   4.73 (0.27) & $10^5$ \\ 
   16 & PO & payoff, KOind, pricesKO &  32.13 (1.23) &  28.51 (0.96) &  25.56 (0.48) & $2 \times 10^3$ outer, $500$ inner \\ 
   16 & PO & prices &  43.73 (0.61) &  43.75 (0.67) &  41.84 (0.65) & $2 \times 10^3$ outer, $500$ inner \\ 
   16 & Tree & payoff, time &  14.24 (0.4) &  48.39 (0.94) &   7.89 (0.09) & $10^5$ \\ 
   16 & Tree & maxprice, time &  13.35 (0.3) &  47.49 (0.74) &   8.31 (0.19) & $10^5$ \\ 
   \hline
 32 & RO & maxprice &   7.88 (1.21) &  22.32 (7.25) & 154.12 (29.91) & $10^3$ training, $10^3$ validation \\ 
   32 & LS & one, pricesKO, KOind, payoff &  15.07 (2.97) &  11.55 (1.37) &  13.68 (2.69) & $10^5$ \\ 
   32 & LS & one, pricesKO, payoff &  15.29 (2.15) &  12.20 (1.68) &  13.04 (3.25) & $10^5$ \\ 
   32 & LS & payoff, KOind, pricesKO &  13.13 (2.59) &  10.96 (1.81) &  11.30 (1.81) & $10^5$ \\ 
   32 & LS & pricesKO, payoff &  15.11 (3.27) &  11.39 (1.91) &  11.62 (1.44) & $10^5$ \\ 
   32 & LS & one, prices, payoff &  16.78 (4.89) &  10.68 (1.52) &  11.92 (3.11) & $10^5$ \\ 
   32 & LS & one &   1.62 (0.6) &   1.51 (0.16) &   1.69 (0.74) & $10^5$ \\ 
   32 & LS & one, KOind, prices &  14.61 (2.44) &  10.77 (0.79) &  11.48 (1.94) & $10^5$ \\ 
   32 & LS & one, prices &  12.49 (3.23) &  10.42 (0.57) &  10.47 (0.84) & $10^5$ \\ 
   32 & LS & one, pricesKO &  13.61 (3.66) &  10.55 (0.43) &  11.02 (1.31) & $10^5$ \\ 
   32 & LS & maxprice, KOind, pricesKO &  12.63 (0.86) &  10.82 (1.36) &  11.60 (1.09) & $10^5$ \\ 
   32 & PO & payoff, KOind, pricesKO &  61.15 (1.43) &  52.57 (4.38) &  52.56 (1.89) & $2 \times 10^3$ outer, $500$ inner \\ 
   32 & PO & prices & 108.40 (2.23) & 108.40 (3.02) &  97.29 (2.01) & $2 \times 10^3$ outer, $500$ inner \\ 
   32 & Tree & payoff, time &  44.24 (19.71) &  15.35 (0.31) &   9.41 (0.32) & $10^5$ \\ 
   32 & Tree & maxprice, time &  43.92 (20.08) &  16.22 (0.62) &  10.92 (0.4) & $10^5$ \\ 
   \hline
\end{tabular}

}
{Problem parameters are $T = 54$, $Y = 3$, $r = 0.05$, $K = 100$, $B_0 = 150$, $\delta = 0.25$, $\sigma_a = 0.1 + s\frac{a \times d}{5}$, $\rho_{a,a'} = 0$ for all $a \neq a'$.  }
\end{table}
\null
\vfill
\clearpage
}

    \afterpage{%
\null
\vfill
\begin{table}[H]
\TABLE{Barrier Option (Asymmetric) - Best Choice of Robustness Parameter. \label{tbl:barrier_epsilon_asymmetric}}
{\centering \small
\begin{tabular}{cccc}
  \hline
  & \multicolumn{3}{c}{Initial Price} \\
 $d$ & $\bar{x} = 90$ & $\bar{x} = 100$ & $\bar{x} = 110$\\
 \hline
  8 & 8.2 (1.75) & 5.9 (1.73) & 4.0 (1.33) \\ 
   16 & 6.2 (1.69) & 4.6 (1.26) & 3.4 (1.17) \\ 
   32 & 5.5 (1.27) & 4.0 (1.05) & 2.3 (0.95) \\ 
   \hline
\end{tabular}
}
%
%
%
{ {\color{black}Best choice of robustness parameter $\epsilon$ found using the validation method from \S\ref{sec:params} in the robust optimization problems constructed from a training dataset of  size $N = 10^3$ and validation set of size $\bar{N} = 10^3$.  The remaining parameters are the same as those shown in Table~\ref{tbl:barrier_asymmetric}.}  }\label{tbl:barrier_asymmetric_epsilon_single}

\end{table}
\null
\vfill
\clearpage
}

\afterpage{%
\null
\vfill
\begin{figure}[H]
\centering
\FIGURE{\includegraphics[width=0.85\linewidth]{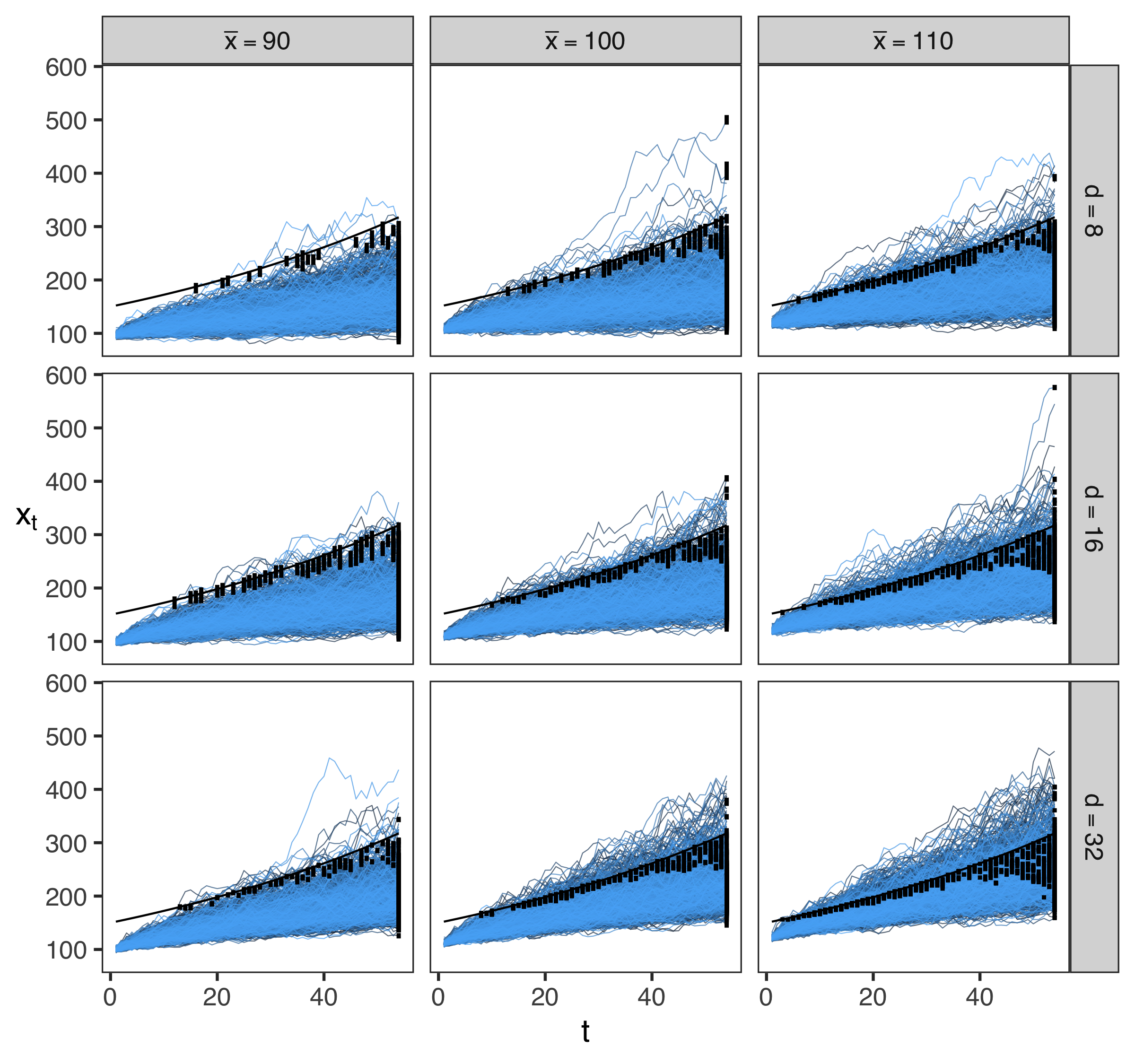}}{Barrier Option (Symmetric) - Visualization of Robust Optimization Stopping Rules.\label{fig:barrier_policy_symmetric}}
{ Each plot shows the exercise policies obtained from solving the {\color{black}heuristic~\eqref{prob:h_bar}} constructed from a training dataset of  size $N = 10^3$ and with the robustness parameter selected using a validation set of size $\bar{N} = 10^3$. The problem parameters are the same as those shown in Table~\ref{tbl:barrier_symmetric}.  }
\end{figure}

\null
\vfill

\clearpage
}

\afterpage{%
\null
\vfill
\begin{figure}[H]
\centering
\FIGURE{\includegraphics[width=0.85\linewidth]{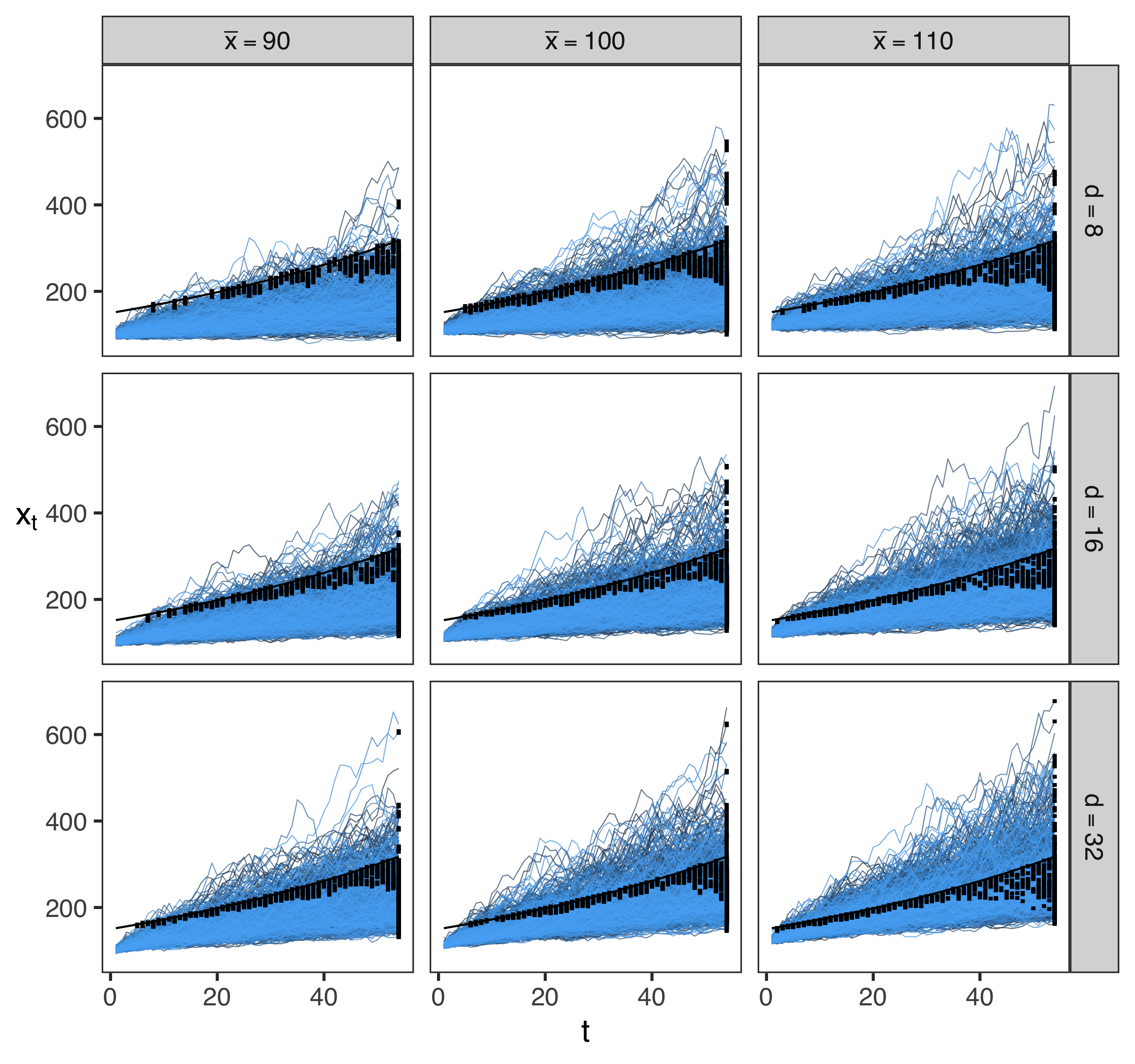}}{Barrier Option (Asymmetric) - Visualization of Robust Optimization Stopping Rules.\label{fig:barrier_policy_asymmetric}}
{ Each plot shows the exercise policies obtained from solving the {\color{black}heuristic~\eqref{prob:h_bar}} constructed from a training dataset of  size $N = 10^3$ and with the robustness parameter selected using a validation set of size $\bar{N} = 10^3$. The problem parameters are the same as those shown in Table~\ref{tbl:barrier_asymmetric}.  }
\end{figure}

\null
\vfill

\clearpage
}

\afterpage{%
\null
\vfill
\begin{figure}[H]
\centering
\FIGURE{\includegraphics[width=0.85\linewidth]{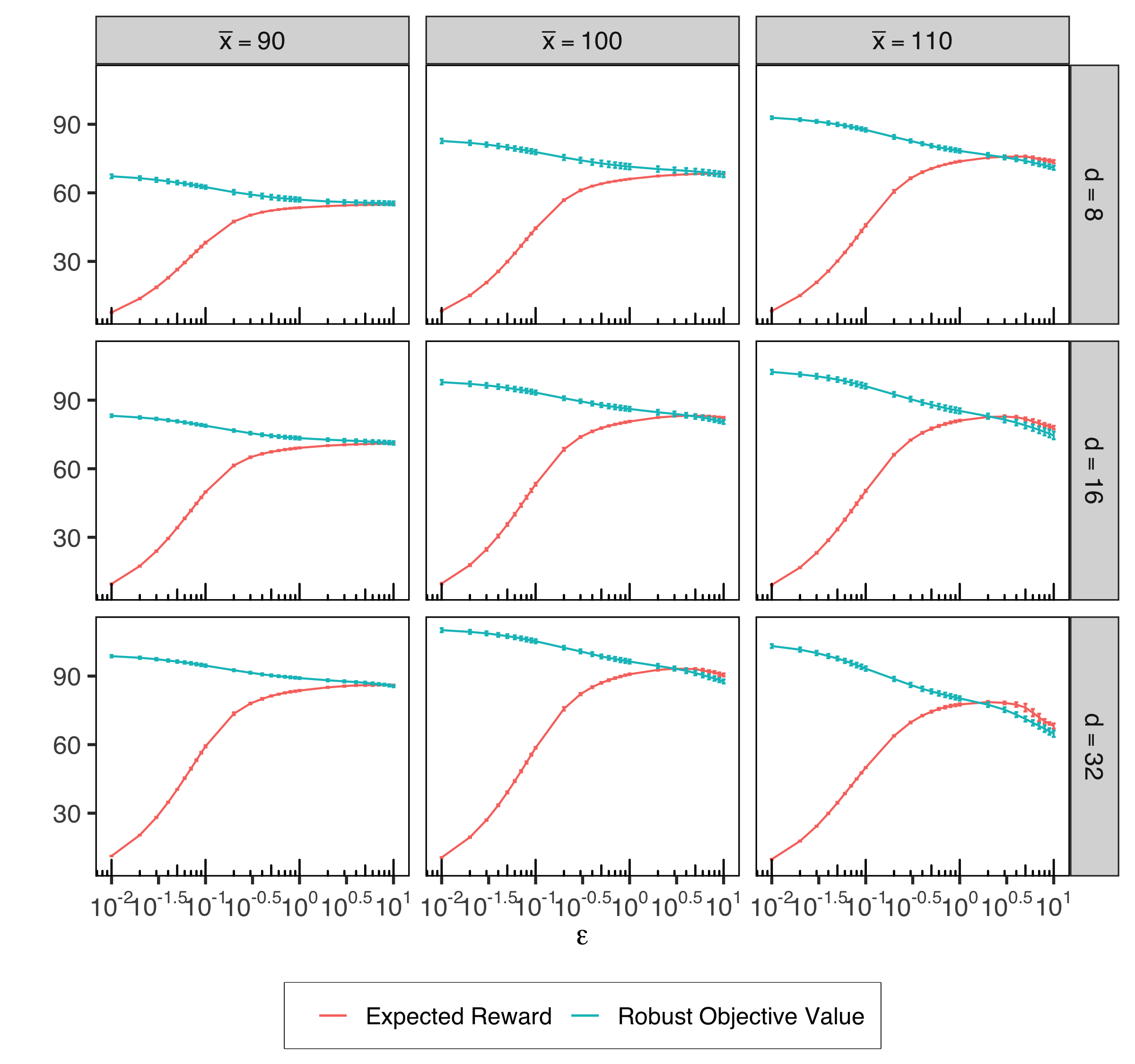}}{Barrier Option (Symmetric) - Impact of Robustness Parameter on Reward.\label{fig:barrier_epsilon_symmetric}}
{  Each plot shows  the robust objective value and expected reward of policies obtained by solving the {\color{black}heuristic~\eqref{prob:h_bar}} constructed from training datasets of size $N = 10^3$.  The problem parameters are the same as those shown in Table~\ref{tbl:barrier_symmetric}. 
 }
\end{figure}

\null
\vfill

\clearpage
}

\afterpage{%
\null
\vfill
\begin{figure}[H]
\centering
\FIGURE{\includegraphics[width=0.85\linewidth]{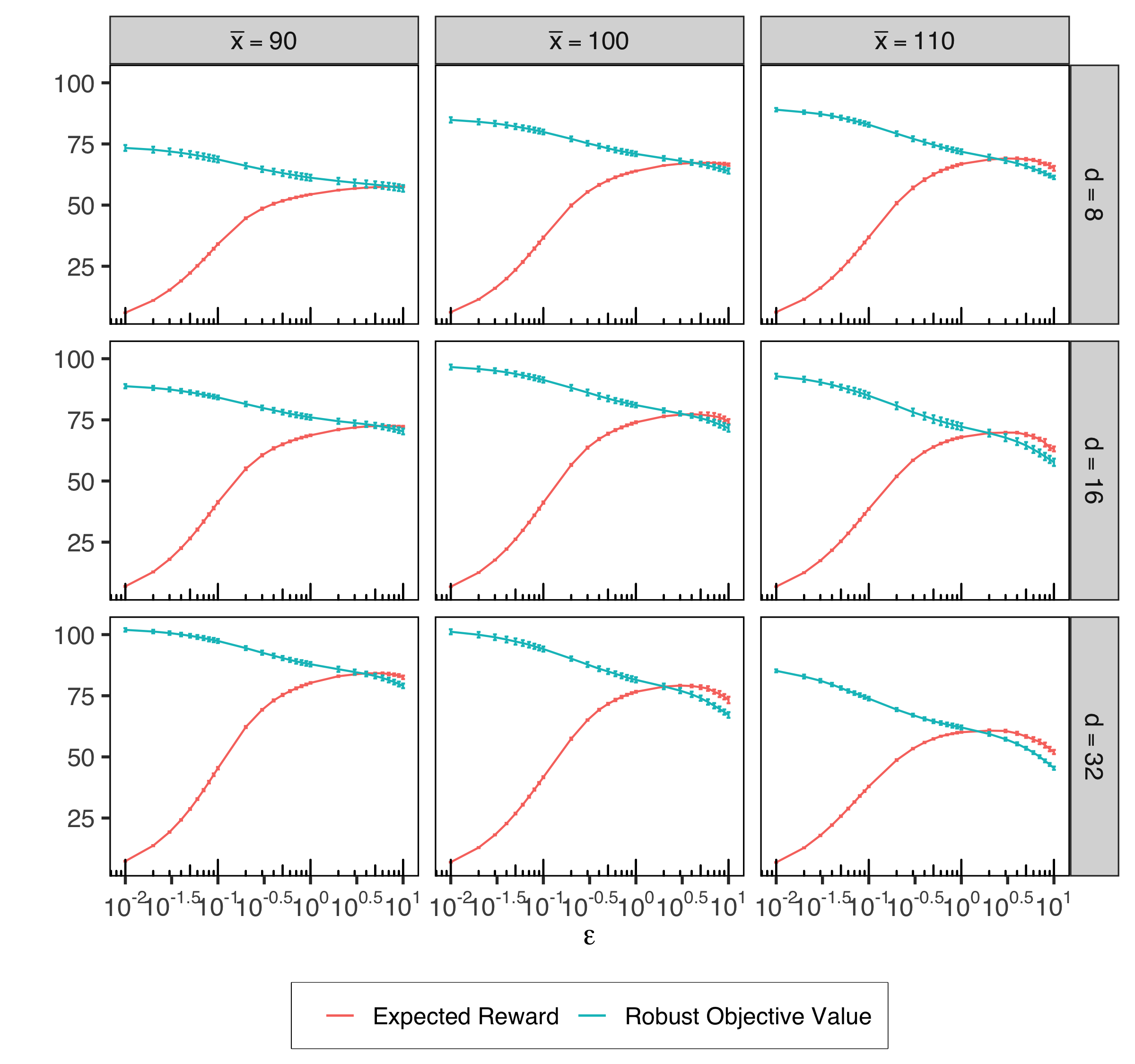}}{Barrier Option (Asymmetric) - Impact of Robustness Parameter on Reward.\label{fig:barrier_epsilon_asymmetric}}
{ Each plot shows  the robust objective value and expected reward of policies obtained by solving the {\color{black}heuristic~\eqref{prob:h_bar}} constructed from training datasets of size $N = 10^3$.  The problem parameters are the same as those shown in Table~\ref{tbl:barrier_asymmetric}.  }
\end{figure}

\null
\vfill

\clearpage
}

\afterpage{%
\null
\vfill
\begin{figure}[H]
\centering
\FIGURE{\includegraphics[width=0.85\linewidth]{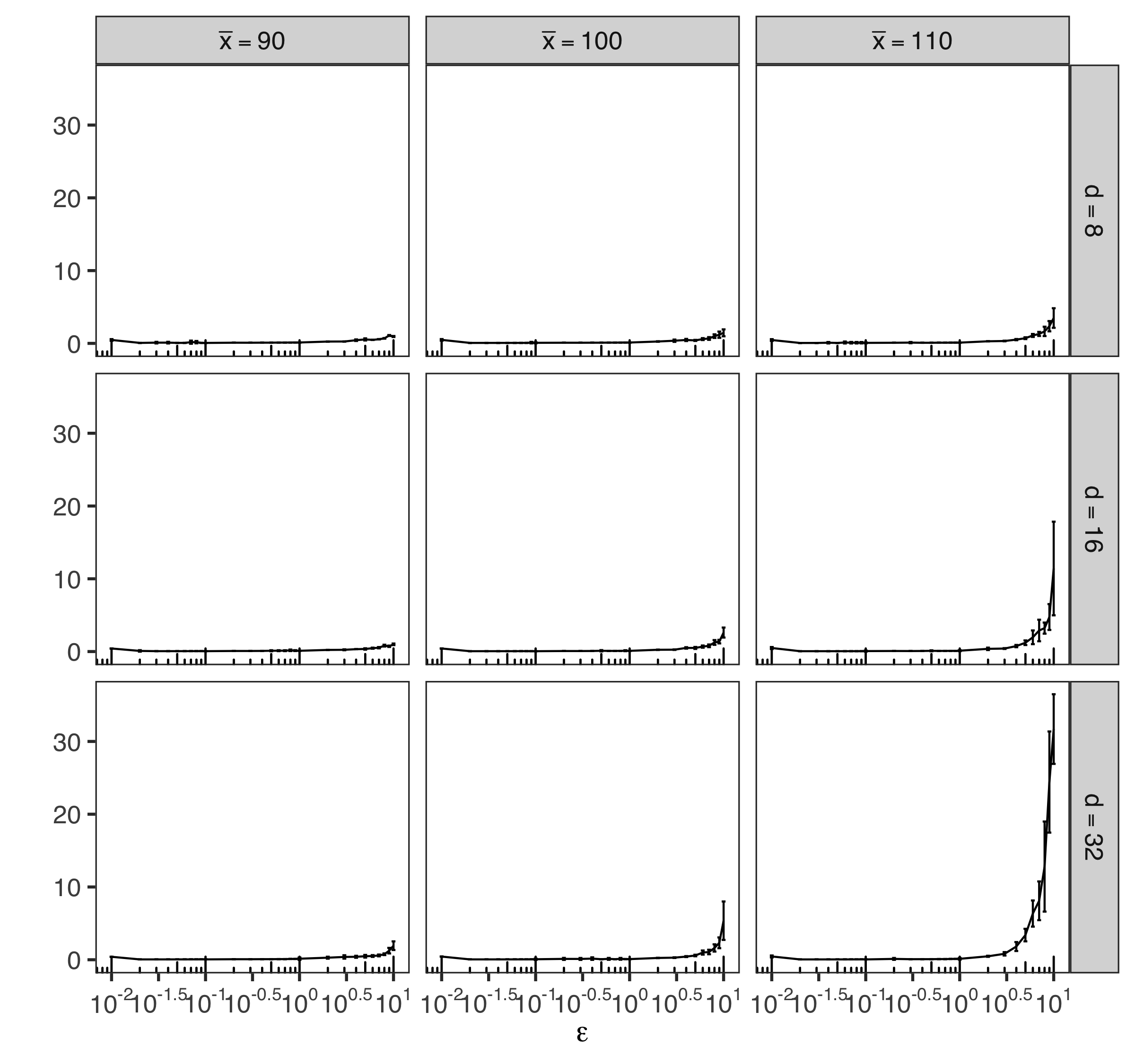}}{Barrier Option (Symmetric) - Impact of Robustness Parameter on Computation Time.\label{fig:barrier_epsilon_symmetric_time}}
{  Each plot shows the computation times from solving the {\color{black}heuristic~\eqref{prob:h_bar}} constructed from training datasets of  size $N = 10^3$. The problem parameters are the same as those shown in Table~\ref{tbl:barrier_symmetric}. } 
\end{figure}


\null
\vfill

\clearpage
}

\afterpage{%
\null
\vfill
\begin{figure}[H]
\centering
\FIGURE{\includegraphics[width=0.85\linewidth]{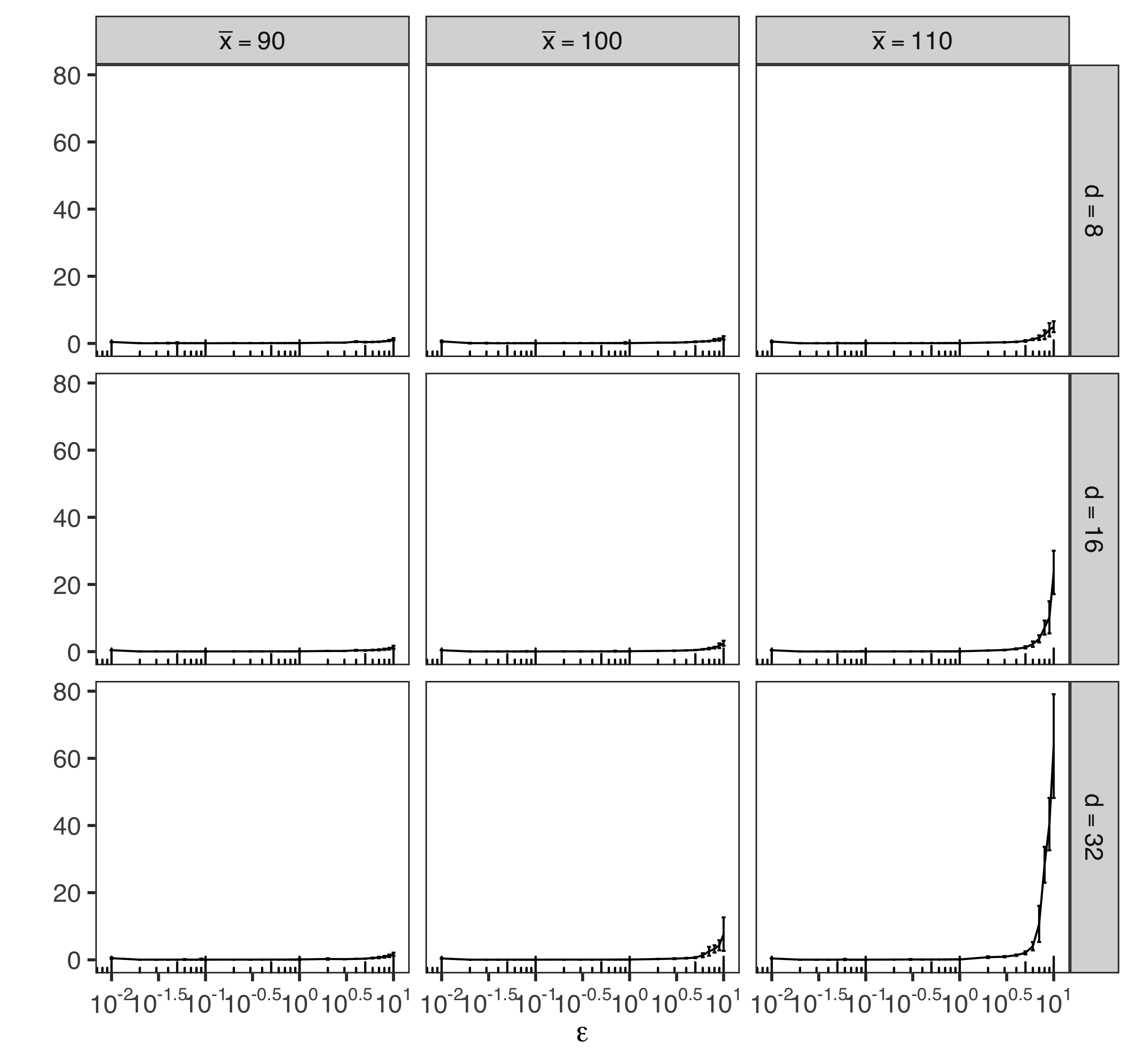}}{Barrier Option (Asymmetric) - Impact of Robustness Parameter on Computation Time.\label{fig:barrier_epsilon_asymmetric_time}}
{ Each plot shows the computation times from solving the {\color{black}heuristic~\eqref{prob:h_bar}} constructed from training datasets of  size $N = 10^3$. The problem parameters are the same as those shown in Table~\ref{tbl:barrier_asymmetric}. }
\end{figure}

\null
\vfill

\clearpage
}

\end{APPENDICES}
\end{document}